\documentclass[11pt,reqno,a4paper]{amsart}
\usepackage{amsmath,amsthm,verbatim,amscd,amssymb,setspace,enumitem,hyperref}
\usepackage{exscale,color}

\usepackage{enumitem}
\usepackage{stackengine,amsmath,lipsum}
\newcommand\addlefttext[3][-1]{%
  \def\stacktype{L}\def\useanchorwidth{T}%
  \ifdim#1pt<0pt\relax%
    \def\stackalignment{c}%
    \stackon[0pt]{$\displaystyle#3$}{\makebox[\textwidth][l]{#2}}%
  \else%
    \def\stackalignment{r}%
    \stackon[0pt]{$\displaystyle#3$}{\makebox[#1\textwidth][l]{#2}}%
  \fi%
}

\newtheorem{innercustomthm}{Theorem}
\newenvironment{customthm}[1]
  {\renewcommand\theinnercustomthm{#1}\innercustomthm}
  {\endinnercustomthm}

\newcommand{\N}{\mathbb{N}}

\newcommand{\R}{\mathbb{R}}

\newcommand{\C}{\mathbb{C}}

\renewcommand{\Re}{\operatorname{Re}}

\newcounter{mtheorem}
\newtheorem{mtheorem}[mtheorem]{Theorem}

\newcommand{\ess}{\operatorname{ess}}
\newcommand{\dis}{\operatorname{dis}}

\newcommand{\supp}{\operatorname{supp}}

\newcommand{\p}{\partial}
\newcommand{\norm}[1]{\Vert #1 \Vert}

\newcommand{\Ric}{\operatorname{Ric}}

\newcommand{\Rm}{\operatorname{Rm}}

\providecommand{\norm}[1]{\lVert#1\rVert}

\def\tr{\operatorname{tr}}

\def\inj{\operatorname{inj}}

\def\Div{\operatorname{div}}

\def\RR{\operatorname{R}}

\newtheoremstyle{fancy}{}{}{\itshape}{}{\textbf\bgroup}{.\egroup}{ }{}
\newtheoremstyle{fancy2}{}{}{\rm}{}{\textbf\bgroup}{.\egroup}{ }{}

\theoremstyle{fancy}
\newtheorem{theorem}{Theorem}[section]
\newtheorem{lemma}[theorem]{Lemma}
\newtheorem{corollary}[theorem]{Corollary}
\newtheorem{prop}[theorem]{Proposition}

\theoremstyle{fancy2}
\newtheorem{definition}[theorem]{Definition}

\newtheorem{remark}[theorem]{Remark}

\newtheorem{claim}[theorem]{Claim}

\setlist{leftmargin=*}

\numberwithin{equation}{section}

\begin{document}
\title{Steady gradient K\"ahler-Ricci solitons on crepant resolutions of Calabi-Yau cones}

\date{\today}
\author{Ronan J.~Conlon}
\address{Department of Mathematics and Statistics, Florida International University, Miami, FL 33199, USA}
\email{rconlon@fiu.edu}
\author{Alix Deruelle}
\address{Institut de math\'ematiques de Jussieu, 4, place Jussieu, Boite Courrier 247 - 75252 Paris}
\email{alix.deruelle@imj-prg.fr}

\date{\today}

\begin{abstract}
We show that, up to the flow of the soliton vector field, there exists a unique complete
steady gradient K\"ahler-Ricci soliton in every K\"ahler class of an equivariant crepant resolution
of a Calabi-Yau cone converging at a polynomial rate to Cao's steady gradient K\"ahler-Ricci soliton on the cone.
\end{abstract}

\maketitle
\markboth{Ronan J.~Conlon and Alix Deruelle}{Steady gradient K\"ahler-Ricci solitons on crepant resolutions of Calabi-Yau cones}

\tableofcontents

\section{Introduction}

\subsection{Overview}
 A \emph{Ricci soliton} is a triple $(M,\,g,\,X)$, where $M$ is a Riemannian manifold endowed with a complete Riemannian metric $g$
and a complete vector field $X$ such that
\begin{equation}\label{soliton111}
\Ric(g)+\frac{\lambda}{2}g=\frac{1}{2}\mathcal{L}_{X}g
\end{equation}
for some $\lambda\in\{-1,\,0,\,1\}$. If $X=\nabla^{g} f$ for some smooth real-valued function $f$ on $M$,
then we say that $(M,\,g,\,X)$ is \emph{gradient}. In this case, the soliton equation \eqref{soliton111}
becomes $$\Ric(g)+\frac{\lambda}{2} g=\operatorname{Hess}(f).$$

If $g$ is complete and K\"ahler with K\"ahler form $\omega$, then we say that $(M,\,g,\,X)$ is a \emph{K\"ahler-Ricci soliton} if
the vector field $X$ is complete and real holomorphic and the pair $(g,\,X)$ satisfies the equation
 \begin{equation}\label{soliton13}
\Ric(g)+\lambda g=\frac{1}{2}\mathcal{L}_{X}g
\end{equation}
for $\lambda$ as above. If $g$ is a K\"ahler-Ricci soliton and if $X=\nabla^{g} f$ for some smooth real-valued function $f$ on $M$,
then we say that $(M,\,g,\,X)$ is \emph{gradient}. In this case, the soliton equation \eqref{soliton13} may be rewritten as
\begin{equation*}
\rho_{\omega}+\lambda\omega=i\partial\bar{\partial}f,
\end{equation*}
where $\rho_{\omega}$ is the Ricci form of $\omega$.

For Ricci and K\"ahler-Ricci solitons $(M,\,g,\,X)$, the vector field $X$ is called the
\emph{soliton vector field}. Its completeness is guaranteed by the completeness of $g$
\cite{Zhang-Com-Ricci-Sol}. If the soliton is gradient, then
the smooth real-valued function $f$ satisfying $X=\nabla^g f$ is called the \emph{soliton potential}. It is unique up to addition of a constant.
Finally, Ricci and K\"ahler-Ricci solitons are called \emph{steady} if $\lambda=0$, \emph{expanding}
if $\lambda=1$, and \emph{shrinking} if $\lambda=-1$ in \eqref{soliton111} and \eqref{soliton13} respectively.

The study of Ricci solitons and their classification is important in the context of Riemannian geometry. For example, they provide a
natural generalisation of Einstein manifolds and on certain Fano manifolds, shrinking K\"ahler-Ricci solitons are known to exist where
there are obstructions to the existence of a K\"ahler-Einstein metric \cite{zhuu}. Also, to each soliton, one may associate a self-similar solution of the Ricci flow
\cite[Lemma 2.4]{Chowchow} which are candidates for singularity models of the flow. The difference in
normalisations between \eqref{soliton111} and \eqref{soliton13} reflects the difference between the constants preceding the Ricci term
in the Ricci flow and in the K\"ahler-Ricci flow respectively when one takes this dynamic point of view.

In this article, we are concerned with the existence and uniqueness of complete steady gradient K\"ahler-Ricci solitons on crepant resolutions of Calabi-Yau cones. Such solitons which are not Ricci-flat are necessarily non-compact \cite{Ivey}. Examples include Hamilton's cigar soliton \cite{hammy}
on $\mathbb{C}$ which was generalised by Cao \cite{Cao-KR-sol} to $\mathbb{C}^{n}$ and $K_{\mathbb{P}^{n}}$. Further generalisations were
then obtained by Dancer-Wang \cite{Wang}, Yang \cite{Boo}, and more recently by Sch\"afer \cite{schafer}. All examples mentioned thus far are highly symmetric and were constructed by
solving an ODE. In \cite{olivier}, Biquard-Macbeth implement a gluing method to
construct examples of complete steady gradient K\"ahler-Ricci solitons
in small K\"ahler classes of an equivariant crepant resolution of
$\mathbb{C}^{n}/\Gamma$, where $\Gamma$ is a finite subgroup of $SU(n)$ acting freely on $\mathbb{C}^{n}\setminus\{0\}$. Our main result is
the construction of a complete steady gradient K\"ahler-Ricci soliton in every K\"ahler class of a crepant
resolution of a Calabi-Yau cone, unique up to the flow of the soliton vector field, converging
at a polynomial rate to Cao's steady gradient K\"ahler-Ricci soliton on the cone.

\subsection{Main result}\label{honeybunch}

Cao's construction of a steady gradient K\"ahler-Ricci soliton on $\mathbb{C}^{n}$
\cite{Cao-KR-sol} allows for an ansatz to construct
 a one-parameter family of incomplete steady gradient K\"ahler-Ricci
  solitons $\tilde{\omega}_{a},\,a\geq0,$ on any Ricci-flat K\"ahler (or ``Calabi-Yau'') cone $(C_{0},\,g_{0})$.
  With this in mind, our main result can be stated as follows.

\begin{mtheorem}[Existence and uniqueness for steadies]\label{main}
Let $(C_{0},\,g_{0})$ be a Calabi-Yau cone of complex dimension $n\geq2$ with complex structure $J_{0}$, Calabi-Yau cone metric $g_{0}$,
radial function $r$, and trivial canonical bundle. Let $\pi:M\to C_{0}$ be a crepant resolution of $C_{0}$ with complex structure $J$ such that
the real holomorphic torus action on $C_{0}$ generated by $J_{0}r\partial_{r}$ extends to $M$ so that the holomorphic
vector field $2r\partial_{r}$ on $C_{0}$ lifts to a real holomorphic vector field $X=\pi^{*}(2r\partial_{r})$ on $M$.
Set $t:=\log(r^{2})$ and define the K\"ahler form $\hat{\omega}:=\frac{i}{2}\partial\bar{\partial}\left(\frac{nt^{2}}{2}\right)$
on $C_{0}$.

Then in every K\"ahler class $\mathfrak{k}$ of $M$, up to the flow of $X$, there exists a unique
complete steady gradient K\"ahler-Ricci soliton $\omega\in\mathfrak{k}$ with soliton vector field $X$ and with $\mathcal{L}_{JX}\omega=0$ such that
for all $\varepsilon\in(0,\,1)$, there exist constants $C(i,\,j,\,\varepsilon)>0$ such that
\begin{equation}\label{amazing}
|\widehat{\nabla}^{i}\mathcal{L}_{X}^{(j)}(\pi_{*}\omega-\hat{\omega})|_{\hat{g}}\leq
C(i,\,j,\,\varepsilon)t^{-\varepsilon-\frac{i}{2}-j}\quad\textrm{for all $i,\,j\in\mathbb{N}_{0}$},
\end{equation}
where $\hat{g}$ denotes the K\"ahler metric associated to $\hat{\omega}$ and
$\widehat{\nabla}$ is the corresponding Levi-Civita connection. More precisely,
for all $\varepsilon\in\left(0,\,\frac{1}{2}\right)$ and for all $a\geq0$, there exist constants $C(i,\,j,\,\varepsilon,\,a)>0$ such that
\begin{equation}\label{amazing2}
|\widehat{\nabla}^{i}\mathcal{L}^{(j)}_{X}(\pi_{*}\omega-\tilde{\omega}_{a}-\hat{\zeta})|_{\hat{g}}
\leq C(i,\,j,\,\varepsilon,\,a)t^{-2+\varepsilon-\frac{i}{2}-j}\quad\textrm{for all $i,\,j\in\mathbb{N}_{0}$,}
\end{equation}
where $\hat{\zeta}$ is a real $(1,\,1)$-form uniquely determined by $\mathfrak{k}$ that is invariant
under the flow of $X$ and $JX$, and $\tilde{\omega}_{a},\,a\geq0,$ denotes Cao's family of incomplete steady gradient K\"ahler-Ricci solitons on $C_{0}$.
If $\mathfrak{k}$ is compactly supported or if $n=2$, then for all $a\geq0$, there exists a smooth real-valued function $\varphi:M\to\mathbb{R}$ and
constants $C(i,\,j,\,a)>0$ such that
\begin{equation}\label{amazing3}
\omega-\tilde{\omega}_{a}=i\partial\bar\partial\varphi,\qquad\textrm{where}\qquad
|\widehat{\nabla}^{i}\mathcal{L}_{X}^{(j)}(t^{-n+1}e^{nt}\varphi)|_{\hat{g}}\leq C(i,\,j,\,a)t^{-\frac{i}{2}-j}\qquad\textrm{for all $i,\,j\in\mathbb{N}_{0}$}.
\end{equation}
\end{mtheorem}

A resolution for which the torus action on the cone extends to the resolution is called \emph{equivariant}.
Such a resolution of a complex cone always exists (cf.~\cite[Proposition 3.9.1]{kollar}). However, for Calabi-Yau cones,
it may not necessarily be crepant. On the other hand, if a crepant resolution of a Calabi-Yau cone is unique, then it is necessarily equivariant; apply the proof of \cite[Lemma 2.13]{conlondez} to see this. Moreover, the steady solitons of Theorem \ref{main} display so-called ``cigar-paraboloid'' asymptotics. Most notably, the volume of a ball of radius $R$ in $M$ grows at rate $O(R^{\frac{1}{2}\dim_{\mathbb{R}}M})$ and the curvature decays linearly.
Finally, Cao's steady gradient K\"ahler-Ricci soliton $\tilde{\omega}_{a}$ on $C_{0}$ also converges to $\hat{\omega}$ at infinity.
The precise asymptotics may be found in Proposition \ref{coro-asy-Cao-met}. Together with
\eqref{amazing2}, they yield the following more refined asymptotics:
\begin{equation*}
|\widehat{\nabla}^{i}\mathcal{L}_{X}^{(j)}(\pi_{*}\omega-\hat{\omega})|_{\hat{g}}\leq
\begin{cases}
C(i,\,j)t^{-1-\frac{i}{2}-j}\log(t) & \textrm{if $j=0$},\\
C(i,\,j)t^{-1-\frac{i}{2}-j} & \text{if $j\geq1$}.
\end{cases}
\end{equation*}

\subsection{The $C^{0}$-estimate}
The problem of constructing a Calabi-Yau metric on a complete K\"ahler manifold $(M,\,\tau)$ of complex dimension $n$ with trivial canonical bundle and K\"ahler form $\tau$ can be reformulated in terms of solving the complex Monge-Amp\`ere equation
\begin{equation}\label{beginn}
\tau_{\psi}^{n}=e^{F}\tau^{n},
\end{equation}
where $\psi$ is an unknown smooth real-valued function, $F$ is a smooth real-valued function (the given data), and
$\tau_{\psi}:=\tau+i\partial\bar{\partial}\psi$ is positive-definite. In the compact case, Yau \cite{Calabiconj} successfully
implemented the continuity method to obtain a solution of this equation. The same strategy has also borne much fruit in the
non-compact case; see for instance \cite{Hein, Nordstrom, Joyce, Tian3, Tian}.

A common feature shared by the proofs of the aforementioned results is the need to establish
a uniform $C^{0}$-bound along the continuity path, where the data $F$ in \eqref{beginn} has either compact support
or, when the underlying complex manifold $M$ is non-compact,
decays sufficiently fast at infinity. This bound has been achieved by
implementing a Nash-Moser iteration. It works as follows. One first establishes an a priori $L^{2}$-bound by considering the difference between the volume forms $\tau^{n}$ and $\tau_{\psi}^{n}$ in the following way:
\begin{equation}\label{eeend}
\int_{M}(1-e^F)\psi\,\tau^n=\int_M\psi\left(\tau^n-\tau_{\psi}^n\right)=\sum_{i\,=\,0}^{n-1}\int_Mi\partial\psi\wedge\bar{\partial}\psi\wedge\tau^i\wedge \tau_{\psi}^{n-1-i}\geq \int_Mi\partial\psi\wedge\bar{\partial}\psi\wedge\tau^{n-1}.
\end{equation}
The first equality uses \eqref{beginn}, the second is an integration by parts (which must be justified when
$M$ is non-compact), and the inequality is obtained by dropping all of the terms in the sum apart from the last
which does not depend on the unknown K\"ahler form $\tau_{\psi}$. A Poincar\'e-type inequality (in the compact case)
or a Sobolev inequality (in the case of maximal volume growth for example) then allows for uniform $L^{2}$-control on $\psi$ (or, if
$F$ is only decaying polynomially at infinity, $L^{p}$-control for $p$ sufficiently large). From this,
the a priori $C^{0}$-bound follows by applying \eqref{eeend} with $\psi|\psi|^{p-2},\,p\geq2,$
in place of $\psi$ and letting $p$ tend to $+\infty$.

Analogously, the problem of constructing a steady gradient K\"ahler-Ricci soliton on $M$ can be reformulated in terms of solving the complex Monge-Amp\`ere equation
\begin{equation}\label{beginnn}
\tau_{\psi}^{n}=e^{F-\frac{X}{2}\cdot\psi}\tau^{n},
\end{equation}
where $\tau,\,\tau_{\psi}$, and $F$ are as before, and $X$ is a given real holomorphic vector field on $M$. In our situation, the data $F$ is
polynomially decaying. As above, our approach to solve this PDE is to implement the continuity method. As such, we also require a uniform $C^{0}$-bound along the continuity path.
In contrast to \eqref{beginn} however, the added difficulty in this case arises from the fact that the unknown function $\psi$ appears on both sides of \eqref{beginnn}.
Nevertheless, to obtain an initial energy estimate we proceed as above, but rather than considering the difference between the volume forms $\tau^{n}$ and $\tau_{\psi}^{n}$, we consider the difference between the \emph{weighted} volume forms $e^{f}\tau^{n}$ and $e^{f_{\psi}}\tau_{\psi}^{n}$, where $-\tau\lrcorner JX=df$ and similarly for $f_{\psi}$, the reason being that the framework of metric measure spaces is the correct one to take to account for the existence of the vector field $X$ that was missing from \eqref{beginn}. Unsurprisingly though, the approach taken in the Calabi-Yau case no longer suffices; not only does the first equality of \eqref{eeend} fail if one replaces $\tau^{n}$ and $\tau_{\psi}^{n}$ with their respective weighted analogues, but we no longer have a global Sobolev inequality that would be needed in order to implement a Nash-Moser iteration.

To overcome these difficulties, we first observe that it is enough to solve \eqref{beginnn} with data $F$ compactly supported, the reason being
that via an application of the implicit function theorem, we are able to reduce to this case when $F$ is polynomially decaying. Once we have reduced to the compactly supported case, we apply the continuity method working in the space $\mathcal{M}^{\infty}_{X,\,\exp}(M)$ of smooth real-valued functions that decay exponentially with derivatives at infinity. Then to obtain an initial a priori energy bound along the continuity path of solutions, we introduce, in line with Tian-Zhu \cite{Tian-Zhu-I} and their work on the uniqueness of shrinking gradient K\"ahler-Ricci solitons on compact K\"ahler manifolds, the following functionals defined on $\mathcal{M}^{\infty}_{X,\,\exp}(M)$:
\begin{equation*}
I_{\tau,\,X}(\varphi):=\int_M\varphi\left(e^f\tau^n-e^{f_{\varphi}}\tau_{\varphi}^n\right)
\qquad\textrm{and}\qquad J_{\tau,\,X}(\varphi):=\int_0^1\int_M\dot{\varphi_s}\left(e^f\tau^n-e^{f_{\varphi_s}}\tau_{\varphi_s}^n\right)\wedge ds.
\end{equation*}
Here, $(\varphi_{t})_{0\,\leq\,t\,\leq\,1}$ is a $C^{1}$-path in $\mathcal{M}^{\infty}_{X,\,\exp}(M)$
from $\varphi_{0}=0$ to $\varphi_{1}=\varphi$,
$\tau_{\varphi_{s}}:=\tau+i\partial\bar{\partial}\varphi_{s}$ is positive-definite,
$-\tau_{\varphi_{s}}\lrcorner X=df_{\varphi_{s}}$, $\tau_{\varphi}:=\tau_{\varphi_{1}}$, and $f_{\varphi}:=f_{\varphi_{1}}$.
The exponential decay guarantees that both of these integrals converge; polynomial decay is not sufficient for this to be the case. The fact that $J_{\tau,\,X}$ is independent of the choice of path and so does indeed define a functional is due to Zhu \cite{Zhu-KRS-C1} in the compact case; we modify his proof accordingly to prove this fact for our situation. We also remark that Aubin \cite{Aub-red-Cas-Pos}, Bando-Mabuchi \cite{Ban-Mab-Uni} and Tian \cite[Chapter $6$]{Tian-Can-Met-Boo} have used these functionals (with $X=0$) to successfully study K\"ahler-Einstein Fano manifolds. By considering separately the continuity path of solutions $(\psi_{t})_{0\,\leq\,t\,\leq\,1}$ of \eqref{beginnn}
in $\mathcal{M}^{\infty}_{X,\,\exp}(M)$ and the linear path $(t\psi_{1})_{0\,\leq\,t\,\leq\,1}$, and making use of a suitable Poincar\'e inequality, we obtain an a priori weighted $L^{2}$-bound for \eqref{beginnn}. Our ability to use a Poincar\'e inequality is crucial for this part of the
argument to work and the existence of such an inequality follows from the existence of a steady gradient Ricci soliton at infinity that we have thanks to Cao's ansatz \cite{Cao-KR-sol} on the Calabi-Yau cone.

The next step involves improving the initial a priori weighted energy estimate to an actual a priori $C^{0}$-estimate.
Since the data $F$ in \eqref{beginnn} is assumed to be compactly supported,
and since $\psi$ in \eqref{beginnn} is a subsolution of the drift Laplacian, i.e.,
$\Delta_{h}\psi+X\cdot\psi\geq F$, where $\Delta_{h}$ is the Laplacian with respect to the K\"ahler metric
$h$ associated to $\tau$, we can assume without loss of generality that along the continuity path
of solutions $(\psi_{t})_{0\,\leq\,t\,\leq\,1}$, $\sup_{M}\psi_{t}$ is contained within the (compact) support of $F$ for each $t$.
A \emph{local} Nash-Moser iteration in a tubular neighbourhood of the support of $F$, using the already-established a priori weighted energy bound, then allows
for a uniform upper bound for $\sup_{M}\psi_{t}$. As for obtaining a uniform lower bound for $\inf_{M}\psi_{t}$,
this is much more delicate. To achieve such a bound, we adapt the proof of B\l{ocki} \cite{Blo-Uni-CY},
whose result comprises an alternative proof of Yau's a priori $C^{0}$-estimate for solutions of \eqref{beginn}
on a compact K\"ahler manifold with vanishing first Chern class \cite{Calabiconj}, making use of the weighted energy estimate in the process.
B\l{ocki}'s proof exploits the $L^{\infty}$-stability of the complex Monge-Amp\`ere operator and has its roots in the pluri-potential theory
developed by Bedford and Taylor \cite{Bed-Tay-Dir}. As B\l{ocki} explains in \cite{Blo-Uni-CY}, the estimate that he utilises is simpler than the finer estimates of Ko\l{odziej} \cite{Kol-Cx-MA}. All that is required is the maximum principle of Alexandrov \cite[Chapter $2$]{Han-Lin} for real Monge-Amp\`ere equations.

As is evident from the above discussion, the fact that $F$ is compactly supported plays a key role in passing from global energy estimates to
pointwise estimates. Assuming rapid decay of $F$ at infinity would not have been sufficient to reach the same conclusions.

\subsection{Outline of paper}
We begin in Section \ref{prelim} by recalling the basics of K\"ahler and Calabi-Yau cones, the relevant aspects of Sasakian geometry that we require,
as well as the definition of an equivariant resolution and a metric measure space. We also define and make some important notes on steady gradient Ricci and K\"ahler-Ricci solitons and introduce the Cao ansatz for the construction of a steady gradient K\"ahler-Ricci soliton on a Calabi-Yau cone. In Section \ref{section-Cao-soliton}, we analyse more precisely the asymptotics of Cao's steady gradient K\"ahler-Ricci soliton,
cumulating in the statement of Proposition \ref{coro-asy-Cao-met}. We follow this up in Section \ref{Section-App-Met} with the construction of a background metric in each K\"ahler class in Proposition \ref{lemma5.7} which we then use in Proposition \ref{equationsetup} to reformulate the problem of existence in terms of solving a scalar PDE, namely the complex Monge-Amp\`ere equation \eqref{beginnn}. Our background metric is asymptotic to Cao's steady gradient K\"ahler-Ricci soliton on the cone and hence serves as an approximate steady gradient K\"ahler-Ricci soliton.

From Section \ref{poincare} onwards, the content takes on a more analytic flavour. In Section \ref{poincare}, we show that the spectrum of the drift Laplacian of a Riemannian metric uniformly equivalent to a steady gradient Ricci soliton at infinity and with comparable potentials for the soliton vector field has a strictly positive lower bound. This observation, comprising Corollary \ref{tarea}, is essential in deriving the a priori weighted energy estimate for \eqref{beginnn} with compactly supported data. In Section \ref{section-small-def-exp}, we study the properties of the drift Laplacian of our background metric acting on exponentially weighted function spaces. More precisely, in Section \ref{section-fct-spa-exp}, we introduce exponentially weighted function spaces and in Section \ref{sec-Fredo-prop-exp} we show that the drift Laplacian of our background metric is an isomorphism between such spaces. This latter result is the content of Theorem \ref{iso-sch-Laplacian-exp}. Using it, we then prove Theorem \ref{Imp-Def-Kah-Ste-exp} that serves as the openness part of the continuity method.
The continuity method itself is outlined at the beginning of Section \ref{apriori} and is the approach that we take in order to solve \eqref{beginnn} with the caveat being however that the data of the PDE is compactly supported. As in \cite{siepmann}, the exponentially weighted function spaces introduced in Section \ref{section-fct-spa-exp} cater specifically for
the compactness of the support of the data. The closedness part of the continuity method involves a priori estimates and these make up the remainder of Section \ref{apriori}.

Our strategy for solving the complex Monge-Amp\`ere equation \eqref{beginnn} for polynomially decaying data involves an application of the implicit function theorem
to reduce to the (previously solved) case of compactly supported data. To achieve this simplification, we work in the space of
polynomially decaying functions. These are introduced in Section \ref{section-fct-spa-pol}. The invertibility of the drift Laplacian
 of our background metric between such spaces is demonstrated in Section \ref{linearpoly}, namely in Theorem \ref{iso-sch-Laplacian-pol}.
Via the implicit function theorem, this invertibility allows for local invertibility
of the complex Monge-Amp\`ere operator at a polynomially decaying solution. This forms the statement of Theorem \ref{Imp-Def-Kah-Ste} in Section \ref{invert-poly}. We also show
in Theorem \ref{surjectivity-drift-Laplacian} that the drift Laplacian is surjective onto the space of polynomially decaying functions.
This result, which forms the bulk of Section \ref{surject}, is used in the proof of the uniqueness part of Theorem \ref{main}.

In Section \ref{prooof}, we complete the proof of Theorem \ref{main}. The existence part is taken care of in Section \ref{existencee} with the
key step allowing us to reduce everything to compactly supported data the content of Proposition \ref{prop-sec-app-poly-dec-comp-supp-inf}.
The proof of this proposition requires Theorem \ref{Imp-Def-Kah-Ste} regarding the local invertibility of the complex Monge-Amp\`ere operator. The uniqueness part of Theorem \ref{main} is then proved in Section \ref{finished}. Finally, Appendix \ref{appendixa} gathers together
the various estimates with respect to $\hat{g}$, the asymptotic model metric of Cao's steady gradient K\"ahler-Ricci soliton on the cone,
that we use throughout.

\subsection{Acknowledgements}
The authors wish to thank Richard Bamler, Aziz El Kacimi-Alaoui, and Song Sun for useful discussions. Part of this work was carried out while the authors were visiting the Institut Henri Poincar\'e as part of the Research in Paris program in July 2019. They wish to thank the institute for their hospitality
and for the excellent working conditions provided.

The first author is supported by NSF grant DMS-1906466 and the second author is supported
by grant ANR-17-CE40-0034 of the French National Research Agency ANR (Project CCEM) and Fondation Louis D., Project ``Jeunes G\'eom\`etres''.

\section{Preliminaries}\label{prelim}

\subsection{Cones}

\subsubsection{Riemannian cones} For us, {the definition of a Riemannian cone will take the following form}.

\begin{definition}\label{cone}
Let $(S,\,g)$ be a compact connected Riemannian manifold. The \emph{Riemannian cone} $C_{0}$
 with link $S$ is defined to be $\R_{+} \times S$ with metric $g_{0} = dr^2 \oplus r^2 g_{S}$ up to isometry. The radius function $r$ is then characterized intrinsically as the distance from the apex in the metric completion.
\end{definition}

\subsubsection{K{\"a}hler cones}  Boyer-Galicki \cite{book:Boyer} is a comprehensive reference here.

\begin{definition}
A \emph{K{\"a}hler cone} $(C_{0},\,g_{0},\,J_{0})$ is a Riemannian cone $(C_{0},\,g_{0})$ such that $g_{0}$ is K{\"a}hler,
together with a choice of $g_{0}$-parallel complex structure $J_0$. This will in fact often be unique up to sign. We then have a K{\"a}hler
 form $\omega_0(X,Y) = g_{0}(J_0X,Y)$, and $\omega_0 = \frac{i}{2}\p\bar{\p} r^2$ with respect to $J_0$.
\end{definition}

The vector field $r\partial_{r}$ is real holomorphic and $\xi:=J_{0}r\partial_r$ is real holomorphic and Killing \cite[Appendix A]{Yau22}. This latter vector field is known as the \emph{Reeb vector field}. The closure of its flow in the isometry
group of the link of the cone generates the holomorphic isometric action of a real torus on the cone that
fixes the apex.

Every K\"ahler cone is affine algebraic.
\begin{theorem}\label{t:affine}
For every K{\"a}hler cone $(C_{0},g_0,J_0)$, the complex manifold $(C_{0},J_0)$ is isomorphic to the smooth part of a normal algebraic variety $V \subset \C^N$ with one singular point. In addition, $V$ can be taken to be invariant under a $\C^*$-action $(t, z_1,\ldots,z_N) \mapsto (t^{w_1}z_1,\ldots,t^{w_N}z_N)$ such that all of the weights $w_i $ are positive
integers.
\end{theorem}
\noindent This can be deduced from arguments written down by van Coevering in \cite[Section 3.1]{vanC4}.

We will frequently make use of the fact that every real-valued pluriharmonic function on a K\"ahler cone that is invariant under the
flow of the Reeb vector field is constant.
\begin{lemma}\label{pluri}
Let $(C_{0},\,g_{0},\,J_{0})$ be a K\"ahler cone with Reeb vector field $\xi$, let $\pi:M\to C_{0}$ be a resolution of $C_{0}$ with exceptional set $E$, and
let $K\subset M$ be a compact subset of $M$ containing $E$ such that $M\setminus K$ is connected.
\begin{enumerate}[label=\textnormal{(\roman{*})}, ref=(\roman{*})]
\item If $u:C_{0}\setminus\pi(K)\to\mathbb{R}$ is a smooth real-valued function defined on $C_{0}\setminus\pi(K)$
that is pluriharmonic (meaning that $\partial\bar{\partial}u=0$) and
 invariant under the flow of $\xi$, then $u$ is constant.
\item If $u:M\to\mathbb{R}$ is a smooth real-valued function defined on $M$ that is pluriharmonic on $M$ and
 invariant under the flow of $(d\pi)^{-1}(\xi)$ on $M\setminus K$, then $u$ is constant.
    \end{enumerate}
\end{lemma}

\begin{proof}
\begin{enumerate}[label=\textnormal{(\roman{*})}, ref=(\roman{*})]
\item Let $r$ denote the radial function of $g_{0}$. Then $r\partial_{r}$ is real holomorphic, and since
$\mathcal{L}_{J_{0}r\partial_{r}}u=\mathcal{L}_{\xi}u=0$, we see that
$$\bar{\partial}(r\partial_{r}u)=\partial\bar{\partial}u\lrcorner(r\partial_{r}-i\xi)=0,$$
that is, $r\partial_{r}u$ is holomorphic. As a real-valued holomorphic function, $r\partial_{r}u$ must be equal to a constant, $c_{0}$ say. Thus,
$$u=c_{0}\log r+c_{1}(x),$$ where $c_{1}(x)$ is a function that depends only on the link $(S,\,g_{S})$ of the cone
$(C_{0},\,g_{0})$. Now, $u$ being pluriharmonic implies that $\Delta_{g_{0}}u=0$, i.e.,
$$\frac{(2n - 2)}{r^{2}}c_{0}+\frac{1}{r^{2}}\Delta_{g_{S}}c_{1}(x)=0,$$
where $n$ is the complex dimension of $C_{0}$. Integrating this equation over $S$ then
shows that $c_{0}=0$ so that $u$ is constant, as claimed.
\item By part (i), we know that $u$ is constant on $M\setminus K$. The result now follows from the maximum principle.
\end{enumerate}
\end{proof}

\subsubsection{Sasaki manifolds and basic cohomology}\label{primitive}
A closed Riemannian manifold $(S,\,g_{S})$ of real dimension $2n-1$ is called \emph{Sasaki} if and only if its Riemannian cone $(C_{0},\,g_{0})$ is a K\"ahler
cone \cite{book:Boyer}, in which case we identify $(S,\,g_{S})$ with the level set $\{r= 1\}$ of $C_{0}$, $r$ here denoting the radial function of $g_{0}$. The restriction of the Reeb vector field to this
level set induces a non-zero
vector field $\xi\equiv J_{0}r\partial_{r}|_{\{r\,=\,1\}}$ on $S$, where
$J_{0}$ denotes the complex structure on the K\"ahler cone associated to $S$.
Let $\eta$ denote the $g_{S}$-dual one-form of
$\xi$ on $S$. This is a contact form and may be written
in terms of $r$ as $\eta=d^{c}\log(r)$ with $d^{c}:=i(\bar{\partial}-\partial)$. Moreover, $\eta$ induces a $g_{S}$-orthogonal decomposition $TS=\mathcal{D}
\oplus\langle\xi\rangle$, where $\mathcal{D}$ is the kernel of $\eta$ and
$\langle\xi\rangle$ is the $\mathbb{R}$-span of $\xi$ in $TS$, and correspondingly a decomposition of the metric $g_{S}$ as $g_{S}=\eta\otimes\eta+g^{T}$ with $g^{T}:=g_{S}|_{\mathcal{D}}$. The metric $g^{T}$ on $\mathcal{D}$
is invariant under the flow of $\xi$ and hence induces a K\"ahler metric
on the local leaf space of the \emph{Reeb foliation}, that is, the foliation of $S$ induced by the flow of $\xi$. We call $g^{T}$
the \emph{transverse metric}. Associated to $g^{T}$ are the
\emph{transverse K\"ahler form} $\omega^{T}$ and the \emph{transverse Ricci curvature} $\operatorname{Ric}(g^{T})$ defined on the local leaf space
in a natural way. The transverse K\"ahler form may be written as $\omega^{T}=\frac{1}{2}d\eta=\frac{1}{2}dd^{c}\log(r)$
which yields the following expression for the K\"ahler form $\omega_{0}$ of the cone metric $g_{0}$:
$$\omega_{0}=\frac{i}{2}\partial\bar{\partial}r^{2}=rdr\wedge\eta+r^{2}\omega^{T}.$$

Differential forms on $S$ that are invariant under the flow of $\xi$ and for which
the contraction with $\xi$ is zero are called \emph{basic}, as seen in the following definition.
\begin{definition}\label{defn:basic}
A $p$-form $\alpha$ on $S$ is called \emph{basic} if
\begin{equation*}
\xi\lrcorner\alpha=0\qquad\textrm{and}\qquad\mathcal{L}_{\xi}\alpha=0.
\end{equation*}
\end{definition}
\noindent We will denote the sheaf of sections of smooth basic $p$-forms and the sheaf of smooth basic functions on $M$ by $\Lambda_{B}^{p}$
and $C_{B}^{\infty}$ respectively. By considering a local foliated chart on $S$, one can always find a local basic orthonormal coframe $\{\theta_{i}\}_{i\,=\,1}^{2n-2}$ of $g^{T}$ such that $\theta_{i}\circ J_{0}=-\theta_{i+1}$.
With respect to this coframe, we may write $$g^{T}=\sum_{i\,=\,1}^{2n-2}\theta_{i}^{2}.$$
If $\alpha$ is a basic form, then one can check that $d\alpha$ is also basic. The
exterior derivative $d$ therefore restricts to a map $d_{B}:\Lambda_{B}^{p}\longrightarrow
\Lambda^{p+1}_{B}$ and we obtain a complex of sheaves
\begin{equation*}
0\longrightarrow C_{B}^{\infty}\stackrel{d_{B}}\longrightarrow \Lambda_{B}^{1}
\stackrel{d_{B}}\longrightarrow\Lambda_{B}^{2}\stackrel{d_{B}}\longrightarrow\ldots
\end{equation*}
Taking the cohomology of this complex, we get the \emph{basic de Rham cohomology groups} $H_{B}^{*}(S)$ of the Reeb foliation. Explicitly, they are given by
\begin{equation*}
H_{B}^{p}(S):=\frac{\ker(d_{B}:\Lambda_{B}^{p}(S)\longrightarrow\Lambda_{B}^{p+1}(S))}
{\operatorname{Im}(d_{B}:\Lambda_{B}^{p-1}(S)\longrightarrow\Lambda_{B}^{p}(S))}.
\end{equation*}
We write $[\alpha]_{B}$ for the cohomology class of a closed basic form $\alpha$. It is a result of El-Kacimi Alaoui et al. \cite{Hector:1985} that the basic de Rham cohomology groups are finite-dimensional.

Naturally associated to the transverse K\"ahler metric $g^{T}$ is the \emph{basic Hodge star operator}
$\bar{\star}:\Lambda_{B}^{r}(S)\longrightarrow\Lambda_{B}^{2n-r-2}(S)$ defined in terms of the Hodge star operator $\star$ of $g_{S}$ by
\begin{equation*}
\bar{\star}\alpha=\star(\eta\wedge\alpha)=(-1)^{r}\xi\lrcorner\star\alpha.
\end{equation*}
Notice that $\bar{\star}^{2}=(-1)^{r^{2}}\operatorname{Id}$
on $\Lambda_{B}^{r}(S)$. We also have a non-degenerate inner product $\langle\cdot\,,\,\cdot\rangle$ on $\Lambda^{r}(S)$ defined by
\begin{equation*}
\langle\cdot\,,\,\cdot\rangle:\Lambda^{r}(S)\times\Lambda^{r}(S)\longrightarrow\R,\qquad
\langle\alpha,\,\beta\rangle:=\int_{M}g_{S}(\alpha,\,\beta)\,{d\mu}_{g_{S}}=\int_{M}\alpha\wedge\star\beta,
\end{equation*}
where ${d\mu}_{g_{S}}$ denotes the volume form of $g_{S}$. This inner product restricts on basic forms to the expression
\begin{equation*}
\langle\cdot\,,\,\cdot\rangle_{B}:\Lambda_{B}^{r}(S)\times\Lambda_{B}^{r}(S)\longrightarrow\R,\qquad
\langle\alpha,\,\beta\rangle_{B}:=\int_{M}g(\alpha,\,\beta)\,\operatorname{dvol}_{g_{S}}=\int_{M}\alpha\wedge\bar{\star}\beta\wedge\eta,
\end{equation*}
and defines a non-degenerate inner product on $\Lambda^{r}_{B}(S)$. We define $L^{2}_{B}(S)$ to be the Hilbert space completion of $C^{\infty}_{B}(S)$ with respect to $\langle\cdot\,,\,\cdot\rangle_{B}$. With respect to this inner product, it is straightforward to check that the adjoint $\delta_{B}:\Lambda_{B}^{r}(S)\longrightarrow
\Lambda_{B}^{r-1}$ of $d_{B}:\Lambda_{B}^{r-1}(S)\longrightarrow\Lambda^{r}_{B}(S)$ is given by $\delta_{B}=-\bar{\star}\circ d_{B}\circ\bar{\star}$. We
then define the \emph{basic Laplacian} $\Delta_{B}$ acting on $\Lambda^{*}_{B}(S)$ by
\begin{equation*}
\Delta_{B}:=d_{B}\delta_{B}+\delta_{B}d_{B}.
\end{equation*}
This differential operator is self-adjoint with respect to $\langle\cdot\,,\,\cdot\rangle_{B}$ and the
kernel of its action on the space of basic $r$-forms is defined to be the space of \emph{basic harmonic $r$-forms}.
In analogy with the Hodge theorem on compact manifolds, there exists a transverse Hodge theorem \cite{hector:1986} for
Sasaki manifolds, which in particular states that each basic cohomology class has a unique basic harmonic representative.

Next, let $p_{S}:C_{0}\simeq\mathbb{R}_{+}\times S\to\{r=1\}\simeq S$ denote the natural projection. Then
we say that a complex-valued basic differential form $\alpha$ on $S$ is of \emph{type $(p,\,q)$} if and only if $p_{S}^{*}\alpha$
is a differential form of type $(p,\,q)$ on $C_{0}$. The sheaf of sections of
such forms on $S$ we denote by $\Lambda_{B}^{p,\,q}$. As in the complex case, there is a splitting
\begin{equation*}
\Lambda_{B}^{r}\otimes\C=\bigoplus_{p+q\,=\,r}\Lambda_{B}^{p,\,q}.
\end{equation*}
We define operators
\begin{equation*}
\p_{B}:\Lambda_{B}^{p,\,q}\longrightarrow\Lambda_{B}^{p+1,\,q}\qquad\textrm{and}\qquad
\bar{\p}_{B}:\Lambda_{B}^{p,\,q}\longrightarrow\Lambda_{B}^{p,\,q+1}
\end{equation*}
by $\p_{B}=\Pi^{p+1,\,q}\circ d_{B}$ and $\bar{\p}_{B}=\Pi^{p,\,q+1}\circ d_{B}$ respectively, where $$\Pi^{r,\,s}:\Lambda_{B}^{r+s}\otimes\C\longrightarrow\Lambda_{B}^{r,\,s}$$
denotes the projection map and where we consider the complex linear extension of $d_{B}$, i.e., $$d_{B}:\Lambda_{B}^{r}\otimes\C\longrightarrow\Lambda_{B}^{r+1}\otimes\C.$$
In analogy with the complex world, we have the following \emph{basic Dolbeault complex}
\begin{equation*}
0\longrightarrow\Lambda_{B}^{p,\,0}\stackrel{\bar{\p}_{B}}\longrightarrow\Lambda_{B}^{p,\,1}\stackrel{\bar{\p}_{B}}\longrightarrow\ldots
\stackrel{\bar{\p}_{B}}\longrightarrow\Lambda^{p,\,n}_{B}\longrightarrow 0,
\end{equation*}
together with the \emph{basic Dolbeault cohomology groups}
\begin{equation*}
H_{B}^{p,\,q}(S):=\frac{\ker(\bar{\p}_{B}:\Lambda_{B}^{p,\,q}(S)\longrightarrow\Lambda_{B}^{p,\,q+1}(S))}{\operatorname{Im}
(\bar{\p}_{B}:\Lambda_{B}^{p,\,q-1}(S)\longrightarrow\Lambda_{B}^{p,\,q}(S))}.
\end{equation*}
These are finite-dimensional \cite{ElKacimi}, and so we define the \emph{basic Hodge numbers} $h^{p,\,q}_{B}(S)$ by
\begin{equation*}
h^{p,\,q}_{B}(S):=\operatorname{dim}H_{B}^{p,\,q}(S).
\end{equation*}

For $k\leq n-1$, we say that a basic $k$-form $\alpha$ is \emph{primitive} if
it lies in the kernel of the adjoint $\Lambda$ of
the map $L:\alpha\mapsto\alpha\wedge d\eta$ with respect to $\langle\cdot\,,\,\cdot\rangle_{B}$. This is equivalent to saying that $\alpha\wedge(d\eta)^{n-k}=0$. The notion of a primitive basic form works equally well for basic $(p,\,q)$-forms. Indeed, we extend $\Lambda$ complex linearly to $$\Lambda_{B}^{*}(S)\otimes\C=\bigoplus_{p\,\geq\,0}\,
\bigoplus_{r\,+\,s\,=\,p}\Lambda_{B}^{r,\,s}S$$ and define a basic $(p,\,q)$-form to be primitive if and only if it lies in the kernel of this linear extension. Using the fact that the commutator $[\Lambda,\,\Delta_{B}]$ vanishes \cite[Lemma 7.2.7]{book:Boyer},
we see that $\Lambda$ maps basic harmonic forms to basic harmonic forms. In particular, from the basic harmonic representation theory of $H^{r}_{B}(S)$,
we find that the map $\Lambda$ descends to a well-defined map on these spaces.
We then define the \emph{$r$th-basic primitive cohomology group} $H_{B}^{r}(S)_{\operatorname{p}}$
as the kernel of the induced map
\begin{equation*}
\begin{split}
\Lambda:H^{r}_{B}(S)\longrightarrow H^{r-2}_{B}(S),\qquad[\beta]_{B}&\longmapsto[\Lambda\beta]_{B},
\end{split}
\end{equation*}
which may equivalently be realised as $$H_{B}^{r}(S)_{\operatorname{p}}=\{[\alpha]_{B}\in H_{B}^{r}(S)\,|\,[L^{n-r}\alpha]_{B}=0\}\quad\textrm{for $r\leq n-1$}.$$
In a similar manner, $\Lambda$ induces a map
\begin{equation*}
\Lambda:H^{p,\,q}_{B}(S)\longrightarrow H^{p-1,\,q-1}_{B}(S),\qquad[\beta]_{B}\longmapsto[\Lambda\beta]_{B},
\end{equation*}
the kernel of which we define as the \emph{$(p,\,q)$th-basic
primitive Dolbeault cohomology group} $H_{B}^{p,\,q}(S)_{\operatorname{p}}$. \cite[Corollary 7.2.10]{book:Boyer} then asserts that
\begin{equation}\label{poof}
H^{r}_{B}(S)_{\operatorname{p}}\otimes\C
=\bigoplus_{p+q\,=\,r}H_{B}^{p,\,q}(S)_{\operatorname{p}}.
\end{equation}
Notice that, as a subset of $H^{r}_{B}(S)$, each element of $H^{r}_{B}(S)_{\operatorname{p}}$
admits a unique basic harmonic representative. Since $[\Lambda,\,\Delta_{B}]=0$,
this representative is necessarily primitive at every point of $S$. By the next proposition, it is therefore harmonic.
This allows us to identify the de Rham cohomology groups with the basic primitive cohomology groups of a Sasaki manifold in a natural way.
\begin{prop}[{\cite[Proposition 7.4.13]{book:Boyer}}]\label{fgh}
Let $(S,\,g_{S})$ be a compact Sasaki manifold of dimension $2n-1$ and let $p$ be an integer satisfying $1\leq p\leq n-1$. Then a $p$-form is harmonic if and only if it is primitive and basic harmonic. Thus, for each $r\geq 0$, the basic primitive cohomology group $H^{r}_{B}(S)_{\operatorname{p}}$ can be naturally identified with the de Rham cohomology group $H^{r}(S,\,\R)$.
\end{prop}

\subsubsection{Calabi-Yau cones}
The particular type of K\"ahler cone that concerns us is the following.
\begin{definition}\label{d:cycone}
We say that $(C_{0},\,g_{0},\,J_0,\,\Omega_0)$ is a \emph{Calabi-Yau cone} if
\begin{enumerate}[label=\textnormal{(\roman{*})}, ref=(\roman{*})]
\item $(C_{0},\,g_{0},\,J_0)$ is a Ricci-flat K\"ahler cone of complex dimension $n$,
\item the canonical bundle $K_{C_{0}}$ of $C_{0}$ with respect to $J_0$ is trivial, and
\item $\Omega_{0}$ is a nowhere vanishing section of $K_{C_{0}}$ with $\omega_0^n = i^{n^2}\Omega_0 \wedge \bar{\Omega}_0$.
\end{enumerate}
\end{definition}

The link of a Calabi-Yau cone is a \emph{Sasaki-Einstein manifold}, an example of
a ``positive'' Sasaki manifold. Such manifolds enjoy the following vanishing property.
\begin{prop}[{\cite[Lemma 5.3]{goto}}]\label{vanishing}
The basic $(p,\,0)$-Hodge numbers of a positive Sasaki manifold
vanish for $p>0$.
\end{prop}

\subsubsection{Crepant resolutions}
A Calabi-Yau cone may be desingularised by a crepant resolution (if one exists).
\begin{definition}
Let $C_{0}$ be a complex space with an isolated normal singularity $o\in C_{0}$ and suppose that the complex manifold $C_{0}\setminus\{o\}$ has trivial canonical bundle. A \emph{crepant resolution} of $C_{0}$ is a pair $(M,\,\pi)$ comprising a smooth complex manifold $M$ with trivial canonical bundle
together with a proper map $\pi:M\longrightarrow C_{0}$ such that the restriction map
$$\pi|_{M\setminus\pi^{-1}(\{o\})}:M\setminus\pi^{-1}(\{o\})\longrightarrow C_{0}\setminus\{o\}$$ is a biholomorphism.
We call the set $E:=\pi^{-1}(\{o\})$ the \emph{exceptional set} of the resolution.
\end{definition}

We have the following vanishing result for such a resolution.
\begin{lemma}[{\cite[Lemma 5.5]{goto}}]\label{321}
Let $C_{0}$ be an affine variety of complex dimension $n\geq2$ with a normal
isolated singularity $o$ and with $K_{C_{0}\setminus\{o\}}$ trivial and let $\pi:M\to C_{0}$
be a crepant resolution of $C_{0}$. Then
\begin{equation}\label{eqn:vanishing}
H^{1}(M,\,\mathcal{O}_{M})=0.
\end{equation}
Furthermore, let $M_{0}$ denote the complement $M\setminus E$, where $E$ is the exceptional set of the resolution. If $n\geq 3$, then it also holds true that
\begin{equation}\label{lenny}
H^{1}(M_{0},\,\mathcal{O}_{M_{0}})=0.
\end{equation}
\end{lemma}

\begin{proof}
First observe from Takegoshi's generalisation of the Grauert-Riemenschneider vanishing theorem \cite[Theorem I]{Grauertriemannschneider} that
\begin{equation*}
R^{q}\pi_{*}K_{M}=0\quad\textrm{for $q>0$}.
\end{equation*}
Since $K_{M}$ is trivial, it follows that $R^{q}\pi_{*}\mathcal{O}_{M}=0$ for $q>0$ as well.

As for the cohomology of the sheaves $R^{q}\pi_{*}\mathcal{O}_{M}$, we know from Oka's coherence theorem and Grauert's direct image theorem that
for $q\geq0$ they are coherent analytic sheaves on $C_{0}$. We also know that $C_{0}$, as a closed analytic subspace of the Stein manifold $\C^{n}$, is itself an example of a Stein space. As a result, Cartan's Theorem B applies, from which we deduce that
$$H^{p}(C_{0},\,R^{q}\pi_{*}\mathcal{O}_{M})=0 \qquad\textrm{for all $p\geq 1$ and $q\geq 0$}.$$

Consider next the Leray spectral sequence \cite[Theorem 4.17.1, p.201]{Godement}
\begin{equation*}
E_{2}^{p,\,q}:=H^{p}(C_{0},\,R^{q}\pi_{*}\mathcal{O}_{M})\Rightarrow H^{p+q}(M,\,\mathcal{O}_{M})
\end{equation*}
and form its exact sequence of terms of low degree \cite[Theorem 4.5.1, p.82]{Godement}
\begin{equation*}
\begin{split}
0\longrightarrow H^{1}(C_{0},\,\pi_{*}\mathcal{O}_{M})\longrightarrow H^{1}(M,\,\mathcal{O}_{M}&)
\longrightarrow H^{0}(C_{0},\,R^{1}\pi_{*}\mathcal{O}_{M})\longrightarrow H^{2}(C_{0},\,\pi_{*}\mathcal{O}_{M})\longrightarrow H^{2}(M,\,\mathcal{O}_{M}).
\end{split}
\end{equation*}
Since $R^{1}\pi_{*}\mathcal{O}_{M}=0$ and $H^{1}(C_{0},\,\pi_{*}\mathcal{O}_{M})=0$, we see that $H^{1}(M,\,\mathcal{O}_{M})=0$.
The vanishing \eqref{eqn:vanishing} now follows.

Finally, let $H_{E}^{i}(M,\,\mathcal{O}_{M})$ denote the cohomology groups with supports in $E$ and coefficients in the structure sheaf $\mathcal{O}_{M}$.
Then we have a long exact sequence of cohomology with supports
\begin{equation}\label{les:supports}
\cdots\longrightarrow H_{E}^{i}(M,\,\mathcal{O}_{M})\longrightarrow H^{i}(M,\,\mathcal{O}_{M})
\longrightarrow H^{i}(M_{0},\,\mathcal{O}_{M}|_{M_{0}})\longrightarrow H_{E}^{i+1}(M,\,\mathcal{O}_{M})\longrightarrow\cdots
\end{equation}
In order to compute the cohomology groups $H_{E}^{i}(M,\,\mathcal{O}_{M})$, we utilise a version of Hartshorne's formal duality theorem \cite[Theorem A.1]{Greb:10}. This requires the hypothesis in the lemma that $C_{0}$ is normal. As $R^{q}\pi_{*}\mathcal{O}_{M}=0$ for $q>0$, the duality theorem
asserts that
\begin{equation*}
H_{E}^{n-q}(M,\,\mathcal{O}^{*}_{M}\otimes K_{M})=0\qquad\textrm{for $q>0$},
\end{equation*}
where $\mathcal{O}_{M}^{*}$ denotes the sheaf dual to $\mathcal{O}_{M}$. Triviality of $K_{M}$ then implies that this vanishing is equivalent to the vanishing of $H_{E}^{i}(M,\,\mathcal{O}_{M})$ for $i<n$,
and so from \eqref{les:supports} we deduce that
\begin{equation*}
H^{1}(M,\,\mathcal{O}_{M})
\cong H^{1}(M_{0},\,\mathcal{O}_{M}|_{M_{0}})\qquad\textrm{if $n\geq 3$.}
\end{equation*}
From the first part of the lemma we already know that $H^{1}(M,\,\mathcal{O}_{M})=0$. Hence for $n\geq 3$, we have that
\begin{equation*}
H^{1}(M_{0},\,\mathcal{O}_{M_{0}})=H^{1}(M_{0},\,\mathcal{O}_{M}|_{M_{0}})=0.
\end{equation*}
\end{proof}

As a result of the vanishing \eqref{eqn:vanishing}, we have a $\partial\bar{\partial}$-lemma on $M$.
\begin{lemma}[$\partial\bar{\partial}$-lemma]\label{deldelbar}
Let $C_{0}$ be an affine variety of complex dimension $n\geq2$ with a normal
isolated singularity $o$ and with $K_{C_{0}\setminus\{o\}}$ trivial, let $\pi:M\to C_{0}$
be a crepant resolution of $C_{0}$, and let $\alpha$ be an exact real $(1,\,1)$-form on $M$. Then there exists a smooth real-valued function $u$ on $M$ such that $\alpha=i\partial\bar{\partial}u$.
\end{lemma}

\begin{proof}
As a real exact two-form, there exists a real one-form $\beta$ on $M$ such that $\alpha=d\beta$. Write $\beta=\beta^{1,\,0}+\beta^{0,\,1}$ for some $\beta^{1,\,0}\in\Lambda^{1,\,0}M$ and $\beta^{0,\,1}\in\Lambda^{0,\,1}M$.
Then, as $\alpha=d\beta$ is real of type $(1,\,1)$, $\beta^{1,\,0}$ and $\beta^{0,\,1}$ must satisfy
$$d\beta=\bar{\p}\beta^{1,\,0}+\p\beta^{0,\,1},\qquad\p\beta^{1,\,0}=0,\qquad\textrm{and}\quad\bar{\p}\beta^{0,\,1}=0.$$
Since $C_{0}$ is a normal affine variety by
Theorem \ref{t:affine} and since $M$ has trivial canonical bundle,
the vanishing $H^{0,\,1}(M)=H^{1}(M,\,\mathcal{O}_{M})=0$ given by \eqref{eqn:vanishing}
together with the $\bar{\p}$-closedness of $\beta^{0,\,1}$ implies that
$\beta^{0,\,1}=\bar{\p}\phi$ for some smooth complex-valued function $\phi$ on $M$.
This yields a simplification of the above expression for $d\beta$, namely
\begin{equation*}
\begin{split}
d\beta=\bar{\p}\beta^{1,\,0}+\p\beta^{0,\,1}
=\overline{\p\beta^{0,\,1}}+\p\beta^{0,\,1}
=\bar{\p}\p\bar{\phi}+\p\bar{\p}\phi=\p\bar{\p}(\phi-\bar{\phi})=i\p\bar{\p}u,
\end{split}
\end{equation*}
where $u:=i(\bar{\phi}-\phi)$ is twice the imaginary part of $\phi$. Thus, $\alpha=d\beta=i\partial\bar{\partial}u$ with $u:M\to\mathbb{R}$ real-valued and smooth, as required.
\end{proof}

Crepant resolutions of Calabi-Yau cones have finite fundamental group.
\begin{lemma}\label{finite}
Let $\pi:M\to C_{0}$ be a crepant resolution of a Calabi-Yau cone $C_{0}$. Then $\pi_{1}(M)$ is finite. In particular, $H^{1}(M,\,\mathbb{R})=0$.
\end{lemma}

\begin{proof}
By \cite{goto, vanC2}, $M$ admits an asymptotically conical Calabi-Yau metric, in particular a complete Ricci-flat Riemannian metric of Euclidean volume growth.
Such a manifold has finite fundamental group by \cite{anderson, lii}.
\end{proof}

\subsubsection{Equivariant resolutions}
The real holomorphic torus action on a K\"ahler cone leads to the notion of an \emph{equivariant resolution}.
\begin{definition}\label{equivariantt}
Let $C_0$ be a K\"ahler cone with complex structure $J_{0}$, let $\pi:M\to C_0$ be a resolution of $C_0$, and let
$G$ be a Lie subgroup of the automorphism group of $(C_{0},\,J_{0})$ fixing the apex of $C_{0}$.
We say that $\pi:M\to C_0$ is an \emph{equivariant resolution with respect to $G$} if the action of $G$ on $C_{0}$ extends to a holomorphic action
on $M$ in such a way that $\pi(g\cdot x)=g\cdot\pi(x)$ for all $x\in M$ and $g\in G$.
\end{definition}

Such a resolution of a K\"ahler cone always exists; see \cite[Proposition 3.9.1]{kollar}.
If a Calabi-Yau cone admits a unique crepant resolution, then the crepant resolution is necessarily equivariant with respect to the real
holomorphic torus action on the cone induced by the Reeb vector field. This follows from
the proof of \cite[Lemma 2.13]{conlondez}.

\subsection{Steady Ricci solitons}
\subsubsection{Definition and properties}
The specific metrics that we are interested in are the following.
\begin{definition}
A \emph{steady Ricci soliton} is a triple $(M,\,g,\,X)$, where $M$ is a Riemannian manifold with a complete Riemannian metric $g$
and a complete vector field $X$ satisfying the equation
\begin{equation}\label{soliton2}
\Ric(g)=\frac{1}{2}\mathcal{L}_{X}g.
\end{equation}
If moreover $X=\nabla^{g} f$ for some smooth real-valued function $f$ on $M$, then we say that
the steady Ricci soliton $(M,\,g,\,X)$ is \emph{gradient}.
In this case, \eqref{soliton2} reduces to
\begin{equation*}
\Ric(g)=\operatorname{Hess}_{g}(f),
\end{equation*}
where $\operatorname{Hess}_{g}$ denotes the Hessian with respect to $g$.

A \emph{steady K\"ahler-Ricci soliton} is a triple $(M,\,g,\,X)$, where $M$
is a K\"ahler manifold, $X$ is a complete real holomorphic vector field on $M$, and $g$ is a complete K\"ahler metric on $M$ whose K\"ahler form $\omega$ satisfies
\begin{equation}\label{soliton}
\rho_{\omega}=\frac{1}{2}\mathcal{L}_{X}\omega,
\end{equation}
with $\rho_{\omega}$
denoting the Ricci form of $\omega$.
If moreover $X=\nabla^{g}f$ for some smooth real-valued function $f$ on $M$,
then we say that the steady K\"ahler-Ricci soliton $(M,\,g,\,X)$ is \emph{gradient}. In this case,
the defining equation of the soliton \eqref{soliton} may be rewritten as
\begin{equation*}
\rho_{\omega}=i\partial\bar{\partial}f.
\end{equation*}

For steady Ricci and K\"ahler-Ricci solitons $(M,\,g,\,X)$, the vector field $X$ is called the \emph{soliton vector field}.
When such solitons are gradient, the smooth real-valued function $f$ on $M$
satisfying $X=\nabla^{g}f$ is called the \emph{soliton potential}. \end{definition}

Two steady K\"ahler-Ricci solitons with the same soliton vector field that differ by $i\partial\bar{\partial}$ of a function
satisfy the following.
\begin{lemma}\label{cma}
Let $\omega_{1}$ and $\omega_{2}$ be two steady K\"ahler-Ricci solitons with the same soliton vector field $X$ on a complex manifold $M$ such that
$\omega_{2}=\omega_{1}+i\partial\bar{\partial}u$ for some smooth real-valued function $u$. Then
$$i\partial\bar{\partial}\left(\log\left(\frac{(\omega_{1}+i\partial\bar{\partial}u)^{n}}{\omega_{1}^{n}}\right)
+\frac{X}{2}\cdot u\right)=0.$$
\end{lemma}

\begin{proof}
With $\rho_{\omega_{i}}$ denoting the Ricci form of $\omega_{i}$, we have that
\begin{equation*}
\begin{split}
0&=\rho_{\omega_{2}}-\frac{1}{2}\mathcal{L}_{X}\omega_{2}\\
&=\rho_{\omega_{2}}-\rho_{\omega_{1}}+\rho_{\omega_{1}}-\frac{1}{2}\mathcal{L}_{X}\omega_{2}\\
&=-i\partial\bar{\partial}\log\left(\frac{\omega_{2}^{n}}{\omega_{1}^{n}}\right)+\rho_{\omega_{1}}-\frac{1}{2}\mathcal{L}_{X}\omega_{2}\\
&=-i\partial\bar{\partial}\log\left(\frac{(\omega_{1}+i\partial\bar{\partial}u)^{n}}{\omega_{1}^{n}}\right)
+\rho_{\omega_{1}}-\frac{1}{2}\mathcal{L}_{X}(\omega_{1}+i\partial\bar{\partial}u)\\
&=-i\partial\bar{\partial}\left(\log\left(\frac{(\omega_{1}+i\partial\bar{\partial}u)^{n}}{\omega_{1}^{n}}\right)
+\frac{X}{2}\cdot u\right)+\underbrace{\rho_{\omega_{1}}-\frac{1}{2}\mathcal{L}_{X}\omega_{1}}_{=\,0}\\
&=-i\partial\bar{\partial}\left(\log\left(\frac{(\omega_{1}+i\partial\bar{\partial}u)^{n}}{\omega_{1}^{n}}\right)
+\frac{X}{2}\cdot u\right).
\end{split}
\end{equation*}
\end{proof}

The next lemma collects together some well-known Ricci soliton identities concerning steady gradient K\"ahler-Ricci solitons that we require.
\begin{lemma}[Ricci soliton identities]\label{solitonid}
Let $(M,\,g,\,X)$ be a connected steady gradient K\"ahler-Ricci soliton with soliton vector field $X=\nabla^{g}f$ for a smooth real-valued function $f:M\to\mathbb{R}$. Then the trace and first order soliton identities are:
\begin{equation*}
\begin{split}
&\Delta_{\omega}f = \frac{\RR_{g}}{2},\\
&\nabla^g \RR_{g}+2\Ric(g)(X)=0, \\
&|\nabla^g f|_g^2+\RR_{g}=c(g),\\
\end{split}
\end{equation*}
where $\RR_{g}$ denotes the scalar curvature of $g$ and $|\nabla^g f|_g^2:=g^{ij}\partial_if\partial_{j}f$.
Here, $c(g)$ is a positive constant and represents the ``charge'' of the soliton at infinity.
\end{lemma}

\begin{proof}
These are proved as in \cite[Proof of Lemma 2.23]{conlondez}.
\end{proof}

Finally, to each complete steady gradient K\"ahler-Ricci soliton, one can associate an eternal solution of the
K\"ahler-Ricci flow that evolves via diffeomorphism. Indeed, if $(M,\,g,\,X)$ is a complete steady gradient K\"ahler-Ricci soliton
with K\"ahler form $\omega$ and soliton potential $f$, set $$\omega(s):=\varphi_{s}^{*}\omega,\qquad s\in(-\infty,\,\infty),$$
where $\varphi_{s}$ is the family of diffeomorphisms generated by the vector field
$-\frac{X}{2}$ with $\varphi_{0}=\operatorname{Id}$, i.e.,
\begin{equation*}
\frac{\partial\varphi_{s}}{\partial s}(x)=-\frac{\nabla^g f(\varphi_{s}(x))}{2},\qquad\varphi_{0}=\operatorname{Id}.
\end{equation*}
Then $\partial_s \omega(s)=-\rho_{\omega(s)}$ for $s\in(-\infty,\,\infty)$ with $\omega(0)=\omega$,
where $\rho_{\omega(s)}$ denotes the Ricci form of $\omega(s)$.

By pulling back the steady gradient K\"ahler-Ricci soliton $g$ by the family of diffeomorphisms generated by
the vector field $-X$ rather than $-\frac{X}{2}$, one obtains an eternal solution of the Ricci flow, i.e.,
a one-parameter family of Riemannian metrics $g(s),\,s\in(-\infty,\,\infty),$ with $g(0)=g$ and $\partial_{s}g(s)=-2\operatorname{Ric}(g(s))$.

\subsubsection{Steady Ricci solitons as metric measure spaces}\label{metric-measure}
A metric measure space is a Riemannian manifold endowed with a weighted volume form.
\begin{definition}
A \emph{metric measure space} is a triple $(M,\,g,\,e^{f}d\mu_{g})$, where $(M,\,g)$ is a complete Riemannian manifold with Riemannian metric $g$,
$d\mu_{g}$ is the volume form associated to $g$, and \linebreak $f:M\to\mathbb{R}$ is a real-valued $C^{1}$-function called the \emph{potential}.
\end{definition}
On such a space, we define the drift Laplacian $\Delta_{f}$ by
$$\Delta_{f}u:=\Delta u+g(\nabla^g f,\,\nabla^g u)$$
on smooth real-valued functions $u\in C^{\infty}(M)$. There is a natural $L^{2}$-inner product
$\langle\cdot\,,\,\cdot\rangle_{L^{2}(e^{f}d\mu_{g})}$ on the space $L^{2}(e^{f}d\mu_{g})$ of square-integrable smooth real-valued functions on $M$
with respect to the measure $e^{f}d\mu_g$ defined by $$\langle u,\,v\rangle_{L^{2}(e^{f}d\mu_{g})}:=\int_{M}uv\,e^{f}d\mu_{g},\qquad u,\,v\in L^{2}(e^{f}d\mu_{g}).$$
The operator $\Delta_{f}$ is symmetric with respect to $\langle\cdot\,,\,\cdot\rangle_{L^{2}(e^{f}d\mu_{g})}$. Indeed, simply observe that
$$\int_{M}(\Delta_{f}u)v\,e^{f}d\mu_{g}=-\int_{M}g(\nabla^{g}u,\,\nabla^{g}v)\,e^{f}d\mu_{g}=\int_{M}u(\Delta_{f}v)\,e^{f}d\mu_{g},\qquad u,\,v\in C^{\infty}_{0}(M).$$

A steady gradient Ricci soliton $(M,\,g,\,X)$ with $X=\nabla^g f$ for $f:M\to\mathbb{R}$ smooth naturally defines a metric measure space $(M,\,g,\,e^{f}d\mu_{g})$.

\subsubsection{Cao's steady gradient K\"ahler-Ricci soliton on a Calabi-Yau cone}
Given a Calabi-Yau
cone of complex dimension $n$, the Calabi-Yau cone metric induces a one-parameter family of incomplete steady gradient K\"ahler-Ricci solitons on the underlying complex space of the cone. This is seen by implementing Cao's ansatz \cite{Cao-KR-sol} which itself involves solving an ODE.
This we now explain. These solitons provide the model for the complete steady gradient K\"ahler-Ricci solitons that we construct.
\begin{prop}\label{caosoliton}
Let $(C_{0},\,g_{0},\,J_0,\,\Omega_{0})$ be a Calabi-Yau cone with radial function $r$ and set $r^{2}=e^{t}$.
Then for all $a\geq0$, there exists a steady gradient K\"ahler-Ricci soliton $\tilde{g}_{a}$ on $C_{0}$
with soliton vector field $2r\partial_{r}=4\partial_{t}$
whose K\"ahler form is given by $\tilde{\omega}_{a}=\frac{i}{2}\partial\bar{\partial}\Phi_{a}(t)$
for a smooth real-valued function $\Phi_{a}(t)$ on $C_{0}$ characterised by the fact that $0\leq a=\lim_{t\to-\infty}\Phi_{a}'(t)$
and whose soliton potential is given by $\varphi_{a}(t):=\Phi'_{a}(t)$.
\end{prop}

\begin{proof}
Recall that $d^{c}=i(\bar{\partial}-\partial)$ so that $dd^{c}=2i\partial\bar{\partial}$
and that $\eta=d^{c}\log(r)$ defines a contact form on the link of the cone $C_{0}$
with $\omega^{T}=\frac{1}{2}dd^{c}\log(r)$
the corresponding transverse K\"ahler metric. We assume an ansatz metric of the form
$\tilde{\omega}=\frac{i}{2}\partial\bar{\partial}\Phi(t)$ for $\Phi:C_{0}\to\mathbb{R}$ a (yet to be determined) smooth real-valued function on $C_{0}$ depending
only on $t$. We compute:
\begin{equation}\label{metric}
\begin{split}
\tilde{\omega}&=\frac{i}{2}\partial\bar{\partial}\Phi(t)=\frac{1}{4}dd^{c}\Phi(t)=\Phi'(t)\frac{1}{4}dd^{c}t+\Phi''(t)
\frac{dt}{2}\wedge\frac{d^{c}t}{2}\\
&=\varphi(t)\omega^{T}+\varphi'(t)\frac{dt}{2}\wedge\eta,
\end{split}
\end{equation}
where $\varphi(t):=\Phi'(t)$. Thus, in order for $\tilde{\omega}$ to define a K\"ahler metric, we require that
both $\varphi,\,\varphi'>0$. With $\tilde{\omega}$ written as above, it is easy to see that
$$\tilde{\omega}^{n}=n\varphi^{n-1}\varphi'\frac{dt}{2}\wedge\eta\wedge(\omega^{T})^{n-1}.$$
As for the K\"ahler form $\omega_{0}$ of $g_{0}$, we have that
$$\omega_{0}^{n}=(rdr\wedge\eta+r^{2}\omega^{T})^{n}=nr^{2n-2}(\omega^{T})^{n-1}\wedge rdr\wedge\eta=
ne^{nt}\frac{dt}{2}\wedge\eta\wedge(\omega^{T})^{n-1}.$$
These last two expressions allow us to write the Ricci form $\rho_{\tilde{\omega}}$ of $\tilde{\omega}$ as
\begin{equation}\label{finito}
\rho_{\tilde{\omega}}=-i\partial\bar{\partial}\log\left(\frac{\tilde{\omega}^{n}}{i^{n^{2}}\Omega_{0}\wedge\bar{\Omega}_{0}}\right)
=-i\partial\bar{\partial}\log\left(\frac{\tilde{\omega}^{n}}{\omega_{0}^{n}}\right)=-i\partial\bar{\partial}\log
\left(\varphi^{n-1}\varphi'e^{-nt}\right),
\end{equation}
where we have made use of the Calabi-Yau condition on the cone in the second equality.

Next, let $X=2r\partial_{r}=4\partial_{t}$. Then the pair $(\tilde{\omega},\,X)$ defines a steady K\"ahler-Ricci soliton with soliton vector field $X$
if and only if $$\rho_{\tilde{\omega}}=\frac{1}{2}\mathcal{L}_{X}\tilde{\omega},$$ that is, if
\begin{equation*}
-i\partial\bar{\partial}\log\left(\varphi^{n-1}\varphi'e^{-nt}\right)=i\partial\bar{\partial}\varphi.
\end{equation*}
Therefore it suffices that $$\varphi^{n-1}\varphi'e^{\varphi}=Ce^{nt}$$
for some constant $C>0$. By a translation in $t$ (which corresponds to a scaling in $r$), we may assume that $C=1$. Thus, we consider the ODE
\begin{equation}\label{ode}
\varphi^{n-1}\varphi'e^{\varphi}=e^{nt}.
\end{equation}

First separating variables in $\varphi$ and $t$ yields
\begin{equation*}
\varphi^{n-1}e^{\varphi}d\varphi=e^{nt}dt.
\end{equation*}
Next, after integrating both sides, we find that
\begin{equation}\label{soliton3}
F(\varphi)e^{\varphi}=\frac{e^{nt}}{n}+C
\end{equation}
for another constant $C$, where
\begin{equation}\label{ronando}
F(s):=\sum_{k=0}^{n-1}(-1)^{n-k-1}\frac{(n-1)!}{k!}s^{k}.
\end{equation}
One can verify that
$$(F(s)e^{s})'=s^{n-1}e^{s}.$$
Consequently, $F(s)e^{s}$ is strictly increasing for $s>0$, hence \eqref{soliton3} implicity defines $\varphi$
so long as $C\geq F(0)e^{0}$. Notice that for $s>0$, $$(F(s)e^{s}-s^{n-1}e^s)'=\underbrace{(F(s)e^{s})'-s^{n-1}e^{s}}_{=\,0}-(n-1)s^{n-2}e^{s}<0$$
so that $$F(s)e^{s}-s^{n-1}e^s<F(0)e^{0}=F(0)\qquad\textrm{for all $s>0$.}$$ As a result, we find that for all $s>s_{0}\geq0$,
\begin{equation*}
0<\int_{s_{0}}^{s}(F(x)e^{x})'\,dx=F(s)e^{s}-F(s_{0})e^{s_{0}}\leq F(s)e^{s}-F(0)e^{0}<s^{n-1}e^{s},
\end{equation*}
which leads to the inequality
\begin{equation}\label{bound}
0<F(s)e^{s}-F(s_{0})e^{s_{0}}<s^{n-1}e^{s}\qquad\textrm{for all $s>s_{0}\geq0$}.
\end{equation}
Finally, taking the limit in \eqref{soliton3} as $t\to-\infty$, we arrive at the fact that
\begin{equation*}
F\left(\lim_{t\to-\infty}\varphi(t)\right)e^{\lim_{t\to-\infty}\varphi(t)}=C.
\end{equation*}

Specify $$0\leq a:=\lim_{t\,\to\,-\infty}\varphi(t)$$ and
accordingly add a subscript $a$ to $\varphi$.
Set $C:=F(a)e^{a}$. Then $\varphi_{a}$ is defined implicitly by the equation
\begin{equation}\label{definition}
F(\varphi_{a}(t))e^{\varphi_{a}(t)}=\frac{e^{nt}}{n}+F(a)e^{a}.
\end{equation}
Since $F(s)e^{s}$ is strictly increasing for $s>0$, its inverse on $(0,\,\infty)$ is also strictly increasing, and so
we deduce from \eqref{definition} that $\varphi_{a}(t)$ is strictly increasing.
Consequently, $\varphi_{a}'(t)>0$ and
$\varphi_{a}(t)>\lim_{t\to-\infty}\varphi_{a}(t)=a\geq0$ for all $t$. Using this latter inequality, we see from
\eqref{bound} that $$\varphi_{a}^{n-1}e^{\varphi_{a}}>F(\varphi_{a})e^{\varphi_{a}}
-F(a)e^{a}=\frac{e^{nt}}{n},$$ hence \eqref{ode} implies that
\begin{equation}\label{proofread}
0<\varphi_{a}'(t)<n.
\end{equation}
It follows that both $\varphi_{a}$ and $\varphi_{a}'$ are strictly positive. We therefore conclude that for $a\geq0$,
$\Phi_{a}(t):=\int\varphi_{a}(t)\,dt$ defines the K\"ahler potential of a steady K\"ahler-Ricci soliton on $C_{0}$ with soliton vector field
$X=2\partial_{r}=4\partial_{t}$. To see that $\tilde{\omega}_{a}$ is gradient with $\varphi_{a}(t)$ serving as a soliton potential, just note that
$\rho_{\tilde{\omega}_{a}}=i\partial\bar{\partial}\varphi_{a}(t)$ which follows by combining \eqref{finito} and \eqref{ode},
and $\tilde{\omega}_{a}\lrcorner X=-d\varphi_{a}$ which is clear from \eqref{metric}.
\end{proof}

\begin{remark}
For clarity, we henceforth drop the subscript $a$ from $\tilde{\omega}_{a},\,\varphi_{a}(t),$ and $\Phi_{a}(t)$ in statements where, for each particular choice of $a\geq0$, the statement
holds true with $\tilde{\omega}$ replaced by $\tilde{\omega}_{a}$, etc.
\end{remark}

\section{Asymptotics of Cao's steady gradient K\"ahler-Ricci soliton}\label{section-Cao-soliton}
Let $(C_{0},\,g_{0})$ be a Calabi-Yau cone of complex dimension $n$ with radial function $r$. Then, as we have just seen, there exists a
steady gradient K\"ahler-Ricci soliton $\tilde{\omega}$ on $C_{0}$ of the form $\tilde{\omega}=\frac{i}{2}\partial\bar{\partial}\Phi(t)$,
where $r^{2}=e^{t}$. In this section, we study in depth the asymptotics of $\tilde{\omega}$. We begin first with
an analysis of the asymptotics of $\varphi(t):=\Phi'(t)$.
\begin{prop}\label{potential}
As $t\to+\infty$, we have
\begin{equation*}
\begin{split}
\varphi(t)&=nt-(n-1)\log t-n\log n+\frac{(n-1)^2}{n}\frac{\log t}{t}+\left(\frac{(n-1)}{n}+(n-1)\log n\right)\frac{1}{t}+O\left(\frac{(\log t)^{2}}{t^{2}}\right),\\
\varphi'(t)&=n-\frac{(n-1)}{t}+O\left(\frac{(\log t)^{2}}{t^2}\right),\qquad\qquad\varphi''(t)=O\left(\frac{(\log t)^{2}}{t^2}\right),
\qquad\qquad\varphi^{(3)}(t)=O\left(\frac{(\log t)^{2}}{t^2}\right).
\end{split}
\end{equation*}
\end{prop}

The asymptotics stated here on the second and third derivative of $\varphi(t)$ are not optimal, but nevertheless suffice for our purposes.

\begin{proof}[Proof of Proposition \ref{potential}]
Recalling from \eqref{proofread} that $\varphi$ is strictly increasing, we derive from \eqref{definition} that for all $t\gg0$,
\begin{equation}
\begin{split}\label{dev-phi-infinity}
\varphi-nt&=-\log(F(\varphi))-\log(n)+\log(1+O(e^{-nt}))\\
&=-\log(\varphi^{n-1}(1+O(\varphi^{-1})))-\log(n)+\log(1+O(e^{-nt}))\\
&=-(n-1)\log(\varphi)-\log(n)+O(\varphi^{-1})+O(e^{-nt})\\
&=-(n-1)\log(nt)-(n-1)\log\left(\frac{\varphi}{nt}\right)-\log(n)+O(\varphi^{-1})+O(e^{-nt}),\\
&=-(n-1)\log(t)-n\log(n)+O(\varphi^{-1})+O(e^{-nt})-(n-1)\log\left(\frac{\varphi}{nt}\right),\\
\end{split}
\end{equation}
where we have used \eqref{ronando} in the second equality. Again, since $\varphi(t)$ is increasing, we see from the third line above that
 $\varphi(t)-nt<0$ for $t$ sufficiently large. Thus, as $\varphi(t)>a$ for all $t$ where $0\leq a:=\lim_{t\to-\infty}\varphi(t)$, we have that
\begin{equation}\label{teeth}
0\leq\frac{a}{nt}<\frac{\varphi(t)}{nt}<1.
\end{equation}
Moreover, from \eqref{dev-phi-infinity}, we see that for $t$ sufficiently large,
\begin{equation*}
\frac{(n-1)}{nt}\log\left(\frac{\varphi}{nt}\right)+\frac{\varphi}{nt}-1=-\frac{(n-1)}{n}\frac{\log(t)}{t}-\frac{\log(n)}{t}+O(t^{-1}\varphi^{-1})
+O(t^{-1}e^{-nt})\geq-C\left(\frac{\log(t)}{t}\right)
\end{equation*}
for some $C>0$. Observing that $\log(x)\leq x$ for $x>0$ then yields the lower bound
$$-C\left(\frac{\log(t)}{t}\right)\leq\frac{(n-1)}{nt}\log\left(\frac{\varphi}{nt}\right)+\frac{\varphi}{nt}-1\leq
\frac{(n-1)}{nt}\left(\frac{\varphi}{nt}\right)+\frac{\varphi}{nt}-1=
\frac{\varphi}{nt}\left(1+\frac{(n-1)}{nt}\right)-1$$
so that
\begin{equation}\label{tooth}
\frac{\varphi}{nt}\geq 1-C\left(\frac{\log(t)}{t}\right)
\end{equation}
for $t$ sufficiently large. Combining \eqref{teeth} and \eqref{tooth}, we arrive at the fact that
\begin{equation*}
-C\left(\frac{\log(t)}{t}\right)\leq\frac{\varphi}{nt}-1<0
\end{equation*}
for $t$ sufficiently large, i.e.,
\begin{equation}\label{lollypop}
\frac{\varphi}{nt}=1+O\left(\frac{\log(t)}{t}\right).
\end{equation}
This in turn implies that $\varphi^{-1}=O(t^{-1})$. Using this and
plugging \eqref{lollypop} back into \eqref{dev-phi-infinity}, we then find that
\begin{equation}\label{dev-phi-order-3}
\varphi=nt-(n-1)\log(t)-n\log(n)+O\left(\frac{\log(t)}{t}\right).
\end{equation}
In particular, we deduce that
\begin{equation}\label{rocky}
\varphi^{-1}=\frac{1}{nt}+O\left(\frac{\log(t)}{t^{2}}\right).
\end{equation}

One can develop further the asymptotic expansion of $\varphi$ by plugging \eqref{dev-phi-order-3} into the last line of \eqref{dev-phi-infinity}. This results in
the expansion
\begin{eqnarray*}
\varphi&=&nt-(n-1)\log t-n\log n-(n-1)\log\left(1-\frac{(n-1)}{n}\frac{\log t}{t}-\frac{\log n}{t}+O\left(\frac{\log t}{t^2}\right)\right)+O(t^{-1})\\
&=&nt -(n-1)\log t-n\log n+\frac{(n-1)^2}{n}\frac{\log t}{t}+O(t^{-1}).
\end{eqnarray*}
Unfortunately this does not suffice to obtain a sharp first order term in the expansion of $\varphi'(t)$. We need to analyse the expansion of $O(\varphi^{-1})$ more carefully.
To this end, recall the definition of $F(\varphi)$ from \eqref{ronando}. We have that
\begin{equation*}
F(s)=s^{n-1}-(n-1)s^{n-2}+O(s^{n-3}),
\end{equation*}
 so that
 \begin{equation*}
 F(\varphi)=\varphi^{n-1}-(n-1)\varphi^{n-2}+O(\varphi^{n-3})=\varphi^{n-1}(1-(n-1)\varphi^{-1}+O(\varphi^{-2})).
 \end{equation*}
Plugging this into the first line of $\eqref{dev-phi-infinity}$ then leads to the expansion
\begin{equation*}
\begin{split}
\varphi-nt&=-\log(F(\varphi))-\log(n)+\log(1+O(e^{-nt}))\\
&=-(n-1)\log \varphi-\log(n)+(n-1)\varphi^{-1}+O(\varphi^{-2})+O(e^{-nt})\\
&=-(n-1)\log\left(\frac{\varphi}{nt}\right)-(n-1)\log(nt)-\log(n)+(n-1)\varphi^{-1}+O(\varphi^{-2})+O(e^{-nt})\\
&=-(n-1)\log\left(\frac{\varphi}{nt}\right)-(n-1)\log t-n\log n+(n-1)\varphi^{-1}+O(\varphi^{-2})+O(e^{-nt})\\
&=-(n-1)\log t-n\log(n)-(n-1)\log\left(1-\frac{(n-1)}{n}\frac{\log(t)}{t}-\frac{\log(n)}{t}+O\left(\frac{\log t}{t^{2}}\right)\right)+\\
&\qquad+(n-1)\varphi^{-1}+O(\varphi^{-2})+O(e^{-nt})\\
&=-(n-1)\log t-n\log(n)+\frac{(n-1)^2}{n}\frac{\log t}{t}+\frac{(n-1)\log(n)}{t}+O\left(\frac{(\log t)^{2}}{t^{2}}\right)+\\
&\qquad+(n-1)\varphi^{-1}+O(\varphi^{-2})+O(e^{-nt})\\
&=-(n-1)\log t-n\log n+\frac{(n-1)^2}{n}\frac{\log t}{t}+\left(\frac{(n-1)}{n}+(n-1)\log n\right)\frac{1}{t}+O\left(\frac{(\log t)^{2}}{t^{2}}\right),
\end{split}
\end{equation*}
where we have used \eqref{dev-phi-order-3} in the fifth equality and \eqref{rocky} in the final equality.
This yields the desired expansion of $\varphi$.

As for $\varphi'(t)$, making use of the above expansion of $\varphi$, we see from \eqref{ode} that
\begin{equation*}
\begin{split}
\log\varphi'&=nt-\varphi-(n-1)\log\varphi\\
&=(n-1)\log(t)+n\log(n)-\frac{(n-1)^2}{n}\frac{\log t}{t}-\left(\frac{(n-1)}{n}+(n-1)\log n\right)\frac{1}{t}-(n-1)\log\left(\varphi\right)\\
&\qquad+O\left(\frac{(\log t)^{2}}{t^{2}}\right)\\
&=(n-1)\log(t)+n\log(n)-\frac{(n-1)^2}{n}\frac{\log t}{t}-\left(\frac{(n-1)}{n}+(n-1)\log n\right)\frac{1}{t}-(n-1)\log\left(\frac{\varphi}{nt}\right)\\
&\qquad-(n-1)\log(nt)+O\left(\frac{(\log t)^{2}}{t^{2}}\right)\\
&=(n-1)\log(t)+n\log(n)-\frac{(n-1)^2}{n}\frac{\log t}{t}-\left(\frac{(n-1)}{n}+(n-1)\log n\right)\frac{1}{t}\\
&\qquad-(n-1)\log\left(1-\frac{(n-1)}{n}\frac{\log(t)}{t}-\frac{\log(n)}{t}+O\left(\frac{\log t}{t^{2}}\right)\right)-(n-1)\log(nt)+O\left(\frac{(\log t)^{2}}{t^{2}}\right)\\
&=\log(n)-\frac{(n-1)^2}{n}\frac{\log t}{t}-\left(\frac{(n-1)}{n}+(n-1)\log n\right)\frac{1}{t}\\
&\qquad-(n-1)\log\left(1-\frac{(n-1)}{n}\frac{\log(t)}{t}-\frac{\log(n)}{t}+O\left(\frac{\log t}{t^{2}}\right)\right)+O\left(\frac{(\log t)^{2}}{t^{2}}\right)\\
&=\log(n)-\frac{(n-1)^2}{n}\frac{\log t}{t}-\left(\frac{(n-1)}{n}+(n-1)\log n\right)\frac{1}{t}+\frac{(n-1)^{2}}{n}\frac{\log(t)}{t}+\frac{(n-1)\log(n)}{t}\\
&\qquad+O\left(\frac{(\log t)^{2}}{t^{2}}\right)\\
&=\log n-\frac{(n-1)}{n}\frac{1}{t}+O\left(\frac{(\log t)^{2}}{t^2}\right)
\end{split}
\end{equation*}
so that
\begin{equation*}
\varphi'(t)=n-\frac{(n-1)}{t}+O\left(\frac{(\log t)^{2}}{t^2}\right),
\end{equation*}
as claimed.

For $\varphi''(t)$, we deduce from \eqref{ode} and our previous expansions of $\varphi(t)$ and $\varphi'(t)$ that
\begin{equation*}
\begin{split}
\varphi''(t)&=\varphi'(t)\left(n-\varphi'(t)-\frac{(n-1)}{\varphi(t)}\varphi'(t)\right)\\
&=O\left(\frac{(\log t)^{2}}{t^2}\right)
\end{split}
\end{equation*}
as desired.

Finally, for $\varphi^{(3)}(t)$, we have that
\begin{equation*}
\begin{split}
\varphi^{(3)}&=\varphi''\left(n-\varphi'-\frac{(n-1)}{\varphi}\varphi'\right)-\varphi'\left(\varphi''+(n-1)\left(\frac{\varphi''}{\varphi}-\left(\frac{\varphi'}{\varphi}\right)^{2}\right)\right)\\
&=O\left(\frac{(\log t)^{2}}{t^2}\right).
\end{split}
\end{equation*}
\end{proof}

The previous proposition allows us to derive the following precise asymptotics of Cao's steady gradient K\"ahler-Ricci soliton.
\begin{prop}\label{coro-asy-Cao-met}
Let $(C_{0},\,g_{0})$ be a Calabi-Yau cone of complex dimension $n\geq2$ with complex structure $J_{0}$, radial function $r$, and transverse metric $g^{T}$,
and set $\eta=d^{c}\log(r)$ and  $r^{2}=e^{t}$. Let $\tilde{g}$ denote Cao's steady gradient K\"ahler-Ricci soliton on $C_{0}$ and set
$\hat{g}:=n\left(\frac{1}{4}dt^{2}+\eta^{2}+tg^{T}\right)$.
Then
\begin{equation}\label{beautiful}
\begin{split}
|\widehat{\nabla}^{i}(\tilde{g}-\hat{g})|_{\hat{g}}&=O\left(t^{-1-\frac{i}{2}}\log(t)\right)\qquad\textrm{for all $i\geq0$},
\end{split}
\end{equation}
and
\begin{equation}\label{even-more-beautiful}
\begin{split}
|\widehat{\nabla}^{i}\mathcal{L}_X^{(j)}(\tilde{g}-\hat{g})|_{\hat{g}}=O(t^{-1-\frac{i}{2}-j})
\qquad\textrm{for all $i\geq0$ and $j\geq1$}.\\
\end{split}
\end{equation}
In particular, for all $\varepsilon\in(0,\,1)$, there exist constants $C(i,\,j,\,\varepsilon)>0$ such that
\begin{equation*}
\begin{split}
|\widehat{\nabla}^{i}\mathcal{L}_X^{(j)}(\tilde{g}-\hat{g})|_{\hat{g}}\leq C(i,\,j,\,\varepsilon)t^{-\varepsilon-\frac{i}{2}-j}
\qquad\textrm{for all $i,\,j\geq0$}.\\
\end{split}
\end{equation*}
\end{prop}

\begin{proof}
We prove this proposition through several claims. We begin with the following initial estimate.
\begin{claim}\label{asymptotics}
$|\tilde{g}-\hat{g}|_{\hat{g}}=O(t^{-1}\log(t)).$
\end{claim}

\begin{proof}
We see from \eqref{metric} that
$\tilde{g}=\varphi(t)g^{T}+\varphi'(t)(\frac{1}{4}dt^{2}+\eta^{2})$. Thus, we can write
\begin{equation}\label{difference}
\tilde{g}-\hat{g}=(\varphi(t)-nt)g^{T}+(\varphi'(t)-n)\left(\frac{1}{4}dt^{2}+\eta^{2}\right).
\end{equation}
In light of Proposition \ref{potential}, we then have that
\begin{equation*}
\begin{split}
|\tilde{g}-\hat{g}|_{\hat{g}}&\leq |\varphi-nt||g^{T}|_{\hat{g}}
+|\varphi'-n|\left|\frac{1}{4}dt^{2}+\eta^{2}\right|_{\hat{g}}\\
&\leq C\left(t^{-1}\log(t)+t^{-1}\right)\\
&\leq Ct^{-1}\log(t),
\end{split}
\end{equation*}
as claimed.
\end{proof}

We next estimate the norm of the curvature tensor $\operatorname{Rm}(\tilde{g})$ of $\tilde{g}$.

\begin{claim}\label{lemma-est-cao-metr-curv}
$|\operatorname{Rm}(\tilde{g})|_{\tilde{g}}=O(t^{-1})$.
\end{claim}

\begin{proof}
As above, we can write $\tilde{g}=\varphi(t)g^{T}+\varphi'(t)\left(\frac{dt^{2}}{4}+\eta^{2}\right)$.
Let $\theta_{1},\ldots,\theta_{2n-2}$ be a local basic orthonormal coframe
for $g^{T}$ with $\theta_{i}\circ J_{0}=-\theta_{i+1}$ for $i$ odd, and let
$(\omega_{ij})_{1\,\leq\, i,\,j\,\leq\,2n-2}$ be the matrix of connection one-forms
of $g^{T}$. Then $(\omega_{ij})$ solves the Cartan structure equations
\begin{equation*}
\left\{ \begin{array}{ll}
d{\theta}_{i}=\sum_{j\,=\,1}^{2n-2}{\omega}_{ji}\wedge\theta_{j}\\
{\omega}_{ij}+{\omega}_{ji}=0.
\end{array} \right.
\end{equation*}
The one-forms
$$\tilde{\theta}_{i}:=\sqrt{\varphi(t)}\theta_{i},\qquad\tilde{\theta}_{2n-1}:=\frac{\sqrt{\varphi'(t)}}{2}dt,\qquad
\tilde{\theta}_{2n}:=\eta\sqrt{\varphi'(t)},$$
where $1\leq i\leq 2n-2$, serve as a local orthonormal coframe of $\tilde{g}$. We first compute the matrix
of connection one-forms $(\tilde{\omega}_{ij})_{1\,\leq\, i,\,j\,\leq\, 2n}$ of $\tilde{g}$ with respect to this coframe. We have that
\begin{equation*}
\begin{split}
d\tilde{\theta}_{i}&=\sum_{j\,=\,1}^{2n-2}\omega_{ji}\wedge\tilde{\theta}_{j}+\frac{\sqrt{\varphi'(t)}}{\varphi(t)}\tilde{\theta}_{2n-1}\wedge\tilde{\theta}_{i},\qquad1\leq i\leq 2n-2,\\
d\tilde{\theta}_{2n-1}&=0,\qquad d\tilde{\theta}_{2n}=\frac{\sqrt{\varphi'(t)}}{\varphi(t)}\sum_{i=1 \atop \textrm{$i$ odd}}^{2n-2}(\tilde{\theta}_{i}\wedge\tilde{\theta}_{i+1}-\tilde{\theta}_{i+1}\wedge\tilde{\theta}_{i})
+\frac{\varphi''(t)}{(\varphi'(t))^{\frac{3}{2}}}\tilde{\theta}_{2n-1}\wedge\tilde{\theta}_{2n}.
\end{split}
\end{equation*}
The matrix $(\tilde{\omega}_{ij})$ is then given by
\begin{equation*}
\begin{split}
\tilde{\omega}_{ji}&=\left\{ \begin{array}{ll}
\omega_{ji}+\delta_{j,\,i+1}\frac{\varphi'(t)}{\varphi(t)}\eta,\qquad\textrm{$1\leq i\leq 2n-2$ odd,\qquad$1\leq j\leq 2n-2$},\\
\omega_{ji}-\delta_{j,\,i-1}\frac{\varphi'(t)}{\varphi(t)}\eta,\qquad\textrm{$1\leq i\leq 2n-2$ even,\qquad$1\leq j\leq 2n-2$},\\
\end{array} \right.\\
\tilde{\omega}_{2n-1,\,i}&=-\sqrt{\frac{\varphi'(t)}{\varphi(t)}}\theta_{i},\qquad 1\leq i\leq 2n-2,\\
\tilde{\omega}_{2n-1,\,2n}&=-\frac{\varphi''(t)}{\varphi'(t)}\eta,\\
\tilde{\omega}_{2n,\,i}&=\left\{ \begin{array}{ll}
\sqrt{\frac{\varphi'(t)}{\varphi(t)}}\theta_{i+1},\qquad\textrm{$1\leq i\leq 2n-2$ odd},\\
-\sqrt{\frac{\varphi'(t)}{\varphi(t)}}\theta_{i-1},\qquad\textrm{$1\leq i\leq 2n-2$ even}.\\
\end{array} \right.\\
\end{split}
\end{equation*}
Next, from Proposition \ref{potential} we derive that
\begin{equation*}
\frac{\varphi'(t)}{\varphi(t)}=O(t^{-1}),\qquad\left(\frac{\varphi'(t)}{\varphi(t)}\right)'=O(t^{-2}),
\qquad\frac{\varphi''(t)}{\varphi'(t)}=O(t^{-\frac{3}{2}}),\qquad\left(\frac{\varphi''(t)}{\varphi'(t)}\right)'=O(t^{-\frac{3}{2}}).\\
\end{equation*}
Thus, with respect to the metric $\hat{g}$, we have the asymptotics
\begin{equation*}
\begin{split}
\tilde{\omega}_{ji}&=O(t^{-\frac{1}{2}}),\qquad d\tilde{\omega}_{ji}=O(t^{-1}),\qquad 1\leq i,\,j\leq 2n-2,\\
\tilde{\omega}_{2n-1,\,i}&=O(t^{-1}),\qquad d\tilde{\omega}_{2n-1,\,i}=O(t^{-\frac{3}{2}}),\qquad 1\leq i\leq 2n-2,\\
\tilde{\omega}_{2n-1,\,2n}&=O(t^{-\frac{3}{2}}),\qquad d\tilde{\omega}_{2n-1,\,2n}=O(t^{-\frac{3}{2}}),\\
\tilde{\omega}_{2n,\,i}&=O(t^{-1}),\qquad d\tilde{\omega}_{2n,\,i}=O(t^{-\frac{3}{2}}),\qquad 1\leq i\leq 2n-2.\\
\end{split}
\end{equation*}
From the Cartan structure equations, it is clear that $|\operatorname{Rm}(\tilde{g})|_{\hat{g}}=O(t^{-1})$. Since $\hat{g}$
and $\tilde{g}$ are equivalent at infinity as a consequence of Claim \ref{asymptotics}, the assertion follows.
\end{proof}

Next employing Shi's derivative estimates, we estimate the norm of the derivatives of $\operatorname{Rm}(\tilde{g})$
with respect to $\tilde{g}$ and its Levi-Civita connection $\widetilde{\nabla}$.
\begin{claim}\label{claim-decay-rm-lie-der}
$|\widetilde{\nabla}^{k}\operatorname{Rm}(\tilde{g})|_{\tilde{g}}=O(t^{-1-\frac{k}{2}})$
for all $k\geq0$.
\end{claim}

\begin{proof}
Claim \ref{lemma-est-cao-metr-curv} asserts the existence of a positive constant $C$
such that $|\operatorname{Rm}(\tilde{g})|_{\tilde{g}}\leq Ct^{-1}$. Set $\lambda:=\frac{1}{2}\left(1+\sqrt{1+\frac{4}{C}}\right)>1$. We will prove the following statement.
$$(\star)\quad\textrm{For all $m\in\mathbb{N}$, $|\widetilde{\nabla}^{k}\operatorname{Rm}(\tilde{g})|_{\tilde{g}}
\leq \frac{C_{k}}{(\lambda^{m})^{1+\frac{k}{2}}}$ on $\{\lambda^{m}\leq t\leq\lambda^{m+1}\}$ for all $k\geq0$.}$$
Here $C_{k}$ is a positive constant independent of $m$.

To this end, we know that $|\operatorname{Rm}(\tilde{g})|_{\tilde{g}}\leq\frac{C}{\lambda^{m}}$
on the region $\{\lambda^{m}\leq t\leq\lambda^{m+2}\}$. Let $(\tilde{g}(s))_{s\,\in\,\mathbb{R}}$ denote the Ricci flow associated to
$\tilde{g}$ with $\tilde{g}(0)=\tilde{g}$. Then by Shi's derivative estimates \cite{shi} (see \cite[Theorem 5.3.2]{zhang} for the precise statement that
we use), there exist constants $\tilde{C}_{k}$ such that
$$|(\nabla^{\tilde{g}(s)})^{k}\operatorname{Rm}(\tilde{g}(s))|_{\tilde{g}(s)}\leq\frac{C\tilde{C}_{k}}{\lambda^{m}s^{\frac{k}{2}}}\qquad\textrm{
on $\{\lambda^{m}\leq t\leq\lambda^{m+2}\}$ for all $s\in\left(0,\,\frac{\lambda^{m}}{C}\right]$ and for all $k\geq0$}.$$ In particular, for $s_{0}=\frac{\lambda^{m}}{C}$, we find that
$$|(\nabla^{\tilde{g}(s_{0})})^{k}\operatorname{Rm}(\tilde{g}(s_{0}))|_{\tilde{g}(s_{0})}\leq\frac{C^{1+\frac{k}{2}}\tilde{C}_{k}}{(\lambda^{m})^{1+\frac{k}{2}}}
\qquad\textrm{on $\{\lambda^{m}\leq t\leq\lambda^{m+2}\}$ for all $k\geq0$}.$$ But since $\tilde{g}(s)$ is obtained from $\tilde{g}$
by flowing along the vector field $-X=-4\partial_{t}$ for time $s$, this last statement is equivalent to
$$|\widetilde{\nabla}^{k}\operatorname{Rm}(\tilde{g})|_{\tilde{g}}\leq\frac{C^{1+\frac{k}{2}}\tilde{C}_{k}}{(\lambda^{m})^{1+\frac{k}{2}}}
\qquad\textrm{on $\Big\{\lambda^{m}-\frac{4\lambda^{m}}{C}\leq t\leq\underbrace{\lambda^{m+2}-\frac{4\lambda^{m}}{C}}_{=\,\lambda^{m+1}}\Big\}$ for all $k\geq0$},$$
so that in particular,
$$|\widetilde{\nabla}^{k}\operatorname{Rm}(\tilde{g})|_{\tilde{g}}\leq\frac{C^{1+\frac{k}{2}}\tilde{C}_{k}}{(\lambda^{m})^{1+\frac{k}{2}}}
\qquad\textrm{on $\left\{\lambda^{m}\leq t\leq\lambda^{m+1}\right\}$ for all $k\geq0$}.$$
This establishes $(\star)$ with $C_{k}:=C^{1+\frac{k}{2}}\tilde{C}_{k}$.

Now for any $x\in M$ with $t(x)\geq\lambda$, there exists $N\in\mathbb{N}$ such that $\lambda^{N}\leq t(x)\leq \lambda^{N+1}$.
Then since $(\star)$ holds true, we see that
$$|\widetilde{\nabla}^{k}\operatorname{Rm}(\tilde{g})|_{\tilde{g}}(x)\leq \frac{C_{k}}{(\lambda^{N})^{1+{\frac{k}{2}}}}
\leq\frac{\lambda^{1+\frac{k}{2}}C_{k}}{(t(x))^{1+\frac{k}{2}}}\qquad\textrm{for all $k\geq0$,}$$
as desired.
\end{proof}

We also have the following estimates on the Lie derivatives of $\operatorname{Rm}(\tilde{g})$ along the soliton vector field $X$.

\begin{claim}\label{claim-decay-rm-lie-der2}
$|\mathcal{L}_X^{(k)}(\Rm(\tilde{g}))|_{\tilde{g}}=O(t^{-1-k})$ for all $k\geq0$.
\end{claim}

\begin{proof}
In order to show that $|\mathcal{L}_X^{(k)}(\operatorname{Rm}(\tilde{g}))|_{\tilde{g}}
=O(t^{-1-k})$ for all $k\geq0$, recall that the curvature operator $\Rm(\tilde{g}(s))$ satisfies the following evolution equation along the
Ricci flow $(\tilde{g}(s))_{s\,\in\,\R},\,\tilde{g}(0)=\tilde{g},$ associated to $\tilde{g}$:
\begin{equation}\label{bonita}
\partial_s\Rm(\tilde{g}(s))=\Delta_{\tilde{g}(s)}\Rm(\tilde{g}(s))+\Rm(\tilde{g}(s))\ast \Rm(\tilde{g}(s)).
\end{equation}
Since $\tilde{g}(s)$ is obtained from $\tilde{g}$ by flowing along the vector field $-X$, we know that
$\mathcal{L}_{X}\Rm(\tilde{g})=-\partial_s\Rm(\tilde{g}(s))|_{s\,=\,0}$.
Using Claim \ref{claim-decay-rm-lie-der} together with \eqref{bonita},
this yields the expected result for $k=1$, namely $|\mathcal{L}_X\Rm(\tilde{g})|_{\tilde{g}}
=O(t^{-2})$.

Next, note the following commutation formula for any tensor $T$ on $M$:
\begin{equation}
\begin{split}\label{comm-time-der-lap}
\left.\left([\partial_s,\Delta_{\tilde{g}(s)}]T\right)\right|_{s\,=\,0}&=\sum_{i\,=\,0}^2\widetilde{\nabla}^{2-i}\Ric(\tilde{g})\ast\widetilde{\nabla}^{i}T.
\end{split}
\end{equation}
This formula can be derived from \cite[Lemma 2.27]{Cho-Lu-Ni-I} using the definition of the rough Laplacian.
In particular, by \eqref{comm-time-der-lap}, if $k=2$, then
\begin{equation*}
\begin{split}
\mathcal{L}_{X}^{(2)}(\Rm(\tilde{g}))&=\mathcal{L}_{X}(-\Delta_{\tilde{g}}\Rm(\tilde{g}))+\mathcal{L}_{X}(\Rm(\tilde{g})\ast\Rm(\tilde{g}))\\
&=\partial_{s}(-\Delta_{\tilde{g}(s)}\Rm(\tilde{g}(s)))|_{s\,=\,0}+\mathcal{L}_{X}(\Rm(\tilde{g})\ast\Rm(\tilde{g}))\\
&=-\Delta_{\tilde{g}}(\mathcal{L}_{X}(\Rm(\tilde{g})))+\sum_{i\,=\,0}^2\widetilde{\nabla}^{2-i}\Rm(\tilde{g})\ast\widetilde{\nabla}^{i}\Rm(\tilde{g})+\Rm(\tilde{g})\ast\Rm(\tilde{g})\ast\Rm(\tilde{g})\\
&=\Delta^2_{\tilde{g}}\Rm(\tilde{g})+\Delta_{\tilde{g}}(\Rm(\tilde{g})\ast\Rm(\tilde{g}))+\sum_{i\,=\,0}^2\widetilde{\nabla}^{2-i}\Rm(\tilde{g})\ast\widetilde{\nabla}^{i}\Rm(\tilde{g})\\
&\qquad+\Rm(\tilde{g})\ast\Rm(\tilde{g})\ast\Rm(\tilde{g})\\
&=\Delta^2_{\tilde{g}}\Rm(\tilde{g})+\sum_{i\,=\,0}^2\widetilde{\nabla}^{2-i}\Rm(\tilde{g})\ast\widetilde{\nabla}^{i}\Rm(\tilde{g})+\Rm(\tilde{g})\ast\Rm(\tilde{g})\ast\Rm(\tilde{g}).
\end{split}
\end{equation*}
This implies that $|\mathcal{L}_X^{(2)}\Rm(\tilde{g})|_{\tilde{g}}
=O(t^{-3})$. The cases $k\geq 3$ can be proved similarly by induction.
\end{proof}

Using this, we can now estimate all of the derivatives of $\varphi$.
\begin{claim}\label{helloo}
\begin{equation*}
\begin{split}
\varphi(t)&=O(t),\\
\left(\varphi'(t)-n\right)^{(k)}&=O(t^{-1-k})\qquad\textrm{for all $k\geq 0$},\\
\varphi^{(k)}(t)&=O(t^{-k})\qquad\textrm{for all $k\geq2$}.\\
\end{split}
\end{equation*}
\end{claim}

\begin{proof}
The first estimate follows immediately from Proposition \ref{potential}.

As for the other estimates, we read from the third soliton identity (Lemma \ref{solitonid}) that for Cao's steady gradient K\"ahler-Ricci soliton $\tilde{g}$,
$$|X|_{\tilde{g}}^{2}+\RR_{\tilde{g}}=c(\tilde{g})$$
for some positive constant $c(\tilde{g})$. Since $\RR_{\tilde{g}}=O(t^{-1})$ as a consequence of Claim \ref{lemma-est-cao-metr-curv} and
since $X=4\partial_{t}$ so that $|X|_{\tilde{g}}^{2}=4n+O(t^{-1}\log(t))$ by Claim \ref{asymptotics}, we deduce
that $c(\tilde{g})=4n$. Also observe from \eqref{metric} that $|X|_{\tilde{g}}^{2}=4\varphi'(t)$. Thus, we may write
\begin{equation}\label{scalar}
4\varphi'(t)=4n-\RR_{\tilde{g}}.
\end{equation}
As a result, we see from Claim \ref{claim-decay-rm-lie-der2} that for all $k\geq0$,
$$|(\varphi'(t)-n)^{(k)}|\leq C|\mathcal{L}^{(k)}_{X}(\varphi'(t)-n)|\leq
C|\mathcal{L}_X^{(k)}\RR_{\tilde{g}}|\leq Ct^{-1-k}$$
and that for all $k\geq2$,
$$|\varphi^{(k)}(t)|\leq C|\mathcal{L}_X^{(k-1)}\RR_{\tilde{g}}|
\leq Ct^{-k}.$$
This yields the second and third estimates of the claim respectively.
\end{proof}

We next estimate the covariant derivatives of $\varphi^{(q)}(t)$ for all $q\geq1$.

\begin{claim}\label{virus}
\begin{equation*}
\begin{split}
|\widehat{\nabla}\varphi^{(q)}(t)|_{\hat{g}}&=O\left(t^{-q-1}\right)\qquad\textrm{for all $q\geq1$},\\
|\widehat{\nabla}^{l}\varphi^{(q)}(t)|_{\hat{g}}&=O\left(t^{-q-1-\frac{l}{2}}\right)\qquad\textrm{for all $l\geq2$ and $q\geq1$}.
\end{split}
\end{equation*}
\end{claim}

\begin{proof}
Recall from Proposition \ref{basis} that with the orthonormal coframe
$$\hat{\theta}_{i}:=\sqrt{nt}\theta_{i}\qquad\textrm{for $i=1,\ldots,2n-2$},\qquad\hat{\theta}_{2n-1}:=\frac{\sqrt{n}}{2}dt,\qquad\textrm{and}\qquad
\hat{\theta}_{2n}:=\eta\sqrt{n}$$
of $\hat{g}$, we have the estimates
\begin{equation*}
\begin{split}
|\widehat{\nabla}^{k}\hat{\theta}_{i}|_{\hat{g}}&=O\left(t^{-\frac{k}{2}}\right)\qquad\textrm{for $1\leq i\leq 2n-2$ and for all $k\geq0$},\\
|\widehat{\nabla}^{k}\hat{\theta}_{2n-1}|_{\hat{g}}&=O\left(t^{-1-\frac{(k-1)}{2}}\right)\qquad\textrm{for all $k\geq1$},\\
|\widehat{\nabla}^{k}\hat{\theta}_{2n}|_{\hat{g}}&=O\left(t^{-1-\frac{(k-1)}{2}}\right)\qquad\textrm{for all $k\geq1$}.
\end{split}
\end{equation*}
Using these facts together with Claim \ref{helloo}, we deduce that for all $q\geq1$,
$$|\widehat{\nabla}\varphi^{(q)}(t)|_{\hat{g}}\leq C|(\varphi'(t)-n)^{(q)}||\hat{\theta}_{2n-1}|_{\hat{g}}
\leq Ct^{-1-q},$$
which yields the first estimate of the claim. For all $q\geq1$ and $l\geq2$, we then derive that
\begin{equation*}
\begin{split}
|\widehat{\nabla}^{l}\varphi^{(q)}(t)|_{\hat{g}}
&\leq C\sum_{i\,=\,1}^{l}\Biggl(|\varphi^{(i+q)}(t)|\sum_{\substack{s_{p}\,\geq\,0 \\ \sum_{p\,=\,0}^{l-i} (p+1)s_{p}\,=\,l \\
\sum_{p\,=\,0}^{l-i}ps_{p}\,=\,l-i}}\prod_{0\,\leq\,p\,\leq\,l-i}|\nabla^{p}\hat{\theta}_{2n-1}|^{s_{p}}_{\hat{g}}\Biggl)\\
&\leq C\sum_{i\,=\,1}^{l}\Biggl(t^{-(i+q)}\sum_{\substack{s_{p}\,\geq\,0 \\ \sum_{p\,=\,0}^{l-i} (p+1)s_{p}\,=\,l \\
\sum_{p\,=\,0}^{l-i}ps_{p}\,=\,l-i}}\prod_{0\,\leq\,p\,\leq\,l-i}|\nabla^{p}\hat{\theta}_{2n-1}|^{s_{p}}_{\hat{g}}\Biggl)\\
&\leq C\sum_{i\,=\,1}^{l}\Biggl(t^{-i-q}\sum_{\substack{0\,\leq\,p\,\leq\,l-i \\
s_{p}\,\geq\,0 \\ \sum_{p\,=\,0}^{l-i}(p+1)s_{p}\,=\,l \\
\sum_{p\,=\,0}^{l-i}ps_{p}\,=\,l-i}}\underbrace{
t^{\sum_{p\,=\,1}^{l-i}\left(-1-\frac{(p-1)}{2}\right)s_{p}}}_{=\,O\left(t^{-\frac{1}{2}\sum_{p\,=\,1}^{l-i}s_{p}}
t^{-\frac{1}{2}\sum_{p\,=\,1}^{l-i}ps_{p}}\right)}
\Biggl)\\
&\leq C\sum_{i\,=\,1}^{l}\Biggl(t^{-i-q}\sum_{\substack{0\,\leq\,p\,\leq\,l-i \\ s_{p}\,\geq\,0 \\ \sum_{p\,=\,0}^{l-i} (p+1)s_{p}\,=\,l \\
\sum_{p\,=\,0}^{l-i}ps_{p}\,=\,l-i}}t^{-\frac{(i-s_{0})}{2}}t^{-\frac{(l-i)}{2}}
\Biggl)\\
&\leq Ct^{-\frac{l}{2}-q}
\sum_{i\,=\,1}^{l}\Biggl(t^{-i}\sum_{\substack{0\,\leq\,p\,\leq\,l-i \\ s_{p}\,\geq\,0 \\ \sum_{p\,=\,0}^{l-i} (p+1)s_{p}\,=\,l \\
\sum_{p\,=\,0}^{l-i}ps_{p}\,=\,l-i}}t^{\frac{s_{0}}{2}}\Biggl)\\
&\leq Ct^{-\frac{l}{2}-q}\Biggl(t^{-l}t^{\frac{l}{2}}+\sum_{i\,=\,1}^{l-1}\underbrace{t^{\frac{s_{0}}{2}-i}}_{=\,O\left(t^{\frac{(i-1)}{2}-i}\right)}\Biggl)\\
&\leq Ct^{-\frac{l}{2}-q}\Biggl(t^{-\frac{l}{2}}+
\sum_{i\,=\,1}^{l-1}t^{-\frac{i}{2}-\frac{1}{2}}\Biggl)\\
&\leq Ct^{-\frac{l}{2}-q}(t^{-\frac{l}{2}}+t^{-1})\\
&\leq Ct^{-q-\frac{l}{2}-1},
\end{split}
\end{equation*}
which is the second estimate of the claim.
\end{proof}

As for $\varphi(t)$ and its covariant derivatives, we have:
\begin{claim}
\begin{equation*}
\begin{split}
|\varphi(t)-nt|&=O\left(\log(t)\right),\\
|\widehat{\nabla}(\varphi(t)-nt)|_{\hat{g}}&=O\left(t^{-1}\right),\\
|\widehat{\nabla}^{l}(\varphi(t)-nt)|_{\hat{g}}&=O\left(t^{-1-\frac{l}{2}}\right)\qquad\textrm{for all $l\geq2$}.
\end{split}
\end{equation*}
\end{claim}

\begin{proof}
The first estimate follows immediately from Proposition \ref{potential}. As for the remaining estimates, we compute using Claim \ref{helloo} that
\begin{equation*}
\begin{split}
|\widehat{\nabla}(\varphi(t)-nt)|_{\hat{g}}&\leq C|\varphi'(t)-n||\hat{\theta}_{2n-1}|_{\hat{g}}\leq Ct^{-1},\\
|\widehat{\nabla}^{2}(\varphi(t)-nt)|_{\hat{g}}&\leq C\left(|\varphi''(t)|+|\varphi'(t)-n||\widehat{\nabla}\hat{\theta}_{2n-1}|_{\hat{g}}\right)\leq
Ct^{-2},\\
|\widehat{\nabla}^{3}(\varphi(t)-nt)|_{\hat{g}}&\leq C\left(|\varphi^{(3)}(t)|+|\varphi''(t)||\widehat{\nabla}\hat{\theta}_{2n-1}|_{\hat{g}}
+|\varphi'(t)-n||\widehat{\nabla}^{2}\hat{\theta}_{2n-1}|_{\hat{g}}
\right)\leq Ct^{-\frac{5}{2}},\\
\end{split}
\end{equation*}
and from Claims \ref{helloo} and \ref{virus} that for all $l\geq4$,
\begin{equation*}
\begin{split}
|\widehat{\nabla}^{l}(\varphi(t)-nt)|_{\hat{g}}&\leq C|\widehat{\nabla}^{l-1}((\varphi'(t)-n)\hat{\theta}_{2n-1})|_{\hat{g}}\\
&\leq C\sum_{k\,=\,0}^{l-1}|\widehat{\nabla}^{k}(\varphi'(t)-n)|_{\hat{g}}|\widehat{\nabla}^{l-1-k}\hat{\theta}_{2n-1}|_{\hat{g}}\\
&\leq C\Biggl(|\varphi'(t)-n||\widehat{\nabla}^{l-1}\hat{\theta}_{2n-1}|_{\hat{g}}
+|\widehat{\nabla}(\varphi'(t)-n)|_{\hat{g}}|\widehat{\nabla}^{l-2}\hat{\theta}_{2n-1}|_{\hat{g}}
+|\widehat{\nabla}^{l-1}(\varphi'(t)-n)|_{\hat{g}}\\
&\qquad+\sum_{k\,=\,2}^{l-2}|\widehat{\nabla}^{k}(\varphi'(t)-n)|_{\hat{g}}|\widehat{\nabla}^{l-1-k}\hat{\theta}_{2n-1}|_{\hat{g}}\Biggr)\\
&\leq C\Biggl(t^{-1}t^{-1-\frac{(l-2)}{2}}+t^{-2}t^{-1-\frac{(l-3)}{2}}+t^{-2-\frac{(l-1)}{2}}
+\sum_{k\,=\,2}^{l-2}t^{-2-\frac{k}{2}}t^{-1-\frac{(l-2-k)}{2}}\Biggl)\\
&\leq C\Biggl(t^{-1-\frac{l}{2}}+t^{-\frac{3}{2}-\frac{l}{2}}+t^{-2-\frac{l}{2}}\Biggl)\\
&\leq Ct^{-1-\frac{l}{2}}.
\end{split}
\end{equation*}
The claim now follows.
\end{proof}

We have already seen in Claim \ref{asymptotics} that \eqref{beautiful} holds true for $k=0$. We now show that it in fact holds
true for all $k\geq1$.
\begin{claim}
\begin{equation*}
\begin{split}
|\widehat{\nabla}^{k}(\tilde{g}-\hat{g})|_{\hat{g}}&=O\left(t^{-1-\frac{k}{2}}\log(t)\right)\qquad\textrm{for all $k\geq1$}.
\end{split}
\end{equation*}
\end{claim}

\begin{proof}
Recall the estimates on the local orthonormal coframe $\{\hat{\theta}_{1},\ldots,\hat{\theta}_{2n-2}\}$
of $ntg^{T}$ from Proposition \ref{basis}. From there, we also read that
\begin{equation*}
\begin{split}
|\widehat{\nabla}^{k}(t^{-1})|_{\hat{g}}&=O\left(t^{-1-k}\right)\qquad\textrm{for $k=0,\,1,\,2,$}\\
|\widehat{\nabla}^{k}(t^{-1})|_{\hat{g}}&=O\left(t^{-2-\frac{k}{2}}\right)\qquad\textrm{for all $k\geq2$}.
\end{split}
\end{equation*}
Using these estimates, we derive that
\begin{equation*}
\begin{split}
\left|\widehat{\nabla}^{l}g^{T}\right|_{\hat{g}}&\leq
C\sum_{i\,=\,1}^{2n-2}\left|\widehat{\nabla}^{l}(t^{-1}\hat{\theta}_{i}\otimes\hat{\theta}_{i})\right|_{\hat{g}}\\
&\leq C\sum_{\substack{p+q+r\,=\,l \\ 1\,\leq\,i\,\leq\,2n-2}}|\widehat{\nabla}^{p}(t^{-1})|_{\hat{g}}
\underbrace{|\widehat{\nabla}^{q}\hat{\theta}_{i}||_{\hat{g}}|\widehat{\nabla}^{r}\hat{\theta}_{i}|_{\hat{g}}}_{=\,O\left(t^{-\frac{(q+r)}{2}}\right)}\\
&\leq C\sum_{k\,=\,0}^{l}\sum_{p\,=\,l-k}t^{-\frac{k}{2}}|\widehat{\nabla}^{p}(t^{-1})|_{\hat{g}}\\
&\leq C\sum_{k\,=\,0}^{l}t^{-\frac{k}{2}}|\widehat{\nabla}^{l-k}(t^{-1})|_{\hat{g}}\\
&\leq Ct^{-\frac{l}{2}}\sum_{k\,=\,0}^{l}t^{\frac{k}{2}}|\widehat{\nabla}^{k}(t^{-1})|_{\hat{g}}\\
&\leq Ct^{-\frac{l}{2}}\Biggl(t^{-1}+t^{-\frac{3}{2}}+\sum_{k\,=\,2}^{l}\underbrace{t^{\frac{k}{2}}|\widehat{\nabla}^{k}(t^{-1})|_{\hat{g}}}_{=\,O\left(t^{-2}\right)}\Biggl)\\
&\leq Ct^{-1-\frac{l}{2}}.
\end{split}
\end{equation*}
Similarly, one can verify that
$$\left|\frac{1}{4}dt^{2}+\eta^{2}\right|_{\hat{g}}=O\left(1\right)$$
and that
$$\left|\widehat{\nabla}^{l}\left(\frac{1}{4}dt^{2}+\eta^{2}\right)\right|_{\hat{g}}=
O\left(t^{-1-\frac{(l-1)}{2}}\right)\qquad\textrm{for all $l\geq1$}.$$
Recalling \eqref{difference}, the above estimates then imply that
\begin{equation*}
\begin{split}
|\widehat{\nabla}(\tilde{g}-\hat{g})|_{\hat{g}}
&\leq C\Biggl(\left|\widehat{\nabla}(\varphi(t)-nt)\right|_{\hat{g}}
|g^{T}|_{\hat{g}}+|\varphi(t)-nt|
|\widehat{\nabla}g^{T}|_{\hat{g}}
+|\widehat{\nabla}(\varphi'(t)-n)|_{\hat{g}}\left|\frac{1}{4}dt^{2}+\eta^{2}\right|_{\hat{g}}\\
&\qquad+|\varphi'(t)-n|\left|\widehat{\nabla}\left(\frac{1}{4}dt^{2}+\eta^{2}\right)\right|_{\hat{g}}\Biggl)\\
&\leq C\left(t^{-\frac{3}{2}}\log(t)+t^{-2}\right)\\
&\leq Ct^{-\frac{3}{2}}\log(t)=Ct^{-1-\frac{1}{2}}\log(t),
\end{split}
\end{equation*}
\begin{equation*}
\begin{split}
&|\widehat{\nabla}^{2}(\tilde{g}-\hat{g})|_{\hat{g}}\\
&\leq \sum_{m\,=\,0}^{2}\left|\widehat{\nabla}^{m}(\varphi(t)-nt)\right|_{\hat{g}}
\underbrace{\left|\widehat{\nabla}^{2-m}g^{T}\right|_{\hat{g}}}_{=\,O\left(t^{-1-\frac{(2-m)}{2}}\right)}
+\sum_{m\,=\,0}^{2}\left|\widehat{\nabla}^{m}(\varphi'(t)-n)\right|\left|\widehat{\nabla}^{2-m}\left(\frac{1}{4}dt^{2}+\eta^{2}\right)\right|_{\hat{g}}\\
&\leq C\Biggl(|\varphi(t)-nt|t^{-2}+
|\widehat{\nabla}(\varphi(t)-nt)|_{\hat{g}}
t^{-\frac{3}{2}}+|\widehat{\nabla}^{2}(\varphi(t)-nt)|_{\hat{g}}t^{-1}+|\varphi'(t)-n|\left|\widehat{\nabla}^{2}\left(\frac{1}{4}dt^{2}+\eta^{2}\right)\right|_{\hat{g}}\\
&\qquad+|\widehat{\nabla}(\varphi'(t)-n)|\left|\widehat{\nabla}\left(\frac{1}{4}dt^{2}+\eta^{2}\right)\right|_{\hat{g}}
+|\widehat{\nabla}^{2}(\varphi'(t)-n)|_{\hat{g}}\left|\frac{1}{4}dt^{2}+\eta^{2}\right|_{\hat{g}}\Biggl)\\
&\leq C\left(t^{-2}\log(t)+t^{-\frac{5}{2}}+t^{-3}\right)\\
&\leq Ct^{-2}\log(t)=Ct^{-1-\frac{2}{2}}\log(t),
\end{split}
\end{equation*}
and that for $k\geq3$,
\begin{equation*}
\begin{split}
&|\widehat{\nabla}^{k}(\tilde{g}-\hat{g})|_{\hat{g}}\\
&\leq \sum_{m\,=\,0}^{k}\left|\widehat{\nabla}^{m}(\varphi(t)-nt)\right|_{\hat{g}}
\underbrace{\left|\widehat{\nabla}^{k-m}g^{T}\right|_{\hat{g}}}_{=\,O\left(t^{-1-\frac{(k-m)}{2}}\right)}
+\sum_{m\,=\,0}^{k}\left|\widehat{\nabla}^{m}(\varphi'(t)-n)\right|\left|\widehat{\nabla}^{k-m}\left(\frac{1}{4}dt^{2}+\eta^{2}\right)\right|_{\hat{g}}\\
&\leq C\Biggl(|\varphi(t)-nt|t^{-1-\frac{k}{2}}+
|\widehat{\nabla}(\varphi(t)-nt)|_{\hat{g}}
t^{-1-\frac{(k-1)}{2}}+\sum_{m\,=\,2}^{k}t^{-1-\frac{(k-m)}{2}}\underbrace{|\widehat{\nabla}^{m}(\varphi(t)-nt)|_{\hat{g}}}
_{=\,O\left(t^{-1-\frac{m}{2}}\right)}\\
&\qquad+|\varphi'(t)-n|\left|\widehat{\nabla}^{k}\left(\frac{1}{4}dt^{2}+\eta^{2}\right)\right|_{\hat{g}}
+\left|\widehat{\nabla}(\varphi'(t)-n)\right|_{\hat{g}}\left|\widehat{\nabla}^{k-1}\left(\frac{1}{4}dt^{2}+\eta^{2}\right)\right|_{\hat{g}}\\
&\qquad+\left|\widehat{\nabla}^{k}(\varphi'(t)-n)\right|_{\hat{g}}+\sum_{m\,=\,2}^{k-1}
\underbrace{\left|\widehat{\nabla}^{m}(\varphi'(t)-n)\right|_{\hat{g}}}_{=\,O\left(t^{-2-\frac{m}{2}}\right)}
\underbrace{\left|\widehat{\nabla}^{k-m}\left(\frac{1}{4}dt^{2}+\eta^{2}\right)\right|_{\hat{g}}}_{=\,O\left(t^{-1-\frac{(k-m-1)}{2}}\right)}\Biggl)\\
&\leq C\Biggl(t^{-1-\frac{k}{2}}\log(t)+t^{-1-\frac{(k-1)}{2}}t^{-1}+t^{-2-\frac{k}{2}}
+t^{-1}t^{-1-\frac{(k-1)}{2}}+t^{-2}t^{-1-\frac{(k-2)}{2}}+t^{-2-\frac{k}{2}}+t^{-\frac{5}{2}-\frac{k}{2}}\Biggl)\\
&\leq C\left(t^{-1-\frac{k}{2}}\log(t)+t^{-\frac{3}{2}-\frac{k}{2}}+t^{-2-\frac{k}{2}}+t^{-\frac{5}{2}-\frac{k}{2}}\right)\\
&\leq Ct^{-1-\frac{k}{2}}\log(t).
\end{split}
\end{equation*}
\end{proof}

Finally, we show that \eqref{even-more-beautiful} holds true, the last step in the proof of the proposition.

\begin{claim}
\begin{equation*}
\begin{split}
|\widehat{\nabla}^{i}\mathcal{L}_X^{(j)}(\tilde{g}-\hat{g})|_{\hat{g}}=O\left(t^{-1-\frac{i}{2}-j}\right)
\qquad\textrm{for all $i\geq0$ and $j\geq1$}.\\
\end{split}
\end{equation*}
\end{claim}

\begin{proof}
Since
\begin{equation*}
\begin{split}
\mathcal{L}_X^{(j)}(\tilde{g}-\hat{g})&=\left(\mathcal{L}_X^{(j)}(\varphi'-n)\right)\cdot
\left(\frac{1}{4}dt^2+\eta\right)+\left(\mathcal{L}_X^{(j)}(\varphi-nt)\right)\cdot g^T\\
&=\left(\mathcal{L}_X^{(j)}(\varphi'-n)\right)\cdot\left(\frac{1}{4}dt^2+\eta\right)+4\left(\mathcal{L}_X^{(j-1)}(\varphi'-n)\right)\cdot g^T,
\end{split}
\end{equation*}
we deduce from Claim \ref{helloo} that
$$|\mathcal{L}_X^{(j)}(\tilde{g}-\hat{g})|_{\hat{g}}=O\left(t^{-1-j}\right)\qquad\textrm{for all $j\geq1$}.$$
Next, for all $j\geq1$, we see that
\begin{equation*}
\begin{split}
|\widehat{\nabla}\mathcal{L}_X^{(j)}(\tilde{g}-\hat{g})|_{\hat{g}}&\leq C\Biggl(\sum_{k\,=\,0}^{1}|\widehat{\nabla}^{k}\varphi^{(j+1)}|_{\hat{g}}\left|\widehat{\nabla}^{1-k}\left(\frac{1}{4}dt^2+\eta\right)\right|_{\hat{g}}+
\sum_{k\,=\,0}^{1}\left|\widehat{\nabla}^{k}\mathcal{L}_X^{(j-1)}(\varphi'-n)\right|_{\hat{g}}\underbrace{|\widehat{\nabla}^{1-k}
g^T|_{\hat{g}}}_{=\,O\left(t^{-1-\frac{(1-k)}{2}}\right)}\Biggl)\\
&\leq C\Biggl(\left|\varphi^{(j+1)}\right|\left|\widehat{\nabla}\left(\frac{1}{4}dt^2+\eta\right)\right|_{\hat{g}}+
|\widehat{\nabla}\varphi^{(j+1)}|_{\hat{g}}+
t^{-\frac{3}{2}}\left|(\varphi'-n)^{(j-1)}\right|
+t^{-1}\left|\widehat{\nabla}\varphi^{(j)}\right|_{\hat{g}}\Biggl)\\
&\leq C\left(t^{-\frac{3}{2}-j}+t^{-2-j}\right)\\
&\leq Ct^{-\frac{3}{2}-j}=Ct^{-1-\frac{1}{2}-j}.
\end{split}
\end{equation*}
For $j\geq1$, we also find that
\begin{equation*}
\begin{split}
|\widehat{\nabla}^{2}\mathcal{L}_X^{(j)}(\tilde{g}-\hat{g})|_{\hat{g}}&\leq C\Biggl(\sum_{k\,=\,0}^{2}|\widehat{\nabla}^{k}\varphi^{(j+1)}|_{\hat{g}}\left|\widehat{\nabla}^{2-k}\left(\frac{1}{4}dt^2+\eta\right)\right|_{\hat{g}}+
\sum_{k\,=\,0}^{2}\left|\widehat{\nabla}^{k}\mathcal{L}_X^{(j-1)}(\varphi'-n)\right|_{\hat{g}}\underbrace{|\widehat{\nabla}^{2-k}g^T|_{\hat{g}}}_{=\,
O(t^{-1-\frac{(2-k)}{2}})}\Biggl)\\
&\leq C\Biggl(
\left|\varphi^{(j+1)}\right|\left|\widehat{\nabla}^{2}\left(\frac{1}{4}dt^2+\eta\right)\right|_{\hat{g}}
+|\widehat{\nabla}\varphi^{(j+1)}|_{\hat{g}}\left|\widehat{\nabla}\left(\frac{1}{4}dt^2+\eta\right)\right|_{\hat{g}}
+|\widehat{\nabla}^{2}\varphi^{(j+1)}|_{\hat{g}}\\
&\qquad+
t^{-2}|(\varphi'-n)^{(j-1)}|+t^{-\frac{3}{2}}\left|\widehat{\nabla}\varphi^{(j)}\right|_{\hat{g}}
+t^{-1}\left|\widehat{\nabla}^{2}\varphi^{(j)}\right|_{\hat{g}}\Biggl)\\
&\leq C\left(t^{-2-j}+t^{-\frac{5}{2}-j}+t^{-3-j}\right)\\
&\leq Ct^{-2-j}=Ct^{-1-\frac{2}{2}-j}.
\end{split}
\end{equation*}
Finally, for all $i\geq3$ and $j\geq1$, we have that
\begin{equation*}
\begin{split}
|\widehat{\nabla}^{i}\mathcal{L}_X^{(j)}(\tilde{g}-\hat{g})|_{\hat{g}}&\leq C\Biggl(\sum_{k\,=\,0}^{i}|\widehat{\nabla}^{k}\varphi^{(j+1)}|_{\hat{g}}\left|\widehat{\nabla}^{i-k}\left(\frac{1}{4}dt^2+\eta\right)\right|_{\hat{g}}+
\sum_{k\,=\,0}^{i}\left|\widehat{\nabla}^{k}\mathcal{L}_X^{(j-1)}(\varphi'-n)\right|_{\hat{g}}\underbrace{|\widehat{\nabla}^{i-k}g^T|_{\hat{g}}}_{=\,O\left(t^{-1-\frac{(i-k)}{2}}\right)}\Biggl)\\
&\leq C\Biggl(\underbrace{\left|\varphi^{(j+1)}\right|\left|\widehat{\nabla}^{i}\left(\frac{1}{4}dt^2+\eta\right)\right|_{\hat{g}}}_{=\,O\left(t^{-\frac{3}{2}-\frac{i}{2}-j}\right)}
+\underbrace{|\widehat{\nabla}\varphi^{(j+1)}|\left|\widehat{\nabla}^{i-1}\left(\frac{1}{4}dt^2+\eta\right)\right|_{\hat{g}}}
_{=\,O\left(t^{-2-\frac{i}{2}-j}\right)}+\underbrace{|\widehat{\nabla}^{i}\varphi^{(j+1)}|}_{=\,O\left(t^{-2-\frac{i}{2}-j}\right)}\\
&\qquad+\sum_{k\,=\,2}^{i-1}\underbrace{|\widehat{\nabla}^{k}\varphi^{(j+1)}|_{\hat{g}}}_{=\,O\left(t^{-j-\frac{k}{2}-2}\right)}
\underbrace{\left|\widehat{\nabla}^{i-k}\left(\frac{1}{4}dt^2+\eta\right)\right|_{\hat{g}}}_{=\,O\left(t^{-1-\frac{(i-k-1)}{2}}\right)}
+t^{-1-\frac{i}{2}}\sum_{k\,=\,0}^{i}t^{\frac{k}{2}}\left|\widehat{\nabla}^{k}\mathcal{L}_{X}^{(j-1)}(\varphi'-n)\right|_{\hat{g}}\Biggl)\\
&\leq C\Biggl(t^{-\frac{3}{2}-\frac{i}{2}-j}+t^{-1-\frac{i}{2}}\Biggl(\underbrace{|(\varphi'-n)^{(j-1)}|}_{=\,O\left(t^{-j}\right)}+
\underbrace{t^{\frac{1}{2}}\left|\widehat{\nabla}\varphi^{(j)}\right|_{\hat{g}}}_{=\,O\left(t^{-\frac{1}{2}-j}\right)}+
\underbrace{\sum_{k\,=\,2}^{i}t^{\frac{k}{2}}\left|\widehat{\nabla}^{k}\varphi^{(j)}\right|_{\hat{g}}}_{=\,O\left(t^{-1-j}\right)}\Biggl)\Biggl)\\
&\leq Ct^{-1-\frac{i}{2}-j}.
\end{split}
\end{equation*}
\end{proof}

\end{proof}

The following lower bound on the scalar curvature of
Cao's steady gradient K\"ahler-Ricci soliton along the end of the cone will prove useful for later.

\begin{lemma}\label{low-pos-bd-scal}
$\RR_{\tilde{g}}\geq \frac{c}{t}$ along the end of $C_{0}$ for some constant $c>0$.
\end{lemma}

\begin{proof}
From \eqref{scalar}, we read that $$\RR_{\tilde{g}}=4n-4\varphi'(t).$$
The asymptotics of $\varphi'(t)$ as dictated by Proposition \ref{potential} then imply that
\begin{equation*}
\begin{split}
\RR_{\tilde{g}}&=4n-4\varphi'(t)\\
&=4n-4\left(n-\frac{(n-1)}{t}+O\left(\frac{(\log t)^{2}}{t^{2}}\right)\right)\\
&=\frac{4}{t}\left(n-1+O\left(\frac{(\log t)^{2}}{t}\right)\right)\\
&\geq\frac{4(n-1-\varepsilon)}{t}
\end{split}
\end{equation*}
for any $\varepsilon\in(0,\,1)$ for $t$ sufficiently large.
\end{proof}

\section{Constructing a background metric and the equation set-up}\label{Section-App-Met}

\subsection{Construction of an approximate soliton}\label{subsection-existence-sol}
In this section, we consider a Calabi-Yau cone $(C_{0 },\,g_{0})$ of
complex dimension $n\geq2$ with complex structure $J_{0}$ and radius function $r$ and an equivariant crepant resolution
$\pi:M\to C_{0}$ of $C_{0}$ with exceptional set $E$ so that $M$ has trivial canonical bundle and the real holomorphic torus action
induced by the flow of the holomorphic vector
field $J_{0}r\partial_{r}$ on $C_{0}$ extends to $M$. We set $r^2=:e^t$ and write $X$ for
the lift of the holomorphic vector field $2r\partial_{r}=4\partial_{t}$ on $C_{0}$ to $M$. We have a transverse K\"ahler
form $\omega^{T}=\frac{1}{2}dd^{c}\log(r)$ on $C_{0}$ as well as a contact form $\eta=d^{c}\log(r)$
on the link $(S,\,g_{S})$ of $C_{0}$ which we identify with the level set
$\{r=1\}$. We also have a natural
projection $p_{S}:C_{0}\simeq\mathbb{R}_{+}\times S\to\{r\,=\,1\}\simeq S$.
Let $J$ denote the complex structure on $M$ and
let $\Phi(t)$ denote the K\"ahler potential
of Cao's steady gradient K\"ahler-Ricci soliton $\tilde{\omega}$ on $C_{0}$ (as in Proposition \ref{caosoliton}),
the asymptotic model of which is the K\"ahler form $\hat{\omega}$ on $C_{0}$ defined by
$$\hat{\omega}:=\frac{i}{2}\partial\bar{\partial}\left(\frac{nt^{2}}{2}\right)
=n\left(\frac{dt}{2}\wedge\eta+t\omega^{T}\right)$$
with associated K\"ahler metric $\hat{g}$. Throughout this section, we identify the complement of the vertex $o$ of $C_{0}$ with $M\setminus E$ via $\pi$. In this way, we treat $r,\,t,$ and $\Phi(t)$ not only as smooth functions on $C_{0}\setminus\{o\}$, but also as smooth functions
on $M\setminus E$.

We begin with the following preliminary lemma.

\begin{lemma}[{\cite[Lemma 5.6]{goto}}]\label{lemma5.6}
Let $\kappa$ be an arbitrary K\"ahler form on $M$ with K\"ahler class $[\kappa]\in H^{2}(M,\,\R)$. Assume that $n=\dim_{\mathbb{C}}C_{0}\geq3$. Then for every $T>1$, there exists a smooth real $(1,\,1)$-form $\tilde{\kappa}_{T}$ on $M$ depending on $T$ with the following properties.
\begin{enumerate}[label=\textnormal{(\roman{*})}, ref=(\roman{*})]
\item $[\tilde{\kappa}_{T}]=[\kappa]\in H^{2}(M,\,\R)$.
\item $\tilde{\kappa}_{T}=\kappa$ on $E\cup\{x\in M:t(x)<T\}$ and the restriction of $\tilde{\kappa}_{T}$ to the subset $\{x\in M:t(x)\,>\,2T\}$ is given by the pullback of a closed, primitive basic $(1,\,1)$-form $\zeta$ on $S$ that is independent of $T$ and determined uniquely by the cohomology class $[\kappa|_{C_{0}}]\in H^{2}(C_{0},\,\R)$. In other words, $$\tilde{\kappa}_{T}|_{\{x\,\in \,M\,|\,t(x)\,>\,2T\}}=p_{S}^{*}(\zeta)\quad\textrm{for every $T>1$}.$$
\end{enumerate}
\end{lemma}

\begin{proof}
By Proposition \ref{fgh}, we know that the vector spaces $H^{2}(S,\,\R)$ and $H^{2}_{B}(S)_{\operatorname{p}}$ coincide. Furthermore, since $S$
is a Sasaki-Einstein manifold which is necessarily a positive Sasaki manifold, we have the vanishing $h^{2,\,0}_{B}(S)=h_{B}^{0,\,2}(S)=0$ from Proposition \ref{vanishing}. Together with \eqref{poof}, these two statements imply that $H^{2}(S,\,\C)=H^{2}_{B}(S)_{\operatorname{p}}\otimes\C=H^{1,\,1}_{B}(S)_{\operatorname{p}}$, and so we have an isomorphism
\begin{equation*}
H^{2}(C_{0},\,\C)\cong H^{2}(S,\,\C)=H^{1,\,1}_{B}(S)_{\operatorname{p}}
\end{equation*}
given by the pullback $p^{*}_{S}:H^{2}(S,\,\C)\longrightarrow H^{2}(C_{0},\,\C)$. Since $\kappa|_{C_{0}}$ defines a cohomology class
in $H^{2}(C_{0},\,\R)$, we can therefore assert that
$$\kappa|_{C_{0}}=p_{S}^{*}(\zeta)+d\theta$$ for some real one-form $\theta$ on $C_{0}$
and for some real primitive basic $(1,\,1)$-form $\zeta$ on $S$
that is basic harmonic and determined uniquely by $[\kappa|_{C_{0}}]$.
Now, $n=\dim_{\C}M\geq 3$ so that $H^{1}(C_{0},\,\mathcal{O}_{C_{0}})=0$ by
Lemma \ref{321}, and $d\theta$ is a real $(1,\,1)$-form. Therefore by arguing as in the proof of
Lemma \ref{deldelbar} with $C_{0}$ and $d\theta$ in place of $M$ and $\alpha$ respectively, we deduce
that $d\theta=i\partial\bar{\partial}\phi$ for some smooth real-valued function $\phi:C_{0}\to\mathbb{R}$.

Next fix $T>1$ and choose a smooth cut-off function $\rho_{T}:M\longrightarrow\R$ satisfying $|\rho_{T}(x)|\leq 1$ for all $x\in M$ and
\begin{displaymath}
\rho_{T}(x) = \left\{ \begin{array}{ll}
0 & \textrm{if $x\in E\cup\{y\in M\,|\,t(y)<T\}$,}\\
1 & \textrm{if $x\in\{y\in M\,|\,t(y)>2T\}$}.\\
\end{array} \right.
\end{displaymath}
We define $\tilde{\kappa}_{T}$ by $$\tilde{\kappa}_{T}=\kappa-i\p\bar{\p}(\rho_{T}.\phi).$$
Then this is a closed real $(1,\,1)$-form on $M$ lying in the same cohomology class as $\kappa$ that interpolates between
$\kappa$ on the set $E\cup\{x\in M\,|\,t(x)<T\}$ and $p_{S}^{*}(\zeta)$ on the set $\{x\in M\,|\,t(x)\,>\,2T\}$, i.e., it
satisfies properties (i) and (ii) of the lemma, as desired.
\end{proof}

The next ingredient we need for the construction of our background metric is the following lemma, akin to \cite[Lemma 2.15]{Conlon}.
\begin{lemma}\label{1convex}
For all $\alpha>\frac{1}{2}$, there exists a smooth plurisubharmonic function $h_{\alpha}$ on $M$ which is strictly plurisubharmonic and equal to
$\frac{1}{2}(\Phi(t))^{\alpha}$ outside a compact subset $K_{\alpha}$ of $M$ containing $E$.
\end{lemma}

\begin{proof}
Let $\psi:\R_{+}\to\R_{+}$ be smooth with $\psi',\psi''\geq 0$ and
$$
\psi(s) = \begin{cases}
T+2 & \textrm{if}\;\,s<T + 1,\\
s & \textrm{if}\;\,s>T+3,
\end{cases}
$$
for some $T>1$ to be specified later. Then $h_{\alpha}(t):=\frac{1}{2}\psi \circ(\Phi(t))^{\alpha}: M \to \R_{+}$ satisfies
$$
i\partial\bar{\partial}h_{\alpha} =
\begin{cases}
0 &\textrm{on}\;E\cup\{x\,|\,\Phi(t(x)) \leq(T+1)^{\frac{1}{\alpha}}\},\\
\psi''\frac{i}{2}\p (\Phi(t))^{\alpha}\wedge\bar{\p}(\Phi(t))^{\alpha}+\psi'\frac{i}{2}\p\bar{\p}(\Phi(t))^{\alpha} &\textrm{on}\;\{x\,|\,\Phi(t(x)) > T^\frac{1}{\alpha}\}.
\end{cases}
$$
Since $\Phi(t)$, as an antiderivative of $\varphi(t)$, tends to $+\infty$ as $t\to+\infty$ by Proposition \ref{potential} so is proper,
we see that $E\cup\{x\,|\,\Phi(t(x))\leq\lambda\}$ is compact for every $\lambda\gg0$ and that on $\{x\,|\,\Phi(t(x)) > T^{\frac{1}{\alpha}}\}$,
$$\frac{i}{2}\partial\bar{\partial}(\Phi(t))^{\alpha}
=\alpha\Phi(t)^{\alpha-2}\biggl(\underbrace{\left((\alpha-1)\varphi(t)^{2}+\varphi'(t)\Phi(t)\right)}_{\sim\,(\alpha-\frac{1}{2})n^{2}t^{2}}\frac{dt}{2}\wedge\eta
+\underbrace{\varphi(t)\Phi(t)}_{>\,0}\omega^{T}\biggr)
 > 0$$
so long as $\alpha>\frac{1}{2}$ and $T \gg 1$, again by virtue of Proposition \ref{potential}. Moreover, notice that $i\p u \wedge\bar{\p}u \geq 0$ for any smooth real-valued function $u$.
Together, these observations imply that $h_{\alpha}$ has the desired properties.
\end{proof}

We can now construct our background metric on $M$ via a construction reminiscent of that in the asymptotically conical
Calabi-Yau case \cite{Conlon, goto, vanC2}.
\begin{prop}[Construction of a background metric]\label{lemma5.7}
Let $\kappa$ be an arbitrary K\"ahler form on $M$ and let $\tilde{\omega}$ denote the K\"ahler form of Cao's steady gradient K\"ahler-Ricci soliton on $C_{0}$. Then there exists a K\"ahler form $\sigma$ on $M$ with the following properties.
\begin{enumerate}[label=\textnormal{(\roman{*})}, ref=(\roman{*})]
\item $[\sigma]=[\kappa]\in H^{2}(M,\,\R)$.
    \item $\mathcal{L}_{JX}\sigma=0$.
\item There exists a compact subset $K\subset M$ containing the exceptional set $E$ of $\pi:M\to C_{0}$ such that on $M\setminus K$,
\begin{equation*}
\pi_{\ast}\sigma=
\begin{cases}
\tilde{\omega} & \textrm{if $n=2$ or $[\kappa]$ is compactly supported},\\
\tilde{\omega}+p^{*}_{S}(\zeta) & \text{if $n\geq3$},
\end{cases}
\end{equation*}
where $\zeta$ is as in Lemma \ref{lemma5.6}.
\end{enumerate}
\end{prop}

\begin{proof}
Fix $\alpha\in\left(\frac{1}{2},\,1\right)$ once and for all. Throughout the proof we assume that $T>1$ is chosen sufficiently large so that
$\tilde{\omega}>0$ as well as $h_{\alpha}= \frac{1}{2}(\Phi(t))^{\alpha}$ and $h_{1}=\frac{1}{2}\Phi(t)$ on $\{t \geq T\}$, and that both of these latter functions
are strictly plurisubharmonic on this region.

We first deal with the case $n=\dim_{\mathbb{C}}C_{0}\geq3$. Let $\zeta$ be the basic $(1,\,1)$-form on the link $S$ of $C_{0}$ associated to the class $[\kappa]$ given by Lemma \ref{lemma5.6}. Since $|p_{S}^{*}(\zeta)|_{\hat{g}}=O(t^{-1})$ (cf.~Proposition \ref{basis}) and since $\hat{g}$ and $\tilde{g}$
are equivalent at infinity as a consequence of Proposition \ref{coro-asy-Cao-met}, by choosing $T$ larger if necessary,
we can assume that $|p_{S}^{*}(\zeta)|_{\tilde{g}}<1$ on $\{t>T\}$. Then on this region, we have that $\tilde{\omega}+p_{S}^{*}(\zeta)>0$. For this choice of $T$, let $\tilde{\kappa}_{T}\in[\kappa]$ be as in Lemma \ref{lemma5.6}.
Then $\tilde{\kappa}_{T}$ is equal to $\kappa$ on $E\cup\{x\in M\,|\,t(x)<T\}$ and $p_{S}^{*}(\zeta)$ on $\{x\in M\,|\,t(x)\,>\,2T\}$. We fix a cut-off function $\chi:M\to\mathbb{R}$ with
$$
\chi(x) =
\begin{cases}
0 &\textrm{if $x\in E\cup\{y\in M\,|\,t(y)<T\}$},\\
1 &\textrm{if $t(x)\,>\,2T$},
\end{cases}
$$ and define $\chi_{\lambda}(x):=\chi(x/\lambda)$ in the obvious way for $\lambda>2$. We then construct a K\"ahler metric $\hat{\sigma}$ in $[\kappa]$
by $$\hat{\sigma}:=\tilde{\kappa}_{T}+Ci\partial\bar{\partial}((1-\chi_{\lambda})h_{\alpha})+i\partial\bar{\partial}h_{1},$$
where $C>0$ and $\lambda>2$ are both to be determined. First observe that
$\hat{\sigma}=\kappa+Ci\partial\bar{\partial}h_{\alpha}+i\partial\bar{\partial}h_{1}\geq \kappa > 0$ on $E \cup \{t< T\}$ because $h_{\alpha}$ and $h_1$ are plurisubharmonic;
$\hat{\sigma} =\tilde{\omega}+p^{*}_{S}(\zeta)+Ci\p\bar{\p} h_\alpha  > 0$ on $\{2T < t < \lambda T\}$ because $|\zeta|_{\tilde{g}}<1$ on this region by choice of $T$; $\hat{\sigma}=\tilde{\omega}+p^{*}_{S}(\zeta)>0$
on $\{t>2\lambda T \}$ since $|p_{S}^{*}(\zeta)|_{\tilde{g}}<1$ on this region, again by choice of $T$; $\hat{\sigma} > 0$ on
$\{T \leq t \leq 2T\}$ by compactness if $C$ is made large enough; and finally,
$\hat{\sigma} > 0$ on $\{\lambda T \leq t \leq 2\lambda T\}$ if $\lambda \gg 1$ depending on all previous choices because
from the equivalence of $\hat{g}$ and $\tilde{g}$, we have that
$$|i\partial\bar{\partial}((1-\chi_{\lambda})h_{\alpha})|_{\tilde{g}}=O\left(\lambda^{2\alpha-2}\right)=o(1)$$
by choice of $\alpha$. In conclusion, $\hat{\sigma}$ is a genuine K\"ahler form on $M$ with $$\hat{\sigma} = \tilde{\omega}+p^{*}_{S}(\zeta)$$
at infinity.

Next assume that $[\kappa]$ is compactly supported and that $n\geq2$. Then the vanishing \eqref{lenny} may no longer hold true and so
we proceed as in \cite{vanC2}. Let $\{E_{i}\}$ be the prime divisors in the exceptional set $E$
of the resolution $\pi:M\to C_{0}$. Since $H_{2n-2}(M)$ is generated by the fundamental classes of the $E_{i}$,
$[\kappa]$ is Poincar\'e dual to $\sum_{i}a_{i}E_{i}$ for some $a_{i}\in\mathbb{R}$.
Thus, there exists a compactly supported closed $(1,\,1)$-form $\beta$ Poincar\'e dual to
$\sum_{i}a_{i}E_{i}$ with $[\beta]=[\kappa]$. Let $\alpha$ be a smooth one-form
with $d\alpha=\kappa-\beta$. Then by Lemma \ref{deldelbar}, there exists a smooth real-valued function $\phi$ on $M$
such that $d\alpha=i\partial\bar{\partial}\phi.$ By choosing $T$ larger if necessary, we can assume that
$\operatorname{supp}(\beta)$, the support of $\beta$, is contained within $\{t< T\}$. Choose a smooth cut-off function $\rho_{T}:M\longrightarrow\R$ satisfying $|\rho_{T}(x)|\leq 1$ for all $x\in M$ and
\begin{displaymath}
\rho_{T}(x) = \left\{ \begin{array}{ll}
0 & \textrm{if $x\in E\cup\{y\in M\,|\,t(y)<T\}$}\\
1 & \textrm{if $x\in\{y\in M\,|\,t(y)>2T\}$}.\\
\end{array} \right.
\end{displaymath}
Then we define $\hat{\sigma}$ in this case
by $$\hat{\sigma}:=\beta+i\partial\bar{\partial}((1-\rho_{T})\cdot\phi)+Ci\partial\bar{\partial}((1-\chi_{\lambda})h_{\alpha})+i\partial\bar{\partial}h_{1},$$
where $C>0$ and $\lambda>2$ are yet to be determined. Observe that
$\hat{\sigma}=\kappa+Ci\partial\bar{\partial}h_{\alpha}+i\partial\bar{\partial}h_{1}\geq \kappa > 0$ on $E \cup \{t< T\}$ because $h_{\alpha}$ and $h_1$ are plurisubharmonic;
$\hat{\sigma} =\tilde{\omega}+Ci\p\bar{\p} h_\alpha  > 0$ on $\{2T < t < \lambda T\}$ because $\operatorname{supp}(\beta)\subset\{t< T\}$ by choice of $T$; $\hat{\sigma}=\tilde{\omega}>0$ on $\{t>2\lambda T \}$ again because
$\operatorname{supp}(\beta)\subset\{t<T\}$; $\hat{\sigma} > 0$ on
$\{T \leq t \leq 2T\}$ by compactness if $C$ is made large enough; and finally,
$\hat{\sigma} > 0$ on $\{\lambda T \leq t \leq 2\lambda T\}$ if $\lambda \gg 1$ depending on all previous choices because
from the equivalence of $\hat{g}$ and $\tilde{g}$, we have that
\begin{equation*}
|i\partial\bar{\partial}((1-\chi_{\lambda})h_{\alpha})|_{\tilde{g}}=O\left(\lambda^{2\alpha-2}\right)=o(1).
\end{equation*}
In conclusion, $\hat{\sigma}$ is a genuine K\"ahler form on $M$ with $$\hat{\sigma} = \tilde{\omega}$$
at infinity.

We now average $\hat{\sigma}$ as given over the action of the real torus $T^k$ on $M$ induced by the flow of the holomorphic vector field $J_{0}r\partial_{r}$ on $C_{0}$ by setting
$$\sigma:=\frac{1}{|T^k|}\int_{T^k}\psi_{g}^{*}\hat{\sigma}\,d\mu(g),$$
where $\psi_{g}:M\to M$ is the automorphism of $M$ induced by $g\in T^k$ and $d\mu$ is the Haar measure on $T^{k}$. Since there is a path in $T^k$ connecting $g$ to the identity,
we have that $\psi_{g}^{*}[\sigma]=[\psi_{g}^{*}\sigma]=[\sigma]$, from which it follows that $[\sigma]=[\hat{\omega}]=[\kappa]$.
Moreover, it is clear that $\mathcal{L}_{JX}\sigma=0$. Finally, since the action of $T^{k}$ preserves $r$ and hence $t$, we have
that $\psi_{g}^{*}\tilde{\omega}=\tilde{\omega}$ and for $n\geq3$, $\psi_{g}^{*}p_{S}^{*}(\zeta)=p_{S}^{*}(\zeta)$ for every $g\in T^{k}$. Thus, $\sigma$ has the desired properties at infinity.

If $n=2$, then we have a long exact sequence of cohomology
\begin{equation}\label{les}
H^1(S,\,\mathbb{R}) \to H^2_c(M,\,\mathbb{R}) \to H^2(M,\,\mathbb{R}) \to H^2(S,\,\mathbb{R}) \to H^3_c(M,\,\mathbb{R}),
\end{equation}
where recall that $S$ is the link of the cone $C_{0}$. In this dimension, $C_{0}=\mathbb{C}^{2}/\Gamma$ for $\Gamma\subset SU(2)$ a finite subgroup acting freely on $\mathbb{C}^{2}\setminus\{0\}$ so that $S=S^{3}$ or some
finite quotient thereof \cite{hamilton}. In particular, $H^1(S,\,\mathbb{R})=H^2(S,\,\mathbb{R})=0$ and so we deduce from \eqref{les} that $H^2_c(M,\,\mathbb{R})=H^2(M,\,\mathbb{R})$.
Hence for $n=2$, every K\"ahler class is compactly supported, a case that has already been dealt with in two dimensions. This completes the proof of the proposition.
\end{proof}

The metric $\sigma$ of Proposition \ref{lemma5.7} enjoys the following asymptotics measured with respect to $\hat{g}$
and its Levi-Civita connection $\widehat{\nabla}$.
\begin{lemma}\label{budwiser}
Let $\sigma$ be the K\"ahler form from Proposition \ref{lemma5.7}. Then
\begin{equation}\label{bubbles}
|\widehat{\nabla}^{i}\mathcal{L}^{(j)}_X\left(\pi_{*}\sigma-\tilde{\omega}\right)|_{\hat{g}}=O\left(t^{-1-\frac{i}{2}-j}\right)\qquad\textrm{for all $i,\,j\geq 0$}.
\end{equation}
\end{lemma}

\begin{proof}
If $n=2$ or $[\sigma]$ is compactly supported, then this is clear. Otherwise for $n\geq3$,
we have by construction that $\pi_{*}\sigma-\tilde{\omega}=p_S^{*}(\zeta)$
for $\zeta$ a basic two-form on the link of the cone. As a pullback, it is clear that
$\mathcal{L}_{X}(p_S^{*}(\zeta))=0$, and it is easy to see that
$\mathcal{L}_{X}(i\partial\bar{\partial}t)=0$. Consequently, \eqref{bubbles} holds true for all $i\geq0$ and
$j\geq1$. If $i\geq0$ and $j=0$, then observe from
Proposition \ref{basis} that
$|\widehat{\nabla}^{k}(i\partial\bar{\partial}t)|_{\hat{g}}=O\left(t^{-1-\frac{k}{2}}\right)$ for all $k\geq0$
with respect to the metric $\hat{g}$, and that
$|\widehat{\nabla}^{k}\beta|_{\hat{g}}=O\left(t^{-1-\frac{k}{2}}\right)$ for all $k\geq0$
for any basic two-form $\beta$ on $C_{0}$, a fact that itself may be proved by induction as demonstrated in
Proposition \ref{basis} for basic one-forms. These asymptotics imply that
\begin{equation*}
|\widehat{\nabla}^{i}\left(\pi_{*}\sigma-\tilde{\omega}\right)|_{\hat{g}}=
|\widehat{\nabla}^{i}(p^{*}_{S}(\zeta))|_{\hat{g}}=O\left(t^{-1-\frac{i}{2}}\right)\qquad\textrm{for all $i\geq0$},
\end{equation*}
from which the lemma follows.
\end{proof}

\subsection{Set-up of the complex Monge-Amp\`ere equation}

We next set up the complex Monge-Amp\`ere equation on the crepant resolution $\pi:M\to C_{0}$ that we will solve in order to construct our steady gradient K\"ahler-Ricci solitons.
\begin{prop}\label{equationsetup}
Let $\sigma$ be the K\"ahler form of Proposition \ref{lemma5.7} with Ricci form $\rho_{\sigma}$,
let $X$ be the lift of the holomorphic vector field $2r\partial_{r}=4\partial_{t}$ on $C_{0}$ to $M$ via $\pi$, and let $J$ denote the complex structure on $M$.
Furthermore, let $\psi\in C^{\infty}(M)$ be such that $\sigma_{\psi}:=\sigma+i\partial\bar{\partial}\psi>0$ and
$\mathcal{L}_{JX}\psi=0$, and consider the equations
\begin{equation}\label{e:soliton}
\log\left(\frac{(\sigma+i\partial\bar{\partial}\psi)^{n}}{\sigma^{n}}\right)+\frac{X}{2}\cdot\psi=F,
\end{equation}
where $F\in C^{\infty}(M)$ and $\mathcal{L}_{JX}F=0$, and
\begin{equation}\label{sexysoliton}
\rho_{\sigma_{\psi}}=\frac{1}{2}\mathcal{L}_{X}\sigma_{\psi},
\end{equation}
where $\rho_{\sigma_{\psi}}$ denotes the Ricci form of $\sigma_{\psi}$. Then:
\begin{enumerate}[label=\textnormal{(\roman{*})}, ref=(\roman{*})]
\item If $\psi$ satisfies \eqref{e:soliton} and $i\partial\bar{\partial}F=\rho_{\sigma}-\frac{1}{2}\mathcal{L}_{X}\sigma$,
then $\sigma_{\psi}$ satisfies \eqref{sexysoliton}.
\item Conversely, if $\sigma_{\psi}$ satisfies \eqref{sexysoliton}, then $\psi$ satisfies \eqref{e:soliton}
for a function $F$ with $i\partial\bar{\partial}F=\rho_{\sigma}-\frac{1}{2}\mathcal{L}_{X}\sigma$
that outside a compact subset of $M$ is given by
\begin{equation*}
F = \left\{
\begin{array}{ll}
0 & \text{if $n=2$ or if $[\sigma]$ is compactly supported},\\
-\log\left(\frac{(\tilde{\omega}+p_S^{*}(\zeta))^{n}}{\tilde{\omega}^{n}}\right) & \text{otherwise}.\\
\end{array} \right.
\end{equation*}
\end{enumerate}
\end{prop}

\begin{proof}
\begin{enumerate}[label=\textnormal{(\roman{*})}, ref=(\roman{*})]
\item If $\psi$ satisfies \eqref{e:soliton} with $F$ as prescribed, then by taking $i\partial\bar{\partial}$ of this equation, we see that
$\sigma_{\psi}$ satisfies \eqref{sexysoliton}.

\item As for the converse, suppose that $\sigma_{\psi}$ satisfies \eqref{sexysoliton}. Then
\begin{equation*}
\begin{split}
0&=\rho_{\sigma_{\psi}}-\frac{1}{2}\mathcal{L}_{X}\sigma_{\psi}\\
&=\rho_{\sigma_{\psi}}-\rho_{\sigma}+\rho_{\sigma}-\frac{1}{2}\mathcal{L}_{X}\sigma_{\psi}\\
&=-i\partial\bar{\p}\log\left(\frac{(\sigma+i\partial\bar{\partial}\psi)^{n}}{\sigma^{n}}\right)+
\rho_{\sigma}-\frac{1}{2}\mathcal{L}_{X}\sigma_{\psi}\\
&=-i\partial\bar{\p}\log\left(\frac{(\sigma+i\partial\bar{\partial}\psi)^{n}}{\sigma^{n}}\right)-\frac{1}{2}i\partial\bar{\partial}\left(X\cdot\psi\right)
+(\rho_{\sigma}-\frac{1}{2}\mathcal{L}_{X}\sigma),\\
\end{split}
\end{equation*}
so that
\begin{equation*}
\begin{split}
i\partial\bar{\p}\left(\log\left(\frac{(\sigma+i\partial\bar{\partial}\psi)^{n}}{\sigma^{n}}\right)+\frac{X}{2}\cdot\psi\right)
&=\rho_{\sigma}-\frac{1}{2}\mathcal{L}_{X}\sigma.\\
\end{split}
\end{equation*}
Now, since $\mathcal{L}_{JX}\sigma=0$, $JX$ is Killing, and so by \cite[Lemma A.6]{con-der}, the $g_{\sigma}$-dual one-form $\eta_{X}:=g_{\sigma}(X,\,\cdot)$ of $X$ is closed, $g_{\sigma}$ denoting the K\"ahler metric associated to $\sigma$. The fact that $H^{1}(M,\,\mathbb{R})=0$ by Lemma \ref{finite} then implies that there is
a smooth real-valued function $\theta_{X}$ on $M$ such that $\eta_{X}=d\theta_{X}$, or equivalently, such that $X=\nabla^{g_{\sigma}}\theta_{X}$, where $\nabla^{g_{\sigma}}$
denotes the Levi-Civita connection of $g_{\sigma}$. It follows that $\sigma\lrcorner X=d\theta_{X}\circ J$, which allows us to write $$\mathcal{L}_{X}\sigma=d(\sigma\lrcorner X)=i\partial\bar{\partial}\theta_{X}.$$
Since $[X,\,JX]=0$ and $\mathcal{L}_{JX}\sigma=0$, we know that $\mathcal{L}_{JX}\mathcal{L}_{X}\sigma=\mathcal{L}_{X}\mathcal{L}_{JX}\sigma
+\mathcal{L}_{[JX,\,X]}\sigma=0$. Hence, by averaging over the real torus action on $M$ induced by that on $C_{0}$,
we may assume that $\mathcal{L}_{JX}\theta_{X}=0.$
Furthermore, as $K_{M}$ is trivial, we may write $$\rho_{\sigma}=i\partial\bar{\partial}v$$ for some smooth real-valued function $v\in C^{\infty}(M)$.
Averaging this equation over the real torus action on $M$, we may then also assume that
$$\rho_{\sigma}=i\partial\bar{\partial}\tilde{v}$$ for some $\tilde{v}\in C^{\infty}(M)$ satisfying $\mathcal{L}_{JX}\tilde{v}=0$. Here we have used the fact that $JX$ is real holomorphic and
$\mathcal{L}_{JX}\sigma=0$ so that $\mathcal{L}_{JX}\rho_{\sigma}=0$. In summary, we can now write
\begin{equation}\label{sexier}
\begin{split}
\rho_{\sigma}-\frac{1}{2}\mathcal{L}_{X}\sigma&=i\partial\bar{\partial}\tilde{v}-i\partial\bar{\partial}\theta_{X}\\
&=i\partial\bar{\partial}F,
\end{split}
\end{equation}
where $F:=\tilde{v}-\theta_{X}\in C^{\infty}(M)$. In particular, notice that $\mathcal{L}_{JX}F=0$.

With $\tilde{\omega}$ as usual denoting Cao's steady gradient K\"ahler-Ricci soliton on $C_{0}$ and $\rho_{\tilde{\omega}}$ its Ricci form, next observe that at infinity we have that
\begin{equation}\label{hoola}
\begin{split}
\rho_{\sigma}-\frac{1}{2}\mathcal{L}_{X}\sigma&=\rho_{\sigma}-\rho_{\tilde{\omega}}+\rho_{\tilde{\omega}}
-\frac{1}{2}\mathcal{L}_{X}(\tilde{\omega}+p_S^{*}(\zeta))\\
&=-i\partial\bar{\p}\log\left(\frac{\sigma^{n}}{\tilde{\omega}^{n}}\right)+\underbrace{\rho_{\tilde{\omega}}
-\frac{1}{2}\mathcal{L}_{X}\tilde{\omega}}_{=\,0}-
\underbrace{\frac{1}{2}\mathcal{L}_{X}(p_S^{*}(\zeta))}_{=\,0}\\
&=i\partial\bar{\p}G(\tilde{\omega}),\\
\end{split}
\end{equation}
where
\begin{equation*}
G=G(\tilde{\omega}) =-\log\left(\frac{\sigma^{n}}{\tilde{\omega}^{n}}\right)= \left\{
\begin{array}{ll}
0 & \text{if $n=2$ or if $[\sigma]$ is compactly supported},\\
-\log\left(\frac{(\tilde{\omega}+p_S^{*}(\zeta))^{n}}{\tilde{\omega}^{n}}\right) & \text{otherwise}.\\
\end{array} \right.
\end{equation*}
Notice that $\mathcal{L}_{JX}G=0$. On subtracting \eqref{hoola} from \eqref{sexier}, we see that at infinity
\begin{equation*}
i\partial\bar{\partial}(F-G)=0.
\end{equation*}
Since $\mathcal{L}_{JX}(F-G)=0$, it then follows from Lemma \ref{pluri}(i)
that $F-G=C$ on the complement of a compact subset of $M$ for some constant $C$.
Therefore, by subtracting a constant from $F$ in \eqref{sexier} if necessary, we may assume that
$$i\partial\bar{\p}\left(\log\left(\frac{(\sigma+i\partial\bar{\partial}\psi)^{n}}{\sigma^{n}}\right)+\frac{X}{2}\cdot\psi-F\right)=0,$$
where
\begin{equation*}
F = \left\{
\begin{array}{ll}
0 & \text{if $n=2$ or if $[\sigma]$ is compactly supported},\\
-\log\left(\frac{(\tilde{\omega}+p_S^{*}(\zeta))^{n}}{\tilde{\omega}^{n}}\right)=o(1) & \text{otherwise}.\\
\end{array} \right.,
\end{equation*}
the asymptotics in the latter case a result of Proposition \ref{basis} and
the fact that $\tilde{g}$ and $\hat{g}$ are equivalent at infinity by Proposition \ref{coro-asy-Cao-met}.
Lemma \ref{pluri}(ii) now asserts that $$\log\left(\frac{(\sigma+i\partial\bar{\partial}\psi)^{n}}{\sigma^{n}}\right)+\frac{X}{2}\cdot\psi=F.$$
Recalling \eqref{sexier} and the fact that $\mathcal{L}_{JX}F=0$, this completes the proof of part (ii) of the proposition.
\end{enumerate}
\end{proof}

\section{Poincar\'e inequality for steady gradient Ricci solitons}\label{poincare}

In this section, we establish a lower bound on the spectrum of the drift
Laplacian of a non-trivial steady gradient Ricci soliton; cf.~Section \ref{metric-measure}.
This allows for a Poincar\'e inequality that we will use in Proposition \ref{prop-a-priori-ene-est}
to establish an a priori weighted $L^{2}$-estimate along the continuity path of solutions to \eqref{e:soliton},
the first step in the derivation of an a priori $C^{0}$-estimate.

We begin with a preliminary result that gives a lower bound on the
spectrum of the drift Laplacian on a Riemannian manifold as soon as a positive eigenfunction exists.
We in fact provide sufficient conditions ensuring the existence of a Hardy inequality. The precise statement is as follows.

\begin{lemma}\label{bd-spec-pos}
Let $(M,\,g,\,e^{\rho}d\mu_{g})$ be a metric measure space endowed with the volume form $d\mu_{g}$ of $g$
and a $C^1$ potential function $\rho$ on $M$ such that
$\int_{M}e^{\rho}d\mu_g=+\infty$. Assume that there exists a positive $C^2$-function $\phi_0$ on $M$
such that $\Delta_{\rho}\phi_0\leq-\lambda_0 \phi_0$ outside a compact subset $K\subset M$ for $\lambda_0$ a positive constant.
Then there exists a positive constant $\lambda_1\leq\lambda_0$ such that the following global Poincar\'e inequality holds true:
\begin{eqnarray*}
\lambda_1\int_M\phi^2\,e^{\rho}d\mu_g\leq\int_M\arrowvert\nabla^g\phi\arrowvert^2_g\,e^{\rho}d\mu_g
\quad\textrm{for any smooth compactly supported function $\phi$ on $M$.}
\end{eqnarray*}
Equivalently, $\inf\sigma(-\Delta_{\rho})\geq\lambda_1>0$, where $\sigma(-\Delta_{\rho})$ denotes the $L^2(e^{\rho}d\mu_g)$-spectrum of the operator $-\Delta_{\rho}$.
\end{lemma}

\noindent The statement and proof of this lemma are straightforward adaptations of \cite{carron}.

\begin{proof}[Proof of Lemma \ref{bd-spec-pos}]
We first prove that
 \begin{equation}
 \inf\sigma_{\ess}(\Delta_{\rho})\geq\lambda_0>0,\label{low-bd-ess-spec}
 \end{equation}
where $\sigma_{\ess}(-\Delta_{\rho})$ denotes the essential $L^2(e^{\rho}d\mu_g)$-spectrum of the operator $-\Delta_{\rho}$.

To this end, let $\psi$ be a smooth function on $M$ with compact support contained in $M\setminus K$ and let $\phi_0$ be as in the statement of the lemma.
Then writing $e^{\rho}d\mu_g=:d\mu_{\rho}$, we have that
 \begin{equation*}
 \begin{split}
\int_M\arrowvert\nabla^g(\phi_0\psi)\arrowvert_g^2\,d\mu_{\rho}&=-\int_M \Delta_{\rho}(\phi_0\psi)(\phi_0\psi)\,d\mu_{\rho}\\
&=-\int_M(\Delta_{\rho}\phi_0)(\phi_0\psi^2)\,d\mu_{\rho}-2\int_M g(\nabla^g\phi_{0},\nabla^g\psi)\phi_0\psi\, d\mu_{\rho}-\int_M(\Delta_{\rho}\psi)(\psi\phi_0^2)\,d\mu_{\rho}\\
&=-\int_M(\Delta_{\rho}\phi_0)(\phi_0\psi^2)\,d\mu_{\rho}-\frac{1}{2}\int_M g(\nabla^g(\psi^2),\nabla^g(\phi_0^2))\,d\mu_{\rho}\\
&\qquad+\int_M g(\nabla^g\psi,\nabla^g(\psi\phi_0^2))\, d\mu_{\rho}\\
&=-\int_M(\Delta_{\rho}\phi_0)(\phi_0\psi^2)\,d\mu_{\rho}+\int_M\arrowvert\nabla^g\psi\arrowvert_g^2\phi_0^2\,d\mu_{\rho}\\
&\geq-\int_M(\Delta_{\rho}\phi_0)(\phi_0\psi^2)\,d\mu_{\rho}\\
&\geq\lambda_0\int_M(\phi_0\psi)^2\,d\mu_{\rho}.
\end{split}
\end{equation*}
Since $\phi_0$ is positive, the previous estimate implies that
\begin{equation*}
\inf\left\{\left.\frac{\|\nabla^g\phi\|^2_{L^2(d\mu_{\rho})}}{\|\phi\|^2_{L^2(d\mu_{\rho})}}\,\right|\,
\phi\in L^2(d\mu_{\rho})\setminus\{0\}\quad\textrm{and}\quad\supp(\phi)\subset M\setminus K\right\}\geq\lambda_0.
\end{equation*}
A straightforward adaptation of \cite[Chapter 2]{Agm-Boo} then yields
the expected lower bound \eqref{low-bd-ess-spec} on $\sigma_{\ess}(\Delta_{\rho})$.

Now, the operator $-\Delta_{\rho}$, being non-negative,
has spectrum $\sigma(-\Delta_{\rho})\subset[0,+\infty)$.
As a result of \eqref{low-bd-ess-spec}, proving $\inf\sigma(-\Delta_{\rho})>0$
is therefore equivalent to showing that $\inf\sigma_{\dis}(-\Delta_{\rho})>0$, where
$\sigma_{\dis}(-\Delta_{\rho})$ denotes the discrete $L^2(e^{\rho}d\mu_g)$-spectrum of $-\Delta_{\rho}$.
Suppose, for sake of a contradiction, that $\inf\sigma_{\dis}(-\Delta_{\rho})=0$.
Then there exists a non-zero function $\phi\in L^2(d\mu_{\rho})$ such that $\Delta_{\rho}\phi=0$.
By a straightforward adaptation of Yau's Liouville theorem \cite[Lemma 7.1]{Li-Pet-Boo}, one
 arrives at a contradiction with the fact that $\int_M d\mu_{\rho}=+\infty$.
\end{proof}

From this, we obtain a lower bound on the spectrum of the drift Laplacian of a Riemannian metric equal to a steady gradient Ricci soliton at infinity.
\begin{prop}\label{Mun-Wan-Har-Sgs}
Let $(M,\,\tilde{g})$ be a one-ended complete Riemannian manifold with infinite volume and
with $\lim_{x\rightarrow+\infty}\RR_{\tilde{g}}=0$
endowed with a smooth proper positive function $\tilde{f}:M\to\mathbb{R}$
such that $\lim_{x\rightarrow+\infty}|\nabla^{\tilde{g}}\tilde{f}|^2_{\tilde{g}}=c(\tilde{g})>0$ for some positive constant $c(\tilde{g})$.
Assume that there exists a compact subset $K\subset M$ and a vector field $X$
on $M\setminus K$ such that $\Ric(\tilde{g})=\frac{1}{2}\mathcal{L}_{\nabla^{\tilde{g}}\tilde{f}}\tilde{g}$ on $M\setminus K$,
i.e.,\linebreak $(M\setminus K,\,\tilde{g},\,X=\nabla^{\tilde{g}}\tilde{f})$ is a steady gradient Ricci soliton. Then
\begin{equation*}
\inf\sigma\left(-\Delta_{\tilde{f}-\alpha\log \tilde{f}}\right)>0
\end{equation*}
for any $\alpha\in\R$.
\end{prop}

Recall that the constant $c(\tilde{g})$ can be interpreted as the ``charge'' at infinity of the (incomplete) steady gradient Ricci
soliton $(M\setminus K,\,\tilde{g},\,X)$; cf.~Lemma \ref{solitonid}.

\begin{remark}
The assumptions made in Proposition \ref{Mun-Wan-Har-Sgs} on the scalar curvature $\RR_{\tilde{g}}$ and $|\nabla^{\tilde{g}}\tilde{f}|^2_{\tilde{g}}$ are not optimal.
One would arrive at the same conclusion by assuming that
$\liminf_{x\rightarrow+\infty}\RR_{\tilde{g}}\geq 0$ and $\lim_{x\rightarrow+\infty}\left(\RR_{\tilde{g}}+|\nabla^{\tilde{g}}\tilde{f}|^2_{\tilde{g}}\right)
=c(\tilde{g})>0$.
\end{remark}

\begin{proof}[Proof of Proposition \ref{Mun-Wan-Har-Sgs}]
Applying Lemma \ref{bd-spec-pos} to $\rho:=\tilde{f}-\alpha\log \tilde{f}$, it suffices
to find a positive smooth function $\phi_0$ with the property that
$\Delta_{\tilde{f}-\alpha \log \tilde{f}}\phi_0\leq -\lambda_0\phi_0$ on $M\setminus K$ for some $\lambda_0>0$ and some compact subset $K\subset M$.
To this end, first observe from the trace version of the Bianchi identity that
\begin{equation*}
2\Div_{\tilde{g}}\Ric(\tilde{g})=\nabla^{\tilde{g}}\RR_{\tilde{g}}.
\end{equation*}
On the other hand, since the steady Ricci soliton equation holds on $M\setminus K$, we have that
\begin{equation*}
\begin{split}
2\Div_{\tilde{g}}\Ric(\tilde{g})&=\Div_{\tilde{g}}\mathcal{L}_{\nabla^{\tilde{g}} \tilde{f}}(\tilde{g})\\
&=\frac{\nabla^{\tilde{g}}\tr_{\tilde{g}}\mathcal{L}_{\nabla^{\tilde{g}} \tilde{f}}(\tilde{g})}{2}+\Delta_{\tilde{g}}\nabla^{\tilde{g}}\tilde{f}+\Ric(\tilde{g})(\nabla^{\tilde{g}}\tilde{f})\\
&=\nabla^{\tilde{g}}\Delta_{\tilde{g}}\tilde{f}+\Delta_{\tilde{g}}\nabla^{\tilde{g}}\tilde{f}+\Ric(\tilde{g})(\nabla^{\tilde{g}}\tilde{f})\\
&=2\nabla^{\tilde{g}}\Delta_{\tilde{g}}\tilde{f}+2\Ric(\tilde{g})(\nabla^{\tilde{g}}\tilde{f})\\
&=2\nabla^{\tilde{g}}\RR_{\tilde{g}}+2\Ric(\tilde{g})(\nabla^{\tilde{g}}\tilde{f})\\
&=2\nabla^{\tilde{g}}\RR_{\tilde{g}}+(\mathcal{L}_{\nabla^{\tilde{g}}\tilde{f}}\tilde{g})(\nabla^{\tilde{g}}\tilde{f})\\
&=\nabla^{\tilde{g}}(2\RR_{\tilde{g}}+|\nabla^{\tilde{g}}\tilde{f}|^2_{\tilde{g}}).
\end{split}
\end{equation*}
Here we have used the soliton identity $\RR_{\tilde{g}}=\Delta_{\tilde{g}}\tilde{f}$ in the fifth line obtained by tracing
the steady Ricci soliton equation, together with the Bochner formula in the third line. It follows that
$\nabla^{\tilde{g}}\left(\RR_{\tilde{g}}+|\nabla^{\tilde{g}}\tilde{f}|^{2}_{\tilde{g}}\right)=0$ on $M\setminus K$ so that by connectedness of this set,
$\RR_{\tilde{g}}+|\nabla^{\tilde{g}}\tilde{f}|^{2}_{\tilde{g}}$ is constant on $M\setminus K$. By assumption, we then find that
\begin{equation}
\begin{split}\label{first-order-sol-steady}
\RR_{\tilde{g}}+|\nabla^{\tilde{g}}\tilde{f}|^{2}_{\tilde{g}}=\lim_{x\rightarrow+\infty}\left(\RR_{\tilde{g}}+|\nabla^{\tilde{g}}\tilde{f}|^{2}_{\tilde{g}}\right)=c(\tilde{g})>0\qquad \text{on $M\setminus K$.}
\end{split}
\end{equation}

We define the function $\phi_0:=e^{-\beta \tilde{f}}$ for $\beta\in(0,\,1)$. Then, making use of \eqref{first-order-sol-steady}, we see that
$\phi_0$ is a positive smooth function satisfying
\begin{equation*}
\begin{split}
\Delta_{\tilde{f}-\alpha\log \tilde{f}}\phi_0&=-\beta\left(\Delta_{\tilde{g}}\tilde{f}+\left(1-\beta-\frac{\alpha}{\tilde{f}}\right)\arrowvert\nabla^{\tilde{g}} \tilde{f}\arrowvert^2_{\tilde{g}}\right)e^{-\beta \tilde{f}}\\
&=-\beta\left(\left(\beta+\frac{\alpha}{\tilde{f}}\right) \RR_{\tilde{g}}+\left(1-\beta-\frac{\alpha}{\tilde{f}}\right)c(\tilde{g})\right)\phi_0.
\end{split}
\end{equation*}
Since $\RR_{\tilde{g}}$ tends to $0$ at infinity, we deduce that
\begin{equation*}
\begin{split}
\Delta_{\tilde{f}-\alpha\log \tilde{f}}\phi_0\leq -\frac{\beta(1-\beta) c(\tilde{g})}{2}\phi_0
\end{split}
\end{equation*}
outside a compact subset of $M$. An application of Lemma \ref{bd-spec-pos} to $\phi_0:=e^{-\frac{\tilde{f}}{2}}$ and $\lambda_0:=\frac{c(\tilde{g})}{8}>0$ now yields the result.
\end{proof}

\begin{remark}
For a complete steady gradient Ricci soliton $(M,\,\tilde{g},\,X=\nabla^{\tilde{g}}\tilde{f})$,
one can show that $\inf\sigma(-\Delta_{\tilde{f}})\geq\frac{c(\tilde{g})}{4}$ in the notation of Proposition \ref{Mun-Wan-Har-Sgs}; see \cite{Mun-Wan-Non-Neg} for a proof. The proof of this proposition can also be refined to show that actually
$\inf\sigma_{\ess}\left(-\Delta_{\tilde{f}-\alpha\log\tilde{f}}\right)\geq\frac{c(\tilde{g})}{4}$ for any $\alpha\in\mathbb{R}$. We do not require this fact here.
\end{remark}

From Proposition \ref{Mun-Wan-Har-Sgs}, we obtain the following corollary that will prove useful in
establishing an a priori weighted energy estimate for the complex Monge-Amp\`ere equation \eqref{e:soliton}.

\begin{corollary}\label{tarea}
Let $(M,\,\tilde{g},\,\tilde{f})$ be as in Proposition \ref{Mun-Wan-Har-Sgs} and let $h$ be a complete Riemannian metric on $M$ uniformly equivalent to $\tilde{g}$
such that $X=\nabla^hf$ for some smooth proper positive function $f:M\to\mathbb{R}$ with $|f-\tilde{f}|=O(1)$. Then the bottom of the $L^2(f^{-\alpha}e^fd\mu_h)$-spectrum of the operator $-\Delta_{f-\alpha\log f}$ corresponding to $h$ is positive.
\end{corollary}

The requirement for the difference between the potentials $f$ and $\tilde{f}$ of $X$ to be bounded is so that the respective weighted volume forms are uniformly comparable.

\begin{proof}[Proof of Corollary \ref{tarea}]
Let $d\mu_{\tilde{g}}$ and $d\mu_{h}$ denote the volume forms of $\tilde{g}$ and $h$ respectively. Then by assumption, there exists a positive constant $C$ such that $C^{-1}\tilde{g}\leq h\leq C\tilde{g}$ and such that
\begin{equation*}
C^{-1}\tilde{f}^{-\alpha}e^{\tilde{f}}\leq f^{-\alpha}e^f\left(\frac{d\mu_{h}}{d\mu_{\tilde{g}}}\right)\leq C\tilde{f}^{-\alpha}e^{\tilde{f}}\qquad\textrm{on $M$}.
\end{equation*}
In particular, we have that
\begin{equation*}
\begin{split}
\inf\Bigg\{\Bigg.&\frac{\|\nabla^h\phi\|^2_{L^2(f^{-\alpha}e^fd\mu_{h})}}{\|\phi\|^2_{L^2(f^{-\alpha}e^fd\mu_{h})}}\,\Bigg|\,\phi\in L^2(f^{-\alpha}e^fd\mu_{h})
\setminus\{0\}\Bigg\}\\
&\geq C^{-3}\inf\left\{
\left.\frac{\|\nabla^{\tilde{g}}\phi\|^2_{L^2(\tilde{f}^{-\alpha}e^{\tilde{f}}d\mu_{\tilde{g}})}}{\|\phi\|^2_{L^2(\tilde{f}^{-\alpha}e^{\tilde{f}}d\mu_{\tilde{g}})}}\,\right|\,\phi\in L^2(\tilde{f}^{-\alpha}e^{\tilde{f}}d\mu_{\tilde{g}})\setminus\{0\}\right\}.
\end{split}
\end{equation*}
The result now follows from Proposition \ref{Mun-Wan-Har-Sgs}.
\end{proof}

\section{Invertibility of the drift Laplacian: exponential case}\label{section-small-def-exp}

In this section, we introduce the exponentially weighted function spaces in which we shall work in order to solve
the complex Monge-Amp\`ere equation \eqref{e:soliton} with compactly supported data. We also analyse various properties
of the drift Laplacian acting on such spaces. We begin by recalling the set-up.

\subsection{Main setting}\label{main-set-sec-pol}
Let $(C_0,\,g_{0},\,J_0,\,\Omega_0)$ be a Calabi-Yau cone of complex dimension $n\geq2$ with radial function $r$.
Set $r^2=:e^t$ and let $\pi:M\to C_{0}$ be an equivariant crepant resolution of $C_{0}$
with respect to the real holomorphic torus action on $C_{0}$ generated by $J_{0}r\partial_{r}$
so that the holomorphic vector field $2r\partial_{r}=4\partial_{t}$ on $C_{0}$ lifts to a real holomorphic vector field $X=\pi^{*}(2r\partial_{r})$
on $M$. Denote by $E$ the exceptional set of the resolution $\pi:M\to C_{0}$ and let $J$ denote the complex structure on $M$.
Throughout, using $\pi$, we identify $M$ and $C_{0}$ on the complement of compact subsets of each containing $E$ and the apex of the cone respectively.

We define a K\"ahler form $\hat{\omega}$ on $C_{0}$ by
$$\hat{\omega}:=\frac{i}{2}\partial\bar{\partial}\left(\frac{nt^{2}}{2}\right)
=n\left(\frac{dt}{2}\wedge\eta+t\omega^{T}\right),$$
where $\omega^{T}$ is the transverse K\"ahler form on $(C_{0},\,g_{0})$. Then the K\"ahler metric $\hat{g}$ associated to $\hat{\omega}$ takes the form
$$\hat{g}=n\left(\frac{1}{4}dt^{2}+\eta^{2}+tg^{T}\right),$$
where $g^{T}$ is the transverse K\"ahler metric associated to $\omega^{T}$ and $\eta=d^{c}\log(r)$ is
a contact form on the link of the cone. Recall that
$$\omega^{T}=\frac{1}{2}d\eta=\frac{1}{2}dd^{c}\log(r)=\frac{1}{4}dd^{c}t=\frac{i}{2}\partial\bar{\partial}t$$
and observe that $-\hat{\omega}\lrcorner JX=d(nt)$. We extend $nt$ to a smooth real-valued function
$\hat{f}:M\to\mathbb{R}$ on $M$ with $\hat{f}\geq1$. Then by definition, $-\hat{\omega}\lrcorner JX=d\hat{f}$ along the end of $C_{0}$.
We also have the following expression for the Riemannian Laplacian with respect to $\hat{g}$ acting on $u\in C^2_{\operatorname{loc}}(\{t>0\})$:
\begin{equation}\label{covid20}
\Delta_{\hat{g}}u=2\Delta_{\hat{\omega}}u=\frac{4}{n}\frac{\partial^{2}u}{\partial t^{2}}+\frac{4(n-1)}{nt}\frac{\partial u}{\partial t}
+\frac{\xi(\xi u)}{n}+\frac{1}{nt}\Delta_{B}u,
\end{equation}
where $\Delta_{B}$ denotes the basic Laplacian on the link of $(C_{0},\,g_{0})$.

Thanks to the Ricci-flatness of $g_{0}$, we have, via the Cao ansatz, a steady gradient K\"ahler-Ricci soliton $\tilde{\omega}$ on
$C_{0}$ with soliton potential $\varphi(t)$ and with $\mathcal{L}_{JX}\tilde{\omega}=0$ that satisfies
\begin{equation*}
|\widehat{\nabla}^{i}\mathcal{L}^{(j)}_X\left(\tilde{\omega}-\hat{\omega}\right)|
_{\hat{g}}=O\left(t^{-\varepsilon-\frac{i}{2}-j}\right)\qquad\textrm{for all $\varepsilon\in(0,\,1)$ and $i,\,j\geq 0$,}
\end{equation*}
where $\widehat{\nabla}$ denotes the Levi-Civita connection of $\hat{g}$. These asymptotics are contained in the statement of
Proposition \ref{coro-asy-Cao-met}. Let $\tau$ be any K\"ahler form on $M$ with $\mathcal{L}_{JX}\tau=0$
such that for some $\varepsilon\in(0,\,1)$,
\begin{equation}\label{covid19}
|\widehat{\nabla}^{i}\mathcal{L}^{(j)}_X\left(\pi_{*}\tau-\hat{\omega}\right)|_{\hat{g}}=O\left(t^{-\varepsilon-\frac{i}{2}-j}\right)\qquad\textrm{for all $i,\,j\geq 0$}.
\end{equation}
We denote by $h$ the K\"ahler metric associated to $\tau$ and by $\nabla^{h}$ its Levi-Civita connection.
Moreover, for any smooth real-valued function $\phi\in C^{\infty}(M)$
such that $\tau+i\partial\bar{\partial}\phi>0$, we
write $\tau_{\phi}:=\tau+i\partial\bar{\partial}\phi$ and denote by $h_{\phi}$ the K\"ahler metric associated to $\tau_{\phi}$.
Since $\hat{g}$ and $h$ are asymptotic with derivatives, one can verify that
measuring the asymptotics of a tensor using either metric is equivalent--that is to say, along the end of $M$,
there exist constants $C_{i,\,j}>0$ such that for every tensor $T$ on $M$,
$$C_{i,\,j}^{-1}|\widehat{\nabla}^{i}\mathcal{L}_{X}^{(j)}T|_{\hat{g}}\leq
|(\nabla^{h})^{i}\mathcal{L}_{X}^{(j)}T|_{h}\leq C_{i,\,j}|\widehat{\nabla}^{i}\mathcal{L}_{X}^{(j)}T|_{\hat{g}}
\qquad\textrm{for all $i,\,j\geq0$}.$$
In what follows, we shall use this fact without further reference.

We first note that $X$ is gradient with respect to $h$.
\begin{lemma}\label{easy}
There exists a smooth proper real-valued function $f:M\to\mathbb{R}$ bounded from below such that $X=\nabla^{h}f$.
\end{lemma}

\begin{proof}
It suffices to show that $-\tau\lrcorner JX=df$ for a smooth real-valued function $f:M\to\mathbb{R}$ with the desired properties.
To this end, observe that
$JX$ is Killing for $h$ and holomorphic so that $0=\mathcal{L}_{JX}\tau=d(\tau\lrcorner JX)$.
A smooth function $f:M\to\mathbb{R}$ with $-\tau\lrcorner JX=df$ therefore exists by Lemma \ref{finite}. To see that $f$ is proper and bounded from below, just note that since  $\hat{g}$ and $h$ are asymptotic along the end of $C_{0}$ and $|X|^{2}_{\hat{g}}=4n$, $|X|^{2}_{h}$ is asymptotic to $4n$ so that $f(x)\to+\infty$ as $x\to+\infty$.
\end{proof}

\begin{remark}
The function $f$ from Lemma \ref{easy} is defined up to a constant. We henceforth fix this constant so that $f\geq1$ on $M$.
\end{remark}

As the next lemma shows, both $f$ and $\hat{f}$ are comparable.

\begin{lemma}\label{tuxedo}
There exists a positive constant $C$ such that $C^{-1}\hat{f}\leq f\leq C\hat{f}$ on $M$. In particular,
$f\sim nt$.
\end{lemma}

\begin{proof}
Let $x\in M\setminus E$ and let $\gamma_{x}(t)$ denote the integral curve of $X$ with $\gamma_{x}(0)=x$. Then
\begin{equation*}
\begin{split}
(f(\gamma_{x}(s))-\hat{f}(\gamma_{x}(s)))-(f(x)-\hat{f}(x))&=\int_{0}^{s}\frac{d}{du}(f(\gamma_{x}(u))-\hat{f}(\gamma_{x}(u)))\,du\\
&=\int_{0}^{s}(X\cdot(f-\hat{f}))(\gamma_{x}(u))\,du\\
&=\int_{0}^{s}(|X|^{2}_{h}-|X|_{\hat{g}}^{2})(\gamma_{x}(u))\,du.
\end{split}
\end{equation*}
Now, the fact that $\frac{d}{du}t(\gamma_{x}(u))=dt(X)=4$ implies that $t(\gamma_{x}(u))=4u+t(x)$, hence it follows from \eqref{covid19} that
\begin{equation*}
\begin{split}
|(f(\gamma_{x}(s))-\hat{f}(\gamma_{x}(s)))-(f(x)-\hat{f}(x))|&\leq C\int_{0}^{s}t(\gamma_{x}(u))^{-\varepsilon}\,du\\
&\leq C\int_{0}^{s}(4u+t(x))^{-\varepsilon}\,du\\
&\leq C\left((4s+t(x))^{1-\varepsilon}-t(x)^{1-\varepsilon}\right)\\
&\leq C(t(\gamma_{x}(s))^{1-\varepsilon}-t(x)^{1-\varepsilon}),
\end{split}
\end{equation*}
and correspondingly, that for all $x,\,y,$ lying on the same flow-line of $X$ along the end of $C_{0}$,
$$|(f(y)-\hat{f}(y))-(f(x)-\hat{f}(x))|\leq C|t(y)^{1-\varepsilon}-t(x)^{1-\varepsilon}|
\leq C|\hat{f}(y)^{1-\varepsilon}-\hat{f}(x)^{1-\varepsilon}|.$$
From this, the result is clear.
\end{proof}

As a consequence of Lemma \ref{tuxedo}, it makes no difference whether one measures \emph{polynomial} rates of growth and decay using $\hat{f}$ or $f$. However, in contrast to $\hat{f}$, $f$ is a globally defined potential for $X$ on $M$. The need for such a function becomes apparent in
the proof of Lemma \ref{lemma-inf-pot-fct}, hence we work with $f$ rather than $\hat{f}$. Note that Lemma \ref{tuxedo} does not imply that $e^{\hat{f}}$ is comparable
to $e^{f}$; in fact, this is not true, and so there is a difference in using $\hat{f}$ and $f$
when measuring \emph{exponential} rates of growth and decay. These are the rates
that we will primarily be dealing with in this section and the next. For this reason, we assume in these sections that in addition to the above,
\begin{equation}\label{add-ass}
\addlefttext{\textbf{Additional assumption:}}{|f-\varphi(t)|=O(1),}
\end{equation}
so that the exponential weights $e^{\tilde{f}}$ and $e^{f}$ are comparable. Notice that this condition does not follow automatically from \eqref{covid19}.
We will use this assumption in this section specifically in Theorem \ref{iso-sch-Laplacian-exp} in the deriviation of the estimates for the drift Laplacian acting between exponentially weighted function spaces, where we must appeal to Corollary \ref{tarea}.

We next state a crucial lemma that will enable us to build good barrier functions at infinity, thereby allowing us to obtain suitable a priori estimates.
This lemma, which mirrors Lemma \ref{solitonid}, can be proved using the estimates from Proposition \ref{basis}.
\begin{lemma}\label{lemm-app-sol-id}
In the above situation, the following asymptotics hold true:
\begin{equation*}
\begin{split}
&|(\nabla^{h})^{k}(f-t)|_{h}=O\left(f^{-\varepsilon-\frac{(k-1)}{2}}\right)\qquad\textrm{for all $k\geq 1$},\\
&\left|\widehat{\nabla}^{i}\mathcal{L}^{(j)}_X\left(\mathcal{L}_{X}\tau\right)\right|_{\hat{g}}=O\left(f^{-1-\frac{i}{2}-j}\right)\qquad\textrm{for all $i,\,j\geq0$,}\\
&\left|\widehat{\nabla}^{i}\mathcal{L}^{(j)}_X\left(\rho_{\tau}-\frac{1}{2}\mathcal{L}_{X}\tau\right)\right|_{\hat{g}}=
O(f^{-1-\varepsilon-\frac{i}{2}-j})\qquad\textrm{for all $i,\,j\geq0$,} \\
&\left|\widehat{\nabla}^{i}\mathcal{L}^{(j)}_X\left(|\nabla^{h}f|^2_h+\RR_{h}-4n\right)\right|_{\hat{g}}=O\left(f^{-\varepsilon
-\frac{i}{2}-j}\right)\qquad\textrm{for all $i,\,j\geq0$,}\\
&\left|\widehat{\nabla}^{i}\mathcal{L}^{(j)}_X\left(\Delta_hf+X\cdot f-4n\right)\right|_{\hat{g}}=O\left(f^{-\varepsilon-\frac{i}{2}-j}\right)
\qquad\textrm{for all $i,\,j\geq0$}.
\end{split}
\end{equation*}
Here, $\rho_{\tau}$ and $\RR_{h}$ denote the Ricci form of $\tau$ and scalar curvature of $h$ respectively, and
$\varepsilon\in(0,\,1)$ is as in \eqref{covid19}.
\end{lemma}

\begin{remark}
In the terminology of Section \ref{section-fct-spa-pol}, these last four estimates can equivalently be written as
\begin{equation*}
\begin{split}
&\mathcal{L}_{X}\tau \in C_{X,\,1}^{\infty}(M),\qquad \rho_{\tau}-\frac{1}{2}\mathcal{L}_{X}\tau\in C_{X,\,1+\varepsilon}^{\infty}(M),\\
&|\nabla^{h}f|^2_h+\RR_{h}-4n\in C_{X,\,\varepsilon}^{\infty}(M),\qquad\textrm{and}\qquad\Delta_hf+X\cdot f-4n\in C_{X,\,\varepsilon}^{\infty}(M),
\end{split}
\end{equation*}
respectively.
\end{remark}

\begin{proof}[Proof of Lemma \ref{lemm-app-sol-id}]
We prove only the second and third estimate. The others can be proved in a similar manner. Regarding the second estimate, using \eqref{covid19},
Lemma \ref{tuxedo}, and Proposition \ref{basis}, and with $\varepsilon\in(0,\,1)$ as in \eqref{covid19}, we have the bounds
\begin{equation*}
\begin{split}
\left|\widehat{\nabla}^{i}\mathcal{L}^{(j)}_X\left(\mathcal{L}_{X}\tau\right)\right|_{\hat{g}}&\leq \left|\widehat{\nabla}^{i}\mathcal{L}^{(j)}_X\left(\mathcal{L}_{X}\hat{\omega}\right)\right|_{\hat{g}}+\left|\widehat{\nabla}^{i}\mathcal{L}^{(j)}_X\left(\mathcal{L}_{X}(\tau-\hat{\omega})\right)\right|_{\hat{g}}\\
&\leq \left|\widehat{\nabla}^{i}\mathcal{L}^{(j)}_X\left(\mathcal{L}_{X}\hat{\omega}\right)\right|_{\hat{g}}+C(i,\,j,\,\varepsilon)f^{-\varepsilon-1-\frac{i}{2}-j}\\
&\leq C\left|\widehat{\nabla}^{i}\mathcal{L}^{(j)}_{X}\omega^{T}\right|_{\hat{g}}+C(i,\,j,\,\varepsilon)f^{-\varepsilon-1-\frac{i}{2}-j}\\
&\leq Cf^{-1-\frac{i}{2}-j},
\end{split}
\end{equation*}
as stated.

As for the third estimate, this encodes the obstruction for $\tau$ to be a steady gradient K\"ahler-Ricci soliton. The existence of such a soliton
on $C_{0}$ is crucial in order for this bound to hold true. Without this, the decay rate would be linear rather than faster than linear with only the latter being
sufficient for us to solve the complex Monge-Amp\`ere equation \eqref{e:soliton}. Regarding this bound, simply observe that for all $i,\,j\geq 0$,
\begin{equation*}
\begin{split}
\widehat{\nabla}^{i}\mathcal{L}^{(j)}_X\left(\rho_{\tau}-\frac{1}{2}\mathcal{L}_{X}\tau\right)&=\widehat{\nabla}^{i}\mathcal{L}^{(j)}_X\left(\rho_{\tau}-\rho_{\tilde{\omega}}+\frac{1}{2}\mathcal{L}_{X}(\tau-\tilde{\omega})\right)\\
&=\widehat{\nabla}^{i}\mathcal{L}^{(j)}_X\left(-i\partial\bar{\partial}\left(\log\left(\frac{\tau^n}{\tilde{\omega}^n}\right)\right)+\frac{1}{2}\mathcal{L}_{X}(\tau-\tilde{\omega})\right)\\
&=O\left(f^{-1-\varepsilon-\frac{i}{2}-j}\right)
\end{split}
\end{equation*}
outside a compact subset of $M$, where we have used Proposition \ref{coro-asy-Cao-met} together with
\eqref{covid19} in the final line
and the fact that $\tilde{\omega}$ is a steady gradient K\"ahler-Ricci soliton on $C_{0}$ in the first line.
\end{proof}

Using Lemma \ref{lemm-app-sol-id}, we derive the following properties of the drift Laplacian acting on exponential weights.
\begin{lemma}\label{lemm-barr-inf}
In the above situation, let $f:M\rightarrow\R$ be a smooth proper real-valued function satisfying $X=\nabla^{h}f$
chosen such that $f\geq1$ on $M$ (which exists by Lemma \ref{easy}).
Then for any $\delta>0$, the function $e^{-\delta f}$ is both a sub- and super-solution of the following equation:
\begin{equation*}
\begin{split}
\left(\Delta_{\tau}+\frac{X}{2}\cdot\right )e^{-\delta f}&=-\delta(1-\delta)\frac{4n}{2}e^{-\delta f}+O(f^{-1})e^{-\delta f}.
\end{split}
\end{equation*}
Moreover, the Laplacian of $f$ with respect to $\tau$ is asymptotically positive and satisfies
\begin{equation*}
\Delta_{\tau}f\geq \frac{c}{f}\qquad\textrm{outside a compact subset of $M$}
\end{equation*}
for some constant $c>0$. In particular,
\begin{equation}
\begin{split}\label{sub-sol-exp-pot-fct}
\left(\Delta_{\tau}+\frac{X}{2}\cdot\right )e^{- f}&\leq-\frac{c}{f}e^{-f}\qquad\textrm{outside a compact subset of $M$}
\end{split}
\end{equation}
for another constant $c>0$.
\end{lemma}

\begin{proof}
We begin with the following computation:
\begin{equation*}
\begin{split}
\left(\Delta_{\tau}+\frac{X}{2}\cdot\right)e^{-\delta f}&=\left(-\delta\left(\Delta_{\tau}+\frac{X}{2}\cdot\right)f+\delta^2\frac{|X|^2_{h}}{2}\right)e^{-\delta f}\\
&=-\delta\left((1-\delta)\frac{|X|^2_{h}}{2}+\Delta_{\tau}f\right)e^{-\delta f}.
\end{split}
\end{equation*}
In particular, notice that for $\delta\in(0,\,1)$,
\begin{equation*}
\begin{split}
\left(\Delta_{\tau}+\frac{X}{2}\cdot\right)e^{-\delta f}&=-\delta\left((1-\delta)\frac{4n}{2}+O(f^{-1})+\Delta_{\tau}f\right)e^{-\delta f}\\
&=-\delta(1-\delta)\frac{4n}{2}e^{-\delta f}+O(f^{-1})e^{-\delta f}
\end{split}
\end{equation*}
by Lemma \ref{lemm-app-sol-id}. If $\delta=1$, then we have that
\begin{equation*}
\begin{split}
\left(\Delta_{\tau}+\frac{X}{2}\cdot\right)e^{- f}&=-\Delta_{\tau}f\cdot e^{-f}.
\end{split}
\end{equation*}
Finally, using equations \eqref{covid20} and \eqref{covid19} together with Lemma \ref{tuxedo}, we find that
\begin{equation*}
\begin{split}
\Delta_{\tau}f&=\Delta_{\hat{\omega}}f+(\Delta_{\tau}-\Delta_{\hat{\omega}})f\\
&=\Delta_{\hat{\omega}}(f-\hat{f})+\underbrace{\Delta_{\hat{\omega}}\hat{f}}_{=\,\frac{4(n-1)}{t}}+(\Delta_{\tau}-\Delta_{\hat{\omega}})f\\
&=\frac{4(n-1)}{t}+\Delta_{\hat{\omega}}(f-\hat{f})+(\Delta_{\tau}-\Delta_{\hat{\omega}})f\\
&=\frac{4(n-1)}{t}+\hat{\omega}\ast i\partial\bar{\partial}(f-\hat{f})+(\tau-\hat{\omega})\ast i\partial\bar{\partial}f\\
&=\frac{4(n-1)}{t}+\frac{1}{2}\underbrace{
\hat{\omega}\ast\mathcal{L}_{X}(\tau-\hat{\omega})}_{=\,O\left(t^{-1-\varepsilon}\right)}+\frac{1}{2}(\tau-\hat{\omega})\ast\mathcal{L}_{X}\tau\\
&=\frac{4(n-1)}{t}+O(t^{-1-\varepsilon})+\frac{1}{2}\underbrace{(\tau-\hat{\omega})\ast\mathcal{L}_{X}(\tau-\hat{\omega})}_{=\,O\left(t^{-1-2\varepsilon}\right)}
+\frac{1}{2}\underbrace{(\tau-\hat{\omega})\ast\mathcal{L}_{X}\hat{\omega}}_{=\,O\left(t^{-1-\varepsilon}\right)}\\
&\geq\frac{c}{t}\geq\frac{c}{f}
\end{split}
\end{equation*}
outside a compact subset of $M$.
\end{proof}

\subsection{Function spaces}\label{section-fct-spa-exp}
We make the following definitions.
\begin{itemize}
\item The \textit{drift Laplacian} (\emph{with respect to $X$}) is defined as
$$\Delta_{\tau,\,X}T:=\Delta_{\tau} T+\nabla^h_{X}T,$$
where $T$ is a tensor on $M$, $\nabla^h$ is the complex linear extension of the Levi-Civita connection of $h$, and
$\Delta_{\tau}$ denotes the Laplacian associated to $h$. In normal coordinates, $\Delta_{\tau}$
takes the form $$\Delta_{\tau}:=\frac{1}{2}\left(\nabla^h_{i}\nabla^h_{\bar{\imath}}+\nabla^h_{\bar{\imath}}\nabla^h_i\right).$$
Recall that the Laplacian acting on functions is given by
\begin{equation*}
\begin{split}
\Delta_{\tau}u&=h^{i\bar{\jmath}}\partial_i\partial_{\bar{\jmath}}u=\tr_{\tau}\left(\frac{i}{2}\partial\bar{\partial}u\right)
\end{split}
\end{equation*}
for $u\in C^{\infty}(M)$ a smooth real-valued function on $M$. Here, the trace operator $\tr_{\tau}$ on $(1,1)$-forms is defined by
\begin{equation*}
\begin{split}
\tr_{\tau}(\alpha):=\frac{n\tau^{n-1}\wedge\alpha}{\tau^n}=h^{i\bar{\jmath}}\alpha_{i\bar{\jmath}},
\end{split}
\end{equation*}
where $\alpha=\frac{i}{2}\alpha_{j\bar{k}}dz^j\wedge dz^{\bar{k}}$ is a $(1,1)$-form on $M$.
\end{itemize}
For clarity, we will omit the reference to the background K\"ahler metric $h$ or to the associated K\"ahler form $\tau$
when there is no possibility of confusion.

\begin{itemize}
\item For $\beta\in\R$ and $k$ a non-negative integer, define $C_{X}^{2k}(M)$ to be the space of $JX$-invariant continuous functions $u$ on $M$ with $2k$ continuous derivatives such that
\begin{equation*}
\norm{u}_{C^{2k}_{X}} :=\sum_{i+2j\,\leq\,2k}\sup_{M}\left|f^{\frac{i}{2}+j}(\nabla^{h})^i\left(\mathcal{L}_{X}^{(j)}u\right)\right|_{h} < \infty,
\end{equation*}
where $$\mathcal{L}_{X}^{(j)}u:=\underbrace{X\cdot...\cdot X\cdot}_{\text{$j$-times}}u.$$
Set $C_{X}^{\infty}(M):=\bigcap_{k\,\geq\,0}C_{X}^{2k}(M)$.\\

\item Let $\delta(h)$ denote the injectivity radius of $h$, write $d_h(x,\,y)$ for the distance with respect to $h$ between two points $x,\,y\in M$,
and let $\phi^{X}_{t}$ denote the flow of $X$ for time $t$. A tensor $T$ on $M$ is said to be in $C^{0,\,2\alpha}(M)$, $\alpha\in\left(0,\,\frac{1}{2}\right)$, if
 \begin{equation*}
 \begin{split}
\left[ T\right]_{C^{0,\,2\alpha}}:=&\sup_{\substack{x\,\neq\,y\,\in\,M \\d_{h}(x,y)\,<\,\delta(h)}}\left[\min(f(x),f(y))^{\alpha}\frac{\arrowvert T(x)-P_{x,\,y}T(y)\arrowvert_h}{d_h(x,y)^{2\alpha}}\right]\\
&+\sup_{\substack{x\,\in\,M \\ t\,\neq\,s\,\geq\,1}}\left[\min(t,s)^{\alpha}\frac{\arrowvert(\phi^X_t)_{\ast}T(x)-(\widehat{P}_{\phi^{X}_s(x),\,\phi^X_t(x)}
((\phi^X_s)_{\ast}T(x)))\arrowvert_h}{|t-s|^{\alpha}}\right]<+\infty,
\end{split}
\end{equation*}
where $P_{x,\,y}$ denotes parallel transport along the unique geodesic joining $x$ and $y$, and
$\widehat{P}_{\phi^X_s(x),\,\phi^X_t(x)}$ denotes parallel transport along the unique flow-line of $X$ joining $\phi^X_s(x)$ and $\phi^X_t(x)$.

\item For $k$ a non-negative integer and $\alpha\in\left(0,\,\frac{1}{2}\right)$, define the H\"older space $C_{X}^{2k,\,2\alpha}(M)$ to be the set of $u\in C_{X}^{2k}(M)$ for which the norm
\begin{equation*}
\norm{u}_{C_{X}^{2k,\,2\alpha}}:=\norm{u}_{C^{2k}_{X}} +\sum_{i+2j\,=\,2k}\left[\left(\nabla^{h}\right)^i\left(\mathcal{L}_{X}^{(j)}u\right)\right]_{C^{0,\,2\alpha}}
\end{equation*}
is finite.

Similarly, define the H\"older space $C_{X,\,\exp}^{2k,\,2\alpha}(M)$ with exponential weight $e^f$ to be the set of $u\in C_{\operatorname{loc}}^{2k}(M)$ for which the norm
\begin{equation*}
\norm{u}_{C_{X,\,\exp}^{2k,\,2\alpha}}:=\left\|e^f\cdot u\right\|_{C^{2k,\,2\alpha}_X}
\end{equation*}
is finite. It is straightforward to check that the space $C^{2k,\,2\alpha}_{X,\,\exp}(M)$ is a Banach space.
We set $C_{X,\,\exp}^{\infty}(M):=\bigcap_{k\,\geq\,0}C_{X,\,\exp}^{2k}(M).$

\item Finally, we define the spaces
\begin{equation*}
\begin{split}
\mathcal{M}^{2k+2,\,2\alpha}_{X,\,\exp}(M)&:=\left\{\psi\in C^2_{\operatorname{loc}}(M)\,|\,\tau_{\psi}:=\tau+i\partial\bar{\partial}\psi>0\right\}\cap C^{2k+2,\,2\alpha}_{X,\,\exp}(M)
\end{split}
\end{equation*}
and
\begin{equation}
\mathcal{M}^{\infty}_{X,\,\exp}(M):=\left\{\psi\in C_{\operatorname{loc}}^{\infty}(M)\,|\,\tau_{\psi}>0
\quad\textrm{and}\quad\psi\in C^{\infty}_{X,\,\exp}(M)\right\}.\label{defn-M-exp-space}
\end{equation}
\end{itemize}

We remark that in \cite{olivier}, the choice of function spaces differs from ours for the case of compactly supported data in that
they work with much larger function spaces where the functions have exponential decay $e^{-\delta f}$, $\delta\in(0,\,1)$.
Their function spaces do have an advantage over ours; mixed polynomial and exponential weights do not appear in their analysis
of the isomorphism properties of the drift Laplacian as is the case for us in Theorem \ref{iso-sch-Laplacian-exp}.

\subsection{Preliminaries and Fredholm properties of the linearised operator}\label{sec-Fredo-prop-exp}
We proceed with the same set-up as in Section \ref{section-fct-spa-exp}.

Define the following map as in \cite{siepmann}:
\begin{equation*}
\begin{split}
MA_{\tau}:\psi\in&\left\{\varphi\in C^2_{\operatorname{loc}}(M)\,|\,\tau_{\varphi}:=\tau+i\partial\bar{\partial}\varphi>0\right\}\mapsto\log\left(\frac{\tau_{\psi}^n}{\tau^n}\right)
+\frac{X}{2}\cdot\psi\in\mathbb{R}.
\end{split}
\end{equation*}
For any $\psi\in C_{\operatorname{loc}}^{2}(M)$, let
$h_{\psi}$ (respectively $h_{s\psi}$) denote the K\"ahler metric associated to the K\"ahler form $\tau_{\psi}$
(resp.~$\tau_{s\psi}$ for any $s\in[0,\,1]$).
Brute force computations show that
\begin{equation}
\begin{split}
MA_{\tau}(0)&=0,\nonumber\\
 D_{\psi}MA_{\tau}(u)&=\Delta_{\tau_{\psi}}u+\frac{X}{2}\cdot u,\quad u\in C^2_{\operatorname{loc}}(M),\nonumber\\
 \frac{d^2}{ds^2}\left(MA_{\tau}(s\psi)\right)&=\frac{d}{ds}(\Delta_{\tau_{s\psi}}\psi)=-\arrowvert  \partial\bar{\partial}\psi\arrowvert^2_{h_{s\psi}}\quad\textrm{for $s\in[0,\,1]$},\label{equ:sec-der}
 \end{split}
 \end{equation}
\begin{equation}\label{equ:taylor-exp}
 \begin{split}
 MA_{\tau}(\psi)&=MA_{\tau}(0)+\left.\frac{d}{ds}\right|_{s\,=\,0}MA_{\tau}(s\psi)+\int_0^1\int_0^{u}\frac{d^2}{ds^2}(MA_{\tau}(s\psi))\,ds\,du\\
 &=\Delta_{\tau}\psi+\frac{X}{2}\cdot\psi-\int_0^1\int_0^{u}\arrowvert \partial\bar{\partial}\psi\arrowvert^2_{h_{s\psi}}\,ds\,du.
 \end{split}
 \end{equation}

We state the first property of the drift Laplacian that we require,
namely that it is an isomorphism between exponentially weighted function spaces.

\begin{theorem}\label{iso-sch-Laplacian-exp}
Let $\alpha\in\left(0,\,\frac{1}{2}\right)$ and $k\in\mathbb{N}$. Then
\begin{equation*}
\begin{split}
\Delta_{\tau}+\frac{X}{2}\cdot:C^{2k+2,\,2\alpha}_{X,\,\exp}(M)\rightarrow f^{-1}\cdot C^{2k,\,2\alpha}_{X,\,\exp}(M)
\end{split}
\end{equation*}
is an isomorphism of Banach spaces.
\end{theorem}

\begin{proof}
We first prove surjectivity.
Let $F\in f^{-1}\cdot C^{2k,\,2\alpha}_{X,\,\exp}(M)$.
Then for $R>0$ sufficiently large such that the level sets $\{f\,=\,R\}$ are smooth closed hypersurfaces of $M$ (recall that $f$ is proper and bounded from below by Lemma
\ref{easy}), let $u_R:\{f\,\leq\,R\}\rightarrow\mathbb{R}$ be the solution of the following Dirichlet problem:
\begin{equation}\label{dir-pb-drift-lap}
\left\{
      \begin{aligned}
        \Delta_{\tau}u_R+\frac{X}{2}\cdot u_R&=F\qquad\textrm{on $\{f\,<\,R\}$},\\
        u_R&=0\qquad\textrm{on $\{f\,=\,R\}$}.
      \end{aligned}
    \right.
\end{equation}
Applying Corollary \ref{tarea} (with $\alpha=0$) to the measure $e^f\tau^n$ and the function $u_R$, we see that
\begin{equation}
\begin{split}\label{poincare-a-priori-lin-est}
\lambda_1\int_{\{f\,\leq\,R\}}|u_R|^2\,e^f\tau^n&\leq \int_{\{f\,\leq\,R\}}|\nabla^{h}u_R|^2_h\,e^f\tau^n\\
&= 2\int_{\{f\,\leq\,R\}}\left(-\Delta_{\tau}u_R-\frac{X}{2}\cdot u_R\right)u_R\,e^f\tau^n\\
&=2\int_{\{f\,\leq\,R\}}(-F)u_R\,e^f\tau^n
\end{split}
\end{equation}
for some constant $\lambda_1>0$ independent of $R$. Using H\"older's inequality on the right-hand side of \eqref{poincare-a-priori-lin-est}, we obtain
the following a priori energy estimate:
\begin{equation}\label{poincare-a-priori-lin-est-bis}
\int_{\{f\,\leq\,R\}}|u_R|^2\,e^f\tau^n\leq c(n,\,\tau)\int_M|F|^2\,e^f\tau^n,
\end{equation}
the right-hand side of which is finite because $F=O(f^{-1}e^{-f})$ implies that $F\in L^2(e^f\tau^n)$.

Next, let $x\in\{f\,<\,R\}$ be such that $B_{h}(x,\,r)\Subset \{f\,<\,R\}$.
We perform a local Nash-Moser iteration on \eqref{dir-pb-drift-lap} in $B_{h}(x,\,r)$.
More precisely, since $(M^{2n},\,h)$ is a Riemannian manifold with Ricci curvature bounded from below,
the results of \cite{Sal-Cos-Uni-Ell} give the following local Sobolev inequality:
\begin{equation*}
\begin{split}
\left(\frac{1}{\operatorname{vol}_{h} (B_{h}(x,\,r))}\int_{B_{h}(x,\,r)}\arrowvert\varphi\arrowvert^{\frac{2n}{n-1}}\,\tau^{n}\right)^{\frac{n-1}{n}}\leq \left(\frac{C(r_0)r^2}{\operatorname{vol}_{h}(B_{h}(x,\,r))}\int_{B_{h}(x,\,r)}\arrowvert\nabla^{h}\varphi\arrowvert_{h}^2\,\tau^{n}\right)
\end{split}
\end{equation*}
for any $\varphi\in H^1_0(B_{h}(x,\,r))$ and for all $x\in M$ and $0<r<r_0$, where $r_0$ is some fixed positive radius.
Now, as $|X|_{h}$ is bounded, the oscillation of $f$ is bounded on $B_{h}(x,\,r)$ by a constant depending only on $r_0$, and so we have the following local weighted Sobolev inequality:
\begin{equation}
\begin{split}
\left(\frac{1}{\operatorname{vol}_{f}(B_{h}(x,\,r))}\int_{B_{h}(x,\,r)}\arrowvert\varphi\arrowvert^{\frac{2n}{n-1}}\,e^f\tau^{n}\right)^{\frac{n-1}{n}}\leq \left(\frac{C(r_0,\tau,n)r^2}{\operatorname{vol}_{f}(B_{h}(x,\,r))}\int_{B_{h}(x,\,r)}\arrowvert\nabla^{h}\varphi\arrowvert_{h}^2\,e^f\tau^{n}\right)\label{sob-inequ-loc-weight}
\end{split}
\end{equation}
for any $\varphi\in H^1_0(B_{h}(x,\,r))$ and for all $x\in M$ and $0<r<r_0$, where $\operatorname{vol}_{f}(B_{h}(x,\,r)):=\int_{B_{h}(x,\,r)}e^f\tau^{n}$.

A Nash-Moser iteration proceeds in several steps. First, one multiplies \eqref{dir-pb-drift-lap}
across by\linebreak  $\eta_{s,\,s'}^2u_R|u_R|^{2(p-1)}$ with $p\geq 1$, where $\eta_{s,\,s'}$, with $0<s+s'<r$ and $s,s'>0$,
is a Lipschitz cut-off function with compact support in $B_{h}(x,\,s+s')$ equal to $1$ on $B_{h}(x,s)$ and with
$|\nabla^{h}\eta_{s,\,s'}|_{h}\leq\frac{1}{s'}$ almost everywhere. One then integrates by parts and uses
the Sobolev inequality of \eqref{sob-inequ-loc-weight} to obtain a so-called ``reversed H\" older inequality'' which, after iteration, leads to the bound
\begin{equation}
\begin{split}
\sup_{B_{h}(x,\,\frac{r}{2})}|u_R|\leq& C\left(\|u_R\|_{L^2(B_{h}(x,\,r),\,e^f\tau^{n})}+\|F\|_{L^{\infty}(B_{h}(x,\,r))}\right)\\
\leq&C\left(\|F\|_{L^2(e^f\tau^{n})}+\|F\|_{C^0}\right)\\
\leq&C\|f\cdot F\|_{C^0_{X,\,\exp}}\label{first-a-priori-c-0-est-cpct-part}
\end{split}
\end{equation}
for $r\leq r_0$, where $C=C(r_0,\tau,n)$. Here we have made use of \eqref{poincare-a-priori-lin-est-bis} in the second line.
This estimate yields an a priori $C^0$-estimate on any fixed subdomain of $\{f\,<\,R\}$.

As for the weighted a priori estimate, observe from [\eqref{sub-sol-exp-pot-fct}, Lemma \ref{lemm-barr-inf}] that
\begin{equation}
\begin{split}\label{sup-diff-inequ-exp-0}
\left(\Delta_{\tau}+\frac{X}{2}\cdot\right)e^{-f}&\leq -\frac{c}{f}e^{-f}
\end{split}
\end{equation}
outside a fixed compact subset $K\subset M$ with smooth boundary independent of $u_R$. Choose $R>0$ sufficiently large so that $K\subset\{f<R\}$.
Then on combining \eqref{dir-pb-drift-lap} and \eqref{sup-diff-inequ-exp-0}, we see that for any positive constant $A$,
\begin{equation}
\begin{split}\label{sub-diff-inequ-psi-barrier-0}
&\left(\Delta_{\tau}+\frac{X}{2}\cdot\right)\left(u_R-Ae^{-f}\right)\geq  \frac{c\cdot A}{f}e^{-f}-\|fe^fF\|_{C^0}\frac{e^{-f}}{f}
\end{split}
\end{equation}
on $\{f\,<\,R\}\setminus K$. In particular, choosing $A$ so that $A>c^{-1}\|fe^fF\|_{C^0}$,
the maximum principle applied to \eqref{sub-diff-inequ-psi-barrier-0} shows that
\begin{equation*}
\sup_{\{f\,<\,R\}\setminus K}(u_R-Ae^{-f})=\max\left\{-Ae^{-R},\,\max_{\partial K}(u_R-Ae^{-f})\right\},
\end{equation*}
since $u_R$ vanishes along the boundary component $\{f\,=\,R\}$. Now, by \eqref{first-a-priori-c-0-est-cpct-part},
$$\max_{\partial K}(u_R-Ae^{-f})\leq C-Ae^{-\max_{\partial K}f}$$ for some uniform constant $C$.
As a consequence, one can choose $A$ large enough such that $$\max_{\partial K}\left(u_R-Ae^{-f}\right)\leq 0.$$
This establishes the expected a priori weighted upper bound. Applying the same line of reasoning to $-u_R$, we obtain a similar a priori lower bound for $u_R$. Thus, we arrive at the following linear a priori estimate:
 \begin{equation}\label{a-priori-wei-c-0-lin-est}
 \left\|e^f\cdot u_R\right\|_{C^0(\{f\,\leq\,R\})}\leq c(n,\,\tau)\left\|fe^f\cdot F\right\|_{C^0}
 \end{equation}
for any $R$ sufficiently large.

To achieve a priori local estimates on higher derivatives of $u_R$, we invoke standard elliptic Schauder estimates on each ball $B_{h}(x,\,\delta)$ with $2\delta=\inj_{h}(M)>0$ compactly contained in $\{f\,<\,R\}$. This results in the bound
 \begin{equation*}
\sup_{B_{h}(x,\,\delta)\Subset \{f\,<\,R\}} e^{f(x)}\|u_R\|_{C^{2k+2,\,2\alpha}_{\operatorname{loc}}(B_{h}(x,\,\delta))}\leq C\left(n,k,\alpha,\tau,\|f\cdot F\|_{C^{2k,\,2\alpha}_{X,\,\exp}}\right),
 \end{equation*}
which, via the Arzel\`a-Ascoli theorem, gives rise to a subsequence still denoted by $(u_R)_{R\,\geq\,R_0}$
that converges to a function $u\in C^{2k+2,\,2\alpha}_{\operatorname{loc}}(M)$
in the $C^{2k+2,\,2\alpha'}_{\operatorname{loc}}$-topology for any $\alpha'\in(0,\,\alpha)$ satisfying
 \begin{equation*}
 \begin{split}
        \Delta_{\tau}u+\frac{X}{2}\cdot u&=F\qquad\text{on $M$},\\
       \sup_{x\in M} e^{f(x)}\|u\|_{C^{2k+2,\,2\alpha}_{\operatorname{loc}}(B_{h}(x,\,\delta))}&\leq C\left(n,k,\alpha,\tau,\|f\cdot F\|_{C^{2k,\,2\alpha}_{X,\,\exp}}\right).
\end{split}
\end{equation*}

Before proving a priori weighted estimates on higher derivatives of the solution $u$, we need to verify that
the operator $\Delta_{\tau}+\frac{X}{2}\cdot$ remains surjective when restricted to the set of $JX$-invariant functions.
This essentially follows from the maximum principle. Indeed, let $F\in f^{-1}\cdot C^{2k,\,2\alpha}_{X,\,\exp}(M)$,
let $\varphi\in C^{2k+2,\,2\alpha}_{X,\,\exp}(M)$ be a solution to
     \begin{equation}\label{eqn-drift-iso}
     \left(\Delta_{\tau}+\frac{X}{2}\cdot\right)\varphi=F,
     \end{equation}
and let $(\phi^{JX}_t)_t$ denote the flow generated by $JX$. Then, since $F$, $X$, and $JX$
are all $JX$-invariant, the function $\varphi_t:=(\phi^{JX}_t)^*\varphi$ also satisfies
\eqref{eqn-drift-iso}. The function $\varphi-\varphi_t$ therefore lies in the kernel of $\Delta_{\tau}+\frac{X}{2}\cdot$.
But $\varphi-\varphi_t$ also tends to $0$ at infinity. Hence,
we deduce from the maximum principle that $\varphi_t=\varphi$ for every $t\in\mathbb{R}$.
In other words, $\varphi$ is $JX$-invariant, as claimed.

In order to obtain a priori weighted estimates on higher derivatives of $u$,
we need to re-interpret the elliptic equation $\Delta_hu+X\cdot u=2F$ as a parabolic one. First, we
conjugate the operator $\Delta_{h}+X\cdot$ with the exponential weight $e^{-f}$ to obtain
\begin{equation*}
\begin{split}
(e^f\circ\left(\Delta_{h}+X\cdot\right)\circ e^{-f})(\widetilde{u})=\Delta_{h}\widetilde{u}-X\cdot \widetilde{u}-\Delta_{h}f\cdot \widetilde{u},
\end{split}
\end{equation*}
where $\widetilde{u}:=e^{f}\cdot u$. Thus, the function $\widetilde{u}$ is a solution of the equation
\begin{equation}\label{comp-conj-op-exp}
\Delta_{h}\widetilde{u}-X\cdot \widetilde{u}-\Delta_{h}f\cdot \widetilde{u}=2e^fF=:\widetilde{F}\in f^{-1}\cdot C^{2k,\,2\alpha}_{X}(M).
\end{equation}
Next, let $(\phi_t^{X})_{t\,\in\,\mathbb{R}}$ be the flow generated by the vector field $X$, a complete
flow since $X$ is complete. Then as in \cite{Bre-Rot}, let $(r_m)_{m}$ be a sequence of radii tending to $+\infty$
 and define $h_m(s):=r_m^{-1}(\phi^{X}_{s r_m})^*h$, $\widetilde{u}_m:=(\phi^{X}_{s r_m})^*u$, and $\widetilde{F}_m:=r_m(\phi^{X}_{s r_m})^*\widetilde{F}$. Then the sequence of functions
 $(\widetilde{u}_m(s))_{s\,\in\,\mathbb{R}}$ satisfies
\begin{equation}\label{sol-drift-lap-exp-rescaled}
 \begin{split}
        \partial_{s}\widetilde{u}_m-\Delta_{h_m(s)}\widetilde{u}_m+\Delta_{h_m(s)}f_m\cdot \widetilde{u}_m&=-\widetilde{F}_m\qquad\text{on $M$},\\
        f_m(s)&:=(\phi^{X}_{sr_m})^*f.
        \end{split}
\end{equation}
By Lemma \ref{lemm-app-sol-id}, the vector field $r_m^{\frac{1}{2}}X$ converges smoothly to the constant vector field $4\partial_t$ on $\mathbb{R}_+$ and the operator $\Delta_{h_m(s)}$ acting on $JX$-invariant functions is asymptotic to the operator $\frac{4}{n}\frac{\partial^{2}}{\partial t^{2}}+\frac{1}{n(1+4s)}\Delta_{B}$ as $r_m$ tends to $+\infty$ for $s\in[-\delta(n),0]$, with $\delta(n)$ positive and small enough such that
\begin{equation*}
\phi^{X}_{s r_m}\left(\left\{r_m-\sqrt{r_m}\leq f\leq r_m+\sqrt{r_m}\right\}\right)\subset \left\{\frac{r_m}{2}\leq f\leq \frac{3}{2}r_m\right\},\qquad s\in[-\delta(n),0].
\end{equation*}
Moreover, by Lemma \ref{lemm-app-sol-id}, $\Delta_{h_m(s)}f_m$ and the derivatives thereof are uniformly bounded on $\left\{\frac{r_m}{2}\leq f\leq \frac{3}{2}r_m\right\}$.

To conclude, we apply interior parabolic Schauder estimates for transversally elliptic operators to \eqref{sol-drift-lap-exp-rescaled} which leads to the estimate
\begin{equation}\label{sch-est-loc-int-resc}
\|\widetilde{u}_m\|_{C^{2k+2,\,2\alpha}(\Omega_m(0))}\leq c(n,\,\tau)\left(\|\widetilde{F}_m\|_{C^{2k,\,2\alpha}(\Omega_m(\delta(n))}+\|\widetilde{u}_m\|_{C^{0}(\Omega_m(\delta(n)))}\right),
\end{equation}
where $\Omega_m(s):=[-s,\,0]\times \left\{r_m-\sqrt{r_m}\leq f\leq r_m+\sqrt{r_m}\right\}$.
Such parabolic Schauder estimates can be derived along the same lines as elliptic Schauder estimates as in \cite[Section $3.5.6$]{ ElKacimi}.
Tracking the scaling properties of the various Schauder norms involved in \eqref{sch-est-loc-int-resc}, we reach the desired conclusion. Indeed,
we find that
\begin{equation*}
\begin{split}
\left\|e^f\cdot u\right\|_{C^{2k+2,\,2\alpha}(\Omega_m(0))}&\leq c(n,\,\tau)r_m\left\|e^f\cdot F\right\|
_{C^{2k,\,2\alpha}(\widetilde{\Omega}_m(\delta(n)))}+c(n,\,\tau)\left\|e^f\cdot u\right\|_{C^{0}(\widetilde{\Omega}_m(\delta(n)))}\\
&\leq c(n,\,\tau)\|f\cdot F\|_{C^{2k,\,2\alpha}_{X,\,\exp}},
\end{split}
\end{equation*}
where $\widetilde{\Omega}_m(\delta(n)):=\cup_{s\in[-\delta(n),\,0]}
\phi^{X}_{s r_m}\left(\left\{r_m-\sqrt{r_m}\leq f\leq r_m+\sqrt{r_m}\right\}\right)$.
Here we have made use of the $C^0$-estimate \eqref{a-priori-wei-c-0-lin-est} on $u$ (once we allow $R\to+\infty$) in the last line.
This proves that $u\in C^{2k+2,\,2\alpha}_{X,\,\exp}(M)$.
\end{proof}

\subsection{Small perturbations of steady gradient K\"ahler-Ricci solitons: exponential case}
In this section we show, using the implicit function theorem, that the
invertibility of the drift Laplacian given by Theorem \ref{iso-sch-Laplacian-exp}
allows for small perturbations in exponentially weighted function spaces
of solutions to the complex Monge-Amp\`ere
equation that we wish to solve. This forms the openness part of the continuity method as will be explained later in Section \ref{apriori}.
We have:
\begin{theorem}\label{Imp-Def-Kah-Ste-exp}
Let $F_0\in f^{-1}\cdot C^{\infty}_{X,\,\exp}(M)$ and let $\psi_0\in\mathcal{M}^{\infty}_{X,\,\exp}(M)$ be a solution of the complex Monge-Amp\`ere equation
\begin{equation*}
\log\left(\frac{\tau^n_{\psi_0}}{\tau^n}\right)+\frac{X}{2}\cdot\psi_0=F_0.
\end{equation*}
Then for all $\alpha\in\left(0,\,\frac{1}{2}\right)$, there exists a neighbourhood
$U_{F_0}\subset f^{-1}\cdot C^{2,\,2\alpha}_{X,\,\exp}(M)$
of $F$ such that for all $F\in U_{F_0}$, there exists a unique function $\psi\in\mathcal{M}^{4,\,2\alpha}_{X,\,\exp}(M)$ such that
\begin{equation*}
\log\left(\frac{\tau^n_{\psi}}{\tau^n}\right)+\frac{X}{2}\cdot\psi=F.
\end{equation*}
\end{theorem}

\begin{remark}\label{heyjude}
This theorem does not assume any finite regularity
on the data $(\psi_0,\,F_0)$ in the corresponding function spaces.
This essentially comes from Theorem \ref{iso-sch-Laplacian-exp} where the closeness of $\tau$ to $\hat{\omega}$ in derivatives, and hence of
$\tau$ to Cao's steady gradient K\"ahler-Ricci soliton in derivatives, is assumed.
\end{remark}

\begin{proof}[Proof of Theorem \ref{Imp-Def-Kah-Ste-exp}]
In order to apply the implicit function theorem for Banach spaces, we must reformulate
the statement of Theorem \ref{Imp-Def-Kah-Ste-exp} in terms of
the map $MA_{\tau}$ introduced formally at the beginning of Section \ref{sec-Fredo-prop-exp}. To this end, consider the mapping
\begin{equation*}
\begin{split}
\widetilde{MA}_{\tau_{\psi_0}}:(\varphi,G)\in &\,\mathcal{M}^{4,\,2\alpha}_{X,\,\exp}(M)\times f^{-1}\cdot C^{2,\,2\alpha}_{X,\,\exp}(M)\\
&\mapsto \log\left(\frac{\tau_{\psi_0+\varphi}^n}{\tau^n}\right)+\frac{X}{2}\cdot (\psi_0+\varphi)-G-F_0\in f^{-1}\cdot C^{2,\,2\alpha}_{X,\,\exp}(M),\qquad \alpha\in\left(0,\,\frac{1}{2}\right).
\end{split}
\end{equation*}
Notice that the function spaces can be defined either by using the metric $h$ or $h_{s\psi_0}$ for any $s\in[0,\,1]$.
To see that $\widetilde{MA}_{\tau_{\psi_0}}$ is well-defined, apply the Taylor expansion \eqref{equ:taylor-exp} to the background metric $\tau_{\psi_0}$
to obtain
\begin{equation}\label{reform-MA-op-exp}
\begin{split}
\widetilde{MA}_{\tau_{\psi_0}}(\varphi,G)&=\log\left(\frac{\tau_{\psi_0+\varphi}^n}{\tau_{\psi_0}^n}\right)+\frac{X}{2}\cdot\varphi-G\\
&=\Delta_{\tau_{\psi_0}}\varphi+\frac{X}{2}\cdot\varphi-G-\int_0^1\int_0^{u}\arrowvert \partial\bar{\partial}\varphi\arrowvert^2_{h_{s(\psi_0+\varphi)}}\,ds\,du.
\end{split}
\end{equation}
Then note that by a computation similar to that undertaken in \eqref{comp-conj-op-exp},
the first three terms of the last line of \eqref{reform-MA-op-exp} lie in $f^{-1}\cdot C^{2,\,2\alpha}_{X,\,\exp}(M)$.

Now, if $S$ and $T$ are tensors in $C^{2k,\,2\alpha}_{\operatorname{loc}}(M)$ that decay as fast as $e^{-f}$ together with their derivatives, then observe that $S\ast T$
shares the same local regularity and decays as fast as $e^{-2f}$, where $\ast$
denotes any linear combination of contractions of tensors with respect to the metric $h$. Notice that $$\arrowvert \partial\bar{\partial}\varphi\arrowvert^2_{h_{s(\psi_0+\varphi)}}=h_{s(\psi_0+\varphi)}^{-1}
\ast (\nabla^{h})^{\,2}\varphi\ast(\nabla^{h})^{\,2}\varphi$$
and that $h_{s(\psi_0+\varphi)}^{-1}-h^{-1}\in C^{2,\,2\alpha}_{\operatorname{loc}}(M)$ decays as fast as $e^{-f}$.
Thus, applying the above reasoning twice to $S=T=(\nabla^{h})^{\,2}\varphi$ and to the
inverse $h_{s(\psi_0+\varphi)}^{-1}$, one finds that
$\arrowvert\partial\bar{\partial}\varphi\arrowvert^2_{h_{s(\psi_0+\varphi)}}\in f^{-1}\cdot C^{2,\,2\alpha}_{X,\,\exp}(M)$ for each $s\in[0,\,1]$ and that
\begin{equation*}
\left\|\int_0^1\int_0^{u}\arrowvert \partial\bar{\partial}\varphi\arrowvert^2_{h_{s(\psi_0+\varphi)}}\,ds\,du\right\|_{f^{-1}\cdot C^{2,\,2\alpha}_{X,\,\exp}}\leq C\left(\alpha,g,\|\psi_0\|_{C^{4,\,2\alpha}_{X,\,\exp}}\right)\|\varphi\|_{C^{4,\,2\alpha}_{X,\,\exp}}\\
\end{equation*}
as long as $\|\varphi\|_{C^{4,\,2\alpha}_{X,\,\exp}}\leq 1$.
Finally, the $JX$-invariance of the right-hand side of \eqref{reform-MA-op-exp} is clear.

By definition, $\widetilde{MA}_{\tau_{\psi_0}}(\varphi,F-F_0)=0$
if and only if $\psi_0+\varphi$ is a solution to $(\ref{MA-neigh-small-per-exp})$ with data $F$.
By \eqref{equ:sec-der},
$$D_{0}\widetilde{MA}_{\tau_{\psi_0}}(\psi)=\Delta_{\tau_{\psi_0}}\psi+\frac{X}{2}\cdot\psi\qquad\textrm{for
 $\psi\in C^{4,\,2\alpha}_{X,\,\exp}(M)$}.$$ Hence, by Theorem \ref{iso-sch-Laplacian-exp} applied to the background metric $\tau_{\psi_0}$
  in place of $\tau$,
  $D_{0}\widetilde{MA}_{\tau_{\psi_0}}$ is an isomorphism of Banach spaces. The result now follows by applying the implicit function theorem
  to the map $\widetilde{MA}_{\tau_{\psi_0}}$ in a neighbourhood of $(0,\,0)\in \mathcal{M}^{4,\,2\alpha}_{X,\,\exp}(M)\times f^{-1}\cdot C^{2,\,2\alpha}_{X,\,\exp}(M)$.
\end{proof}

\section{A priori estimates}\label{apriori}
In this section, $(C_0,\,g_{0},\,J_0,\,\Omega_0)$ will denote a Calabi-Yau cone of complex dimension $n\geq2$ with radial function $r$.
We set $r^2=:e^t$ and let $\pi:M\to C_{0}$ be an equivariant crepant resolution of $C_{0}$
with respect to the real holomorphic torus action on $C_{0}$ generated by $J_{0}r\partial_{r}$
so that the holomorphic vector field
$2r\partial_{r}=4\partial_{t}$ on $C_{0}$ lifts to a real holomorphic vector field $X=\pi^{*}(2r\partial_{r})$
on $M$. $J$ will denote the complex structure on $M$ and $\tau$ will again be any K\"ahler form on $M$ with $\mathcal{L}_{JX}\tau=0$
satisfying for some $\varepsilon\in(0,\,1)$ the asymptotic bounds
\begin{equation}\label{asy-cao-def-sec5}
|\widehat{\nabla}^{i}\mathcal{L}^{(j)}_X\left(\pi_{*}\tau-\hat{\omega}\right)|_{\hat{g}}=O\left(t^{-\varepsilon-\frac{i}{2}-j}\right)\qquad\textrm{for all $i,\,j\geq 0$},
\end{equation}
where $\hat{\omega}$ denotes the K\"ahler form $\frac{i}{2}\partial\bar{\partial}\left(\frac{nt^{2}}{2}\right)$ on $C_{0}$,
$\hat{g}$ denotes the corresponding K\"ahler metric, and $\widehat{\nabla}$ denotes the associated Levi-Civita connection.
The metric $h$ will denote the K\"ahler metric associated to $\tau$
and for any smooth real-valued function $\phi\in C^{\infty}(M)$
such that $\tau+i\partial\bar{\partial}\phi>0$, we write
$\tau_{\phi}:=\tau+i\partial\bar{\partial}\phi$ and let $h_{\phi}$ denote the corresponding K\"ahler metric and $\nabla^{h_{\phi}}$ the Levi-Civita connection associated to $h_{\phi}$.
By Lemma \ref{easy}, there exists a smooth proper
real-valued function $f:M\to\mathbb{R}$ that is bounded from below with $X=\nabla^{h}f$, which is chosen so that $f\geq1$ on $M$.
We also have the incomplete steady gradient K\"aher-Ricci soliton $\tilde{\omega}$ on $C_{0}$ given to us by the Cao ansatz with soliton potential $\varphi(t)$.
As we are working with exponentially weighted function spaces in this section, we assume in addition that \eqref{add-ass} holds true, i.e.,
$|f-\varphi(t)|=O(1).$ This will allow us to appeal to Corollary \ref{tarea} and in doing so, establish an a priori weighted energy estimate
in Section \ref{section-C-0-est}, the precursor to the $C^{0}$-estimate.

Our goal in this section is to solve the complex Monge-Amp\`ere equation
\begin{equation}\label{MA-cpct-supp}
\log\left(\frac{(\tau+i\partial\bar{\partial}\psi)^{n}}{\tau^{n}}\right)+\frac{X}{2}\cdot\psi=F,\quad \tau+i\partial\bar{\partial}\psi>0,\quad F\in C^{\infty}_0(M),
\end{equation}
on $M$ in the space of functions that decay exponentially at infinity. More precisely, we seek a solution $\psi$ of \eqref{MA-cpct-supp}
that lies in $\mathcal{M}^{\infty}_{X,\,\exp}(M)$, the space of admissible K\"ahler potentials defined in \eqref{defn-M-exp-space} by
\begin{equation*}
\mathcal{M}^{\infty}_{X,\,\exp}(M)=\left\{\psi\in
 C_{\operatorname{loc}}^{\infty}(M)\,|\,\tau_{\psi}>0\quad\textrm{and}\quad\psi\in C^{\infty}_{X,\,\exp}(M)\right\}.
\end{equation*}
The main result we prove here is:
\begin{theorem}\label{theo-exi-sol-cpct-supp-data}
Let $F$ be a compactly supported smooth $JX$-invariant function on $M$. Then there exists a solution $\psi\in\mathcal{M}^{\infty}_{X,\,\exp}(M)$ to \eqref{MA-cpct-supp}.
\end{theorem}

Our approach to solve \eqref{MA-cpct-supp} is to implement the continuity method. We consider the following
one-parameter family of complex Monge-Amp\`ere equations:
\begin{equation}\label{MA-cpct-supp-t}
\log\left(\frac{(\tau+i\partial\bar{\partial}\psi_t)^{n}}{\tau^{n}}\right)+\frac{X}{2}\cdot\psi_t=t\cdot F,\qquad F\in C^{\infty}_0(M),\qquad t\in[0,\,1],
\qquad \psi_t\in\mathcal{M}^{\infty}_{X,\,\exp}(M).
\end{equation}
When $t=0$, there is the trivial solution to \eqref{MA-cpct-supp-t}, namely $\psi_0\equiv0$.
When $t=1$, \eqref{MA-cpct-supp-t} corresponds to \eqref{MA-cpct-supp}, that is, the equation that we wish to solve. Via the a priori
estimates to follow, we will show that the set $t\in[0,\,1]$ for which
\eqref{MA-cpct-supp-t} has a solution is closed. As we have just seen, this set is non-empty. Openness of
this set follows from the isomorphism properties of the drift Laplacian given by Theorem \ref{Imp-Def-Kah-Ste-exp}.
Connectedness of $[0,\,1]$ then implies that \eqref{MA-cpct-supp-t} has a solution for $t=1$, resulting in the desired solution
of \eqref{MA-cpct-supp}.

\subsection{A priori $C^0$-estimate}\label{section-C-0-est}

We begin with the a priori estimate on the $C^0$-norm of $(\psi_t)_{0\,\leq\, t\,\leq\, 1}$
which is uniform in $t\in[0,\,1]$. For the sake of clarity, we omit the dependence on the parameter $t$ while estimating
various norms of the solutions $(\psi_t)_{0\,\leq\, t\,\leq\, 1}$. We begin with two crucial observations.
Our first is:
\begin{lemma}[Localising the supremum and infimum of a solution]\label{lemma-loc-crit-pts}
Let $\psi\in\mathcal{M}^{\infty}_{X,\,\exp}(M)$ be a solution to \eqref{MA-cpct-supp}. Then
\begin{enumerate}
\item either $\psi$ attains a maximum (respectively a minimum) at a point contained in the support $\supp(F)$ of $F$ and $\sup_M\psi=\max_{\supp(F)}\psi\geq 0$ (resp.~ $\inf_M\psi=\min_{\supp(F)}\psi\leq 0$),\\
\item or $\sup_M\psi\leq 0$ (resp.~$\inf_M\psi\geq 0$).
\end{enumerate}
\end{lemma}
\begin{proof}
We prove the assertions of Lemma \ref{lemma-loc-crit-pts} that concern the supremum of a solution $\psi$ only. The statements
on the infimum of $\psi$ can be proved in a similar manner.

Observe that $\psi$ is a subsolution of the following differential inequality:
\begin{equation}
\Delta_{\tau}\psi+\frac{X}{2}\cdot\psi\geq F.\label{subsol-psi}
\end{equation}
If $\psi$ attains a maximum at a point in $\supp(F)$, then $\sup_M\psi=\max_{\supp(F)}\psi$ and since $\psi$ tends to $0$ at infinity, we deduce that $\max_M\psi\geq 0$. If $\psi$ attains a
maximum in $M\setminus\supp(F)$, then the strong maximum principle of Hopf \cite[Theorem 3.5]{gilbarg} applied to \eqref{subsol-psi} implies that $\psi$ is constant outside of $\supp(F)$. In this particular case, as $\psi$ tends to $0$ at infinity, we have that $\sup_M\psi=\max_{M\setminus\supp(F)}\psi=0$.

Now, if $\psi$ does not attain a maximum at a point in $M$,
then there is a sequence of radii $(R_k)_{k\,\geq\,0}$ tending to $+\infty$ such that $\psi(x)\leq\sup_{\partial B_{h}(p,\,R_k)}\psi$
for all $x\in B_{h}(p,\,R_k)$ and for all $k\geq 0$,
where $B_{h}(p,\,R_k)$ denotes the geodesic ball with respect to $h$ of radius $R_k>0$
centered at a point $p\in M$. By letting $R_k\to+\infty$ together with the fact that $\psi$ tends to $0$ at infinity, one reaches the desired conclusion.
\end{proof}

Next we have:
\begin{lemma}[A first rough lower bound on $X\cdot\psi$]\label{lemma-inf-pot-fct}
Let $\psi\in\mathcal{M}^{\infty}_{X,\,\exp}(M)$ be a solution to (\ref{MA-cpct-supp}). Then
$$\inf_M\left(f+\frac{X}{2}\cdot \psi\right)=\min_{\{X\,=\,0\}}\left(f+\frac{X}{2}\cdot \psi\right)=\min_Mf\geq 1.$$
In particular, $\frac{X}{2}\cdot\psi\geq -f$ on $M$.
\end{lemma}

\begin{proof}
A crucial observation is that the gradient of the function
$f+\frac{X}{2}\cdot\psi$ with respect to the K\"ahler metric $h_{\psi}$ induced by $\tau_{\psi}$ is $X$, i.e.,
\begin{equation}\label{crucial-rk-pot}
X=\nabla^{h_{\psi}}\left(f+\frac{X}{2}\cdot\psi\right).
\end{equation}
To see this, note that \eqref{crucial-rk-pot} is equivalent to the statement that
$\tau_{\psi}\lrcorner JX=-d\left(f+\frac{X}{2}\cdot\psi\right)$.
Then with $X^{1,\,0}=\frac{1}{2}(X-iJX)$ and keeping in mind the fact that $\mathcal{L}_{JX}\psi=0$, this latter statement follows from the imaginary part of the sequence of equalities:
\begin{equation}\label{botox}
\begin{split}
\frac{1}{2}\tau_{\psi}\lrcorner X-\frac{i}{2}\tau\lrcorner JX&=\tau_{\psi}\lrcorner X^{1,\,0}\\
&=\tau\lrcorner X^{1,\,0}+(i\partial\bar{\partial}\psi)\lrcorner X^{1,\,0}\\
&=\frac{1}{2}\left(\tau\lrcorner X-i\tau\lrcorner JX\right)+i\bar{\partial}(X^{1,\,0}\cdot\psi)\\
&=\frac{1}{2}\tau\lrcorner X+\frac{i}{2}df+\frac{i}{2}\bar{\partial}(X\cdot\psi)\\
&=\frac{1}{2}\tau\lrcorner X+\frac{i}{2}df+\frac{i}{4}\left(d(X\cdot\psi)+id(X\cdot\psi)\circ J\right)\\
&=\frac{1}{2}\left(\tau\lrcorner X-\frac{1}{2}d(X\cdot\psi)\circ J\right)+\frac{i}{2}d\left(f+\frac{X}{2}\cdot\psi\right).\\
\end{split}
\end{equation}

Now, since $\psi$ together with its derivatives
decay exponentially to $0$ at infinity, we see that
$X\cdot\psi=2\left(\frac{X}{2}\cdot\psi\right)$ decays to $0$ at infinity as the norm of $X$ is bounded.
Thus, $f+\frac{X}{2}\cdot\psi$ is a proper function bounded from below.
In particular, it attains a minimum at a point $p_{\min}\in M$, a point at which $X$ vanishes by virtue of \eqref{crucial-rk-pot}. From this last remark, the result follows.
\end{proof}

\subsubsection{Aubin-Tian-Zhu's functionals}

We now introduce two functionals that have been defined and used by
Aubin \cite{Aub-red-Cas-Pos}, Bando and Mabuchi \cite{Ban-Mab-Uni}, and Tian \cite[Chapter $6$]{Tian-Can-Met-Boo}
in the study of Fano manifolds, and by Tian and Zhu \cite{tianzhu1} in the study of shrinking gradient K\"ahler-Ricci solitons
on compact K\"ahler manifolds.
\begin{definition}
Let $(\varphi_t)_{0\,\leq\, t\,\leq\, 1}$ be a $C^1$-path in $\mathcal{M}^{\infty}_{X,\,\exp}(M)$ from $\varphi_{0}=0$ to $\varphi_{1}=\varphi$.
We define the following two generalised weighted energies:
\begin{equation*}
\begin{split}
I_{\tau,\,X}(\varphi)&:=\int_M\varphi\left(e^f\tau^n-e^{f+\frac{X}{2}\cdot\varphi}\tau_{\varphi}^n\right),\\
J_{\tau,\,X}(\varphi)&:=\int_0^1\int_M\dot{\varphi_s}\left(e^f\tau^n-e^{f+\frac{X}{2}\cdot\varphi_s}\tau_{\varphi_s}^n\right)\wedge ds.
\end{split}
\end{equation*}
\end{definition}

At first sight, these two functionals resemble relative weighted mean values of a potential $\varphi$ in
$\mathcal{M}^{\infty}_{X,\,\exp}(M)$ or of
a path $(\varphi_t)_{0\,\leq\, t\,\leq\, 1}$ in $\mathcal{M}^{\infty}_{X,\,\exp}(M)$ respectively. When $X\equiv 0$ and $(M,\,\tau)$ is a compact K\"ahler manifold,
an integration by parts together with some algebraic manipulations (see Aubin's seminal paper \cite{Aub-red-Cas-Pos} or Tian's book \cite[Chapter $6$]{Tian-Can-Met-Boo}) show that
\begin{equation}
\begin{split}\label{formulae--fct-I-J-ein}
I_{\tau,\,0}(\varphi)&=\sum_{k\,=\,0}^{n-1}\int_{M}i\partial\varphi\wedge\bar{\partial}\varphi\wedge\tau^k\wedge\tau_{\varphi}^{n-1-k},\\
J_{\tau,\,0}(\varphi)&=\sum_{k\,=\,0}^{n-1}\frac{k+1}{n+1}\int_{M}i\partial\varphi\wedge\bar{\partial}\varphi\wedge\tau^{k}\wedge\tau_{\varphi}^{n-1-k}.
\end{split}
\end{equation}
This justifies
the description of $I_{\tau,\,0}(\varphi)$ and $J_{\tau,\,0}(\varphi)$ as modified energies.
Moreover, it demonstrates that on a compact K\"ahler manifold $J_{\tau,\,0}$ is a true functional, that is to say, it does not depend on the choice of path.

Such formulae \eqref{formulae--fct-I-J-ein} for $I_{\tau,\,X}$ and $J_{\tau,\,X}$
for a non-vanishing vector field $X$ and a non-compact K\"ahler manifold $(M,\,\tau)$
do not seem to be readily available for a good reason; the exponential function is not algebraic.
However, our aim here is to prove that the essential properties shared by both
$I_{\tau,\,0}$ and $J_{\tau,\,0}$ hold true for a non-vanishing vector field $X$ in a non-compact setting.
We follow closely Tian and Zhu's work \cite{tianzhu1}, beginning with:

\begin{theorem}\label{main-thm-I-J}
$I_{\tau,\,X}(\varphi)$
and $J_{\tau,\,X}(\varphi)$ are well-defined for $\varphi\in\mathcal{M}^{\infty}_{X,\,\exp}(M)$.
Moreover, $J_{\tau,\,X}$ does not depend on the choice of path $(\varphi_t)_{0\,\leq\, t\,\leq\, 1}$ in
$\mathcal{M}^{\infty}_{X,\,\exp}(M)$ from $\varphi_{0}=0$ to $\varphi_{1}=\varphi$, hence defines
a functional on $\mathcal{M}^{\infty}_{X,\,\exp}(M)$. Finally, the first variation of the difference $(I_{\tau,\,X}-J_{\tau,\,X})$ is
given by
\begin{equation}\label{first-var-I-J}
\frac{d}{dt}\left(I_{\tau,\,X}-J_{\tau,\,X}\right)(\varphi_t)=-\int_M\varphi_t\left(\Delta_{\tau_{\varphi_t}}\dot{\varphi_t}+\frac{X}{2}\cdot\dot{\varphi_t}\right)\,e^{f_{\varphi_t}}\tau_{\varphi_t}^n,
\end{equation}
where $f_{\varphi_t}:=f+\frac{X}{2}\cdot\varphi_t$ satisfies
$X=\nabla^{\tau_{\varphi_t}}f_{\varphi_t}$
and where $(\varphi_t)_{0\,\leq\, t\,\leq\, 1}$ is any $C^1$-path in $\mathcal{M}^{\infty}_{X,\,\exp}(M)$ from $\varphi_{0}=0$ to $\varphi_{1}=\varphi$.
\end{theorem}

\begin{proof}
We begin by showing that $I_{\tau,\,X}(\varphi)$ is well-defined. By linearising
the weighted measure $e^{f_{\varphi}}\tau_{\varphi}^n$ at $\varphi=0$, one sees that
\begin{equation*}
\begin{split}
e^{f_{\varphi}}\tau_{\varphi}^n&=e^f\tau^n+e^f\left(\left(\frac{X}{2}\cdot\varphi\right)\,\tau^n+ni\partial\bar{\partial}\varphi\wedge\tau^{n-1}\right)+\int_0^1(1-s)\frac{\partial^{2}}{\partial s^2}(e^{f_{s\varphi}}\tau_{s\varphi}^n)\wedge ds\\
&=e^f\tau^n+e^f\left(\left(\frac{X}{2}\cdot\varphi\right)\,\tau^n+ni\partial\bar{\partial}\varphi\wedge\tau^{n-1}\right)+e^fQ(X\cdot\varphi,i\partial\bar{\partial}\varphi),
\end{split}
\end{equation*}
where $Q$ is a $2n$-form satisfying
\begin{equation*}
\begin{split}
\left|Q(X\cdot\varphi,i\partial\bar{\partial}\varphi)\right|\leq ce^{\frac{|X\cdot\varphi|}{2}}\left(|X\cdot\varphi|_{h}^2+|i\partial\bar{\partial}\varphi|^2_{h}\right)\tau^n
\end{split}
\end{equation*}
pointwise on $M$. Since $\varphi\in\mathcal{M}^{\infty}_{X,\,\exp}(M)$,
we know that $\varphi$ and its derivatives decay exponentially with respect to $f$.
Moreover, since $X$ is bounded, $X\cdot\varphi$ also decays exponentially,
and by construction, we know that the volume of the level sets of $f$ have polynomial growth, that is to say,
\begin{equation*}
\int_{\{f\,=\,R\}}\tau^n\lrcorner X=O\left(R^{n-1}\right)\qquad\textrm{as $R\to+\infty$.}
\end{equation*}
Using the co-area formula, these observations together imply that
\begin{equation*}
\int_{\{f\,\geq\,\lambda\}}|\varphi|\left|e^f\tau^n-e^{f_{\varphi}}\tau_{\varphi}^n\right|\leq c\int_{\lambda}^{+\infty}e^{-s}s^{n-1}ds<+\infty,
\end{equation*}
where $\lambda$ is a positive constant large enough so that $\{X\,=\,0\}\subset\left\{f\leq\frac{\lambda}{2}\right\}$. This shows that $I_{\tau,\,X}$ is well-defined.
The assertion that $J_{\tau,\,X}(\varphi)$ is well-defined for any path $(\varphi_t)_{0\,\leq\, t\,\leq\, 1}$ in
$\mathcal{M}^{\infty}_{X,\,\exp}(M)$ from $\varphi_{0}=0$ to $\varphi_{1}=\varphi$ is proved in a similar manner.

Before proving that $J_{\tau,\,X}$ is path-independent and hence defines a functional,
we establish the first variation of $I_{\tau,\,X}-J_{\tau,\,X}$. To this end,
let $(\varphi_t)_{0\,\leq\, t\,\leq\, 1}$ be a $C^1$-path in $\mathcal{M}^{\infty}_{X,\,\exp}(M)$ with $\varphi_{0}=0$ and $\varphi_{1}=\varphi$.
Then differentiating with respect to the parameter $t$ under the integral sign in the definition of $I_{\tau,\,X}(\varphi_t)$, we find that
\begin{equation}
\begin{split}\label{first-var-I}
\frac{d}{dt}I_{\tau,\,X}(\varphi_t)&=\int_M\dot{\varphi_t}\left(e^f\tau^n-e^{f_{\varphi_t}}\tau_{\varphi_t}^n\right)-
\int_M\varphi_t\cdot\,e^{f_{\varphi_t}}\left(\left(\frac{X}{2}\cdot\dot{\varphi_t}\right)\,\tau_{\varphi_t}^{n}+ni\partial\bar{\partial}\dot{\varphi_t}\wedge\tau_{\varphi_t}^{n-1}\right),
\end{split}
\end{equation}
where the dominated convergence theorem justifies passing the derivative under the integral.

Regarding the first variation of $J_{\tau,\,X}$, for each fixed $t\in[0,\,1]$, define the path $\tilde{\varphi}^t_s:=\varphi_{st}$
in $\mathcal{M}^{\infty}_{X,\,\exp}(M)$ for $s\in[0,\,1]$. Then $(\tilde{\varphi}^t_s)_{0\,\leq\,s\,\leq\,1}$
is a $C^1$-path from $\tilde{\varphi}^t_0=0$ to $\tilde{\varphi}^t_1=\varphi_t$ and
 \begin{equation}
\begin{split}\label{rewrite-J-path}
J_{\tau,\,X}(\varphi_t)&=\int_0^1\int_M\left(\frac{\partial\tilde{\varphi}^t_s}{\partial s}\right)\left(e^f\tau^n-e^{f+\frac{X}{2}\cdot\tilde{\varphi}^t_s}\tau_{\tilde{\varphi}^t_s}^n\right)\wedge ds\\
&=\int_0^1\int_Mt\dot{\varphi}_{st}\left(e^f\tau^n-e^{f+\frac{X}{2}\cdot\varphi_{st}}\tau_{\varphi_{st}}^{n}\right)\wedge ds\\
&=\int_0^t\int_M\dot{\varphi_{u}}\left(e^f\tau^n-e^{f+\frac{X}{2}\cdot\varphi_{u}}\tau_{\varphi_{u}}^n\right)\wedge du,\\
\end{split}
\end{equation}
where the change of variables $u:=st$ was performed in the third line.
One can check that the integrand (with respect to $du$) in \eqref{rewrite-J-path} is a real-valued continuous function. As such,
\begin{equation}\label{first-var-J}
\frac{d}{dt}J_{\tau,\,X}(\varphi_t)=\int_M\dot{\varphi_{t}}\left(e^f\tau^n-e^{f+\frac{X}{2}\cdot\varphi_{t}}\tau_{\varphi_{t}}^n\right).
\end{equation}
Taking the difference of \eqref{first-var-I} and \eqref{first-var-J} then yields \eqref{first-var-I-J}.

What remains to be shown is that
$J_{\tau,\,X}(\varphi)$ does not depend on the chosen path
$(\varphi_{t})_{0\,\leq\, t\,\leq\, 1}$ in $\mathcal{M}^{\infty}_{X,\,\exp}(M)$ from $0$ to $\varphi$.
We follow Bando and Mabuchi's seminal paper \cite{Ban-Mab-Uni} together with \cite[Lemma $3.1$]{Zhu-KRS-C1} to prove this.
By concatenating paths, it suffices to show
that if $(\varphi_t)_{0\,\leq\, t\,\leq\, 1}$ is a
$C^1$-path in $\mathcal{M}^{\infty}_{X,\,\exp}(M)$ from $\varphi_{0}=0$ to $\varphi_{1}=\varphi\equiv0$,
then $J_{\tau,\,X}(\varphi)=0$. To prove this, we need to enlarge the space of parameters $[0,\,1]$ to a square $[0,\,1]\times[0,\,1]$
in the following way. Define the following two-parameter path:
$$\textrm{$\varphi_{t,\,\delta}:=(1-\delta)\cdot \varphi_t$ for $(t,\,\delta)\in[0,\,1]\times[0,\,1]$}.$$
This path has the following properties:
\begin{itemize}
  \item $\varphi_{t,\,\delta}\in\mathcal{M}^{\infty}_{X,\,\exp}(M)$ for all $(t,\,\delta)\in[0,\,1]\times[0,\,1]$,
  \item $\varphi_{t,\,0}=\varphi_t$ for all $t\in[0,\,1]$,
  \item $\varphi_{t,\,1}=0$ for all $t\in[0,\,1]$,
  \item $\varphi_{0,\,\delta}=0$ for all $\delta\in[0,\,1]$,
  \item $\varphi_{1,\,\delta}=0$ for all $\delta\in[0,\,1]$.
\end{itemize}
Next, define the $2n$-forms $(\Omega_{t,\,\delta})_{(t,\,\delta)\,\in\,[0,\,1]\times[0,\,1]}$ by
\begin{equation*}
\Omega_{t,\,\delta}:=e^{f}\tau^n-e^{f_{\varphi_{t,\,\delta}}}\tau_{\varphi_{t,\,\delta}}^n\qquad\textrm{for all $(t,\,\delta)\in[0,\,1]\times[0,\,1]$}.
\end{equation*}
Then by definition
$$\Omega_{t,\,1}=\Omega_{0,\,\delta}=\Omega_{1,\,\delta}=0_{\Lambda^{2n}M}\qquad\textrm{for all $t,\,\delta\in[0,\,1]$},$$
and we can rewrite $J_{\tau,\,X}(\varphi)$ as
\begin{equation*}
J_{\tau,\,X}(\varphi)=\int_0^1\int_M\left(\frac{\partial}{\partial t}(\varphi_{t,\,0})\,\Omega_{t,\,0}\right)\wedge dt.
\end{equation*}
An important observation is the following.
\begin{claim}\label{claim-J-2-parameter}
$J_{\tau,\,X}(\varphi)$ can be computed as
\begin{equation}\label{compute-J-square}
J_{\tau,\,X}(\varphi)=\int_{0}^1\int_0^1\int_Md_{t,\,\delta}\varphi_{t,\,\delta}\wedge d_{t,\,\delta}\Omega_{t,\,\delta},
\end{equation}
where $d_{t,\,\delta}$ is the exterior derivative with respect to the parameters $t$ and $\delta$.
\end{claim}

\begin{proof}
We first compute the integrand of \eqref{compute-J-square} pointwise:
\begin{equation*}
\begin{split}
d_{t,\,\delta}\varphi_{t,\,\delta}\wedge d_{t,\,\delta}\Omega_{t,\,\delta}&=\left(\left(\frac{\partial \varphi_{t,\,\delta}}{\partial t}\right)dt+\left(\frac{\partial \varphi_{t,\,\delta}}{\partial \delta}\right)d\delta\right)\wedge
 \left(\left(\frac{\partial \Omega_{t,\,\delta}}{\partial t}\right)\wedge dt+\left(\frac{\partial \Omega_{t,\,\delta}}{\partial \delta}\right)\wedge d\delta\right)\\
&=\left(\frac{\partial \varphi_{t,\,\delta}}{\partial t}\cdot \frac{\partial \Omega_{t,\,\delta}}{\partial \delta}-\frac{\partial \varphi_{t,\,\delta}}{\partial \delta}\cdot \frac{\partial \Omega_{t,\,\delta}}{\partial t}\right)\wedge dt\wedge d\delta.
\end{split}
\end{equation*}
Next, integration by parts, first with respect to $\delta$ and then with respect to $t$, yields
\begin{equation*}
\begin{split}
\int_{0}^{1}\left(\int_{0}^{1}\frac{\partial \varphi_{t,\,\delta}}{\partial t}\cdot \frac{\partial \Omega_{t,\,\delta}}{\partial \delta}\wedge d\delta\right)\wedge dt&=\int_0^1\left[\frac{\partial \varphi_{t,\,\delta}}{\partial t}\cdot \Omega_{t,\,\delta}\right]\bigg\vert_{\delta=0}^{\delta=1}\wedge dt-\int_{0}^{1}\left(\int_{0}^{1}\frac{\partial^2\varphi_{t,\,\delta}}{\partial \delta\partial t}\,\Omega_{t,\,\delta}\wedge d\delta\right)\wedge dt\\
&=-\int_{0}^{1}\frac{\partial\varphi_{t,\,0}}{\partial t}\cdot\Omega_{t,\,0}\wedge dt-\int_{0}^{1}\left(\int_{0}^{1}\frac{\partial^2\varphi_{t,\,\delta}}{\partial \delta\partial t}\,\Omega_{t,\,\delta}\wedge d\delta\right)\wedge dt,\\
\int_{0}^{1}\left(\int_{0}^{1}\frac{\partial \varphi_{t,\,\delta}}{\partial \delta}\cdot \frac{\partial \Omega_{t,\,\delta}}{\partial t}\wedge dt\right)
\wedge d\delta&=\int_0^1\left[\frac{\partial \varphi_{t,\,\delta}}{\partial \delta}\cdot \Omega_{t,\,\delta}\right]\bigg\vert_{t=0}^{t=1}\wedge d\delta-\int_{0}^{1}\left(\int_{0}^{1}\frac{\partial^2\varphi_{t,\,\delta}}{\partial t\partial \delta}\,\Omega_{t,\,\delta}\wedge dt\right)\wedge d\delta\\
&=-\int_{0}^{1}\left(\int_{0}^{1}\frac{\partial^2\varphi_{t,\,\delta}}{\partial t\partial \delta}\,\Omega_{t,\,\delta}\wedge dt\right)\wedge d\delta.
\end{split}
\end{equation*}
The claim now follows from the fact that $\frac{\partial^2\varphi_{t,\,\delta}}{\partial \delta\partial t}=
\frac{\partial^2\varphi_{t,\,\delta}}{\partial t\partial \delta}=-\dot{\varphi}_{t}$ by definition of the path $\varphi_{t,\,\delta}$.
\end{proof}

Now, from the definition of the $2n$-forms $\Omega_{t,\,\delta}$, we have that
 \begin{equation*}
 \begin{split}
-d_{t,\,\delta}\Omega_{t,\,\delta}&=d_{t,\,\delta}\left(e^{f_{\varphi_{t,\,\delta}}}\tau_{\varphi_{t,\,\delta}}^n\right)\\
&=\left(\left(\frac{X}{2}\cdot d_{t,\,\delta}\varphi_{t,\,\delta}\right)\wedge \tau_{\varphi_{t,\,\delta}}^n+ni\partial\bar{\partial}\left(d_{t,\,\delta}\varphi_{t,\,\delta}\right)\wedge\tau^{n-1}_{\varphi_{t,\,\delta}}\right)e^{f_{\varphi_{t,\,\delta}}}\\
&=\left[\left(\frac{X}{2}\cdot\left(\frac{\partial\varphi_{t,\,\delta}}{\partial t}\right)\right)dt+
\left(\frac{X}{2}\cdot\left(\frac{\partial\varphi_{t,\,\delta}}{\partial \delta}\right)\right)d\delta\right]\wedge e^{f_{\varphi_{t,\,\delta}}} \tau_{\varphi_{t,\,\delta}}^n\\
&\qquad+n\left[i\partial\bar{\partial}\left(\frac{\partial\varphi_{t,\,\delta}}{\partial t}\right)\wedge dt+i\partial\bar{\partial}\left(\frac{\partial\varphi_{t,\,\delta}}{\partial \delta}\right)\wedge d\delta\right]\wedge e^{f_{\varphi_{t,\,\delta}}}\tau^{n-1}_{\varphi_{t,\,\delta}}.
\end{split}
\end{equation*}
This allows us to express the integrand of the right-hand side of \eqref{compute-J-square} as:
\begin{equation*}
\begin{split}
d_{t,\,\delta}\varphi_{t,\,\delta}\wedge d_{t,\,\delta}\Omega_{t,\,\delta}&=\left[\frac{\partial\varphi_{t,\,\delta}}{\partial \delta}\left(\frac{X}{2}\cdot\left(\frac{\partial\varphi_{t,\,\delta}}{\partial t}\right)\right)
-\frac{\partial\varphi_{t,\,\delta}}{\partial t}\left(\frac{X}{2}\cdot\left(\frac{\partial\varphi_{t,\,\delta}}{\partial \delta}\right)\right)\right]e^{f_{\varphi_{t,\,\delta}}}
\tau^n_{\varphi_{t,\,\delta}}\wedge dt\wedge d\delta\\
&\qquad+n\left[\frac{\partial \varphi_{t,\,\delta}}{\partial \delta}i\partial\bar{\partial}\left(\frac{\partial\varphi_{t,\,\delta}}{\partial t}\right)
-\frac{\partial\varphi_{t,\,\delta}}{\partial t}i\partial\bar{\partial}\left(\frac{\partial\varphi_{t,\,\delta}}{\partial \delta}\right)\right]\wedge e^{f_{\varphi_{t,\,\delta}}}\tau^{n-1}_{\varphi_{t,\,\delta}}\wedge dt\wedge d\delta.
\end{split}
\end{equation*}
Integration by parts with respect to the weighted measure $e^{f_{\varphi_{t,\,\delta}}}\tau^{n}_{\varphi_{t,\,\delta}}$
for some fixed parameters $t,\delta\in[0,\,1]$ then gives us
\begin{equation*}
\begin{split}
n\int_M&\left[\frac{\partial \varphi_{t,\,\delta}}{\partial \delta}i\partial\bar{\partial}\left(\frac{\partial\varphi_{t,\,\delta}}{\partial t}\right)
-\frac{\partial\varphi_{t,\,\delta}}{\partial t}i\partial\bar{\partial}\left(\frac{\partial\varphi_{t,\,\delta}}{\partial \delta}\right)\right]\wedge e^{f_{\varphi_{t,\,\delta}}}\tau^{n-1}_{\varphi_{t,\,\delta}}\\
&=-n\int_M\left[i\bar{\partial}\left(\frac{\partial\varphi_{t,\,\delta}}{\partial t}\right)\wedge\partial\left(\frac{\partial\varphi_{t,\,\delta}}{\partial \delta}\right)+i\partial\left(\frac{\partial\varphi_{t,\,\delta}}{\partial \delta}\right)\wedge\bar{\partial}\left(\frac{\partial\varphi_{t,\,\delta}}{\partial t}\right)\right]\wedge e^{f_{\varphi_{t,\,\delta}}}\tau^{n-1}_{\varphi_{t,\,\delta}}\\
&\qquad+n\int_M\left[\frac{\partial\varphi_{t,\,\delta}}{\partial t}\cdot i\partial f_{\varphi_{t,\,\delta}}\wedge \bar{\partial}\left(\frac{\partial\varphi_{t,\,\delta}}{\partial \delta}\right)-\frac{\partial\varphi_{t,\,\delta}}{\partial \delta}\cdot i\partial f_{\varphi_{t,\,\delta}}\wedge\bar{\partial}\left(\frac{\partial\varphi_{t,\,\delta}}{\partial t}\right)\right]\wedge e^{f_{\varphi_{t,\,\delta}}}\tau^{n-1}_{\varphi_{t,\,\delta}}\\
&=n\int_M\left[\frac{\partial\varphi_{t,\,\delta}}{\partial t}\cdot i\partial f_{\varphi_{t,\,\delta}}\wedge \bar{\partial}\left(\frac{\partial\varphi_{t,\,\delta}}{\partial \delta}\right)-\frac{\partial\varphi_{t,\,\delta}}{\partial \delta}\cdot i\partial f_{\varphi_{t,\,\delta}}\wedge\bar{\partial}\left(\frac{\partial\varphi_{t,\,\delta}}{\partial t}\right)\right]\wedge e^{f_{\varphi_{t,\,\delta}}}\tau^{n-1}_{\varphi_{t,\,\delta}}\\
&=\int_M\left[\frac{\partial\varphi_{t,\,\delta}}{\partial t}\frac{X}{2}\cdot\left(\frac{\partial\varphi_{t,\,\delta}}{\partial \delta}\right)
-\frac{\partial\varphi_{t,\,\delta}}{\partial \delta}\frac{X}{2}\cdot\left(\frac{\partial\varphi_{t,\,\delta}}{\partial t}\right)\right]e^{f_{\varphi_{t,\,\delta}}}\tau^{n}_{\varphi_{t,\,\delta}}.
\end{split}
\end{equation*}
Here, the exponential decay of the functions $\varphi_{t,\,\delta}$ justifies the use of
Stokes' theorem in the first equality, and in the last line we have applied the identity
\begin{equation*}
n\,i\partial f_{\varphi_{t,\,\delta}}\wedge\bar{\partial}u\wedge
\tau^{n-1}_{\varphi_{t,\,\delta}}=[(\nabla^{h_{\varphi_{t,\,\delta}}})^{0,\,1}f_{\varphi_{t,\,\delta}}\cdot u]\,\tau^{n}_{\varphi_{t,\,\delta}}
=\left(\frac{X}{2}\cdot u\right)\,\tau^{n}_{\varphi_{t,\,\delta}}\quad\textrm{for $JX$-invariant $u\in C^{\infty}(M)$}
\end{equation*}
to $\frac{\partial\varphi_{t,\,\delta}}{\partial\delta}$ and $\frac{\partial\varphi_{t,\,\delta}}{\partial t}$, both of which are $JX$-invariant by virtue of the fact
that $\varphi_{t}$ is for all $t\in[0,\,1]$.
\end{proof}

We next show how Theorem \ref{main-thm-I-J} can be applied to obtain
a priori energy estimates along a path of solutions to \eqref{MA-cpct-supp-t}
in $\mathcal{M}^{\infty}_{X,\,\exp}(M)$. Here we make use of the assumption that
$|f-\varphi(t)|$ is bounded in applying Corollary \ref{tarea}.

\begin{prop}[A priori energy estimates]\label{prop-a-priori-ene-est}
Let $(\psi_t)_{0\,\leq\, t\,\leq\, 1}$ be a path of solutions in $\mathcal{M}^{\infty}_{X,\,\exp}(M)$ to (\ref{MA-cpct-supp-t}). Then there exists a positive constant $C=C\left(n,\tau,\|f\cdot F\|_{C^0_{X,\,\exp}}\right)$ such that
\begin{equation*}
\sup_{0\,\leq\, t\,\leq\, 1}\int_M|\psi_t|^2\,\frac{e^{f}}{f^2}\tau^n\leq C.
\end{equation*}
\end{prop}

\begin{proof}
As a consequence of Theorem \ref{main-thm-I-J}, we can use any $C^1$-path $(\varphi_t)_{0\,\leq\, t\,\leq\, 1}$
in $\mathcal{M}^{\infty}_{X,\,\exp}(M)$ from $\varphi_{0}=0$ to $\varphi_{1}=\varphi\in\mathcal{M}^{\infty}_{X,\,\exp}(M)$
to compute $J_{\tau,\,X}(\varphi)$. As in \cite{tianzhu1}, we choose two different paths
to compute $J_{\tau,\,X}(\psi)$, the first being the linear path
defined by $\varphi_t:=t\psi$, $t\in[0,\,1]$, for $\psi\in\mathcal{M}^{\infty}_{X,\,\exp}(M)$ a solution to \eqref{MA-cpct-supp}. For this path, Theorem \ref{main-thm-I-J} asserts that
\begin{equation*}
\left(I_{\tau,\,X}-J_{\tau,\,X}\right)(\psi)=-\int_0^1\int_Mt\psi\left(\Delta_{\tau_{t\psi}}\psi+\frac{X}{2}\cdot\psi\right)\,e^{f+t\frac{X}{2}\cdot\psi}\tau_{t\psi}^n\wedge dt.
\end{equation*}
Integration by parts with respect to the weighted volume form $e^{f+t\frac{X}{2}\cdot\psi}\tau_{t\psi}^n$ then leads to
\begin{equation}
\begin{split}\label{bded-below-I-J}
\left(I_{\tau,\,X}-J_{\tau,\,X}\right)(\psi)&=n\int_0^1\int_Mt\, i\partial\psi\wedge\bar{\partial}\psi\wedge\,\left(e^{f+t\frac{X}{2}\cdot\psi}\tau_{t\psi}^{n-1}\right)\wedge dt\\
&=n\int_0^1\int_M t\,i\partial\psi\wedge\bar{\partial}\psi\wedge\,\left(e^{f+t\frac{X}{2}\cdot\psi}\left((1-t)\tau+t\tau_{\psi}\right)^{n-1}\right)\wedge dt\\
&= n\sum_{k\,=\,0}^{n-1}{{n-1}\choose{k}}\left(\int_0^1t^{k+1}(1-t)^{n-1-k}\int_M i\partial\psi\wedge\bar{\partial}\psi\wedge\left(e^{f+t\frac{X}{2}\cdot\psi}\tau^{n-1-k}
\wedge\tau_{\psi}^k\right)\right)\wedge dt\\
&\geq n \int_0^1t(1-t)^{n-1}\int_M i\partial\psi\wedge\bar{\partial}\psi\wedge\left(e^{f+t\frac{X}{2}\cdot\psi}\tau^{n-1}\right)\wedge dt\\
&\geq  n \int_0^1t(1-t)^{n-1}\int_M i\partial\psi\wedge\bar{\partial}\psi\wedge\left(e^{(1-t)f}\tau^{n-1}\right)\wedge dt\\
&=n \int_M\left(\int_0^1t(1-t)^{n-1}e^{(1-t)f}dt\right)i\partial\psi\wedge\bar{\partial}\psi\wedge\tau^{n-1},
\end{split}
\end{equation}
where we have used Lemma \ref{lemma-inf-pot-fct} to bound the weight $e^{f+t\frac{X}{2}\cdot\psi}$ from below in the penultimate line.
From this, the following claim will allow us to obtain a lower bound.

\begin{claim}\label{claim-est-bded-below}
There exists a positive constant $c_n$ such that
\begin{equation*}
\int_0^1t(1-t)^{n-1}e^{(1-t)f}dt\geq c_n\frac{e^f}{f^2}.
\end{equation*}
\end{claim}

\begin{proof}
Via the change of variables $s=1-t$, notice that for $k\geq 1$,
\begin{equation*}
\begin{split}
\int_0^1t(1-t)^{k-1}e^{(1-t)f}dt&=\int_0^1(1-s)s^{k-1}e^{sf}ds\\
&=\int_0^1s^{k-1}e^{sf}ds-\int_0^1s^{k}e^{sf}ds\\
&=c_{k-1}(f)-c_k(f),
\end{split}
\end{equation*}
where $c_k(f):=\int_0^1s^{k}e^{sf}\,ds$ for $k\in\mathbb{N}$. An induction argument using the relations
\begin{eqnarray*}
c_0(f)=\frac{e^f-1}{f},\qquad c_k(f)=\frac{e^f}{f}-\frac{k}{f}c_{k-1}(f),\qquad k\geq 1,
\end{eqnarray*}
derived using integration by parts then shows that for all $k\geq 0$, $c_k(f)$ is equivalent to $f^{-1}e^f$ as $f$ tends to $+\infty$.

Next, a computation shows that for all $k\geq 2$,
\begin{equation*}
\begin{split}
c_{k-1}(f)-c_k(f)&=\frac{e^f}{f}-\frac{(k-1)}{f}c_{k-2}(f)-\left(\frac{e^f}{f}-\frac{k}{f}c_{k-1}(f)\right)\\
&=\frac{k}{f}c_{k-1}(f)-\frac{(k-1)}{f}c_{k-2}(f)\\
&=\frac{c_{k-2}(f)}{f}-\frac{k}{f}(c_{k-2}(f)-c_{k-1}(f)).
\end{split}
\end{equation*}
Another induction argument on $k$ (the case $k=1$ can be handled easily) further yields the fact that
for all $k\geq 2$, $c_{k-1}(f)-c_k(f)$ is equivalent to $f^{-2}e^f$ as $f$ tends to $+\infty$. This in turn
implies Claim \ref{claim-est-bded-below}.
\end{proof}

Applying Claim \ref{claim-est-bded-below} to \eqref{bded-below-I-J} results in the lower bound
\begin{equation}\label{inequ-bded-bel-I-J-fin}
(I_{\tau,\,X}-J_{\tau,\,X})(\psi)\geq c\int_M\frac{e^f}{f^2}\,i\partial\psi\wedge\bar{\partial}\psi\wedge\tau^{n-1}
\end{equation}
for some positive constant $c_{n}$. We also require an upper bound on
$(I_{\tau,\,X}-J_{\tau,\,X})(\psi)$ to complete the proof of the proposition. To achieve such a bound,
we use the continuity path of solutions
$\varphi_t:=\psi_t$, $t\in[0,\,1]$, to \eqref{MA-cpct-supp-t}
to compute $(I_{\tau,\,X}-J_{\tau,\,X})(\psi)$. First observe that the
first variations $(\dot{\psi_t})_{0\,\leq\, t\,\leq\, 1}$ satisfy the following PDE
obtained from \eqref{MA-cpct-supp-t} by differentiating with respect to the parameter $t$:
\begin{equation*}
\begin{split}
\Delta_{\tau_{\psi_t}}\dot{\psi_t}+\frac{X}{2}\cdot\dot{\psi_t}=F,\qquad 0\leq t\leq 1.
\end{split}
\end{equation*}
Combined with Theorem \ref{main-thm-I-J}, this leads to the estimate:
\begin{equation}
\begin{split}\label{inequ-bded-abo-I-J-fin}
(I_{\tau,\,X}-J_{\tau,\,X})(\psi)&=-\int_0^1\int_M\psi_t\cdot F\,e^{f_{\psi_t}}\tau_{\psi_t}^n\wedge dt\\
&=-\int_0^1\int_M\psi_t\cdot F\,e^{f+tF}\tau^n\wedge dt\\
&\leq Ce^{\|F\|_{C^0}}\int_0^1\int_{\supp(F)}|F||\psi_t|\,e^f\tau^n\wedge dt\\
&\leq Ce^{\|F\|_{C^0}}\|F\|_{L^2(f^{2}e^f\tau^n)}\int_0^1\|\psi_t\|_{L^2(f^{-2}e^f\tau^n)}\,dt\\
&=:C\left(n,\|f\cdot F\|_{C^0_{X,\,\exp}}\right)\int_0^1\|\psi_t\|_{L^2(f^{-2}e^f\tau^n)}\,dt,
\end{split}
\end{equation}
where we have used \eqref{MA-cpct-supp-t} in the third line and the Cauchy-Schwarz inequality in the penultimate line.
Comparing \eqref{inequ-bded-bel-I-J-fin} with \eqref{inequ-bded-abo-I-J-fin}, we deduce that
 \begin{equation}\label{ireland}
\|\nabla^{h}\psi\|^2_{L^2(f^{-2}e^f\tau^n)}\leq C\int_0^1\|\psi_t\|_{L^2(f^{-2}e^f\tau^n)}\,dt,
\end{equation}
where $C=C\left(n,\,\|f\cdot F\|_{C^0_{X,\,\exp}}\right)$ is a
positive constant that depends only on $n$ and $F$ that may vary from line to line.
Now, an application of Corollary \ref{tarea} to $(M,\,\tau,\,f)$ and $(M\setminus K,\,\tilde{\omega},\,\varphi(t))$ for $K\subset M$ compact,
keeping in mind the fact that the difference between $f$ and the soliton potential $\varphi(t)$ of $\tilde{\omega}$ is bounded on $M\setminus K$
by assumption, shows that
\begin{equation*}
\begin{split}
\lambda(\tau)\|\psi\|^2_{L^2(f^{-2}e^f\tau^n)}\leq \|\nabla^{h}\psi\|^2_{L^2(f^{-2}e^f\tau^n)}
\end{split}
\end{equation*}
for some positive constant $\lambda(\tau)$ independent of the parameter $t\in[0,\,1]$. Concatenating this inequality with \eqref{ireland}, we therefore see that
\begin{equation*}
\begin{split}
\|\psi\|^2_{L^2(f^{-2}e^f\tau^n)}\leq C\int_0^1\|\psi_t\|_{L^2(f^{-2}e^f\tau^n)}\,dt,
\end{split}
\end{equation*}
where $C=C\left(n,\,\tau,\,\|f\cdot F\|_{C^0_{X,\,\exp}}\right).$
This last inequality applies to any truncated path of the one-parameter family of solutions $(\psi_t)_{0\,\leq\, t\,\leq\, 1}$. Thus,
 \begin{equation}
\begin{split}\label{crucial-a-priori-ineq-bis}
\|\psi_t\|^2_{L^2(f^{-2}e^f\tau^n)}&\leq C\int_0^1\|\psi_{st}\|_{L^2(f^{-2}e^f\tau^n)}\,ds\\
&=\frac{C}{t}\int_0^t\|\psi_{s}\|_{L^2(f^{-2}e^f\tau^n)}\,ds.
\end{split}
\end{equation}
This is a Gr\"onwall-type differential inequality and can be integrated as follows.
Let $$H(t):=\int_0^t\|\psi_{s}\|_{L^2(f^{-2}e^f\tau^n)}\,ds$$ and observe that \eqref{crucial-a-priori-ineq-bis} may be rewritten as
\begin{equation*}
\begin{split}
H'(t)\leq \frac{C}{\sqrt{t}}\sqrt{H(t)},\quad t\in(0,1].
\end{split}
\end{equation*}
Integrating then implies that $H(t)\leq C\left(n,\tau,\|f\cdot F\|_{C^0_{X,\,\exp}}\right)\cdot t$
for all $t\in[0,\,1]$ which, after applying \eqref{crucial-a-priori-ineq-bis} once more, yields to the desired upper bound.
\end{proof}

\subsubsection{A priori estimate on $\sup_M\psi$}
Let $\psi_t$ be a solution to \eqref{MA-cpct-supp-t} for some fixed parameter $t\in[0,\,1]$.
We next obtain an upper bound for $\sup_M\psi_t$ uniform in $t$.
To obtain such a bound, it suffices by Lemma \ref{lemma-loc-crit-pts} to only bound $\max_{\supp(F)}\psi_t$ from above.
We do this by implementing a local Nash-Moser iteration as in the proof
of Theorem \ref{iso-sch-Laplacian-exp} using the fact that $\psi_{t}$ is a super-solution of the
linearised complex Monge-Amp\`ere equation of which the drift Laplacian with respect to the known metric $\tau$
forms a part.

\begin{prop}[A priori upper bound on $\sup_M\psi$]\label{prop-bd-abo-uni-psi}
Let $(\psi_t)_{0\,\leq\, t\,\leq\, 1}$ be a path of solutions in $\mathcal{M}^{\infty}_{X,\,\exp}(M)$ to (\ref{MA-cpct-supp-t}). Then there exists a positive constant $C=C\left(n,\tau,\|f\cdot F\|_{C^0_{X,\,\exp}}\right)$ such that
\begin{equation*}
\sup_{0\,\leq\, t\,\leq\, 1}\sup_{\supp(F)}\psi_t\leq C.
\end{equation*}
\end{prop}

\begin{proof}
Let $t\in[0,\,1]$ and set $\psi:=\psi_t$ to simplify notation. Let
$\psi_+:=\max\{\psi,\,0\}$. This is a non-negative Lipschitz function. The strategy of proof is standard and follows along the lines of
the proof of Theorem \ref{iso-sch-Laplacian-exp}; we use a Nash-Moser iteration
to obtain an a priori upper bound on $\sup_{\supp(F)}\psi_+$ in terms of the
(weighted) energy of $\psi_+$ on a tubular neighbourhood of $\supp(F)$. The result then follows by invoking Proposition \ref{prop-a-priori-ene-est}.

To this end, notice that since $\log(1+x)\leq x$ for all $x>-1$ and since $\psi$ is a solution to \eqref{MA-cpct-supp-t}, $\psi$ satisfies the differential inequality
\begin{equation}\label{sub-diff-inequ-psi}
\Delta_{\tau}\psi+\frac{X}{2}\cdot \psi\geq -|F|\qquad\text{on $M$.}
\end{equation}
As in the proof of Theorem \ref{iso-sch-Laplacian-exp}, let $x\in\{f\,<\,R\}$ be such that $B_{h}(x,\,r)\Subset \{f\,<\,R\}$ and multiply \eqref{sub-diff-inequ-psi} across
by $\eta_{s,\,s'}^2u_R|u_R|^{2(p-1)}$ with $p\geq 1$, where $\eta_{s,\,s'}$, with $0<s+s'<r$ and $s,\,s'>0$, is a Lipschitz cut-off function with compact support in $B_{h}(x,\,s+s')$ equal to $1$ on $B_{h}(x,\,s)$ and with
$|\nabla^{h}\eta_{s,\,s'}|_{h}\leq\frac{1}{s'}$ almost everywhere. Next, integrate by parts and use the Sobolev inequality \eqref{sob-inequ-loc-weight}
to obtain a reversed H\" older inequality which after iteration leads to the bound
\begin{equation*}
\begin{split}\label{first-a-priori-c-0-est-cpct-part-non-lin}
\sup_{B_{h}(x,\,\frac{r}{2})}\psi_+&\leq C(n,\tau,r)\left(\|\psi_+\|_{L^2(B_{h}(x,\,r),\,e^f\tau^{n})}^2+\|F\|^2_{C^0}\right)^{\frac{1}{2}}\\
&\leq C(n,\tau,r)\left(\int_{T_r(\supp(F))}\psi_+^2\,f^{-2}e^f\tau^n+\|F\|^2_{C^0}\right)^{\frac{1}{2}}\\
&\leq C\left(n,\tau,r,\|f\cdot F\|_{C^0_{X,\,\exp}}\right),
\end{split}
\end{equation*}
where $T_r(\supp(F)):=\{x\in M\,|\,d_h(x,\,\supp(F))\leq r\}$. Here, we have made use of Proposition \ref{prop-a-priori-ene-est} in the last line.
\end{proof}

Obtaining a lower bound on $\psi_{t}$ is more difficult. The function $\psi_{t}$ is a sub-solution of the linearised equation,
however with respect to the drift Laplacian of the unknown metric, and so an alternative approach is required. We use the weighted $L^{2}$-bound given by Proposition \ref{prop-a-priori-ene-est} together with an adaption of B\l{ocki}'s method \cite{Blo-Uni-CY} to achieve the desired estimate. The fact that the data $F$ in \eqref{MA-cpct-supp-t} is compactly supported is crucial for the proof to work.

\begin{prop}[A priori lower bound on $\inf_M\psi$]\label{prop-bd-bel-uni-psi}
Let $(\psi_t)_{0\,\leq\, t\,\leq\, 1}$ be a path of solutions in $\mathcal{M}^{\infty}_{X,\,\exp}(M)$ to \eqref{MA-cpct-supp-t}. Then there exists a positive constant $C=C\left(n,\tau,\supp(F),\|f\cdot F\|_{C^0_{X,\,\exp}}\right)$ such that
\begin{equation*}
\inf_{0\,\leq\, t\,\leq\, 1}\inf_{\supp(F)}\psi_t\geq -C.
\end{equation*}
\end{prop}

\begin{proof}
Fix $t\in[0,\,1]$ and set $\psi:=\psi_t$ to lighten notation.
By Lemma \ref{lemma-loc-crit-pts}, we can assume that $\psi$ attains a minimum at a point $x_0\in\supp(F)$.
Following \cite{Blo-Uni-CY}, one can find a local coordinate
chart $U$ with $x_0\in U$ together with a smooth strictly plurisubharmonic function $G$ defined on $U$
with $i\partial\bar{\partial}G=\tau$. After adding a pluriharmonic function to $G$ if necessary,
one can then find two positive numbers $a$ and $r$
depending only on the local geometry of $(M,\,\tau)$ around $x_0$ such that $G<0$
on $B_{h}(x_0,\,2r)$, $G$ attains its minimum at $x_0$ on $B_{h}(x_0,\,2r)$, and $G\geq G(x_0)+a$
on the annulus $B_{h}(x_0,\,2r)\setminus B_{h}(x_0,\,r)$.

Consider the non-positive function $u$ defined on $B_{h}(x_0,\,2r)$ by
\begin{equation*}
u :=
\begin{cases}
\psi+G & \text{if $\sup_{M}\psi\leq 0$},\\
\psi-\sup_{\supp(F)}\psi+G & \text{otherwise.}
\end{cases}
\end{equation*}
We are now in a position to apply \cite[Proposition 3]{Blo-Uni-CY} which asserts that
\begin{equation}
\begin{split}\label{blocki-fund-est}
\|u\|_{L^{\infty}(B_{h}(x_0,\,2r))}\leq a+\left(c_n\cdot (2r)\cdot a^{-1}\right)^{2n}\|u\|_{L^1(B_{h}(x_0,\,2r))}\cdot\left\|\frac{\tau_{\psi}^n}{\tau^n}\right\|_{L^{\infty}(B_{h}(x_0,\,2r))}.
\end{split}
\end{equation}
In the case that $\sup_M\psi=\max_{\supp(F)}\psi\geq0$, we obtain, after rearranging (\ref{blocki-fund-est}),
the following sequence of inequalities:
\begin{equation*}
\begin{split}
-\inf_M\psi&\leq\sup_M\psi-\inf_M\psi=\sup_{\supp(F)}\psi-\psi(x_0)\\
&=G(x_{0})-u(x_{0})\\
&\leq\|u\|_{L^{\infty}(B_{h}(x_0,\,2r))}\\
&\leq C(\tau,a,r,n)\|u\|_{L^1(B_{h}(x_0,\,2r))}\cdot\left\|\frac{\tau_{\psi}^n}{\tau^n}\right\|_{L^{\infty}(B_{h}(x_0,\,2r))}+a\\
&= C(\tau,a,r,n)\|u\|_{L^1(B_{h}(x_0,\,2r))}\cdot\left\|e^{-\frac{X}{2}\cdot\psi+F}\right\|_{L^{\infty}(B_{h}(x_0,\,2r))}+a\\
&\leq C(\tau,a,r,n)\|u\|_{L^1(B_{h}(x_0,\,2r))}\cdot\left\|e^{f+F}\right\|_{L^{\infty}(B_{h}(x_0,\,2r))}+a\\
&\leq C(\tau,a,r,n,F)\|u\|_{L^1(B_{h}(x_0,\,2r))}+a\\
&\leq C(\tau,a,r,n,F)\left(\|\psi\|_{L^1(B_{h}(x_0,\,2r))}+\sup_{\supp(F)}\psi+1\right)\\
&\leq C(\tau,a,r,n,F)\left(\|\psi\|_{L^2(B_{h}(x_0,\,2r))}+1\right)\\
&\leq C(\tau,a,r,n,F)\left(\|\psi\|_{L^2(f^{-2}e^f\tau^{n})}+1\right)\\
&\leq C(\tau,a,r,n,F),
\end{split}
\end{equation*}
where $c(\tau,a,r,n)$ denotes a positive constant that may vary from line to line.
Here we have used Lemma \ref{lemma-inf-pot-fct} in the sixth line
to bound $\left|\frac{\tau_{\psi}^n}{\tau^n}\right|$ uniformly from above since $B_{h}(x_0,\,2r)$ is contained in
the tubular neighbourhood $T_{2r}(\supp(F))$ of $\supp(F)$ of radius $2r$, we use Proposition \ref{prop-bd-abo-uni-psi} to bound $\sup_{\supp(F)}\psi$ uniformly together and H\"older's inequality in the antepenultimate line, and finally, we use Proposition \ref{prop-a-priori-ene-est} in the last line to bound $\|\psi\|_{L^2(f^{-2}e^f\tau^{n})}$ uniformly from above. This concludes the proof of Proposition \ref{prop-bd-bel-uni-psi} in the case that $\sup_M\psi=\max_{\supp(F)}\psi\geq 0$.
The case $\sup_{M}\psi\leq 0$ proceeds similarly.
\end{proof}

\subsection{A priori estimates on higher derivatives}
We next derive a priori \emph{local} bounds on higher derivatives of solutions to the complex Monge-Amp\`ere equation \eqref{MA-cpct-supp}, beginning with the radial
derivative.

\newpage
\subsubsection{A priori estimate on the radial derivative}

\begin{prop}[A priori estimate on $X\cdot\psi$]\label{prop-bd-uni-X-psi}
Let $(\psi_t)_{0\,\leq\, t\,\leq\, 1}$ be a path of solutions in $\mathcal{M}^{\infty}_{X,\,\exp}(M)$ to (\ref{MA-cpct-supp-t}). Then there exists a positive constant $C=C\left(n,\tau,\supp(F),\|f\cdot F\|_{C^0_{X,\,\exp}}\right)$ such that
\begin{equation*}
\sup_{0\,\leq\, t\,\leq\, 1}\sup_M|X\cdot\psi_t|\leq C.
\end{equation*}
\end{prop}

\begin{remark}
The $JX$-invariance of $\psi$ is crucial for the proof of this proposition to go through.
\end{remark}

\begin{proof}[Proof of Proposition \ref{prop-bd-uni-X-psi}]
Our proof is based on that of Siepmann in the case of an expanding gradient K\"ahler-Ricci soliton; see \cite[Lemma 5.4.14]{siepmann}.
We adapt his proof here to our particular setting.

The proof comprises two parts. The first gives rise to an upper bound for $X\cdot\psi$, whereas the latter part
yields a lower bound for $X\cdot\psi$. Before proceeding with the first part though, we make the following claim.
\begin{claim}\label{est-sec-der-vec-fiel}
Let $X^{1,\,0}=\frac{1}{2}(X-iJX)$. Then
$$X^{1,\,0}\cdot\left(X^{1,\,0}\cdot \psi\right)=2i\partial\bar{\partial}\psi\left(\Re\left(X^{1,\,0}\right),\,J\Re\left(X^{1,\,0}\right)\right)
\geq -2\left|\Re\left(X^{1,\,0}\right)\right|_{h}^2.$$
\end{claim}

\begin{proof}
Since $\psi$ is invariant under the flow of $JX$, we know that $$JX\cdot (X\cdot \psi)=0.$$
In particular, we have that $X^{1,\,0}\cdot\left(X^{1,\,0}\cdot\psi\right)=
\Re\left(X^{1,\,0}\right)\cdot\left(\Re\left(X^{1,\,0}\right)\cdot\psi\right)=
\overline{X^{1,\,0}}\cdot\left(X^{1,\,0}\cdot\psi\right)$. A straightforward computation then shows that
\begin{equation*}
\overline{X^{1,\,0}}\cdot\left(X^{1,\,0}\cdot \psi\right)=\partial\bar{\partial}\psi\left(X^{1,\,0},\,
\overline{X^{1,\,0}}\right)=2i\partial\bar{\partial}\psi\left(\Re\left(X^{1,\,0}\right),\,J\Re\left(X^{1,\,0}\right)\right).
\end{equation*}
The result now follows from the fact that $\tau_{\psi}> 0$ so that
\begin{equation*}
\begin{split}
i\partial\bar{\partial}\psi\left(\Re\left(X^{1,\,0}\right),\,J\Re\left(X^{1,\,0}\right)\right)&=\tau_{\psi}\left(\Re\left(X^{1,\,0}\right),\,J\Re\left(X^{1,\,0}\right)\right)-\left|\Re\left(X^{1,\,0}\right)\right|_{h}^2\\
&\geq -\left|\Re\left(X^{1,\,0}\right)\right|_{h}^2.
\end{split}
\end{equation*}
\end{proof}

To achieve an upper bound for $X\cdot \psi$, we introduce the flow
$(\phi^{X}_t)_{t\,\in\,\mathbb{R}}$ generated by the vector field $\frac{X}{2}$. This flow is complete as $X$ is complete.
Define $\psi_x(t):=\psi(\phi^{X}_t(x))$ for $(x,\,t)\in M\times\mathbb{R}.$ Then for any cut-off function $\eta:\mathbb{R}_+\rightarrow[0,\,1]$
such that $\eta(0)=1$, $\eta'(0)=0$, we have that
\begin{eqnarray*}
\int_0^{+\infty}\eta''(t)\psi_x(t)\,dt&=&-\int_0^{+\infty}\eta'(t)\psi_x'(t)\,dt\\
&=&\psi_x'(0)+\int_0^{+\infty}\eta(t)\psi_x''(t)\,dt.
\end{eqnarray*}
Hence it follows from the boundedness of the soliton vector field $X$ with respect to the norm induced by $\tau$ and Claim \ref{est-sec-der-vec-fiel} that
\begin{equation*}
\begin{split}
\frac{X}{2}\cdot\psi(x)=\psi_x'(0)&\leq -\int_{\supp(\eta)}\frac{X}{2}\cdot \left(\frac{X}{2}\cdot \psi\right)(\phi^{X}_t(x))\,dt+\sup_{t\,\in\,\supp(\eta'')}\arrowvert\psi_x(t)\arrowvert\int_{\supp(\eta'')}\arrowvert\eta''(t)\arrowvert\,dt\\
&\leq\frac{1}{2}\int_{\supp(\eta)}\arrowvert X\arrowvert^2_{h}(\phi^{X}_t(x))\,dt+\sup_{t\,\in\,\supp(\eta'')}\arrowvert\psi_x(t)\arrowvert\int_{\supp(\eta'')}\arrowvert\eta''(t)\arrowvert\,dt\\
&\leq c_n\int_{\supp(\eta)}dt+\sup_{t\,\in\,\supp(\eta'')}\arrowvert\psi_x(t)\arrowvert\int_{\supp(\eta'')}\arrowvert\eta''(t)\arrowvert\,dt.
\end{split}
\end{equation*}
Choose $\eta$ such that for some $\varepsilon>0$ to be chosen later, $\eta\equiv1$ on $[0,\frac{\varepsilon}{2}]$,
$\supp(\eta)\subset [0,\varepsilon]$, and such that
$\arrowvert\eta''\arrowvert\leq c/\varepsilon^2$ for some uniform positive constant $c$. Then for all $\varepsilon>0$,
\begin{equation}
\frac{X}{2}\cdot\psi(x)\leq c_n\varepsilon+c\|\psi\|_{C^0}\varepsilon^{-1}.\label{upper-bd-rough-eps}
\end{equation}
Minimising the right-hand side of \eqref{upper-bd-rough-eps} seen as a function of $\varepsilon>0$,
one obtains the inequality
\begin{equation*}
\frac{X}{2}\cdot\psi(x)\leq c_n\|\psi\|_{C^0}^{\frac{1}{2}},
\end{equation*}
reminiscent of an interpolation inequality with the lower bound given by Claim \ref{est-sec-der-vec-fiel} on the second derivatives of $\psi$ in the direction of $X$ succinctly contained within the constant $c_{n}$. The upper bound now follows from Propositions \ref{prop-bd-abo-uni-psi} and \ref{prop-bd-bel-uni-psi}.
The lower bound can be proven analogously by working on an interval $[-\varepsilon,\,0]$ and choosing
$\varepsilon>0$ in a manner similar to above.
\end{proof}

\subsubsection{$C^2$ a priori estimate}
The $C^{2}$-estimate is next.
\begin{prop}[A priori $C^2$-estimate]\label{prop-C^2-est}
Let $(\psi_t)_{0\,\leq\, t\,\leq\, 1}$ be a path of solutions in $\mathcal{M}^{\infty}_{X,\,\exp}(M)$ to (\ref{MA-cpct-supp-t}). Then there exists a positive constant $C=C\left(n,\tau,\supp(F),\|f\cdot F\|_{C^2_{X,\,\exp}}\right)$ such that the following $C^2$ a priori estimate holds true:
\begin{equation*}
\sup_{0\,\leq\, t\,\leq\, 1}\|i\partial\bar{\partial}\psi_t\|_{C^0(M)}\leq C.
\end{equation*}
\end{prop}

\begin{proof}
We follow closely \cite[Proposition 6.6]{con-der} where the
approach taken is based on standard computations performed in Yau's seminal paper \cite[pp.347--351]{Calabiconj}; see \cite[Lemma 5.4.16]{siepmann}
for a modification of these computations to the setting of expanding gradient K\"ahler-Ricci solitons. Only the presence of the
vector field $X$ has to be taken into account, therefore we only outline the main steps.

For the sake of clarity, we suppress the dependence of the function $\psi_t$ on the parameter $t\in[0,\,1]$.
According  to \eqref{MA-cpct-supp-t}, $\psi$ satisfies
 \begin{eqnarray*}
\log\left(\frac{\tau_{\psi}^n}{\tau^n}\right)=F-\frac{X}{2}\cdot\psi=:F(\psi).
\end{eqnarray*}
As in \cite{Calabiconj}, we compute the Laplacian of $F(\psi)$ with respect to $\tau$ in local holomorphic coordinates around a point $x\in M$ such that at $x$, the Riemannian metrics $h$ and $h_{\psi}$ associated to $\tau$ and $\tau_{\psi}$
take the form $h_{i\bar{\jmath}}(x)=\delta_{i\bar{\jmath}}$ and $(h_{\psi})_{i\bar{\jmath}}(x)=(1+\psi_{i\bar{\imath}}(x))\delta_{i\bar{\jmath}}$ respectively. After a lengthy computation, one arrives at the fact that
\begin{equation}\label{eq-sec-der-yau}
\begin{split}
\Delta_{\tau}(F(\psi))&=\Delta_{\tau_{\psi}}(\tr_{h}(h_{\psi}))
-\frac{\psi_{i\bar{\jmath}k}\psi_{\bar{\imath}j\bar{k}}}{(1+\psi_{i\bar{\imath}})(1
+\psi_{k\bar{k}})}+\Rm(h)_{i\bar{\imath}k\bar{k}}\left(1-\frac{1}{1+\psi_{i\bar{\imath}}}
-\frac{\psi_{i\bar{\imath}}}{1+\psi_{k\bar{k}}}\right).
\end{split}
\end{equation}
Now, a standard computation shows that
\begin{equation*}
\begin{split}
-\sum_{i,\,k}\Rm(h)_{i\bar{\imath}k\bar{k}}\left(1-\frac{1}{1+\psi_{i\bar{\imath}}}-\frac{\psi_{i\bar{\imath}}}{1+\psi_{k\bar{k}}}\right)
\geq\inf_M \Rm(h)\left(\tr_h(h_{\psi}^{-1})\tr_{h}(h_{\psi})-C(n)\right),
\end{split}
\end{equation*}
where $\Rm(h)$ is the complex linear extension of the curvature operator of the metric $h$ and where $\inf_M\Rm(h):=\inf_{i\neq k}\Rm(h)_{i\bar{\imath}k\bar{k}}$.

Next we study the term $\Delta_{\tau}\left(\frac{X}{2}\cdot\psi\right)$. Let $X^{1,\,0}=\frac{1}{2}(X-iJX)$.
Then since $X$ is real holomorphic and both $\tau$ and $\psi$ are $JX$-invariant, we find that
\begin{equation*}
\begin{split}
\Delta_{\tau}\left(\frac{X}{2}\cdot\psi\right)&=\Delta_{\tau}\left(X^{1,\,0}\cdot\psi\right)\\
&=\nabla^h_i(X^{1,\,0})^k\psi_{\bar{\imath}k}+X^{1,\,0}\cdot\Delta_{\tau}\psi\\
&=\nabla^{h}(X^{1,\,0})\ast\partial\bar{\partial}\psi+\frac{X}{2}\cdot\Delta_{\tau}\psi\\
&\leq C\tr_{h}(h_{\psi})+C(n)\|\nabla^hX\|_{C^0(M)}+\frac{X}{2}\cdot\tr_{h}(h_{\psi}),
\end{split}
\end{equation*}
where we have used the fact that $0<h_{\psi}\leq (n+\Delta_{\tau}\psi)h$ together with the boundedness of $\nabla^hX$ given by Lemma \ref{lemm-app-sol-id}. To summarise, we have the following first crucial estimate:
\begin{equation}\label{crucial-est-tr-C2}
\begin{split}
\Delta_{\tau_{\psi}}\tr_{h}(h_{\psi})+\frac{X}{2}\cdot \tr_{h}(h_{\psi})&\geq \frac{\psi_{i\bar{\jmath}k}\psi_{\bar{\imath}j\bar{k}}}{(1+\psi_{i\bar{\imath}})(1+\psi_{k\bar{k}})}\\
&\qquad+\Delta_{\tau}F-C\tr_{h}(h_{\psi})\left(1+\inf_M \Rm(h)\tr_h(h_{\psi}^{-1})\right)-C(n,\,\tau).
\end{split}
\end{equation}

Now, if $u:=e^{-\alpha\psi}\tr_{h}(h_{\psi})$, where $\alpha\in\mathbb{R}$ will be specified later,
then, as in the proof of \cite[Lemma 5.4.16]{siepmann}, one estimates the Laplacian of $u$ with respect to $\tau_{\psi}$ in the following way:
 \begin{equation*}
 \begin{split}
&\Delta_{\tau_{\psi}}u\geq e^{-\alpha \psi}\left(\Delta_{\tau}F(\psi)-\inf_{M}\operatorname{Rm}(h)\tr_{h}(h_{\psi}^{-1})\tr_{h}(h_{\psi})-C(n)-\alpha \Delta_{\tau_{\psi}}\psi\tr_{h}(h_{\psi})\right).
\end{split}
\end{equation*}
Here, one has to take advantage of the non-negative term involving the third derivatives of $\psi$ on the right-hand side of (\ref{crucial-est-tr-C2}) to absorb the term $h_{\psi}(\nabla^{h_{\psi}}\psi,\nabla^{h_{\psi}}\Delta_{\tau}\psi)$.
Thus, for some positive constant $C$ independent of $\psi$, it follows that
\begin{equation*}
\begin{split}
\Delta_{\tau_{\psi}}u+\frac{X}{2}\cdot u&\geq e^{-\alpha\psi}\left(\Delta_{\tau}F-\inf_{M}\operatorname{Rm}(h)\tr_h(h_{\psi}^{-1})\tr_{h}(h_{\psi})\right)\\
&\qquad-C(n,\,\tau)e^{-\alpha \psi}-\alpha\left(\frac{X}{2}\cdot\psi\right)u-C(n,\,\tau)u-\alpha(n-\tr_{h}(h_{\psi}^{-1}))
u\\
&\geq-C\left(n,\tau,\|\psi\|_{C^0(M)},\|F\|_{C^2(M)}\right)-C\left(n,\tau,\|X\cdot\psi\|_{C^0(M)}\right)u\\
&\qquad+\tr_h(h_{\psi}^{-1})u\\
&\geq-C-Cu+\tr_h(h_{\psi}^{-1})u,
\end{split}
\end{equation*}
where we set $\alpha:=\max\{1+\inf_{M}\operatorname{Rm}(h),\,1\}$ and
$C=C\left(n,\tau,\supp(F),\|f\cdot F\|_{C^2_{X,\,\exp}}\right)$, and we have used
Propositions \ref{prop-bd-abo-uni-psi}, \ref{prop-bd-bel-uni-psi}, and \ref{prop-bd-uni-X-psi} in the last line.
Another estimate using the geometric inequality
\begin{eqnarray*}
\sum_i\frac{1}{1+\psi_{i\bar{\imath}}}\geq\left(\frac{\sum_i(1+\psi_{i\bar{\imath}})}{\Pi_i(1+\psi_{i\bar{\imath}})}\right)^{\frac{1}{n-1}},
\end{eqnarray*}
or equivalently,
\begin{eqnarray*}
\tr_h(h_{\psi}^{-1})\geq \left(\frac{\tr_h(h_{\psi})}{\det_h(h_{\psi})}\right)^{\frac{1}{n-1}},
\end{eqnarray*}
then shows that $u$ satisfies the following differential inequality:
\begin{eqnarray*}
\Delta_{\tau_{\psi}}u+\frac{X}{2}\cdot u\geq -C(1+u)+Cu^{\frac{n}{n-1}}
\end{eqnarray*}
for some positive constant $C=C\left(n,\tau,\supp(F),\|f\cdot F\|_{C^2_{X,\,\exp}}\right)$.
Since $u$ is non-negative and converges to $n$ at infinity as $\psi\in\mathcal{M}^{\infty}_{X,\,\exp}(M)$,
an application of the maximum principle to an exhausting sequence of domains of $M$ finally yields the desired upper bound on $n+\Delta_{\tau}\psi$.
\end{proof}

A useful consequence of Proposition \ref{prop-C^2-est} is that the K\"ahler metrics
induced by $\tau$ and $\tau_{\psi}$ are uniformly equivalent.

\begin{corollary}\label{coro-equiv-metrics-0}
Let $(\psi_t)_{0\,\leq\, t\,\leq\, 1}$ be a path of solutions in $\mathcal{M}^{\infty}_{X,\,\exp}(M)$ to \eqref{MA-cpct-supp-t} and
for $t\in[0,\,1]$, let $h_{\psi_t}$  be the K\"ahler metric induced by $\tau_{\psi_t}$. Then the tensors
$h^{-1}h_{\psi_t}$ and $h_{\psi_t}^{-1}h$ satisfy the following uniform estimate:
\begin{equation*}
\sup_{0\,\leq\, t\,\leq\, 1}\|h^{-1}h_{\psi_t}\|_{C^{0}}+\sup_{0\,\leq\, t\,\leq\, 1}\|h_{\psi_t}^{-1}h\|_{C^{0}}\leq C
\end{equation*}
for some positive constant $C=C\left(n,\tau,\supp(F),\|f\cdot F\|_{C^2_{X,\,\exp}}\right)$.
In particular, the metrics $h$ and $(h_{\psi_t})_{0\,\leq\, t\,\leq\, 1}$ are uniformly equivalent.
\end{corollary}

\begin{proof}
By Proposition \ref{prop-C^2-est}, we know that $$\sup_{0\,\leq\, t\,\leq\, 1}\|h^{-1}h_{\psi_t}\|_{C^{0}}\leq C\left(n,\tau,\supp(F),\|f\cdot F\|_{C^2_{X,\,\exp}}\right).$$
Moreover, by Proposition \ref{prop-bd-uni-X-psi}, $h^{-1}h_{\psi_t}$ satisfies
\begin{equation*}
\det(h^{-1}h_{\psi_t})=e^{F-\frac{X}{2}\cdot\psi_t}\geq e^{-C}
\end{equation*}
for some uniform positive constant $C=C(n,\tau,F)$.
Furthermore, each eigenvalue of $h_{\psi_t}$ is uniformly bounded from below by a positive constant.
Hence we conclude that $$\sup_{0\,\leq\, t\,\leq\, 1}\|h_{\psi_t}^{-1}h\|_{C^0}\leq C\left(n,\tau,\supp(F),\|f\cdot F\|_{C^2_{X,\,\exp}}\right).$$
\end{proof}

\subsubsection{$C^3$ a priori estimate}
We now present the $C^{3}$-estimate.
\begin{prop}[A priori $C^3$-estimate]\label{prop-C^3-est}
Let $(\psi_t)_{0\,\leq\, t\,\leq\, 1}$ be a path of solutions in $\mathcal{M}_{X,\,\exp}^{4,\,2\alpha}(M)$ to \eqref{MA-cpct-supp-t} and let
$h$ be the K\"ahler metric induced by $\tau$ with Levi-Civita connection $\nabla^{h}$. Then
\begin{equation*}
\sup_{0\,\leq\, t\,\leq\, 1}\|\nabla^{h}\partial\bar{\partial}\psi_t\|_{C^0}\leq C\left(n,\tau,\|f\cdot F\|_{C^{4,\,2\alpha}_{X,\,\exp}}\right).
\end{equation*}
\end{prop}

\begin{proof}
We follow closely the proof given in \cite[Proposition $6.9$]{con-der} which itself is based on \cite{Pho-Ses-Stu}.

For the sake of clarity, we drop the dependence of the potential $\psi_t$ and the data $tF$ on the parameter $t\in[0,\,1]$.
Set
$$S(h_{\psi},h):=\arrowvert\nabla^hh_{\psi}\arrowvert^2_{h_{\psi}}.$$
Then from the definition of $S$, we see that
\begin{equation*}
\begin{split}
S(h_{\psi},h)=&h_{\psi}^{i\bar{\jmath}}h_{\psi}^{k\bar{l}}h_{\psi}^{p\bar{q}}\nabla^{h}_i{h_{\psi}}_{kp}\overline{\nabla^{h}_{j}{h_{\psi}}_{lq}}\\
=&|\Psi|_{h_{\psi}}^2,
\end{split}
\end{equation*}
where
\begin{equation*}
\begin{split}
\Psi_{ij}^k(h_{\psi},h)&:=\Gamma(h_{\psi})_{ij}^k-\Gamma(h)_{ij}^k\\
&=h_{\psi}^{k\bar{l}}\nabla^h_i(h_{\psi})_{j\bar{l}}.
\end{split}
\end{equation*}
Now, since $\psi$ solves \eqref{MA-cpct-supp-t}, $(M,\,h_{\psi},\,X)$
is an ``approximate'' steady gradient K\"ahler-Ricci soliton in the following precise sense:
if $h_{\psi}(s):=(\phi^{X}_{s})^*h_{\psi}$ and $h(s):=(\phi^{X}_{s})^*h$,
where $(\phi^{X}_{s})_{s\,\in\,\R}$ is the
one-parameter family of diffeomorphisms generated by $-\frac{X}{2}$,
then $(h_{\psi}(s))_{s\,\in\,\R}$
is a solution of the following perturbed K\"ahler-Ricci flow with initial condition $h_{\psi}$:
\begin{equation*}
\begin{split}
\partial_{s}h_{\psi}(s)&=-\Ric(h_{\psi}(s))+(\phi^{X}_{s})^*\left(-\mathcal{L}_{\frac{X}{2}}h+\Ric(h)+\nabla^{h}\bar{\nabla}^hF\right),\qquad s\in\R,\\
h_{\psi}(0)&=h_{\psi}.
 \end{split}
\end{equation*}
In particular, $\partial_{s}h_{\psi}=-\Ric(h_{\psi})
+(\phi^{X}_{\tau})^*\Lambda$, where $\Lambda:=
-\mathcal{L}_{\frac{X}{2}}h+\Ric(h)+\nabla^{h}\bar{\nabla}^hF$ has uniformly controlled $C^1$-norm
as $h$ is asymptotic to $\hat{g}$ with derivatives (cf.~\eqref{asy-cao-def-sec5})
and $F$ is compactly supported.

Define $S(s):=S(h_{\psi}(s),\,h(s))$ and correspondingly set $\Psi(s):=\Psi(h_{\psi}(s),\,h(s))$. We adapt \cite[Proposition 3.2.8]{Bou-Eys-Gue} to our setting. By a
brute force computation, we have that
\begin{equation*}
\begin{split}
\Delta_{\tau_{\psi}}S&=2\Re\left(h_{\psi}^{i\bar{\jmath}}h_{\varphi}^{p\bar{q}}{h_{\psi}}_{k\bar{l}}\left(\Delta_{\tau_{\psi},\,1/2}\Psi_{ip}^k\right)
\overline{\Psi_{jq}^l}\right)+|\nabla^{h_{\psi}} \Psi|^2_{h_{\psi}}+|\overline{\nabla}^{h_{\psi}}\Psi|_{h_{\psi}}^2\\
&\qquad+\Ric(h_{\psi})^{i\bar{\jmath}}h_{\psi}^{p\bar{q}}{h_{\psi}}_{k\bar{l}}\Psi_{ip}^k\overline{\Psi_{jq}^l}
+h_{\psi}^{i\bar{\jmath}}\Ric(h_{\psi})^{p\bar{q}}{h_{\psi}}_{k\bar{l}}\Psi_{ip}^k\overline{\Psi_{jq}^l}-h_{\psi}^{i\bar{\jmath}}h_{\psi}^{p\bar{q}}\Ric(h_{\psi})_{k\bar{l}}\Psi_{ip}^k\overline{\Psi_{jq}^l},
\end{split}
\end{equation*}
where
\begin{equation*}
\begin{split}
&\Delta_{\tau_{\psi},\,1/2}:=h_{\varphi}^{i\bar{\jmath}}\nabla^{h_{\varphi}}_i\nabla^{h_{\varphi}}_{\bar{\jmath}},\label{def-lap-half}\\
&T^{i\bar{\jmath}}:=h_{\psi}^{i\bar{k}}h_{\psi}^{l\bar{\jmath}}T_{k\bar{l}},
\end{split}
\end{equation*}
for $T_{k\bar{l}}\in\Lambda^{1,\,0}M\otimes\Lambda^{0,\,1}M$. We also have that
\begin{equation*}
\begin{split}
\partial_{u}\Psi(u)_{ip}^k|_{u\,=\,0}&=\partial_{u}|_{u\,=\,0}(\Gamma(h_{\psi}(u))-\Gamma(h(u)))_{ip}^k\\
&=\nabla^{h_{\psi}}_i(-\Ric(h_{\psi})_p^k+\Lambda_p^k)-\nabla^{h}_i(-\mathcal{L}_{\frac{X}{2}}h_{p}^{k}),\\
\partial_{u}h_{\psi}^{i\bar{\jmath}}|_{u\,=\,0}&=\Ric(h_{\psi})^{i\bar{\jmath}}-\Lambda^{i\bar{\jmath}}.
\end{split}
\end{equation*}
Finally, using the second Bianchi identity, we compute that
\begin{equation*}
\Delta_{h_{\psi},\,1/2}\Psi_{ip}^k=h_{\psi}^{a\bar{b}}\nabla_a^{h_{\psi}}\Rm(h)_{i\bar{b}p}^k-\nabla^{h_{\psi}}_i\Ric(h_{\psi})_p^k,
\end{equation*}
which in turn implies that the following evolution equation is satisfied by $\Psi$:
\begin{equation*}
{\partial_{u}\Psi_{ip}^k(u)|_{u\,=\,0}}=\Delta_{h_{\psi},\,1/2}\Psi_{ip}^k+T_{ip}^k
\end{equation*}
for a tensor $T$ of the form
\begin{equation*}
\begin{split}
T&=h_{\psi}^{-1}\ast\nabla^{h_{\psi}}\Rm(h)+\nabla^{h_{\psi}}\Lambda+\nabla^h(\mathcal{L}_{\frac{X}{2}}(h))\\
&=h_{\psi}^{-1}\ast\nabla^h\Rm(h)+h_{\psi}^{-1}\ast h_{\psi}^{-1}\ast\Rm(h)\ast\Psi+h_{\psi}^{-1}\ast\Psi\ast \Lambda+\nabla^h(\Lambda+\mathcal{L}_{\frac{X}{2}}(h)).
\end{split}
\end{equation*}
Notice the simplification here regarding the ``bad'' term $-\nabla^{h_{\psi}}\Ric(h_{\psi})$.
Since this flow is evolving only by diffeomorphism, we know that
\begin{equation*}
\begin{split}
S(s)&=(\phi^{X}_{s})^*S(h_{\psi},\,h),\\
\partial_{u}S|_{u\,=\,0}&=-\frac{X}{2}\cdot S(h_{\psi},\,h).
\end{split}
\end{equation*}
Hence Young's inequality, together with
the boundedness of $\|h_{\psi}^{-1}h\|_{C^0(M)}$ and $\|h_{\psi}h^{-1}\|_{C^0(M)}$ ensured by Corollary \ref{coro-equiv-metrics-0}
and the boundedness of the covariant derivatives of the tensors $\Rm(h)$ and $\Lambda$, imply that
\begin{eqnarray*}
\Delta_{h_{\psi}}S+\frac{X}{2}\cdot S\geq -C(S+1)
\end{eqnarray*}
for some positive uniform constant $C$.

We use as a barrier function the trace $\tr_{\tau}(\tau_{\psi})$ which, by (\ref{crucial-est-tr-C2}) and the uniform equivalence of the metrics $h$ and $h_{\psi}$ provided by Corollary \ref{coro-equiv-metrics-0}, satisfies
\begin{equation*}
\Delta_{\tau_{\psi}}\tr_{\tau}(\tau_{\psi})+\frac{X}{2}\cdot \tr_{\tau}(\tau_{\psi})\geq C^{-1}S-C,
\end{equation*}
where $C$ is a uniform positive constant that may vary from line to line. By applying the maximum principle to $\varepsilon S+\tr_{\tau}(\tau_{\psi})$ for some sufficiently small $\varepsilon>0$, one arrives at the desired a priori estimate.
\end{proof}

We next establish H\"older regularity of $h^{-1}h_{\psi_{t}}$ and $h^{-1}_{\psi_{t}}h$,
an improvement on Corollary \ref{coro-equiv-metrics-0}.

\begin{corollary}\label{coro-equiv-metrics}
Let $(\psi_t)_{0\,\leq\, t\,\leq\, 1}$ be a path of solutions in $\mathcal{M}_{X,\,\exp}^{4,\,2\alpha}(M)$ to \eqref{MA-cpct-supp-t}
and for $t\in[0,\,1]$, let $h_{\psi_t}$ be the K\"ahler metric induced by $\tau_{\psi_t}$.
Then for any $\alpha\in\left(0,\,\frac{1}{2}\right)$, the tensors $h^{-1}h_{\psi_t}$ and $h_{\psi_t}^{-1}h$ satisfy the following uniform estimate:
 \begin{equation*}
\sup_{0\,\leq\, t\,\leq\, 1}\left(\|h^{-1}h_{\psi_t}\|_{C_{\operatorname{loc}}^{0,\,2\alpha}}+\|h_{\psi_t}^{-1}h\|_{C_{\operatorname{loc}}^{0,\,2\alpha}}\right)\leq C\left(n,\alpha,\tau,\supp(F),\|f\cdot F\|_{C^{4,\,2\alpha}_{X,\,\exp}}\right).
\end{equation*}
\end{corollary}

 \begin{proof}
As usual, we suppress the dependence of the solutions $\psi_t$ on the parameter $t\in[0,\,1]$ to lighten the notation.
The same statement applies to the data $tF$.

By standard local interpolation inequalities applied to Propositions \ref{prop-C^2-est} and \ref{prop-C^3-est}, we see that
\begin{equation*}
\|h^{-1}h_{\psi}\|_{C_{\operatorname{loc}}^{0,\,2\alpha}}\leq C\left(n,\alpha,\supp(F),\tau,\|f\cdot F\|_{C^{4,\,2\alpha}_{X,\,\exp}}\right).
\end{equation*}
 Combining the previous estimate with Corollary \ref{coro-equiv-metrics-0}, it suffices to prove a uniform bound on the local $2\alpha$-H\"older norm of $h_{\psi}^{-1}h$.
We conclude with the following observation: if $u$ is a positive function on $M$ in $C_{\operatorname{loc}}^{2\alpha}(M)$ uniformly bounded from below by a positive constant, then $[u^{-1}]_{2\alpha}\leq [u]_{2\alpha}(\inf_Mu)^{-2}$. By invoking Corollary \ref{coro-equiv-metrics-0} once more, this last remark applied to $h_{\psi}^{-1}h$ implies that
\begin{equation*}
\|h_{\psi}^{-1}h\|_{C_{\operatorname{loc}}^{0,\,2\alpha}}\leq C\left(n,\alpha,\supp(F),\tau,\|f\cdot F\|_{C^{4,\,2\alpha}_{X,\,\exp}}\right)
\end{equation*}
 as well.
\end{proof}

\subsubsection{Local bootstrapping}
We now improve the local regularity of our continuity path of solutions to \eqref{MA-cpct-supp-t}. This estimate will be used
in deriving the subsequent weighted a priori estimates.

\begin{prop}\label{prop-loc-holder-C-3}
Let $(\psi_t)_{0\,\leq\, t\,\leq\, 1}$ be a path of solutions in $\mathcal{M}_{X,\,\exp}^{4,\,2\alpha}(M)$, $\alpha\in\left(0,\,\frac{1}{2}\right)$,
to \eqref{MA-cpct-supp-t}. Then for any $\alpha\in\left(0,\,\frac{1}{2}\right)$,
\begin{equation*}
\sup_{0\,\leq\, t\,\leq\, 1}\|\psi_t\|_{C_{\operatorname{loc}}^{3,\,2\alpha}}\leq C\left(n,\alpha,\tau,\|f\cdot F\|_{C^{4,\,2\alpha}_{X,\,\exp}}\right).
\end{equation*}
\end{prop}

\begin{proof}
Again suppressing the dependence of the solutions $\psi_t$ and data $tF$
on the parameter $t\in[0,\,1]$, we see from the proof of the a priori $C^2$-estimate (cf.~Proposition \ref{prop-C^2-est}
and \eqref{eq-sec-der-yau}) that
\begin{equation}\label{weloveele}
\begin{split}
\Delta_{\tau_{\psi}}\left(\Delta_{\tau}\psi+\frac{X}{2}\cdot\psi\right) =& \Delta_{\tau}F+h_{\psi}^{-1}\ast h^{-1}\ast\Rm(h)+\Rm(h)\ast\nabla\bar{\nabla}\psi\ast h_{\psi}^{-1}\\
&+h^{-1}\ast h^{-1}\ast\Rm(h)+h^{-1}\ast h_{\psi}^{-1}\ast h_{\psi}^{-1}\ast \bar{\nabla}\nabla\bar{\nabla}\psi\ast\nabla\bar{\nabla}\nabla \psi\\
&+\left(\Delta_{\tau_{\psi}}-\Delta_{\tau}\right)\left(\frac{X\cdot \psi}{2}\right),
\end{split}
\end{equation}
where $\ast$ denotes the ordinary contraction of two tensors. Notice that
\begin{equation}
\begin{split}\label{easy-obs-diff-lap}
\left|\left(\Delta_{\tau_{\psi}}-\Delta_{\tau}\right)(X\cdot \psi)\right|&=\left|h_{\psi}^{-1}\ast\partial\bar{\partial}\psi\ast \partial\bar{\partial}(X\cdot\psi)\right|\\
&\leq \|h_{\psi}^{-1}h\|_{C^0}\cdot \|i\partial\bar{\partial}\psi\|_{C^0}\cdot \left(\|\nabla^{h}\psi\|_{C^0}+\|i\partial\bar{\partial}\psi\|_{C^0}+\|i\partial\bar{\partial}\partial\psi\|_{C^0}\right).
\end{split}
\end{equation}
Here we have used the boundedness of the derivatives of the vector field $X$ given by Lemma \ref{lemm-app-sol-id} with respect to the norm induced by $\tau$.

By Propositions \ref{prop-C^2-est} and \ref{prop-C^3-est} together with \eqref{easy-obs-diff-lap},
the $C^0$-norm of the right-hand side of \eqref{weloveele} is uniformly bounded and, thanks to Corollary \ref{coro-equiv-metrics}, so too are the coefficients of $\Delta_{\tau_{\psi}}$ in the $C^{0,\,2\alpha}_{\operatorname{loc}}$-sense. As a result, by applying the Morrey-Schauder $C^{1,\,2\alpha}$-estimates, we see that for any $x\in M$ and for $\delta<\inj_{h}(M)$,
\begin{equation*}
\left\|\Delta_{\tau}\psi+\frac{X}{2}\cdot\psi\right\|_{C^{1,\,2\alpha}(B_{h}(x,\,\delta))}\leq C\left(n,\alpha,\tau,\supp(F),\|f\cdot F\|_{C^{4,\,2\alpha}_{X,\,\exp}}\right).
\end{equation*}
Finally, applying standard interior Schauder estimates for elliptic equations once again with respect to $\Delta_{\tau}$ leads to the bound
\begin{equation*}
\begin{split}
\|\psi\|_{C^{3,\,2\alpha}(B_{h}(x,\,\frac{\delta}{2}))}&\leq C(n,\alpha,\tau)\left(\|\Delta_{\tau}\psi\|_{C^{1,\,2\alpha}(B_{h}(x,\,\delta))}+\|\psi\|_{C^{1,\,2\alpha}(B_{h}(x,\,\delta))}\right)\\
&\leq C\left(n,\alpha,\tau,\supp(F),\|f\cdot F\|_{C^{4,\,2\alpha}_{X,\,\exp}}\right).
\end{split}
\end{equation*}
\end{proof}

We next establish the following well-known local regularity result for solutions to (\ref{MA-cpct-supp}).
\begin{prop}\label{prop-loc-reg}
Let $F\in C^{k,\,\alpha}_{\operatorname{loc}}(M)$ for some $k\geq1$ and $\alpha\in(0,\,1)$ and let $\psi\in C^{3,\,\alpha}_{\operatorname{loc}}(M)$ be a solution to (\ref{MA-cpct-supp}) with data $F$. Then $\psi\in C^{k+2,\alpha}_{\operatorname{loc}}(M)$. Moreover, for all $k\geq 1$, $\alpha\in(0,\,1)$, and $x\in M$,
\begin{equation*}
\begin{split}
\|\psi\|_{C^{k+2,\alpha}(B_{h}(x,\,\frac{\delta}{2}))}\leq C(n,k,\alpha,\tau)\left(\|F\|_{C^{k,\,\alpha}(B_{h}(x,\,\delta))}+\|\psi\|_{C^{3,\,\alpha}(B_{h}(x,\,\delta))}\right),\qquad \delta<\inj_{h}(M).
\end{split}
\end{equation*}
\end{prop}
\begin{proof}
We prove this proposition by induction on $k\geq 1$. The case $k=1$ is true by assumption, so let $F\in C^{k+1,\alpha}_{\operatorname{loc}}(M)$ and let $\psi\in C^{3,\,\alpha}_{\operatorname{loc}}(M)$ be a solution of
\eqref{MA-cpct-supp}. Then by induction, $\psi\in C^{k+2,\alpha}_{\operatorname{loc}}(M)$.
Let $x\in M$ and choose local holomorphic coordinates defined on $B_{h}(x,\,\delta)$ for some $0<\delta<\inj_{h}(M)$. Then since $\psi$ satisfies
\begin{equation*}
F=\log\left(\frac{\tau_{\psi}^n}{\tau^n}\right)+\frac{X}{2}\cdot\psi,
\end{equation*}
we know that for $j=1,...,2n$, the derivative $\partial_j\psi$ satisfies
\begin{eqnarray*}
\Delta_{\tau_{\psi}}\left(\partial_j\psi\right)=\partial_j\left(F-\frac{X}{2}\cdot\psi\right)\in C^{k,\,\alpha}_{\operatorname{loc}}(M).
\end{eqnarray*}
As the coefficients of $\Delta_{\tau_{\psi}}$ are in $C^{k,\,\alpha}_{\operatorname{loc}}(M)$,
an application of the standard interior Schauder estimates for elliptic equations now gives us the
desired local regularity result, namely $\partial_j\psi\in C^{k+2,\alpha}_{\operatorname{loc}}(M)$ for all $j=1,...,2n$, or equivalently, $\psi\in C^{k+3,\alpha}_{\operatorname{loc}}(M)$ together with the expected estimate.
\end{proof}

\subsection{Weighted a priori estimates}
We next deal with the weighted a priori estimates on derivatives of solutions to
\eqref{MA-cpct-supp} along the continuity path, beginning first with the weighted $C^0$-estimate.

\subsubsection{Weighted $C^0$ a priori estimate}
\begin{prop}[Weighted $C^0$ a priori estimate]\label{prop-a-prio-wei-C0-est}
Let $(\psi_t)_{0\,\leq\, t\,\leq\, 1}$ be a path of solutions in $\mathcal{M}_{X,\,\exp}^{2k+2,\,\alpha}(M)$,
$k\geq1$, $\alpha\in\left(0,\,\frac{1}{2}\right)$, to (\ref{MA-cpct-supp-t}). Then there exists a positive constant $C$ such that
\begin{equation}
\sup_{0\,\leq\, t\,\leq\, 1}\sup_M\left|e^f\psi_t\right|\leq C\left(n,\tau,\|f\cdot F\|_{C^{4,\,2\alpha}_{X,\,\exp}}\right),\label{a-priori-wei-bd-c-0}
\end{equation}
where $C\left(n,\tau,\|f\cdot F\|_{C^{4,\,2\alpha}_{X,\,\exp}}\right)$ is bounded by a constant $C(n,\tau,\Lambda)$ depending only on an upper bound $\Lambda$ of $\|f \cdot F\|_{C^{4,\,2\alpha}_{X,\,\exp}}$.
\end{prop}

\begin{proof}
We begin with an upper bound for $e^f\psi$. First note that $\psi$ satisfies the differential inequality
\begin{equation*}
 -\|fe^fF\|_{C^0}\frac{e^{-f}}{f}\leq tF=\log\left(\frac{\tau_{\psi}^n}{\tau^{n}}\right)+\frac{X}{2}\cdot\psi\leq \Delta_{\tau}\psi+\frac{X}{2}\cdot \psi.
\end{equation*}
Moreover, by \eqref{sub-sol-exp-pot-fct} of Lemma \ref{lemm-barr-inf}, we see that
\begin{equation*}
\begin{split}
\left(\Delta_{\tau}+\frac{X}{2}\cdot\right)e^{-f}&=-e^{-f}\Delta_{\tau} f \leq -\frac{c}{f}e^{-f}
\end{split}
\end{equation*}
outside some compact subset $K\subset M$ independent of $\psi$ for some $c>0$.
Thus, one obtains, for any positive constant $A$, the lower bound
\begin{equation}
\begin{split}\label{sub-diff-inequ-psi-barrier}
&\left(\Delta_{\tau}+\frac{X}{2}\cdot\right)\left(\psi-Ae^{-f}\right)\geq  \frac{c\cdot A}{f}e^{-f}-\|fe^fF\|_{C^0}\frac{e^{-f}}{f}
\end{split}
\end{equation}
on $M\setminus K$. In particular, choosing $A$ so that $A>c^{-1}\|fe^fF\|_{C^0}$, the maximum principle
applied to \eqref{sub-diff-inequ-psi-barrier} shows that
\begin{equation*}
\sup_{M\setminus K}(\psi-Ae^{-f})=\max\left\{0,\,\max_{\partial K}(\psi-Ae^{-f})\right\},
\end{equation*}
as both $\psi$ and $e^{-f}$ converge (exponentially) to $0$ at infinity.
Now, by Lemma \ref{lemma-loc-crit-pts} and Proposition \ref{prop-bd-abo-uni-psi}, we know that $$\max_{\partial K}(\psi-Ae^{-f})\leq C-Ae^{-\max_{\partial K}f}$$
for some uniform constant $C$. Hence one can choose $A$ large enough so that $\max_{\partial K}\left(\psi-Ae^{-f}\right)\leq 0$. This establishes the expected a priori weighted upper bound.

As for the lower bound, we proceed in two steps, the first establishing a ``rough'' lower bound on $\psi$ in the following claim.

\begin{claim}\label{claim-first-dec-low-bd-c-0}
There exists $\delta=\delta(F,\tau,n)\in(0,\,1)$ such that $\psi\geq-Ce^{-\delta f}$ for some uniform positive constant $C$.
\end{claim}

\begin{proof}
Consider the function $\chi_{B,\,\delta}:=-B e^{-\delta f}$ for some $B>0$ and $\delta\in(0,\,1)$ to be specified later.
Since $|i\partial\bar{\partial}f|_{h}=O(f^{-1})$ and $|X|_{h}$ is bounded, there exists a uniform constant $c_{1}(\tau,\,n)$
such that $|i\partial\bar{\partial}(Be^{-\delta f})|_{h}<1$ on the set where $B\delta e^{-\delta f}\leq c_{1}$.
By linearising around $\tau$ and once again using the fact that $|i\partial\bar{\partial}f|_{h}=O(f^{-1})$ and $|X|_{h}$ is bounded,
one obtains the following estimate on this set:
 \begin{equation*}
 \begin{split}
\log\left(\frac{\tau_{\chi_{B,\,\delta}}^n}{\tau^n}\right)&=-B\Delta_{\tau}e^{-\delta f}-\int_0^1\int_0^u|i\partial\bar{\partial}\chi_{B,\,\delta}|^{2}_{h_{s\chi_{B,\,\delta}}}\,ds\,du\\
&\geq -B\Delta_{\tau}e^{-\delta f}-c\left|Bi\partial\bar{\partial}e^{-\delta f}\right|^2_{h}\\
&\geq -B\Delta_{\tau}e^{-\delta f}-c_{2}(\tau,\,n)(B\delta)^2e^{-2\delta f}
\end{split}
\end{equation*}
for $c_{2}(\tau,\,n)$ a sufficiently large positive constant. As a consequence, we find using Lemma \ref{lemm-barr-inf} that
\begin{equation*}
\begin{split}
\log\left(\frac{\tau_{\chi_{B,\,\delta}}^n}{\tau^n}\right)+\frac{X}{2}\cdot\chi_{B,\,\delta}&\geq -B
\left(\Delta_{\tau}+\frac{X}{2}\cdot\right)e^{-\delta f}-c_{2}(\tau,\,n)(B\delta)^2e^{-2\delta f}\\
&\geq B\delta(1-\delta)\frac{4n}{2}e^{-\delta f}-c_{2}(\tau,\,n)(B\delta)^2e^{-2\delta f}-Bc_{3}(\tau,\,n)f^{-1}e^{-\delta f}\\
&\geq Bn\delta(1-\delta)e^{-\delta f}
\end{split}
\end{equation*}
on the set where
\begin{equation}\label{finally}
B\delta e^{-\delta f}\leq c_{1}\qquad\textrm{and}\qquad c_{2}B\delta^{2}e^{-\delta f}+c_{3}f^{-1}\leq\delta(1-\delta)n.
\end{equation}
It follows that on this set we have that
\begin{equation*}
\begin{split}
\log\left(\frac{(\tau_{\psi}+i\partial\bar{\partial}(\chi_{B,\,\delta}-\psi))^n}{\tau_{\psi}^n}\right)+\frac{X}{2}\cdot(\chi_{B,\,\delta}-\psi)
&=\log\left(\frac{\tau_{\chi_{B,\,\delta}}^n}{\tau_{\psi}^n}\right)+\frac{X}{2}\cdot(\chi_{B,\,\delta}-\psi)\\
&=\log\left(\frac{\tau_{\chi_{B,\,\delta}}^n}{\tau^n}\right)-\log\left(\frac{\tau_{\psi}^n}{\tau^n}\right)+\frac{X}{2}\cdot(\chi_{B,\,\delta}-\psi)\\
&\geq Bn\delta(1-\delta)e^{-\delta f}-F.\\
&\geq-F.
\end{split}
\end{equation*}
Thus, letting $R(B,\,\delta)>0$ be such that \eqref{finally} holds true on $\{f\geq R\}$ and
$\{f\geq R\}\subset M\setminus\supp(F)$, we find from the maximum principle that
\begin{equation*}
\max_{\{f\,\geq\, R\}}(\chi_{B,\,\delta}-\psi)=\max\left\{0,\,\max_{\{f\,=\,R\}}(\chi_{B,\,\delta}-\psi)\right\}.
\end{equation*}
Now, Lemma \ref{lemma-loc-crit-pts} and Proposition \ref{prop-bd-bel-uni-psi} imply that
\begin{eqnarray*}
\max_{\{f\,=\,R\}}(\chi_{B,\,\delta}-\psi)\leq -\inf_M\psi-B e^{-\delta R}\leq c_{4}(\tau,n,F)-B e^{-\delta R}\leq 0
\end{eqnarray*}
as soon as $B \geq c_{4}(\tau,\,n,\,F)e^{\delta R}$. Choose $B>c_{4}e^{c_{3}}$
and $\delta\in(0,\,1)$ such that
\begin{equation*}
B\delta e^{-c_{3}}\leq c_{1},\qquad c_{2}B\delta^{2}e^{-c_{3}}+\delta\leq\delta(1-\delta)n,\qquad\textrm{and}\qquad
\left\{f\geq\frac{c_{3}}{\delta}\right\}\subset M\setminus\supp(F),
\end{equation*}
and set $R:=\frac{c_{3}}{\delta}$. Then with these choices, we obtain
the desired lower bound. This completes the proof of the claim.
\end{proof}

The second step mimics the approach already taken to establish the upper bound on $e^f\psi$. Namely,
we linearise the complex Monge-Amp\`ere equation \eqref{MA-cpct-supp} using Claim \ref{claim-first-dec-low-bd-c-0} and view it as a linear PDE with data decaying as fast as $F$. We then apply the minimum principle to ameliorate the decay of $\psi$ given by Claim \ref{claim-first-dec-low-bd-c-0} before iterating the whole argument, incrementally improving the decay of $\psi$ each time. To this end, recall the Taylor expansion of order two with integral remainder that was established in \eqref{equ:taylor-exp}:
\begin{equation*}
 \begin{split}
\log\left(\frac{\tau_{\psi}^n}{\tau^n}\right)+\frac{X}{2}\cdot\psi&=0+\left.\frac{d}{ds}\right|_{s\,=\,0}\left(\log\left(\frac{\tau_{s\psi}^n}{\tau^n}\right)+\frac{X}{2}\cdot s\psi\right)\\
&\qquad+\int_0^1\int_0^{u}\frac{d^2}{ds^2}\left(\log\left(\frac{\tau_{s\psi}^n}{\tau^n}\right)+\frac{X}{2}\cdot s\psi\right)\,ds\,du\\
 &=\Delta_{\tau}\psi+\frac{X}{2}\cdot\psi-\int_0^1\int_0^{u}\arrowvert \partial\bar{\partial}\psi\arrowvert^2_{h_{s\psi}}\,ds\,du.
 \end{split}
 \end{equation*}
As a solution of \eqref{MA-cpct-supp}, $\psi$ must satisfy
 \begin{equation}\label{equ-lin-MA-cpst-supp}
 \Delta_{\tau}\psi+\frac{X}{2}\cdot\psi=F+\int_0^1\int_0^{u}\arrowvert \partial\bar{\partial}\psi\arrowvert^2_{h_{s\psi}}\,ds\,du.
 \end{equation}
We consider the right-hand side of \eqref{equ-lin-MA-cpst-supp} as the data, that is to say,
we view \eqref{equ-lin-MA-cpst-supp} as a linear equation in $\psi$. Then by Claim \ref{claim-first-dec-low-bd-c-0}
and the first part of the proof of Proposition \ref{prop-a-prio-wei-C0-est}, we know that
\begin{equation}
|\psi|\leq c(n,\tau,F)e^{-\delta_0 f}\label{rough-dec-psi-wei-a-priori-covid}
\end{equation}
 for some $\delta:=\delta_0(n,\tau,F)\in(0,\,1)$. Let $x\in M$ and choose holomorphic coordinates centered at $x$ in a ball $B_{h}(x,\,\delta)$
for some $\delta<\inj_{h}(M)$. Then the reminder form of Taylor's theorem shows that in these coordinates,
 \begin{equation*}
 \begin{split}
F&=\log\left(\frac{\tau_{\psi}^n}{\tau^n}\right)+\frac{X}{2}\cdot\psi\\
&=\left(\int_0^1 h_{s\psi}^{i\bar{\jmath}}\,ds\right)\partial_i\partial_{\bar{\jmath}}\psi+\frac{X}{2}\cdot\psi\\
&=:a^{i\bar{\jmath}}\partial_i\partial_{\bar{\jmath}}\psi+\frac{X}{2}\cdot\psi.
\end{split}
\end{equation*}
Now, by Corollary \ref{coro-equiv-metrics}, $\|a^{i\bar{\jmath}}\|_{C^{0,\,2\alpha}(B_{h}(x,\,\delta))}$ is uniformly bounded from above and $a^{i\bar{\jmath}}\geq \Lambda^{-1}\delta_{i\bar{\jmath}}$ on $B_{h}(x,\,\delta)$ for some uniform constant $\Lambda>0$. The standard interior Schauder estimates for elliptic equations therefore
apply and tell us that
\begin{equation*}
\|\psi\|_{C^{2,\,2\alpha}(B_{h}(x,\,\frac{\delta}{2}))}\leq C\left(\left\|\psi\right\|_{C^{0}(B_{h}(x,\,\delta))}+\|F\|_{C^{0,\,2\alpha}(B_{h}(x,\,\delta))}\right)
\end{equation*}
for some uniform positive constant $C=C\left(n,\alpha,\tau,\supp(F),\|f\cdot F\|_{C^{4,\,2\alpha}_{X,\,\exp}}\right)$. Using \eqref{rough-dec-psi-wei-a-priori-covid},
this leads to the estimate
\begin{equation*}
\|\psi\|_{C^{2,\,2\alpha}(B_{h}(x,\,\frac{\delta}{2}))}\leq C\left(n,k,\tau,\supp(F),\|f\cdot F\|_{{C^{4,\,2\alpha}_{X,\,\exp}}}\right)e^{-\delta_0 f(x)},\qquad x\in M,
\end{equation*}
from which we deduce that $|i\partial\bar{\partial}\psi|_{h}\leq c(n,\tau,F)e^{-\delta_0 f}$.
In light of \eqref{equ-lin-MA-cpst-supp}, it subsequently follows that
 \begin{equation*}
 \left|\Delta_{\tau}\psi+\frac{X}{2}\cdot\psi\right|\leq c(n,\tau,F)e^{-2\delta_0f}+F.
 \end{equation*}

If $2\delta_0<1$, then observe from
Lemma \ref{lemm-barr-inf} that the function $e^{-2\delta_0f}$ is a good barrier function in the sense that
 \begin{equation*}
\left(\Delta_{\tau}+\frac{X}{2}\cdot\right)e^{-2\delta_0f}\leq -2\delta_0(1-2\delta_0)\frac{4n}{2}e^{-2\delta_0f}
\end{equation*}
outside a compact subset of $M$. In particular, the function $\psi+Ce^{-2\delta_0f}$ for $C$ a positive constant to be determined satisfies
\begin{equation*}
\left(\Delta_{\tau}+\frac{X}{2}\cdot\right)\left(\psi+Ce^{-2\delta_0f}\right)< 0
\end{equation*}
outside a compact set (which itself does not depend on the solution).
By applying the minimum principle to the function $\psi+Ce^{-2\delta_0f}$
and arguing as for the supremum bound in the first part of the proof,
one obtains the lower bound $\psi\geq -C(n,\tau,F)e^{-2\delta_0f}$
for some positive constant $C=C(n,\tau,F)$ sufficiently large.
Notice that $2\delta_0>\delta_0$ so that the a priori decay of $\psi$ has improved.
Iterating this argument a finite number of times, we end up with the estimate
 \begin{equation*}
 \left|\Delta_{\tau}\psi+\frac{X}{2}\cdot\psi\right|\leq c(n,\tau,F)\frac{e^{-f}}{f}.
 \end{equation*}
Here, we treat $F$ in \eqref{equ-lin-MA-cpst-supp} as data decaying like $f^{-1}e^{-f}$.
To conclude the proof of the proposition, we argue as for the supremum bound
in the first part of this proof.
 \end{proof}

\subsubsection{$C^4$-weighted estimate}

We next obtain the weighted $C^{4}$ a priori estimate.
\begin{prop}[Weighted $C^{4}$ a priori estimate]\label{prop-C^2-wei-est}
Let $(\psi_t)_{0\,\leq\, t\,\leq\, 1}$ be a path of solutions in $\mathcal{M}_{X,\,\exp}^{4,\,2\alpha}(M)$, $\alpha\in\left(0,\,\frac{1}{2}\right)$ to (\ref{MA-cpct-supp-t}) with $F\in f^{-1}\cdot C^{4,\,2\alpha}_{X,\,\exp}(M)$. Then
\begin{equation*}
\sup_{0\,\leq\, t\,\leq\, 1}\|\psi_t\|_{C^{4,\,2\alpha}_{X,\,\exp}}\leq C\left(n,\alpha,\tau,\supp(F),\|f\cdot F\|_{C^{4,\,2\alpha}_{X,\,\exp}}\right).
\end{equation*}
\end{prop}

\begin{proof}
Let $x\in M$ and choose holomorphic coordinates centred at $x$ in a ball $B_{h}(x,\,\delta)$
for some $\delta<\inj_{h}(M)$. Then we have that
 \begin{equation*}
 \begin{split}
F&=\log\left(\frac{\tau_{\psi}^n}{\tau^n}\right)+\frac{X}{2}\cdot\psi\\
&=\left(\int_0^1 h_{s\psi}^{i\bar{\jmath}}\,ds\right)\partial_i\partial_{\bar{\jmath}}\psi+\frac{X}{2}\cdot\psi\\
&=:a^{i\bar{\jmath}}\partial_i\partial_{\bar{\jmath}}\psi+\frac{X}{2}\cdot\psi.
\end{split}
\end{equation*}
Now, by Propositions \ref{prop-loc-holder-C-3} and \ref{prop-loc-reg}, $\|a^{i\bar{\jmath}}\|_{C^{2,\,2\alpha}(B_{h}(x,\,\delta))}$ is uniformly bounded from above and $a^{i\bar{\jmath}}\geq \Lambda^{-1}\delta_{i\bar{\jmath}}$ on $B_{h}(x,\,\delta)$ for some uniform constant $\Lambda>0$. In particular, standard interior Schauder estimates for elliptic equations imply that
\begin{equation*}
\|\psi\|_{C^{4,\,2\alpha}(B_{h}(x,\,\frac{\delta}{2}))}\leq C\left(\left\|\psi\right\|_{C^{0}(B_{h}(x,\,\delta))}+\|F\|_{C^{2,\,2\alpha}(B_{h}(x,\,\delta))}\right)
\end{equation*}
for some uniform positive constant $C=C\left(n,\alpha,\tau,\supp(F),\|f\cdot F\|_{C^{4,\,2\alpha}_{X,\,\exp}}\right)$.
From the bound [\eqref{a-priori-wei-bd-c-0}, Proposition \ref{prop-a-prio-wei-C0-est}], it then follows that
\begin{equation}
\|\psi\|_{C^{4,\,2\alpha}(B_{h}(x,\,\frac{\delta}{2}))}\leq C\left(n,k,\tau,\supp(F),\|f\cdot F\|_{{C^{4,\,2\alpha}_{X,\,\exp}}}\right)e^{-f(x)},\qquad x\in M.\label{a-priori-wei-bd-c-3-loc-rough}
\end{equation}
Rewriting \eqref{MA-cpct-supp-t} as in \eqref{equ-lin-MA-cpst-supp}, observe that the
right-hand side of \eqref{equ-lin-MA-cpst-supp} now lies in $f^{-1}\cdot C^{2,\,2\alpha}_{X,\,\exp}(M)$ by (\ref{a-priori-wei-bd-c-3-loc-rough}) and that the following uniform estimate holds true:
\begin{equation*}
\left\|f\cdot\int_0^1\int_0^{u}
\arrowvert \partial\bar{\partial}\psi\arrowvert^2_{h_{s\psi}}\,ds\,du\right\|_{C_{X,\,\exp}^{2,\,2\alpha}}\leq
C\left(n,\alpha,\tau,\supp(F),\|f\cdot F\|_{C^{4,\,2\alpha}_{X,\,\exp}}\right).
\end{equation*}
One now obtains the desired result by applying Theorem \ref{iso-sch-Laplacian-exp} with $k=1$ and $\alpha\in\left(0,\,\frac{1}{2}\right)$.
\end{proof}

\subsubsection{Bootstrapping at infinity}
We now bootstrap to obtain higher regularity on $\psi$.
\begin{prop}[Weighted a priori estimates on higher derivatives]\label{prop-a-prio-wei-Ck-est}
Let $(\psi_t)_{0\,\leq\, t\,\leq\, 1}$ be a path of solutions in $\mathcal{M}_{X,\,\exp}^{2k+2,\,\alpha}(M)$, $k\geq 1$, $\alpha\in\left(0,\,\frac{1}{2}\right)$, to (\ref{MA-cpct-supp-t}) with $F\in f^{-1}\cdot C^{\infty}_{X,\,\exp}(M)$. Then $\psi\in C^{\infty}_{X,\,\exp}(M)$. Moreover, one has the following estimate:
\begin{equation*}
\|\psi\|_{C^{2k+2,\,2\alpha}_{X,\,\exp}}\leq C\left(n,k,\alpha,\tau,\supp(F),\|f\cdot F\|_{C^{\max\{2k,\,4\},\,2\alpha}_{X,\,\exp}}\right).
\end{equation*}
\end{prop}

\noindent The proof of this proposition is identical to that of Proposition \ref{prop-C^2-wei-est} and is therefore omitted.

\newpage
\subsection{Proof of Theorem \ref{theo-exi-sol-cpct-supp-data}}
We finally prove Theorem \ref{theo-exi-sol-cpct-supp-data}. Recall the statement:
\begin{customthm}{\ref{theo-exi-sol-cpct-supp-data}}
Let $F$ be a smooth compactly supported $JX$-invariant function on $M$. Then there exists a solution $\psi\in\mathcal{M}^{\infty}_{X,\,\exp}(M)$ to
\begin{equation*}
\log\left(\frac{(\tau+i\partial\bar{\partial}\psi)^{n}}{\tau^{n}}\right)+\frac{X}{2}\cdot\psi=F,\qquad
\tau+i\partial\bar{\partial}\psi>0,\qquad F\in C^{\infty}_0(M).
\end{equation*}
\end{customthm}

\begin{proof}
Given $F\in f^{-1}C^{\infty}_{X,\,\exp}(M)$, define $F_t:=tF\in f^{-1}C^{\infty}_{X,\,\exp}(M)$
for $t\in[0,\,1]$, fix $\alpha\in\left(0,\,\frac{1}{2}\right)$, and set
\begin{equation*}
\begin{split}
S&:=\left\{t\in[0,\,1]\,|\,\textrm{there exists $\psi_t\in C^{\infty}_{X,\,\exp}(M)$ satisfying (\ref{MA-cpct-supp-t}) with data $F_t\in f^{-1}C^{\infty}_{X,\,\exp}(M)$}\right\}.
\end{split}
\end{equation*}
Note that $S\neq\emptyset$ since $0\in S$ (take $\psi_{0}=0$).

We first claim that $S$ is open. Indeed, this follows from Theorem \ref{Imp-Def-Kah-Ste-exp}; if $t_0\in S$, then by Theorem \ref{Imp-Def-Kah-Ste-exp}, there exists $\epsilon_{0}>0$ such that for all $t\in(t_{0}-\epsilon_{0},\,t_{0}+\epsilon_{0})$,
there exists a solution $\varphi_{t}\in \mathcal{M}^{4,\,2\alpha}_{X,\,\exp}(M)$ to $\eqref{MA-cpct-supp-t}$ with data $tF\in f^{-1}C^{2,\,2\alpha}_{X,\,\exp}(M)$. Since the data $tF$ lies in $f^{-1}C^{\infty}_{X,\,\exp}(M)$,
Proposition \ref{prop-a-prio-wei-Ck-est} ensures that for each $t$ in this interval, $\varphi_{t}\in \mathcal{M}^{\infty}_{X,\,\exp}(M)$.
It follows that $(t_{0}-\epsilon_{0},\,t_{0}+\epsilon_{0})\cap[0,\,1]\subseteq S$.

We next claim that $S$ is closed. To see this, take a sequence $(t_k)_{k\,\geq\,0}$ in $S$ converging to some $t_{\infty}\in S$. Then for $F_k:=t_kF$, $k\geq 0$, the corresponding solutions $\psi_{t_k}=:\psi_k$, $k\geq 0$, of \eqref{MA-cpct-supp-t} satisfy
 \begin{equation}
(\tau+i\partial\bar{\partial}\psi_k)^n=e^{F_{{k}}-\frac{X}{2}\cdot\psi_k}\tau^{n},\qquad k\geq 0.\label{MA-seq}
\end{equation}
It is straightforward to check that the sequence $(F_{{k}})_{k\,\geq\,0}$ is uniformly bounded in $f^{-1}C^{4,\,2\alpha}_{X,\,\exp}(M)$. As a consequence, the sequence $(\psi_k)_{k\,\geq\,0}$ is uniformly bounded in $C^{4,\,2\alpha}_{X,\,\exp}(M)$ by Proposition \ref{prop-C^2-wei-est}. The
Arzel\`a-Ascoli theorem therefore allows us to pull out a subsequence of $(\psi_k)_{k\,\geq\,0}$ that converges to some $\psi_{\infty}\in C^{4,\,2\beta}_{\operatorname{loc}}(M)$, $\beta\in(0,\alpha)$. As $(\psi_k)_{k\,\geq\,0}$ is uniformly bounded in $C^{4,\,2\alpha}_{X,\,\exp}(M)$, $\psi_{\infty}$ will also lie in $C^{4,\,2\alpha}_{X,\,\exp}(M)$. We need to show that $(\tau+i\partial\bar{\partial}\psi_{\infty})(x)>0$
at every point $x\in M$. For this, it suffices to show that $(\tau+i\partial\bar{\partial}\psi_{\infty})^n(x)>0$ for every $x\in M$. This is seen to hold true by letting $k$ tend to $+\infty$ (up to a subsequence) in \eqref{MA-seq}. The fact that $\psi_{\infty}\in\mathcal{M}^{\infty}_{X,\,\exp}(M)$ follows from Proposition \ref{prop-a-prio-wei-Ck-est}.

Finally, as an open and closed non-empty subset of $[0,\,1]$, connectedness of $[0,\,1]$ implies that $S=[0,\,1]$. This completes the proof of the theorem.
\end{proof}

\section{Invertibility of the drift Laplacian: polynomial case}\label{section-small-def}

In this section, we introduce function spaces which, rather than being modeled on exponential
weights, are modeled on polynomial weights. We then carry out the corresponding analysis for these spaces as
was implemented in Section \ref{section-small-def-exp} for function spaces with exponential weights.
Our set-up is the same as that outlined at the beginning of Section \ref{main-set-sec-pol}. The definitions from Section \ref{section-fct-spa-exp} carry forward as well.
However, note that as we are now working with polynomial weights, the weight $f$ is comparable to $\hat{f}$ by Lemma \ref{tuxedo}, hence we do \textbf{not} need to assume that \eqref{add-ass} holds true in this section. We begin with the definition of the relevant function spaces.

\subsection{Function spaces}\label{section-fct-spa-pol}

We make the following definitions.
\begin{itemize}
\item For $\beta\in\R$ and $k$ a non-negative integer, define $C_{X,\,\beta}^{2k}(M)$ to be the space of $JX$-invariant continuous functions $u$ on $M$ with $2k$ continuous derivatives such that
\begin{equation*}
\norm{u}_{C^{2k}_{X,\,\beta}} :=\sum_{i+2j\leq2k}\sup_{M}\left|f^{\frac{i}{2}+j+\beta}(\nabla^{h})^i\left(\mathcal{L}_{X}^{(j)}u\right)\right|_h < \infty.
\end{equation*}
Define $C_{X,\,\beta}^{\infty}(M)$ to be the intersection of the spaces $C_{X,\,\beta}^{2k}(M)$ over all $k\in \N_0$.\\

\newpage
\item
 Let $\delta(h)$ denote the injectivity radius of $h$, write $d_h(x,\,y)$ for the distance with respect to $h$ between two points $x,\,y\in M$,
and let $\phi^{X}_{t}$ denote the flow of $X$ for time $t$. A tensor $T$ on $M$ is said to be in $C^{0,\,2\alpha}(M)$, $\alpha\in\left(0,\,\frac{1}{2}\right)$, if
 \begin{equation*}
 \begin{split}
\left[T\right]_{C^{0,\,2\alpha}_{\gamma}}:=&\sup_{\substack{x\,\neq\,y\,\in\,M \\d_{h}(x,y)\,<\,\delta(h)}}\left[\min(f(x),f(y))^{\gamma+\alpha}\frac{\arrowvert T(x)-P_{x,\,y}T(y)\arrowvert_h}{d_{h}(x,\,y)^{2\alpha}}\right]\\
&+\sup_{\substack{x\,\in\, M \\ t\,\neq\,s\,\geq\,1}}\left[\min(t,s)^{\gamma+\alpha}\frac{\arrowvert (\phi^X_t)_{\ast}T(x)-(\widehat{P}_{\phi^X_s(x),\,\phi^X_t(x)}((\phi^X_s)_{\ast}T(x)))\arrowvert_h}{|t-s|^{\alpha}}\right]<+\infty,
\end{split}
\end{equation*}
where $P_{x,\,y}$ denotes parallel transport along the unique geodesic joining $x$ and $y$, and
$\widehat{P}_{\phi^X_s(x),\,\phi^X_t(x)}$ denotes parallel transport along the unique flow-line of $X$ joining $\phi^X_s(x)$ and $\phi^X_t(x)$.

\item For $\beta\in\R$, $k$ a non-negative integer, and $\alpha\in\left(0,\,\frac{1}{2}\right)$, define the H\"older space $C_{X,\,\beta}^{2k,\,2\alpha}(M)$ with polynomial weight $f^{\beta}$ to be the set of $u\in C_{X,\,\beta}^{2k}(M)$ for which the norm
\begin{equation*}
\norm{u}_{C_{X,\,\beta}^{2k,\,2\alpha}}:=\norm{u}_{C^{2k}_{X,\,\beta}} +\sum_{i+2j\,=\, 2k}\left[\left(\nabla^{h}\right)^i\left(\mathcal{L}_{X}^{(j)}u\right)\right]_{C^{0,\,2\alpha}_{\beta}}
\end{equation*}
is finite. It is straightforward to check that the space $C^{2k,\,2\alpha}_{X,\,\beta}(M)$ is a Banach space.
The H\"older space $C_{X,\,\beta}^{2k,\,2\alpha}(M)$ with $\beta=0$
coincides with $C^{2k,\,2\alpha}_X(M)$ endowed with the norm
introduced in Section \ref{section-fct-spa-exp} as $\|\cdot\|_{C^{2k,\,2\alpha}_X}$. The intersection
$\bigcap_{k\,\geq\,0}C^{2k}_{X,\,\beta}(M)$ we denote by $C^{\infty}_{X,\,\beta}(M)$.

\item Finally, we define the spaces
\begin{equation*}
\begin{split}
\mathcal{M}^{2k+2,\,2\alpha}_{X,\,\beta}(M)&:=\left\{\varphi\in C^2_{\operatorname{loc}}(M)\,|\,\tau+i\partial\bar{\partial}\varphi>0\right\}\bigcap C^{2k+2,\,2\alpha}_{X,\,\beta}(M),
\end{split}
\end{equation*}
and akin to $\mathcal{M}^{\infty}_{X,\exp}(M)$ defined in \eqref{defn-M-exp-space}, one considers the following convex set of K\"ahler potentials:
\begin{equation*}
\mathcal{M}^{\infty}_{X,\,\beta}(M)=\bigcap_{k\,\geq\,0}\,\mathcal{M}^{2k+2,\,2\alpha}_{X,\,\beta}(M).
\end{equation*}
\end{itemize}

\subsection{Preliminaries and Fredholm properties of the linearised operator}\label{linearpoly}

As demonstrated for exponentially weighted function spaces in Theorem \ref{iso-sch-Laplacian-exp}, we show that the drift
Laplacian is also an isomorphism between polynomially weighted function spaces.

\begin{theorem}\label{iso-sch-Laplacian-pol}
Let $\alpha\in\left(0,\,\frac{1}{2}\right)$, $k\in\mathbb{N}$, and $\beta>0$. Then the drift Laplacian
\begin{equation*}
\begin{split}
\Delta_{\tau}+\frac{X}{2}\cdot:C^{2k+2,\,2\alpha}_{X,\,\beta}(M)\rightarrow C^{2k,\,2\alpha}_{X,\,\beta+1}(M)
\end{split}
\end{equation*}
is an isomorphism of Banach spaces.
\end{theorem}

\begin{proof}
Let $\beta>0$. We compute the following conjugate operator associated to $2\Delta_{\tau}+X\cdot=\Delta_{h}+X\cdot$:
\begin{equation*}
\begin{split}
\left[f^{\beta+1}\circ\left(\Delta_{h}+X\cdot\right)\circ f^{-\beta}\right]U&=f\cdot\left(\Delta_{h}+X\cdot\right)U-2\beta X\cdot U+f^{\beta+1}\cdot\left(\Delta_{h}+X\cdot\right)f^{-\beta}\cdot U\\
&=f\cdot\left(\Delta_{h}+X\cdot\right)U-2\beta X\cdot U-\beta\left(\Delta_{h}f+X\cdot f-(\beta+1)f^{-1}X\cdot f\right)\\
&=f\cdot\left(\Delta_{h}U+X\cdot U-\frac{4n\beta}{f}\cdot U\right)\\
&\qquad-2\beta X\cdot U+\beta\left(4n-\Delta_{h}f-X\cdot f+\frac{(\beta+1)}{f}X\cdot f \right)U\\
&=:f\cdot\left(\Delta_{h}U+X\cdot U-\frac{4n\beta}{f}\cdot U\right)+K(U)
\end{split}
\end{equation*}
for any function $U\in C^2_{\operatorname{loc}}(M)$. We analyse each term in this expression separately.
Our first claim asserts that the perturbed drift Laplacian is an isomorphism of Banach spaces.

\begin{claim}\label{claim-iso-conj-op}
For $\alpha\in\left(0,\,\frac{1}{2}\right)$ and $\beta>0$, the operator
\begin{equation*}
\begin{split}
\Delta_{\tau}+\frac{X}{2}\cdot-\frac{4n\beta}{f}:C^{2k+2,\,2\alpha}_{X}(M)\rightarrow C^{2k,\,2\alpha}_{X,\,1}(M)
\end{split}
\end{equation*}
is an isomorphism of Banach spaces.
\end{claim}

\begin{proof}
This operator is well-defined and continuous by definition of the relevant function spaces.
Let $F\in C^{2k,\,2\alpha}_{X,\,1}(M)$. Then for $R>0$ sufficiently large such that the level sets $\{f\,=\,R\}$ are smooth closed hypersurfaces, let $u_R:\{f\,\leq\,R\}\rightarrow\mathbb{R}$ be the solution of
the following Dirichlet problem:
\begin{equation*}
\left\{
      \begin{aligned}
        \Delta_{\tau}U_R+\frac{X}{2}\cdot U_R-\frac{4n\beta}{f}\cdot U_R&=F\qquad\text{on $\{f\,<\,R\}$},\\
        U_R&=0\qquad\text{on $\{f\,=\,R\}$}.
      \end{aligned}
    \right.
\end{equation*}
For $A\in\mathbb{R}$, observe that
\begin{equation*}
\begin{split}
\Delta_{\tau}U_R+\frac{X}{2}\cdot U_R&=\frac{4n\beta}{f}\cdot \left(U_R-\frac{A}{4n\beta}\right)+F+\frac{A}{f},\\
\end{split}
\end{equation*}
which, upon setting $A:=\|f\cdot F\|_{C^0}$, is bounded below by $\frac{4n\beta}{f}\cdot \left(U_R-\frac{A}{4n\beta}\right)$.
Thus, the maximum principle applied to $U_R$ shows that $$\max_{\{f\,\leq\,R\}}\left(U_R-(4n\beta)^{-1}\|f\cdot F\|_{C^0}\right)\leq 0.$$ The previous argument applied to $-U_R$ further yields the fact that
$$\min_{\{f\,\leq\,R\}}U_R\geq -(4n\beta)^{-1}\|f\cdot F\|_{C^0}.$$ Together, these two bounds imply that
$\max_{\{f\,\leq\,R\}}|U_R|\leq (4n\beta)^{-1}\|f\cdot F\|_{C^0}.$

Next, standard elliptic Schauder estimates on each
ball $B_{h}(x,\,\delta)$ with $2\delta=\inj_{h}(M)>0$ compactly contained in $\{f\,<\,R\}$ give the following a priori local estimates on higher derivatives of $U_R$:
 \begin{equation*}
\sup_{B_{h}(x,\,\delta)\Subset \{f\,<\,R\}}\|U_R\|_{C^{2k+2,\,2\alpha}_{\operatorname{loc}}(B_{h}(x,\,\delta))}\leq C(n,k,\alpha,\tau)\|F\|_{C^{2k,\,2\alpha}_{X,\,1}}.
 \end{equation*}
As a consequence, we may appeal to the Arzel\`a-Ascoli theorem to pass to a subsequence still denoted by $(U_R)_{R\,\geq\,R_0}$ converging to a function $U\in C^{2k+2,\,2\alpha}_{\operatorname{loc}}(M)$ in the $C^{2k+2,\,2\alpha'}_{\operatorname{loc}}$-topology for any $\alpha'\in(0,\alpha)$ satisfying
 \begin{equation}\label{sol-drift-lap-pol-conj}
        \Delta_{\tau}U+\frac{X}{2}\cdot U-\frac{4n\beta}{f}\cdot U=F\qquad\text{on $M$},
\end{equation}
and $$\sup_{x\,\in\,M} \|U\|_{C^{2k+2,\,2\alpha}_{\operatorname{loc}}(B_{h}(x,\,\delta))}\leq C(n,k,\alpha,\tau)\|F\|_{C^{2k,\,2\alpha}_{X,\,1}}.$$
We claim that this solution $U$ of \eqref{sol-drift-lap-pol-conj} is unique among all
bounded solutions in $C^2_{\operatorname{loc}}(M)$. Indeed, by subtracting $U$ from another solution $U'$ with the same data $F$, it suffices to show that the kernel of the operator defined by the left-hand side of \eqref{sol-drift-lap-pol-conj} restricted to $C^2_{\operatorname{loc}}(M)$-bounded functions is zero-dimensional.
To this end, let $V\in C^2_{\operatorname{loc}}(M)\cap C^0(M)$ be such that
 \begin{equation*}
  \begin{split}
        \Delta_{\tau}V+\frac{X}{2}\cdot V-\frac{4n\beta}{f}\cdot V&=0\qquad\text{on $M$}.
       \end{split}
\end{equation*}
Then for $\epsilon>0$, the function $V-\epsilon \log f$ satisfies
\begin{equation*}
  \begin{split}
        \Delta_{\tau}\left(V-\epsilon \log f\right)+\frac{X}{2}\cdot \left(V-\epsilon \log f\right)&\geq\frac{4n\beta}{f}\cdot \left(V-\epsilon \log f\right)+\frac{4n\epsilon\beta}{f}\log f-\frac{\epsilon}{f} \left( \Delta_{\tau}f+\frac{X}{2}\cdot f\right).
       \end{split}
\end{equation*}
Since $f$ is proper and bounded from below by $1$ (cf.~Lemma \ref{easy}), and since $V$ is bounded, the function $V-\epsilon\log f$ attains a maximum at
some point $x_0\in M$. Applying the maximum principle to $V-\epsilon\log f$ at this point yields the upper bound
\begin{equation*}
\begin{split}
 4n\beta\left(V(x_0)-\epsilon \log f(x_0)\right)&\leq \epsilon\left( \Delta_{\tau}f+\frac{X}{2}\cdot f- 4n\beta\log f(x_0)\right)\\
 &\leq \epsilon\left( \Delta_{\tau}f+\frac{X}{2}\cdot f\right)\leq c(n,\,\tau)\epsilon.
 \end{split}
\end{equation*}
To conclude, observe that for all $x\in M$ and every $\epsilon>0$, $V(x)-\epsilon \log f(x)\leq c(n,\,\tau)\epsilon$.
Letting $\epsilon$ tend to $0$, one sees that $V\leq 0$. Arguing with $-V$ in place of $V$ shows that $V\geq 0$, and
so $V=0$ and the operator $\Delta_{\tau}+\frac{X}{2}\cdot -\frac{4n\beta}{f}\cdot$ is injective. The fact that $U$ belongs to $C^{2k+2,\,2\alpha}_{X}(M)$ follows along the same lines as
for the exponential case, specifically from \eqref{comp-conj-op-exp} onwards in the proof of Theorem \ref{iso-sch-Laplacian-exp}. This completes the proof of Claim \ref{claim-iso-conj-op}.
\end{proof}

Our second claim asserts that the operator $K$ is compact.

\begin{claim}\label{claim-cpct-op}
The operator
\begin{equation*}
\begin{split}
U\in C^{2k+2,\,2\alpha}_{X}(M)\rightarrow K(U):=-2X\cdot U+\beta\left(4n-\Delta_{h}f-X\cdot f+\frac{(\beta+1)}{f}X\cdot f \right)U\in C^{2k,\,2\alpha}_{X}(M)
\end{split}
\end{equation*}
is a compact operator between Banach spaces.
\end{claim}

\begin{proof}
From the very definition of the operator $K$, it is straightforward to check that $K(U)$ is $JX$-invariant if $U$ is. Let $(U_i)_{i}$ be a sequence of functions in $C^{2k+2,\,2\alpha}_{X}(M)$ that is bounded (by $1$ say).
Then by the Arzel\`a-Ascoli theorem, there exists a subsequence still denoted by $(U_i)_i$ that converges in the $C^{2k+2,\,2\alpha'}
_{\operatorname{loc}}(M)$-topology for any $\alpha'\in(0,\alpha)$ to a function $U\in C^{2k+2,\,2\alpha}_{X}(M)$. In particular, $(K(U_i))_i$ converges in the $C^{2k+1,\,2\alpha'}_{\operatorname{loc}}(M)$-topology for any $\alpha'\in(0,\alpha)$ to $K(U)\in C^{2k,\,2\alpha}_{X}(M)$.
We will show that $(K(U_i))_i$ converges to $K(U)$ in the $C^{2k,\,2\alpha}_{X}(M)$-topology. We explain the proof in the case $k=0$ only. The proof for the cases $k\geq 1$
is similar.

Fix a cut-off function $\chi:M\to\mathbb{R}$ with $0\leq \chi\leq1$ and
$$
\chi(x) =
\begin{cases}
1 &\textrm{if $x\in E\cup\{y\in M\,|\,t(y)\leq 1\}$},\\
0 &\textrm{if $t(x)>2$},
\end{cases}
$$
define $\chi_{R}(x):=\chi(x/R)$ in the obvious way for $R>0$, and write $U_i=(1-\chi_{R})U_i+\chi_{R}U_i$.
Then for $\varepsilon\in(0,\,1)$ as in \eqref{covid19} and $\varepsilon'>0$, let $R>0$ be large enough such that for all indices $i\geq 0$,
\begin{equation*}
\begin{split}
&\left|\beta\left(4n-\Delta_{h}f-X\cdot f+\frac{(\beta+1)}{f}X\cdot f \right)\right|\cdot|(1-\chi_{R})(U_i-U)|\leq \left(\|U\|_{C^0}+\|U_i\|_{C^0}\right)\frac{c(n,\,\tau)}{R^{\varepsilon}}\leq \frac{\varepsilon'}{2},\\
&2|X\cdot((1-\chi_{R})(U-U_i))|\leq c(n,\tau)\frac{\|X\cdot(U-U_i)\|_{C^{0}_1}}{R}+2|X\cdot\chi_{R}|\left(\|U\|_{C^0}+\|U_i\|_{C^0}\right)\leq \frac{c(n,\tau)}{R}\leq \frac{\varepsilon'}{2}.
\end{split}
\end{equation*}
Here we have used Lemma \ref{lemm-app-sol-id}. Similar estimates also hold true for the corresponding $\alpha$-semi-norms by increasing $R$ if necessary.
For such an $R>0$, observe that $\lim_{i\rightarrow+\infty}\|\chi_{R}(U-U_i)\|_{C^{0,\,2\alpha}_{X}}=\lim_{i\rightarrow+\infty}\|\chi_{R}(U-U_i)\|_{C_{\operatorname{loc}}^{0,\,2\alpha}}=0$.
This concludes the case $k=0$.
\end{proof}

Claims \ref{claim-iso-conj-op} and \ref{claim-cpct-op} show that the operator
\begin{equation*}
\begin{split}
\Delta_{\tau}+\frac{X}{2}\cdot:C^{2k+2,\,2\alpha}_{X,\,\beta}(M)\rightarrow C^{2k,\,2\alpha}_{X,\,\beta+1}(M)
\end{split}
\end{equation*}
is Fredholm of index $0$. Since this operator is also injective by the maximum principle, the isomorphism property follows.
\end{proof}

\subsection{Small perturbations of steady gradient K\"ahler-Ricci solitons: polynomial case}\label{invert-poly}
In this section we show, using the implicit function theorem, that the invertibility of the drift
Laplacian allows for small perturbations in polynomially weighted function spaces of solutions to the complex Monge-Amp\`ere equation that concerns us.
The precise statement that we prove is the following.

\begin{theorem}\label{Imp-Def-Kah-Ste}
Let $F_0\in C^{\infty}_{X,\,\beta+1}(M)$ for some $\beta>0$ and let $\psi_0\in\mathcal{M}^{\infty}_{X,\,\beta}(M)$ be a solution of the complex Monge-Amp\`ere equation
\begin{equation*}
\log\left(\frac{\tau^n_{\psi_0}}{\tau^n}\right)+\frac{X}{2}\cdot\psi_0=F_0.
\end{equation*}
Then for $\alpha\in\left(0,\,\frac{1}{2}\right)$, there exists a neighbourhood $U_{F_0}\subset C^{2,\,2\alpha}_{X,\,\beta+1}(M)$ of $F_0$ such that for all $F\in U_{F_0}$, there exists a unique function $\psi\in\mathcal{M}^{4,\,2\alpha}_{X,\,\beta}(M)$ such that
\begin{equation}
\log\left(\frac{\tau^n_{\psi}}{\tau^n}\right)+\frac{X}{2}\cdot\psi=F.\label{MA-neigh-small-per-pol}
\end{equation}
\end{theorem}

\begin{remark}
Theorem \ref{Imp-Def-Kah-Ste} does not assume any finite regularity on the data $(\psi_0,\,F_0)$ in the relevant
function spaces; see Remark \ref{heyjude}.
\end{remark}

\begin{proof}[Proof of Theorem \ref{Imp-Def-Kah-Ste}]
In order to apply the implicit function theorem for Banach spaces, we must reformulate
the statement of Theorem \ref{Imp-Def-Kah-Ste} in terms of
the map $MA_{\tau}$ introduced formally at the beginning of Section \ref{sec-Fredo-prop-exp}. To this end, consider the mapping
\begin{equation*}
\begin{split}
\widetilde{MA}_{\tau_{\psi_0}}:(\varphi,G)\in &\,\mathcal{M}^{4,\,2\alpha}_{X,\,\beta}(M)\times C^{2,\,2\alpha}_{X,\,\beta+1}(M)\\
&\mapsto \log\left(\frac{\tau_{\psi_0+\varphi}^n}{\tau^n}\right)+\frac{X}{2}\cdot (\psi_0+\varphi)-G-F_0\in C^{2,\,2\alpha}_{X,\,\beta+1}(M),\quad \alpha\in\left(0,\,\frac{1}{2}\right).
\end{split}
\end{equation*}
Notice that the function spaces can be defined either by using the metric $h$ or $h_{s\psi_0}$ for any $s\in[0,\,1]$.
To see that $\widetilde{MA}_{\tau_{\psi_0}}$ is well-defined, apply the Taylor expansion \eqref{equ:taylor-exp} to the background metric $\tau_{\psi_0}$
to obtain
\begin{equation}\label{reform-MA-op}
\begin{split}
\widetilde{MA}_{\tau_{\psi_0}}(\varphi,G)&=\log\left(\frac{\tau_{\psi_0+\varphi}^n}{\tau_{\psi_0}^n}\right)+\frac{X}{2}\cdot\varphi-G\\
&=\Delta_{\tau_{\psi_0}}\varphi+\frac{X}{2}\cdot\varphi-G-\int_0^1\int_0^{u}\arrowvert \partial\bar{\partial}\varphi\arrowvert^2_{h_{s(\psi_0+\varphi)}}\,ds\,du.
\end{split}
\end{equation}
Then by the very definition of $C^{4,\,2\alpha}_{X,\,\beta}(M)$,
the first three terms of the last line of \eqref{reform-MA-op} lie in $C^{2,\,2\alpha}_{X,\,\beta+1}(M)$.

Now, if $S$ and $T$ are tensors in $C^{2k,\,2\alpha}_{X,\,\gamma_1}(M)$ and $C^{2k,\,2\alpha}_{X,\gamma_2}(M)$ respectively, with $\gamma_i>0$, $i=1,2$,
then observe that $S\ast T$ lies in $C^{2k,\,2\alpha}_{X,\,\gamma_1+\gamma_2}(M)$,
where $\ast$ denotes any linear combination of contractions of tensors with respect to the metric $h$. Moreover,
\begin{equation}\label{mult-inequ-holder}
\|S\ast T\|_{C^{2k,\,2\alpha}_{X,\,\gamma_1+\gamma_2}}\leq C(k,\,\alpha)\|S\|_{C^{2k,\,2\alpha}_{X,\,\gamma_1}}\cdot \|T\|_{C^{2k,\,2\alpha}_{X,\,\gamma_2}}.
\end{equation}
Notice that $$\arrowvert \partial\bar{\partial}\varphi\arrowvert^2_{h_{s(\psi_0+\varphi)}}=h_{s(\psi_0+\varphi)}^{-1}
\ast (\nabla^{h})^{\,2}\varphi\ast(\nabla^{h})^{\,2}\varphi$$ and that $$h_{s(\psi_0+\varphi)}^{-1}-h^{-1}\in C^{2,\,2\alpha}_{X,\,\beta+1}(M).$$
Thus, applying \eqref{mult-inequ-holder} twice to $S=T=(\nabla^{h})^{2}\varphi$ and to the inverse $h_{s(\psi_0+\varphi)}^{-1}$ with weights
$\gamma_1=\gamma_{2}=\beta+1$ and $k=1$, one finds that $\arrowvert \partial\bar{\partial}\varphi\arrowvert^2_{h_{s(\psi_0+\varphi)}}\in C^{2,\,2\alpha}_{X,\,2\beta+2}(M)\subset C^{2,\,2\alpha}_{X,\,\beta+1}(M)$ for each $s\in[0,\,1]$ and that
\begin{equation*}
\left\|\int_0^1\int_0^{u}\arrowvert \partial\bar{\partial}\varphi\arrowvert^2_{h_{s(\psi_0+\varphi)}}\,ds\,du
\right\|_{C^{2,\,2\alpha}_{X,\,\beta+1}}\leq C\left(k,\alpha,h,\|\psi_0\|_{C^{2,\,2\alpha}_{X,\,\beta}}\right)\|\varphi\|_{C^{4,\,2\alpha}_{X,\,\beta}},\\
\end{equation*}
so long as $\|\varphi\|_{C^{4,\,2\alpha}_{X,\,\beta}}\leq 1$. Finally, the $JX$-invariance of the right-hand side of \eqref{reform-MA-op} is clear.

By definition, $\widetilde{MA}_{\tau_{\psi_0}}(\varphi,F-F_0)=0$ if and only if $\psi_0+\varphi$ is a solution of $\eqref{MA-neigh-small-per-pol}$ with data $F$.
By \eqref{equ:sec-der}, we have that
$$D^1_{(0,\,0)}\widetilde{MA}_{\tau_{\psi_0}}(\psi):=D_{(0,\,0)}\widetilde{MA}_{\tau_{\psi_0}}((\psi,\,0))=\Delta_{\tau_{\varphi_0}}\psi+\frac{X}{2}\cdot\psi
\qquad\textrm{for $\psi\in C^{4,\,2\alpha}_{X,\,\beta}(M)$}.$$ Hence, after applying
Theorem \ref{iso-sch-Laplacian-pol} to the background metric $\tau_{\psi_0}$ in place of $\tau$, we conclude that
$D^1_{(0,\,0)}\widetilde{MA}_{\tau_{\psi_0}}$ is an isomorphism of Banach spaces. The result now follows by applying the implicit function theorem
to the map $\widetilde{MA}_{\tau_{\psi_0}}$ in a neighbourhood of $(0,\,0)\in\mathcal{M}^{4,\,2\alpha}_{X,\,\beta}(M)\times C^{2,\,2\alpha}_{X,\,\beta+1}(M)$.
\end{proof}

\subsection{Surjectivity of the drift Laplacian}\label{surject}

We next show that the drift Laplacian of $\tau$ is surjective on the space of functions that decay polynomially. This result is required in order to
prove uniqueness in each K\"ahler class of the steady gradient K\"ahler-Ricci solitons that we construct.
 \begin{theorem}\label{surjectivity-drift-Laplacian}
There exists a finite set $\mathcal{F}\subset\left(0,\,\frac{1}{2}\right)$ such that for all $\beta\in(0,\,1)\setminus \mathcal{F}$, the drift Laplacian
\begin{equation*}
\begin{split}
\Delta_{\tau}+\frac{X}{2}\cdot:C^{\infty}_{X,\,\beta-1}(M)\rightarrow C^{\infty}_{X,\,\beta}(M),
\end{split}
\end{equation*}
is surjective.
\end{theorem}

\begin{proof}
Let $Q\in C^{\infty}_{X,\,\beta}(M)$ be fixed once and for all. In order to find a $JX$-invariant solution $u$ to $\Delta_hu+X\cdot u=Q$ that grows at most like $f^{1-\beta}$, we first solve the corresponding equation with respect to $\hat{g}$ outside of
a compact subset of $M$ which we henceforth identify with the end of $C_{0}$ via $\pi$. To this end, we invoke the spectral decomposition of the basic Laplacian $\Delta_{B}$ acting on $L^{2}_{B}(S)$, the space of basic $L^2$-integrable functions defined on the link $S$ of the cone $C_0$, which exists by virtue of \cite[Proposition 3.1]{ken2}. Let $(\phi^{B}_i)_{i\,\geq\,0}$ denote a complete orthonormal basis of smooth eigenfunctions of $\Delta_{B}$
in $L^{2}_{B}(S)$ and let $(\lambda^{B}_i)_{i\,\geq\,0}$ denote the corresponding eigenvalues with $\lambda^{B}_{i}\geq\lambda^{B}_{j}$ for $i\geq j$.
Then we have that $-\Delta_{B}\phi_i^{B}=\lambda_i^{B}\phi_i^{B}$ for each $i\geq 0$, $\lambda_0^{B}=0$, $\lambda_{i}^{B}\to+\infty$ as $i\to+\infty$, and $\phi_0^{B}=1$, this last equality following from the fact that on Sasaki manifolds, the basic Laplacian coincides with the Laplacian acting on basic functions \cite{ken2}.
Moreover, as the cone is Ricci-flat, we can assert from \cite{ken} that $\lambda_{1}^{B}\geq 2n\left(1+\frac{1}{2n-3}\right)$.

We seek a $JX$-invariant solution of the equation $$\Delta_{\hat{g}}\tilde{u}+X\cdot \tilde{u}=Q$$ of the form $\tilde{u}(t,\cdot):=\sum_{i\,\geq\,0}u_i(t)\cdot\phi_i^{B}$ on the end $\{t\geq 1\}$ of $C_{0}$ say. When a solution of this form exists, it is
clear that it is $JX$-invariant. Thanks to \eqref{covid20}, one can decompose this PDE into an infinite system of second order linear ODEs on $\{t\geq 1\}$, namely
\begin{equation*}
\begin{split}
&4\partial_t^2u_i+\left(4n+\frac{4(n-1)}{t}\right)\partial_tu_i-\frac{\lambda_i^{B}}{t}u_i=n\cdot Q_i,\qquad u_i(t)=O\left(t^{1-\beta}\right),\\
&Q_i(t):=\int_SQ(t,\cdot)\phi_i^{B}\,d\mu_{g_S},\qquad i\geq 0
\end{split}
\end{equation*}
with $g_{S}$ the metric on $S$ and $d\mu_{g_{S}}$ the associated volume form. As one shall see, in order to solve this system, it suffices to solve the following infinite system of first order linear ODEs which have the added advantage of being more explicit in terms of the data $Q$:
\begin{equation}
\begin{split}\label{inf-seq-ODE-2}
&4n\partial_tu_i-\frac{\lambda_i^{B}}{t}u_i=n\cdot Q_i,\qquad u_i(t)=O\left(t^{1-\beta}\right),\\
&Q_i(t):=\int_{S}Q(t,\,\cdot)\phi_i^{B}\,d\mu_{g_S},\qquad i\geq 0.
\end{split}
\end{equation}
The solution of this latter system depends on the sign of $1-\beta-\frac{\lambda_i^{B}}{4n}$ and indeed, is given by
\begin{equation}
\begin{split}\label{ODE-solve-expli-u_i}
u_i(t)&:=\frac{1}{4}t^{\frac{\lambda_i^{B}}{4n}}\int_1^tQ_i(s)\cdot s^{-\frac{\lambda_i^{B}}{4n}}\,ds\qquad\text{if $1-\beta-\frac{\lambda_i^{B}}{4n}> 0$},\\
u_i(t)&:=-\frac{1}{4}t^{\frac{\lambda_i^{B}}{4n}}\int_t^{+\infty}Q_i(s)\cdot s^{-\frac{\lambda_i^{B}}{4n}}\,ds\qquad\text{if $1-\beta-\frac{\lambda_i^{B}}{4n}<0$}.
\end{split}
\end{equation}
When $1-\beta-\frac{\lambda_i^{B}}{4n}> 0$, the solutions $u_i(t)$ are defined up to a solution of the homogeneous equation
corresponding to \eqref{inf-seq-ODE-2}. By our lower bound on $\lambda_{1}^{B}$, we know that $1-\beta-\frac{\lambda_i^{B}}{4n}<\frac{1}{2}$.
Hence, the critical case $1-\beta-\frac{\lambda_i^{B}}{4n}= 0$ can be avoided by only considering
positive $\beta$ lying outside some finite subset $\mathcal{F}$ of $\left(0,\,\frac{1}{2}\right)$, namely those $\beta\in\left(0,\,\frac{1}{2}\right)$ with $\beta=1-\frac{\lambda_i^{B}}{4n}$ for some $\lambda_{i}^{B}$ of which there are only finitely many
as $\lambda_{i}^{B}\to+\infty$ as $i\to+\infty$.

Working with $\beta\in(0,\,1)\setminus \mathcal{F}$, we next investigate the common properties of the functions $u_i$, $i\geq 0$, beginning with the following estimate.

\begin{claim}\label{claim-asy-est-u_i}
For all $j,\,k\geq 0$, there exists a positive constant $C_{j,\,k}$ such that
\begin{equation}\label{covid-inequ}
|\partial_t^{j}u_i|\leq C_{j,\,k}(\lambda_i^{B})^{-k}t^{1-\beta-j}\qquad\textrm{for all $j\geq 0$ and $t\geq 1$.}
\end{equation}
\end{claim}

\begin{proof}
By induction on $j\geq 0$, it is straightforward to check that $|\partial^j_tu_i|\leq C_{j}t^{1-\beta-j}$ for all $j\geq 0$ using
\eqref{inf-seq-ODE-2} and the fact that $\mathcal{L}^{(j)}_XQ=O(t^{-\beta-j})$ for all $j\geq 0$ by assumption.
Let us show that \eqref{covid-inequ} holds true for $j=0$. For this, it suffices to prove this estimate for all indices $i$ with $1-\beta-\frac{\lambda_i^{B}}{4n}<0$. To this end, using \eqref{inf-seq-ODE-2} and the fact that $\Delta_{B}^{(k)}\phi_i^{B}=(-\lambda_i^{B})^{k}\phi_i^{B}$ for all $k\geq0$, we derive the following estimate on the Fourier coefficients of $Q$:
\begin{equation}
\begin{split}\label{est-Q_i}
|Q_i(t)|=&\left|(-\lambda_i^{B})^{-k}\int_S\Delta_{g_S}^{(k)}Q(t,\,\cdot)\phi_i^{B}\,d\mu_{g_S}\right|\\
&\leq (\lambda_i^{B})^{-k}\|\Delta_{g_S}^{(k)}Q(t,\,\cdot)\|_{C^0(S)}\|\phi_i^{B}\|_{L_{B}^2(S)}\operatorname{vol}_{g_{S}}(S)\\
&\leq C_{k}(\lambda_i^{B})^{-k}t^{-\beta},
\end{split}
\end{equation}
where we have used the fact that $\|\phi_i^{B}\|_{L^2(S)}=1$ and $\Delta_{tg_S}^{(k)}Q(t,\,\cdot)=t^{-k}\Delta_{g_S}^{(k)}Q(t,\,\cdot)=O(t^{-k-\beta})$ for all $k\geq 0$. Plugging \eqref{est-Q_i} into \eqref{ODE-solve-expli-u_i}
now yields the desired conclusion for $j=0$.

One can similarly show that $|\partial_t^{j}Q_i(t)|\leq C_{j,\,k}(\lambda_i^{B})^{-k}t^{-\beta-j}$ for all $j\geq 1$ and $k\geq 0$. The higher order estimates \eqref{covid-inequ} for $j\geq 1$ can be proved by induction by repeatedly using the ODE \eqref{inf-seq-ODE-2} satisfied by $u_i$.
\end{proof}

Claim \ref{claim-asy-est-u_i} together with a version of Weyl's law for the spectrum of $\Delta_{B}$ on
$L^{2}_{B}(S)$ \cite[Propostion 3.4]{ken2} asserting that $\lambda_i^{B}\geq C i^{\frac{2}{2n-1}}$ for some positive constant $C$ independent of $i$, imply that the function $\tilde{u}=\sum_{i\,\geq\,0}u_i\cdot \phi_i^{B}$ is a genuine $JX$-invariant smooth function defined on the complement of a compact subset $K$ of $M$ and lying in $C^{\infty}_{X,\,\beta-1}(M\setminus K)$. Moreover, by construction, each $u_i$, $i\geq 0$, satisfies
the following second order linear ODE:
\begin{equation*}
\begin{split}
\Delta_{\hat{g}}u_i+X\cdot u_i&=\frac{4}{n}\partial_t^2u_i+\left(4+\frac{4(n-1)}{nt}\right)\partial_tu_i-\frac{\lambda_i^{B}}{nt}u_i\\
&=\frac{4}{n}\partial_t^2u_i+\frac{4(n-1)}{nt}\partial_tu_i+4\partial_tu_i-\frac{\lambda_i^{B}}{nt}u_i\\
&=\frac{1}{n}\partial_t\left(\frac{\lambda_i^{B}}{nt}u_i+Q_i\right)+\frac{(n-1)}{nt}\left(\frac{\lambda_i^{B}}{nt}u_i+Q_i\right)+Q_i\\
&=:Q_i+R_i,
\end{split}
\end{equation*}
where the remainder term $R_i$ lies in $C^{\infty}_{X,\,1+\beta}(M)$. Consequently, after writing $R:=\sum_{i\,\geq\,0}R_i\cdot \phi_i^{B}$, we see that $R\in C^{\infty}_{X,\,1+\beta}(M\setminus K)$ and that $\tilde{u}$ satisfies the PDE:
\begin{equation}\label{ell-drift-lap-tilde-u}
\begin{split}
\Delta_{h}\tilde{u}+X\cdot\tilde{u}=Q+\underbrace{\left(\Delta_h-\Delta_{\hat{g}}\right)\tilde{u}+R}_{\textrm{$\in\, C^{\infty}_{X,\,\min\{1+\beta,\,\beta+\varepsilon\}}(M\setminus K)$}}\qquad\textrm{on $M\setminus K$}.
\end{split}
\end{equation}
Here we have used \eqref{covid19}, i.e., $\hat{g}-h\in C^{\infty}_{X,\,\varepsilon}(M)$ for some given $\varepsilon\in(0,\,1)$.

Now, after localizing $\tilde{u}$ with the help of a $JX$-invariant smooth function equal to $1$ outside a sufficiently large compact subset of $M$,
we end up with a function defined on $M$ which we still denote by $\tilde{u}$ satisfying
an equation of the same type as \eqref{ell-drift-lap-tilde-u}. More precisely,
$\tilde{u}$ satisfies
\begin{equation*}
\begin{split}
\Delta_{h}\tilde{u}+X\cdot\tilde{u}=Q+\tilde{R}
\end{split}
\end{equation*}
for some $\tilde{R}\in C^{\infty}_{X,\,\min\{1+\beta,\,\beta+\varepsilon\}}(M)$.

If $\beta+\varepsilon>1$, then Theorem \ref{iso-sch-Laplacian-pol} asserts that
$\Delta_{h}\tilde{v}+X\cdot\tilde{v}=-\tilde{R}$ for some function $\tilde{v}\in C^{2k+2,\,2\alpha}_{X,\,\min\{\beta,\,\beta+\varepsilon-1\}}(M)$.
Since $\tilde{R}\in C^{\infty}_{X,\,\min\{1+\beta,\,\beta+\varepsilon\}}(M)$,
injectivity of the drift Laplacian between $C^{2k+2,\,2\alpha}_{X,\,\beta}(M)$ and $C^{2k,\,2\alpha}_{X,\,\beta+1}(M)$ for $\beta>0$
implies that $\tilde{v}\in C^{2k+2,\,2\alpha}_{X,\,\min\{\beta,\,\beta+\varepsilon-1\}}(M)$ for every
$k\geq0$, and so $\tilde{v}\in C^{\infty}_{X,\,\min\{\beta,\,\beta+\varepsilon-1\}}(M)$.
The function $\tilde{u}+\tilde{v}$ therefore lies in $C^{\infty}_{X,\,\beta-1}(M)$ and solves the equation $\Delta_{h}u+X\cdot u=Q$.

In the case that $\beta+\varepsilon \leq1$, shrink $\varepsilon>0$ if necessary so that
$\beta+k\varepsilon\notin \mathcal{F}\cup\{1\}$ for all $k\in\mathbb{N}$. Then by applying the first
part of this proof with data $\tilde{R}\in C^{\infty}_{X,\,\min\{1+\beta,\,\beta+\varepsilon\}}(M)$ in place of $Q$,
one can find a function $\tilde{u}_1\in C^{\infty}_{\beta+\varepsilon-1}(M)$ with $\Delta_{h}(\tilde{u}-\tilde{u}_1)+X\cdot(\tilde{u}-\tilde{u}_1)-Q\in C^{\infty}_{X,\,\min\{1+\beta,\,\beta+2\varepsilon\}}(M)$. Iterating this argument a finite number of times reduces to the case $\beta+\varepsilon>1$ which can then be solved by invoking Theorem \ref{iso-sch-Laplacian-pol} as before.
\end{proof}

\begin{remark}
If $\tau$ in Theorem \ref{surjectivity-drift-Laplacian} is a complete steady gradient K\"ahler-Ricci soliton on $M$, then the kernel of the drift Laplacian restricted to $C^{\infty}_{X,\,\beta}(M)$ for any $\beta>-1$ comprises only constants. This is a direct consequence of \cite[Corollary 1.4]{Hua-Zha-Zha}, where one only needs to assume that the function lying in the kernel of the drift Laplacian grows sublinearly, that is, as $o(t)$.
\end{remark}

\section{Proof of Theorem \ref{main}}\label{prooof}

Let $(C_0,\,g_{0},\,J_0,\,\Omega_0)$ be a Calabi-Yau cone
of complex dimension $n\geq2$ with link $S$, radial function $r$, and transverse K\"ahler form $\omega^{T}$.
Set $r^2=:e^t$ and let $\pi:M\to C_{0}$ be a crepant resolution of $C_{0}$ with exceptional set $E$ which is
equivariant with respect to the real torus action on $C_{0}$ generated by $J_{0}r\partial_{r}$
so that the holomorphic vector field $2r\partial_{r}=4\partial_{t}$ on $C_{0}$ lifts to a real holomorphic vector field $X=\pi^{*}(2r\partial_{r})$
on $M$. Let $J$ denote the complex structure on $M$ and recall from Proposition \ref{caosoliton} that we have Cao's one-parameter family of
steady gradient K\"ahler-Ricci solitons $\tilde{\omega}_{a},\,a\geq0,$ on $C_{0}$ with respective soliton potentials $\varphi_{a}(t)$,
as well as the K\"ahler form $\hat{\omega}=\frac{i}{2}\partial\bar{\partial}\left(\frac{nt^{2}}{2}\right)$.

In this section, we use the results acquired thus far to prove Theorem \ref{main}. We begin with the existence part before moving on to uniqueness.
Throughout, we identify $M\setminus E$ with the complement of the apex of $C_{0}$ via $\pi$.

\subsection{Existence}\label{existencee}
Fix $a\geq0$ and for a given K\"ahler class $\mathfrak{k}$ on $M$, take the K\"ahler form
$\sigma$ in $\mathfrak{k}$ asymptotic to $\tilde{\omega}_{a}$ with $\mathcal{L}_{JX}\sigma=0$
given by Proposition \ref{lemma5.7}. We now add a subscript $a$ to $\sigma$ to indicate that this K\"ahler form is asymptotic at infinity to
$\tilde{\omega}_{a}$. Combining Proposition \ref{coro-asy-Cao-met} and Lemma \ref{budwiser}, one can see from
the triangle inequality that $\sigma_{a}$ satisfies \eqref{covid19}.
Therefore by Proposition \ref{equationsetup}, the problem of constructing a steady gradient K\"ahler-Ricci soliton in $\mathfrak{k}$ with
the desired properties can be reformulated in terms of solving a scalar PDE on $M$, namely the complex Monge-Amp\`ere equation \eqref{e:soliton} which we now
recall:
\begin{equation}\label{eq:soliton}
\log\left(\frac{(\sigma_{a}+i\partial\bar{\partial}\psi)^{n}}{\sigma_{a}^{n}}\right)+\frac{X}{2}\cdot\psi=F,
\end{equation}
where $\psi$ and $F$ are smooth functions invariant under the flow of $JX$ and outside a compact subset of $M$,
\begin{equation}\label{looove}
F = \left\{
\begin{array}{ll}
0 & \text{if $n=2$ or if $\mathfrak{k}$ is compactly supported},\\
-\log\left(\frac{(\tilde{\omega}_{a}+p_S^{*}(\zeta))^{n}}{\tilde{\omega}_{a}^{n}}\right) & \text{otherwise},\\
\end{array} \right.
\end{equation}
for $\zeta$ a basic primitive $(1,\,1)$-form on $S$ determined uniquely by $\mathfrak{k}$ with $p_{S}:C_{0}\to\{r=1\}\cong S$ denoting the projection.
Recall from Section \ref{primitive} that $\zeta\wedge(\omega^{T})^{n-2}=0$ and notice that, by Lemma \ref{finite}
and the $JX$-invariance of $\sigma_{a}$ and $\psi$, any steady K\"ahler-Ricci soliton resulting from the solution of \eqref{eq:soliton} is necessarily gradient. Finally, observe that the smooth proper real-valued function $f_{a}$ on $M$ defined by $-\sigma_{a}\lrcorner JX=df_{a}$ and chosen such that $f_{a}\geq1$ (which is guaranteed to exist by Lemma \ref{easy}) differs from
the soliton potential $\varphi_{a}(t)$ of $\tilde{\omega}_{a}$ by a constant. Indeed, to see this last point, just note that
$$i\partial\bar{\partial}(f_{a}-\varphi_{a})=\frac{1}{2}\mathcal{L}_{X}(\sigma_{a}-\tilde{\omega}_{a})=\frac{1}{2}\mathcal{L}_{X}p_S^{*}(\zeta)=0$$
and $\mathcal{L}_{JX}(f_{a}-\varphi_{a})=0$ and appeal to Lemma \ref{pluri}(i).

\subsubsection{Compactly supported K\"ahler classes or $n=2$}
When $\mathfrak{k}$ is a compactly supported K\"ahler class or $n=2$, we read from \eqref{looove} that the function $F$ in
\eqref{eq:soliton} is compactly supported. Since $|f_{a}-\varphi_{a}|$ is bounded along the end of $M$, Theorem \ref{theo-exi-sol-cpct-supp-data} applies and provides us
with a solution of \eqref{eq:soliton}. More precisely, Theorem \ref{theo-exi-sol-cpct-supp-data} asserts the existence of a function $\psi\in\mathcal{M}^{\infty}_{X,\,\exp}(M)$
with $\mathcal{L}_{JX}\psi=0$ such that $\omega_{a}:=\sigma_{a}+i\partial\bar{\partial}\psi$ defines a steady gradient K\"ahler-Ricci soliton in the K\"ahler class $\mathfrak{k}$ of $M$. In particular, $JX$ will be Killing for $\omega_{a}$ and from Proposition \ref{lemma5.7} we see that outside a compact subset,
$\omega_{a}:=\tilde{\omega}_{a}+i\partial\bar{\partial}\psi$.
The asymptotics \eqref{amazing3} then follow from the fact that $\psi\in\mathcal{M}^{\infty}_{X,\,\exp}(M)$
and the aforementioned fact that the function $f_{a}$, the exponential of which being the weight used in the definition of $\mathcal{M}^{\infty}_{X,\,\exp}(M)$ with respect to $\sigma_{a}$,
differs from $\varphi_{a}(t)$ by a constant. The independence of $\omega_{a}$ from the parameter $a$ will be shown in Section \ref{independent}.

\subsubsection{Non-compactly supported K\"ahler classes}
Supressing  the subscript $a$ for the moment, when $\mathfrak{k}$ is not compactly supported, we see that at infinity,
\begin{equation*}
F=-\log\left(\frac{(\tilde{\omega}+p^{*}_{S}(\zeta))^{n}}{\tilde{\omega}^{n}}\right)
=n\frac{\tilde{\omega}^{n-1}\wedge p_{S}^{*}(\zeta)}{\tilde{\omega}^{n}}
+\underbrace{\sum_{k\,=\,2}^{n}\binom{n}{k}\frac{\tilde{\omega}^{n-k}\wedge (p_{S}^{*}(\zeta))^{k}}{\tilde{\omega}^{n}}}_{\in\,C^{\infty}_{X,\,2}(M)}.
\end{equation*}
Now, $\zeta$ being a basic primitive $(1,\,1)$-form on $S$ leads to the simplification
\begin{equation*}
\begin{split}
\tilde{\omega}^{n-1}\wedge p^{*}_{S}(\zeta)&=(n-1)\varphi^{n-2}\varphi'(\omega^{T})^{n-2}\wedge\frac{dt}{2}\wedge\eta\wedge p^{*}_{S}(\zeta)\\
&=(n-1)\varphi^{n-2}\varphi'\frac{dt}{2}\wedge\eta\wedge\underbrace{(\omega^{T})^{n-2}\wedge p^{*}_{S}(\zeta)}_{=\,0},\\
&=0.
\end{split}
\end{equation*}
Consequently, the right-hand side of
\eqref{eq:soliton} lies in $C^{\infty}_{X,\,2}(M)$. We next show, via a modification of the K\"ahler form $\sigma$, that solving \eqref{eq:soliton} can be reduced to the case where the data $F$ is compactly supported. This will then allow us to appeal directly to Theorem \ref{theo-exi-sol-cpct-supp-data} to assert the existence of an exponentially decaying solution of this reduced equation, thereby resulting in
a solution of \eqref{eq:soliton}. The reduction of \eqref{eq:soliton} to an equation with data $F$ compactly supported we now present.

\begin{prop}\label{prop-sec-app-poly-dec-comp-supp-inf}
Let $\sigma$ be the K\"ahler form from Proposition \ref{lemma5.7}.
Then for all $\delta\in\left(0,\,\frac{1}{2}\right)$, there exists a constant $T=T(\delta)>0$ and a smooth $JX$-invariant function
 $\varphi_T\in C_{X,\,1-\delta}^{\infty}(M)$ such that $\sigma+i\partial\bar{\partial}\varphi_T>0$ and
\begin{equation*}
\log\left(\frac{\left(\sigma+i\partial\bar{\partial}\varphi_T\right)^n}{\sigma^n}\right)+\frac{X}{2}\cdot\varphi_T=\chi_{T}\cdot F,
\end{equation*}
where $\chi_T$ is a smooth $JX$-invariant cut-off function supported on $\{t\geq T\}$.
\end{prop}

\begin{proof}
For $\alpha\in\left(0,\,\frac{1}{2}\right)$ fixed, Theorem \ref{Imp-Def-Kah-Ste} asserts that for all $\delta\in(0,\,1)$, there exists a neighbourhood $\mathcal{U}_{0}\subset C^{2,\,2\alpha}_{X,\,2-\delta}(M)$
of the constant function $0$ such that for all data $G\in \mathcal{U}_{0}$,
there exists a unique solution $\varphi\in C^{4,\,2\alpha}_{X,\,1-\delta}(M)$ such that $\sigma+i\partial\bar{\partial}\varphi>0$ and such that
\begin{equation*}
\log\left(\frac{\left(\sigma+i\partial\bar{\partial}\varphi\right)^n}{\sigma^n}\right)+\frac{X}{2}\cdot\varphi=G.
\end{equation*}
We fix a cut-off function $\chi:M\to\mathbb{R}$ with $|\chi|\leq1$ and
$$
\chi(x) =
\begin{cases}
0 &\textrm{if $x\in E\cup\{y\in M\,|\,t(y)\leq 1\}$},\\
1 &\textrm{if $t(x)>2$},
\end{cases}
$$
and define $\chi_{T}(x):=\chi(x/T)$ in the obvious way for $T>0$, a constant to be determined. Then
for any $\delta\in(0,\,1)$,
\begin{equation*}
\begin{split}
\sup_{M}|t^{2-\delta}\chi_{T}\cdot F|&\leq\sup_{t(x)\,\geq\,T}t^{2-\delta}|F(x)|\leq C\sup_{t(x)\,\geq\,T}t^{2-\delta}t^{-2}\leq
CT^{-\delta},\\
\sup_{M}|t^{\frac{5}{2}-\delta}\tilde{\nabla}(\chi_T\cdot F)|&\leq
C\sup_{t(x)\,\geq\,T}t^{\frac{5}{2}-\delta}\left(|\chi'|\frac{|F(x)|}{T}+|\tilde{\nabla}F|_{\tilde{g}}\right)\\
&\leq C\left(\sup_{T\,\leq\,t\,\leq 2T}t^{\frac{5}{2}-\delta}\frac{|F(x)|}{T}+\sup_{t(x)\,\geq\,T}t^{\frac{5}{2}-\delta}|\tilde{\nabla}F|_{\tilde{g}}\right)\\
&\leq C\left(\sup_{T\,\leq\,t\,\leq 2T}t^{\frac{1}{2}-\delta}\frac{1}{T}+\sup_{t(x)\,\geq\,T}t^{-\delta}\right)\\
&\leq CT^{-\delta}.
\end{split}
\end{equation*}
Continuing in this manner, one sees that
$\|\chi_T\cdot F\|_{C^{2,\,2\alpha}_{X,\,2-\delta}}=O(T^{-\delta})$ so that for
$T$ sufficiently large, one can find a unique function $\varphi_T\in C^{4,\,2\alpha}_{X,\,1-\delta}(M)$ with $\sigma+i\partial\bar{\partial}\varphi_T>0$ such that
\begin{equation}
\log\left(\frac{\left(\sigma+i\partial\bar{\partial}\varphi_{T}\right)^n}{\sigma^n}\right)+\frac{X}{2}\cdot\varphi_{T}=\chi_T\cdot F.\label{MA-ad-hoc-Fchi}
\end{equation}

In order to complete the proof of the proposition, we bootstrap the regularity of $\varphi_T$ at infinity in the following way.
Assume that $\varphi\in C^{2k+2,\,2\alpha}_{X,\,\beta}(M)$ is a solution to \eqref{MA-ad-hoc-Fchi}
for some $k\geq 0$, $\beta>0$, and $\alpha\in\left(0,\,\frac{1}{2}\right)$. Then observe that $\varphi$ is a solution of the following linear equation in disguise:
\begin{equation}\label{lin-equ-ad-hoc-G}
 \Delta_{\sigma}\varphi+\frac{X}{2}\cdot\varphi=\chi_T\cdot F-\int_0^1\int_0^{u}\arrowvert \partial\bar{\partial}\varphi\arrowvert^2_{g_{\sigma_{s\varphi}}}\,ds\,du,
 \end{equation}
where $g_{\sigma_{s\varphi}}$ is the K\"ahler metric associated to the K\"ahler form $\sigma_{s\varphi}:=\sigma+i\partial\bar{\partial}(s\varphi)$.
Applying Theorem \ref{iso-sch-Laplacian-pol} to \eqref{lin-equ-ad-hoc-G}, one sees that
$\varphi\in C^{2(k+1)+2,\,2\alpha}_{X,\,\beta}(M)$ as soon as the right-hand side of \eqref{lin-equ-ad-hoc-G} lies in $C^{2k+2,\,2\alpha}_{X,\,\beta+1}(M)$. Since $\varphi\in C^{2k+2,\,2\alpha}_{X,\,\beta}(M)$,
we deduce that
$i\partial\bar{\partial}\varphi\in C^{2k,\,2\alpha}_{X,\,\beta+1}(M) $. Arguing as in the proof of Theorem \ref{Imp-Def-Kah-Ste} using \eqref{mult-inequ-holder}, one immediately sees that the integral term on the right-hand side of \eqref{lin-equ-ad-hoc-G} lies in $C^{2k,\,2\alpha}_{X,\,2\beta+2}(M)$. Now, the local estimates provided by Proposition \ref{prop-loc-reg} imply that $\varphi$ is bounded together with all of its derivatives on $M$. Standard local interpolation inequalities then show that the
integral term on the right-hand side of \eqref{lin-equ-ad-hoc-G} lies in $C^{2k+2,\,2\alpha}_{X,\,2\beta+1-\alpha-\epsilon}(M)$ for
all $\epsilon>0$ sufficiently small. In particular, with $\beta=1-\delta$, $\delta\in(0,\frac{1}{2})$,
we obtain the desired result by choosing $0<\epsilon\leq\beta-\alpha=1-\delta-\alpha$
after recalling that $\chi_{T}\cdot F\in C^{\infty}_{X,\,2-\delta}(M)$.
\end{proof}

Proposition \ref{prop-sec-app-poly-dec-comp-supp-inf} allows us to finally solve \eqref{eq:soliton}.

\begin{theorem}\label{theo-main-sol-data-pol-dec}
For all $\delta\in\left(0,\,\frac{1}{2}\right)$, there exists a smooth $JX$-invariant solution $\psi$ to \eqref{eq:soliton} that has the following decomposition:
\begin{equation*}
\psi=\psi_1+\psi_2,
\end{equation*}
where $\psi_i$, $i=1,\,2$, are smooth $JX$-invariant functions with $\psi_1\in C^{\infty}_{X,\,1-\delta}(M)$ and $\psi_2\in C^{\infty}_{X,\,\exp}(M)$.
\end{theorem}

\begin{proof}
Fix $\delta\in\left(0,\,\frac{1}{2}\right)$. Then by Proposition \ref{prop-sec-app-poly-dec-comp-supp-inf}, there exists a smooth function
$\psi_1\in C_{X,\,1-\delta}^{\infty}(M)$ such that $\sigma+i\partial\bar{\partial}\psi_1>0$, $\mathcal{L}_{JX}\psi_1=0$, and such that
\begin{equation}\label{MA-proof-main-thm-2}
\log\left(\frac{\left(\sigma+i\partial\bar{\partial}\psi_1\right)^n}{\sigma^n}\right)+\frac{X}{2}\cdot \psi_1=\chi_TF,
\end{equation}
where $\chi_T$ is a smooth $JX$-invariant cut-off function supported on $\{t\geq T\}$. Setting $\sigma_{\psi}:=\sigma+i\partial\bar{\partial}\psi$ and
using \eqref{MA-proof-main-thm-2}, we rewrite \eqref{eq:soliton} as
\begin{equation*}
\begin{split}
\log\left(\frac{\sigma_{\psi}^n}{\sigma^n}\right)+\frac{X}{2}\cdot \psi&=\chi_T\cdot F+(1-\chi_T)\cdot F\\
&=\log\left(\frac{\left(\sigma+i\partial\bar{\partial}\psi_1\right)^n}{\sigma^n}\right)+\frac{X}{2}\cdot \psi_1+(1-\chi_T)\cdot F,
\end{split}
\end{equation*}
which leads to the reduced equation
\begin{equation}
\begin{split}\label{MA-proof-main-thm-3}
\log\left(\frac{\left(\sigma_{\psi_1}+i\partial\bar{\partial}\varphi\right)^n}{\sigma_{\psi_1}^n}
\right)+\frac{X}{2}\cdot \varphi&=\underbrace{(1-\chi_T)\cdot F}_{\text{compactly supported}},\qquad\varphi:=\psi-\psi_1,
\end{split}
\end{equation}
where $\sigma_{\psi_{1}}:=\sigma+i\partial\bar{\partial}\psi_{1}$. Now, since $\psi_{1}\in C^{\infty}_{X,\,1-\delta}(M)$, the
asymptotics \eqref{asy-cao-def-sec5} hold true, and akin to \eqref{botox}, we have that
$-\sigma_{\psi_{1}}\lrcorner JX=d\left(f_{a}+\frac{X}{2}\cdot\psi_{1}\right)$. After noting in addition that
$f_{a}+\frac{X}{2}\cdot\psi_{1}=\varphi_{a}(t)+O(1)$, we see that Theorem \ref{theo-exi-sol-cpct-supp-data} applies with $\tau=\sigma_{\psi_1}$
and data $(1-\chi_T)\cdot F$. This yields a solution $\psi_2\in C^{\infty}_{X,\,\exp}(M)$
of \eqref{MA-proof-main-thm-3}. With this, we conclude the proof of Theorem \ref{theo-main-sol-data-pol-dec}.
\end{proof}

Restoring now the subscript $a$ to $\sigma$, let $\delta\in\left(0,\,\frac{1}{2}\right)$
and for this particular choice of $\delta$, let $\psi_{\delta}$ denote the solution of
\eqref{eq:soliton} given by Theorem \ref{theo-main-sol-data-pol-dec} for $\sigma_{a}$.
Write $\omega_{\delta,\,a}:=\sigma_{a}+i\partial\bar{\partial}\psi_{\delta}$. Then
$\omega_{\delta,\,a}$ is a steady gradient K\"ahler-Ricci soliton in $\mathfrak{k}$ with soliton vector field $X$
which, by virtue of Proposition \ref{lemma5.7}(iii) and the fact that $\psi_{\delta}\in C^{\infty}_{X,\,1-\delta}(M)$,
satisfies \eqref{amazing2} with $\varepsilon=\delta$, and correspondingly \eqref{amazing} by Proposition \ref{coro-asy-Cao-met}.
Moreover, the $JX$-invariance of $\psi_{\delta}$ implies that $\mathcal{L}_{JX}\omega_{\delta,\,a}=0$.

\subsubsection{Independence of the parameters}\label{independent}
All that remains to show is that $\omega_{\delta,\,a}\in\mathfrak{k}$ (respectively $\omega_{a}\in\mathfrak{k}$)
is independent of $\delta$ and $a$ (resp.~$a$). This will then allow us to set
$\omega:=\omega_{\delta,\,a}$ (resp.~$\omega:=\omega_{a}$),
resulting in the steady gradient K\"ahler-Ricci soliton $\omega$ of Theorem \ref{main} with the desired properties. We prove this for
$\omega_{\delta,\,a}$ only. The proof for $\omega_{a}$ is similar.

To this end, consider the K\"ahler forms $\omega_{1}:=\omega_{\delta_{1},\,a_{1}}$ and $\omega_{2}:=\omega_{\delta_{2},\,a_{2}}$ in the same
K\"ahler class $\mathfrak{k}$ of $M$, where $0\leq a_{1}\neq a_{2}$ and without loss of generality we assume that
$0<\delta_{1}<\delta_{2}<\frac{1}{2}$. Then by Lemma \ref{deldelbar}, there exists a smooth real-valued function $u:M\to\mathbb{R}$
such that
\begin{equation}\label{stanco2}
\omega_{2}-\omega_{1}=i\partial\bar{\partial}u.
\end{equation}
By averaging $u$ over the real torus action on $M$ induced by the torus action on $C_{0}$
generated by $J_{0}r\partial_{r}$, we may assume that $\mathcal{L}_{JX}u=0$. Now, we know from the asymptotics that
\begin{equation}\label{tired2}
\omega_{2}-\omega_{1}=\tilde{\omega}_{a_{2}}-\tilde{\omega}_{a_{1}}+i\partial\bar{\partial}\phi
=i\partial\bar{\partial}(\Phi_{a_{2}}-\Phi_{a_{1}}+\phi)
\end{equation}
for some $\phi\in C^{\infty}_{X,\,1-\delta_{1}}(M)$, where $\Phi_{a}(t)$ is the K\"ahler potential of
$\tilde{\omega}_{a}$ as in Proposition \ref{caosoliton}. On subtracting \eqref{stanco2} from \eqref{tired2}, we see that at infinity
\begin{equation*}
i\partial\bar{\partial}(\Phi_{a_{2}}-\Phi_{a_{1}}+\phi-u)=0.
\end{equation*}
Set $G:=\Phi_{a_{2}}-\Phi_{a_{1}}+\phi-u$. Then $G$ is a smooth real-valued pluriharmonic function on the end of $M$ with
$\mathcal{L}_{JX}G=0$ and so by Lemma \ref{pluri}(i) must be equal to a constant. Therefore by subtracting a constant from
$u$, we may assume that
\begin{equation}\label{star}
u=\Phi_{a_{2}}-\Phi_{a_{1}}+\phi
\end{equation}
outside a compact subset of $M$.

Returning now to \eqref{stanco2}, we deduce from Lemma \ref{cma} that
\begin{equation*}
i\partial\bar{\partial}\left(\log\left(\frac{(\omega_{1}+i\partial\bar{\partial}u)^{n}}{\omega_{1}^{n}}\right)
+\frac{X}{2}\cdot u\right)=0.
\end{equation*}
Since $\mathcal{L}_{JX}\left(\log\left(\frac{(\omega_{1}+i\partial\bar{\partial}\varphi)^{n}}{\omega_{1}^{n}}\right)
+\frac{X}{2}\cdot\varphi\right)=0$, it subsequently follows from Lemma \ref{pluri}(ii) that
\begin{equation}\label{soccers}
\log\left(\frac{(\omega_{1}+i\partial\bar{\partial}u)^{n}}{\omega_{1}^{n}}\right)
+\frac{X}{2}\cdot u=C
\end{equation}
for some constant $C$. Next recalling \eqref{star}, we see that at infinity,
\begin{equation*}
\begin{split}
X\cdot u &=4\Phi'_{a_{2}}(t)-4\Phi'_{a_{1}}(t)+X\cdot\phi\\
&=4\varphi_{a_{2}}(t)-4\varphi_{a_{1}}(t)+X\cdot\phi\\
&=O\left(\frac{(\log t)^{2}}{t^{2}}\right)+O(t^{-2+\delta_{1}})\\
&=o(1),
\end{split}
\end{equation*}
where we have used Proposition \ref{potential} and the fact that $\phi\in C^{\infty}_{X,\,1-\delta_{1}}(M)$ in the penultimate line.
Moreover, we also find that at infinity,
\begin{equation*}
\begin{split}
i\partial\bar{\partial}u &=\omega_{2}-\omega_{1}\\
&=\tilde{\omega}_{a_{2}}-\tilde{\omega}_{a_{1}}+\underbrace{i\partial\bar{\partial}\psi_{\delta_{2}}}_{=\,O\left(t^{-2+\delta_{2}}\right)}
-\underbrace{i\partial\bar{\partial}\psi_{\delta_{1}}}_{=\,O\left(t^{-2+\delta_{1}}\right)}\\
&=\underbrace{(\tilde{\omega}_{a_{2}}-\hat{\omega})}_{=\,O\left(t^{-1}\log(t)\right)}+\underbrace{(\hat{\omega}-\tilde{\omega}_{a_{1}})}_{=\,O\left(t^{-1}\log(t)\right)}+O(t^{-2+\delta_{1}})\\
&=O\left(t^{-1}\log(t)\right)+O\left(t^{-2+\delta_{1}}\right)\\
&=o(1),
\end{split}
\end{equation*}
where this time we have used Proposition \ref{coro-asy-Cao-met} in the penultimate line. These last two observations imply that
$C=0$ in \eqref{soccers} and so
\begin{equation*}
\log\left(\frac{(\omega_{1}+i\partial\bar{\partial}u)^{n}}{\omega_{1}^{n}}\right)
+\frac{X}{2}\cdot u=0.
\end{equation*}
The strong maximum principle of Hopf applied to this equation (as it was for instance in \cite[p.13]{olivier}) now implies that
$u$ is a constant. Hence $\omega_{1}=\omega_{2}$, as required.

\subsection{Uniqueness}\label{finished}

In the setting of Theorem \ref{main}, suppose that for some $\varepsilon>0$, $\nu$ is a complete steady gradient K\"ahler-Ricci soliton in
the K\"ahler class $\mathfrak{k}$ of $M$ with $\mathcal{L}_{JX}\nu=0$ satisfying
\begin{equation*}
\left|\widehat{\nabla}^{i}\left(\mathcal{L}_{X}^{(j)}(\pi_{*}\nu-\hat{\omega})\right)
\right|_{\hat{g}}\leq C(i,\,j)t^{-\varepsilon-\frac{i}{2}-j}\quad\textrm{for all $i,\,j\in\mathbb{N}_{0}$}.
\end{equation*}
With $\omega$ denoting the steady gradient K\"ahler-Ricci soliton in $\mathfrak{k}$ from
Theorem \ref{main}, write $\rho_{\nu}$ and $\rho_{\omega}$ for the Ricci form of $\nu$ and $\omega$ respectively.
Throughout, we identify $M$ and $C_{0}$ at infinity using the resolution map. We begin with the following auxiliary claim
 which essentially asserts that the difference between $\nu$ and $\omega$ must be of order $O(t^{-1})$.

\begin{claim}\label{hotwire}
There exists $c\in\mathbb{R}$ and a function $\phi\in C^{\infty}_{X,\,\delta}(M)$
for some $\delta>0$ such that
$$\nu-\omega-ci\partial\bar{\partial}f=i\partial\bar{\partial}\phi.$$
\end{claim}

\begin{proof}
By Lemma \ref{deldelbar}, there exists a smooth real-valued function $\varphi$ on $M$ such that
\begin{equation*}
\nu=\omega+i\partial\bar{\partial}\varphi
\end{equation*}
that necessarily satisfies
\begin{equation*}
\left|\widehat{\nabla}^k\left(\mathcal{L}_{X}^{(j)}(\pi_{*}(i\partial\bar{\partial}\varphi))\right)\right|_{\hat{g}}\leq C_{k}t^{-\varepsilon-\frac{k}{2}-j}\quad\textrm{for all $j,k\in\mathbb{N}_{0}$}.
\end{equation*}
Since $\mathcal{L}_{JX}\omega=\mathcal{L}_{JX}\nu=0$, by averaging $\varphi$ over the real torus action on $M$ induced by the torus action on $C_{0}$
generated by $J_{0}r\partial_{r}$, we can assume that $\mathcal{L}_{JX}\varphi=0$. Then from Lemma \ref{cma} we
see that
\begin{equation*}
0=i\partial\bar{\partial}\left(\log\left(\frac{(\omega+i\partial\bar{\partial}\varphi)^{n}}{\omega^{n}}\right)
+\frac{X}{2}\cdot\varphi\right).
\end{equation*}
Now, $\mathcal{L}_{JX}\left(\log\left(\frac{(\omega+i\partial\bar{\partial}\varphi)^{n}}{\omega^{n}}\right)
+\frac{X}{2}\cdot\varphi\right)=0$. Lemma \ref{pluri}(ii) therefore asserts that
\begin{equation*}
\log\left(\frac{(\omega+i\partial\bar{\partial}\varphi)^{n}}{\omega^{n}}\right)
+\frac{X}{2}\cdot\varphi=2nc
\end{equation*}
for some constant $c\in\mathbb{R}$. In particular, we deduce that $X\cdot\varphi=O(1)$ so that $\varphi$ grows at most linearly. Indeed,
$X\cdot\varphi=4nc+O(t^{-\varepsilon})$ and so writing $\gamma_{x}(t)$ for the integral curve of $X$ with $\gamma_{x}(0)=x$, we find that
\begin{equation}\label{novel}
\varphi(\gamma_{x}(t))=\varphi(x)+4nct+O(t^{1-\varepsilon}).
\end{equation}

By Lemma \ref{solitonid} and the asymptotics of $\omega$, we know that
\begin{equation*}
|X|^{2}_{g}+\RR_{g}=4n,
\end{equation*}
where $g$ is the K\"ahler metric associated to $\omega$.
In particular, we have that
\begin{equation*}
\frac{d}{dt}(f(\gamma_{x}(t)))=|X|_{g}^{2}(\gamma_{x}(t))=4n-\RR_{g}(\gamma_{x}(t))
\end{equation*}
so that
\begin{equation*}
\begin{split}
f(\gamma_{x}(t))-f(x)&=\int_{0}^{t}\frac{d}{ds}(f(\gamma_{x}(s)))\,ds\\
&=\int_{0}^{t}(4n-\RR_{g}(\gamma_{x}(s)))\,ds\\
&=4nt-\int_{0}^{t}\RR_{g}(\gamma_{x}(s))\,ds.
\end{split}
\end{equation*}
Solving for $t$ and plugging into \eqref{novel} then yields:
\begin{equation}\label{miami}
\varphi(\gamma_{x}(t))-cf(\gamma_{x}(t))=\varphi(x)-cf(x)+c\int_{0}^{t}\RR_{g}(\gamma_{x}(s))\,ds+O(t^{1-\varepsilon}).
\end{equation}

Next consider the equation
\begin{equation}\label{complex}
\Delta_{\omega}\phi_{0}+\frac{X}{2}\cdot\phi_{0}=\log\left(\frac{(\omega+i\partial\bar{\partial}\varphi)^{n}}{\omega^{n}}\right)
-\Delta_{\omega}\varphi.
\end{equation}
At infinity, the right-hand side of this PDE takes the form
\begin{equation*}
\begin{split}
\log\left(\frac{(\omega+i\partial\bar{\partial}\varphi)^{n}}{\omega^{n}}\right)
-\Delta_{\omega}\varphi&=\Delta_{\omega}\varphi+\sum_{k\,=\,2}^{n}{n\choose k}\frac{\omega^{n-k}\wedge(i\partial\bar{\partial}\varphi)^{k}}{\omega^{n}}-\Delta_{\omega}\varphi\\
&=\sum_{k\,=\,2}^{n}{n\choose k}\frac{\omega^{n-k}\wedge(i\partial\bar{\partial}\varphi)^{k}}{\omega^{n}}\\
&=\frac{(i\partial\bar{\partial}\varphi)^{2}\wedge \Psi}{\omega^n},
\end{split}
\end{equation*}
where $\Psi$ denotes a bounded $(n-2,\,n-2)$-form together with its derivatives.
In particular, since $i\partial\bar{\partial}\varphi\in C^{\infty}_{X,\,\delta_{0}}(M)$
for every $\delta_{0}\in(0,\,\varepsilon]$, we see that
\begin{equation*}
\log\left(\frac{(\omega+i\partial\bar{\partial}\varphi)^{n}}{\omega^{n}}\right)-\Delta_{\omega}\varphi\in C^{\infty}_{X,\,2\delta_{0}}(M)
\end{equation*}
for every $\delta_{0}\in(0,\,\varepsilon]$. Consequently, Theorem \ref{surjectivity-drift-Laplacian} applies and tells us that there exists a
function $\phi_{0}\in C^{\infty}_{X,\,-1+2\delta_{0}}(M)$ solving \eqref{complex}
for every $\delta_{0}\in(0,\,\min\{\varepsilon,\,1\})\setminus \mathcal{F}$ for some $\mathcal{F}\subset\left(0,\,\frac{1}{2}\right)$ a fixed finite subset. It follows that
\begin{equation}\label{nor-uni-thm}
\begin{split}
\Delta_{\omega}(\varphi+\phi_{0})+\frac{X}{2}\cdot(\varphi+\phi_{0})&=
\Delta_{\omega}\varphi+\frac{X}{2}\cdot\varphi+\log\left(\frac{(\omega+i\partial\bar{\partial}\varphi)^{n}}{\omega^{n}}\right)
-\Delta_{\omega}\varphi\\
&=\log\left(\frac{(\omega+i\partial\bar{\partial}\varphi)^{n}}{\omega^{n}}\right)+\frac{X}{2}\cdot\varphi\\
&=2nc.
\end{split}
\end{equation}
Since $f$ satisfies $\Delta_{\omega}f+\frac{X}{2}\cdot f=2n$ by Lemma \ref{solitonid}, we deduce from
\eqref{nor-uni-thm} that
$$\left(\Delta_{\omega}+\frac{X}{2}\cdot\right)\left(\varphi-cf+\phi_{0}\right)=0.$$

Recalling now \eqref{miami}, we have that
\begin{equation*}
\left(\varphi-cf+\phi_{0}\right)(\gamma_{x}(t))=\varphi(x)-cf(x)+\phi_{0}(\gamma_{x}(t))+c\int_{0}^{t}\RR_{g}(\gamma_{x}(s))\,ds+O(t^{1-\varepsilon}).
\end{equation*}
Since the fixed point set of $X$ is compact so that the flow-lines of $X$ flow into a compact set (see \cite[Proposition 2.28]{conlondez}),
and since $\RR_{g}=O(t^{-1})$, we know that $\varphi-cf+\phi_{0}$ grows sublinearly.
As a sublinearly growing function lying in the kernel of $\Delta_{\omega}+\frac{X}{2}$,
the Liouville theorem \cite[Corollary 1.4]{Hua-Zha-Zha} asserts that $\varphi-cf+\phi_{0}$ must be equal to a constant.
Thus, modifying $\phi_{0}$ by this constant and setting $\phi_{1}:=cf-\phi_{0}$, we arrive at the fact that
$$\nu=\omega+i\partial\bar{\partial}\phi_{1},$$
where this time
\begin{equation*}
\left|\widehat{\nabla}^k\left(\mathcal{L}_{X}^{(j)}(\pi_{*}(i\partial\bar{\partial}\phi_{1}))\right)\right|_{\hat{g}}
\leq C_{k}t^{-\delta_{1}-\frac{k}{2}-j}\quad\textrm{for all $j,k\in\mathbb{N}_{0}$}
\end{equation*}
and $cf-\phi_{1}\in C^{\infty}_{X,\,-1+\delta_{1}}(M)$
for every $\delta_{1}:=2\delta_0\in(0,\,\min\{2\varepsilon,\,1\})\setminus \mathcal{F}$. Iterating the above argument with $\delta_{1}$
in place of $\varepsilon$, we can find a function $\phi_{2}$ with
$$\nu=\omega+i\partial\bar{\partial}\phi_{2},$$
where
\begin{equation*}
\left|\widehat{\nabla}^k\left(\mathcal{L}_{X}^{(j)}(\pi_{*}(i\partial\bar{\partial}\phi_{2}))\right)\right|_{\hat{g}}
\leq C_{k}t^{-\delta_{2}-\frac{k}{2}-j}\quad\textrm{for all $j,k\in\mathbb{N}_{0}$}
\end{equation*}
and $cf-\phi_{2}\in C^{\infty}_{X,\,-1+2\delta_{2}}(M)$
for every $\delta_{2}\in (0,\,\min\{2\delta_1,\,1\})\setminus \mathcal{F}=(0,\,\min\{2^2\varepsilon,\,1\})\setminus \mathcal{F}$. Continuing in this manner,
we can find a function $\phi_{l}$ with
$$\nu=\omega+i\partial\bar{\partial}\phi_{l},$$
where
\begin{equation*}
\left|\widehat{\nabla}^k\left(\mathcal{L}_{X}^{(j)}(\pi_{*}(i\partial\bar{\partial}\phi_{l}))\right)\right|_{\hat{g}}
\leq C_{k}t^{-\delta_{l}-\frac{k}{2}-j}\quad\textrm{for all $j,k\in\mathbb{N}_{0}$}
\end{equation*}
and where $cf-\phi_{l}\in C^{\infty}_{X,\,-1+2\delta_{l}}(M)$
for every $\delta_{l}\in(0,\,\min\{2^{l}\varepsilon,\,1\})\setminus \mathcal{F}$.
Choosing $l$ large enough such that $2^{-l-1}<\varepsilon\leq 2^{-l}$, we can then write
$$\nu=\omega+i\partial\bar{\partial}\phi$$
for a smooth function $\phi$ satisfying
$cf-\phi\in C^{\infty}_{X,\,-1+2\delta_{l}}(M)$
for every $\delta_{l}\in\left(\frac{1}{2},\,2^l\varepsilon\right)\setminus \mathcal{F}$. In particular,
$cf-\phi\in C^{\infty}_{X,\,\delta}(M)$ for some $\delta>0$, as desired.
\end{proof}

Next, let $\psi_{s}$ denote the family of diffeomorphisms generated by the vector field $\frac{X}{2}$ with $\psi_{0}=\operatorname{Id}$
i.e.,
\begin{equation*}
\frac{\partial\psi_{s}}{\partial s}(x)=\frac{X(\psi_{s}(x))}{2},\qquad\psi_{0}=\operatorname{Id}.
\end{equation*}
Then $\omega(s):=\psi^{*}_{s}\omega,\,s\in\mathbb{R},$ defines a backward K\"ahler-Ricci flow on $M$ with $\omega(0)=\omega$
so that $\partial_{s}\omega(s)=\rho_{\omega(s)}$, where $\rho_{\omega(s)}$ denotes the Ricci form of $\omega(s)$.
Our next observation, contained in the following claim, concerns the asymptotics of $\omega(s)$.

\begin{claim}
For all $s\in\R$,
$$\omega(s)-\omega\in C_{X,\,1}^{\infty}(M).$$
\end{claim}

\begin{proof}
Let $g$ and $g(s)$ denote the K\"ahler metrics determined by $\omega$ and $\omega(s)$
respectively. It suffices to prove that
\begin{equation*}
\left|(\nabla^{g})^{i}(g(s)-g)\right|
_{g}\leq C_{k}t^{-1-\frac{i}{2}}\quad\textrm{for all $i\geq0$}
\end{equation*}
and that
\begin{equation*}
\left|(\nabla^{g})^{j}\left(\mathcal{L}_{X}^{(k)}(\omega(s)-\omega)\right)\right|
_{g}\leq C_{k}t^{-1-\frac{j}{2}-k}\quad\textrm{for all $j\geq0$ and $k\geq1$}.
\end{equation*}
To this end, we proceed as in the proof of \cite[Theorem $3.8$]{conlondez}.

For any $x\in M$, let $\gamma_{x}(v):=\psi_v(x)$ denote the flow of $\frac{X}{2}$ with $\gamma_{x}(0)=x$.
Then $$|\operatorname{Rm}(g(s))|_{g(s)}(x)=|\operatorname{Rm}(g)|_{g}(\gamma_{x}(s))$$
which is bounded above by some positive constant $K$, and so by integrating the backward Ricci flow equation,
it follows that
\begin{equation}\label{equiv}
e^{-Ks}g(x)\leq g(s)(x)\leq e^{Ks}g(x)\qquad\textrm{for all $x\in M$}.
\end{equation}
Thus, for any $x\in M\setminus E$,
\begin{equation*}
\begin{split}
|g(s)-g|_{g}(x)&\leq\int_{0}^{s}|\partial_{u}g(u)|_{g}(x)\,du\leq C\int_{0}^{s}|\Ric(g(u))|_{g}(x)\,du\leq C(s)\int_{0}^{s}|\Ric(g(u))|_{g(u)}(x)\,du\\
&=C(s)\int_{0}^{s}|\Ric(g)|_{g}(\gamma_{u}(x))\,du\leq C(s)\int_{0}^{s}(t(\gamma_{u}(x)))^{-1}\,du,\\
\end{split}
\end{equation*}
where we have used \eqref{equiv}. Now, since $$\frac{\partial}{\partial v}(t(\gamma_{x}(v)))=dt(X)=4,$$
we have that $t(\gamma_{x}(v))=4v+t(x).$ Hence if $t(x)$ is larger than $8s$ say, so that $t(\gamma_{x}(v))>0$ for all $|v|\leq s$, then
\begin{equation*}
\begin{split}
|g(s)-g|_{g}(x)&\leq C(s)\int_{0}^{s}(t(\gamma_{u}(x)))^{-1}\,du=C(s)\int_{0}^{s}(4v+t(x))^{-1}\,du\\
&=C(s)\ln\left(\frac{4s+t(x)}{t(x)}\right)=C(s)\ln\left(1+\frac{4s}{t(x)}\right)\leq Ct(x)^{-1}.
\end{split}
\end{equation*}

As for the covariant derivative, we must work slightly harder. Recall that if $T$ is a tensor on $M$, then $\nabla^{g(s)}T
=\nabla^gT+g(s)^{-1}\ast\nabla^g(g(s)-g)\ast T$ since at the level of Christoffel symbols, one has that
$$\Gamma(g(s))_{ij}^k=\Gamma(g)_{ij}^k+\frac{1}{2}g(s)^{km}\left(\nabla_i^g(g(s)-g)_{jm}+\nabla_j^g(g(s)-g)_{im}-\nabla_m^g(g(s)-g)_{ij}\right).$$
In light of this, we have for all $x\in M\setminus E$,
\begin{equation*}
\begin{split}
\partial_u&\left(|\nabla^{g}(g(u)-g)|^2_{g}\right)(x)\leq C|\nabla^g\Ric(g(u))|_g(x)|\nabla^g(g(u)-g)|_g(x)\\
&\leq C\left(|\nabla^{g(u)}\Ric(g(u))|_g(x)+\left(\left|\left(\nabla^g-\nabla^{g(u)}\right)\Ric(g(u))\right|_g(x)\right)\right)|\nabla^g(g(u)-g)|_g(x)\\
&\leq C\left(|\nabla^{g(u)}\Ric(g(u))|_{g(u)}(x)+\left(\left|\left(\nabla^g-\nabla^{g(u)}\right)\Ric(g(u))\right|_g(x)\right)\right)|\nabla^g(g(u)-g)|_g(x)\\
&=C\left(|\nabla^{g}\Ric(g)|_g(\gamma_{x}(u))+\left|\left(\nabla^g-\nabla^{g(u)}\right)\Ric(g(u))\right|_g(x)\right)|\nabla^g(g(u)-g)|_g(x)\\
&\leq C\left(|\nabla^{g}\Ric(g)|_g(\gamma_{x}(u))+|\nabla^g(g(u)-g)|_g(x)|\Ric(g(u))|_g(x)\right)|\nabla^g(g(u)-g)|_g(x)\\
&\leq C\left(|\nabla^{g}\Ric(g)|_g(\gamma_{x}(u))+|\nabla^g(g(u)-g)|_g(x)|\Ric(g(u))|_{g(u)}(x)\right)|\nabla^g(g(u)-g)|_g(x)\\
&\leq C\left(|\nabla^{g}\Ric(g)|_g(\gamma_{x}(u))+|\nabla^g(g(u)-g)|_g(x)|\Ric(g)|_{g}(\gamma_{x}(u))\right)|\nabla^g(g(u)-g)|_g(x)\\
&\leq C\left(|\nabla^{g}\Ric(g)|_g(\gamma_{x}(u))+|\nabla^g(g(u)-g)|_g(x)\right)|\nabla^g(g(u)-g)|_g(x)\\
&\leq C|\nabla^g(g(u)-g)|^2(x)+C|\nabla^{g}\Ric(g)|^{2}_g(\gamma_{x}(u)),
\end{split}
\end{equation*}
where throughout $C$ denotes a positive constant depending on $s$ that may vary from line to line and where Young's inequality has been used in the last line.
This inequality may be rewritten as
$$\partial_u\left(e^{-Cu}|\nabla^{g}(g(u)-g)|^2_{g}\right)(x)\leq
Ce^{-Cu}|\nabla^{g}\Ric(g)|^{2}_g(\gamma_{x}(u)),$$
which, upon integrating over $[0,\,s]$, yields the fact that
$$e^{-Cs}|\nabla^{g}(g(s)-g)|^2_{g}(x)
-|\nabla^{g}(g(0)-g)|^2_{g}(x)\leq C\int_{0}^{s}e^{-Cu}|\nabla^{g}\Ric(g)|^{2}_g(\gamma_{x}(u))\,du.$$
Since $|\nabla^{g}(g(0)-g)|^2_{g}(x)=0$, we deduce that
\begin{equation*}
\begin{split}
|\nabla^{g}(g(s)-g)|^2_{g}(x)&\leq Ce^{Cs}\int_{0}^{s}e^{-Cu}|\nabla^{g}\Ric(g)|^{2}_g(\gamma_{x}(u))\,du\\
&\leq C(s)\int_{0}^{s}|\nabla^{g}\Ric(g)|^{2}_g(\gamma_{x}(u))\,du\\
&\leq C(s)\int_{0}^{s}(t(\gamma_{x}(u)))^{-3}\,du\\
&\leq C(s)\int_{0}^{s}(4u+t(x))^{-3}\,du\\
&\leq C(s)\left(\frac{1}{(4s+t(x))^{2}}-\frac{1}{t(x)^{2}}\right)\\
&\leq C(s)t(x)^{-3}
\end{split}
\end{equation*}
so that
$$|\nabla^{g}(g(s)-g)|_{g}(x)\leq Ct(x)^{-\frac{3}{2}}.$$
The cases $i\geq 2$ are proved by induction.

Next, concerning the Lie derivatives, we know that
$$\mathcal{L}_{X}(\omega(s)-\omega)=2\rho_{\omega(s)}-2\rho_{\omega}.$$
Since $|(\nabla^{g})^{j}(\omega(s)-\omega)|_{g}=O(t^{-\frac{j}{2}})$ for all $j\geq0$, we then find that
$$|(\nabla^{g})^{j}\mathcal{L}_{X}(\omega(s)-\omega)|_g=2|(\nabla^{g})^{j}(\rho_{\omega(s)}-\rho_{\omega})|_{g}=O\left(t^{-2-\frac{j}{2}}\right)
\qquad\textrm{for all $j\geq0$}.$$
The conclusion now follows from another induction argument using the fact that
$$\rho_{\omega(s)}-\rho_{\omega}=-i\partial\bar{\partial}\log\left(\frac{\omega(s)^{n}}{\omega^{n}}\right)$$
so that
$$|(\nabla^{g})^{j}\mathcal{L}^{(k)}_{X}(\omega(s)-\omega)|_{g}
\leq C\left|(\nabla^{g})^{\,j+2}X^{(k-1)}\cdot\log\left(\frac{\omega(s)^{n}}{\omega^{n}}\right)\right|_{g}
\qquad\textrm{for all $j\geq0$ and $k\geq1$}.$$
\end{proof}

Applying Claim \ref{hotwire} with $\omega(s)$ in place of $\nu$, we see accordingly that there exists a constant $c_{s}$ (depending on $s$) such that
\begin{equation}\label{wtf}
\omega(s)-\omega-c_{s}i\partial\bar{\partial}f=i\partial\bar{\partial}\phi_{s}
\end{equation}
for some $\phi_{s}\in C^{\infty}_{X,\,\delta_{s}}(M)$, where $\delta_{s}>0$. Next note:

\begin{claim}
$$|\omega(s)-\omega-s\rho_{\omega}|_{g}=O\left(t^{-2}\right).$$
\end{claim}

\begin{proof}
For any point $x$ on the complement of the exceptional set of $M$, we have that
\begin{equation*}
\begin{split}
(\omega(s)-\omega-s\rho_{\omega})(x)&=\int_{0}^{s}\partial_{u}(\omega(u)-\omega-u\rho_{\omega})(x)\,du\\
&=\int_{0}^{s}(\rho_{\omega(u)}-\rho_{\omega})(x)\,du\\
&=\int_{0}^{s}\int_{0}^{u}\frac{\partial\rho_{\omega(v)}}{\partial v}(x)\,dv\,du\\
&=\int_{0}^{s}\int_{0}^{u}\left(-\frac{1}{2}\Delta_{\omega(v)}\rho_{\omega(v)}+\operatorname{Rm}(\omega(v))\ast\rho_{\omega(v)}\right)(x)\,dv\,du\\
&=\int_{0}^{s}\int_{v}^{s}\left(-\frac{1}{2}\Delta_{\omega(v)}\rho_{\omega(v)}+\operatorname{Rm}(\omega(v))\ast\rho_{\omega(v)}\right)(x)\,du\,dv\\
&=\int_{0}^{s}(s-v)\left(-\frac{1}{2}\Delta_{\omega(v)}\rho_{\omega(v)}+\operatorname{Rm}(\omega(v))\ast\rho_{\omega(v)}\right)(x)\,dv,
\end{split}
\end{equation*}
where $\Delta_{\omega(v)}$ denotes the usual rough Laplacian. Here, we have used the evolution equation satisfied by the Ricci curvature $\rho_{\omega(s)}$ along the (backward) K\"ahler-Ricci flow in the fourth line; see \cite[Chapter $3$, Section $3.2.6$]{Bou-Eys-Gue}. In particular, we obtain the bound
\begin{equation*}
\begin{split}
|\omega(s)-\omega-s\rho_{\omega}|_{g}(x)&\leq\int_{0}^{s}|s-v|\left|-\frac{1}{2}\Delta_{\omega(v)}\rho_{\omega(v)}+\operatorname{Rm}(\omega(v))\ast\rho_{\omega(v)}\right|_{g}(x)\,dv\\
&\leq C(s)\int_{0}^{s}\left|-\frac{1}{2}\Delta_{\omega(v)}\rho_{\omega(v)}+\operatorname{Rm}(\omega(v))\ast\rho_{\omega(v)}\right|_{g}(x)\,dv\\
&\leq C(s)\int_{0}^{s}e^{Kv}\left|-\frac{1}{2}\Delta_{\omega(v)}\rho_{\omega(v)}+\operatorname{Rm}(\omega(v))\ast\rho_{\omega(v)}\right|_{g(v)}(x)\,dv\\
&\leq C(s)\int_{0}^{s}\left|-\frac{1}{2}\Delta_{\omega(v)}\rho_{\omega(v)}+\operatorname{Rm}(\omega(v))\ast\rho_{\omega(v)}\right|_{g(v)}(x)\,dv\\
&=C(s)\int_{0}^{s}\left|-\frac{1}{2}\Delta_{\omega}\rho_{\omega}+\operatorname{Rm}(\omega)\ast\rho_{\omega}\right|_{g}(\gamma_{v}(x))\,dv\\
&\leq C(s)\int_{0}^{s}t(\gamma_{v}(x))^{-2}\,dv\\
&\leq C(s)\int_{0}^{s}(4v+t(x))^{-2}\,dv\\
&\leq Ct(x)^{-2},
\end{split}
\end{equation*}
as desired.
\end{proof}

Now on one hand, we see from \eqref{wtf} that
$$|s-c_{s}||\rho_{\omega}|_{g}=|s-c_{s}||i\partial\bar{\partial}f|_{g}=
|\omega(s)-\omega-s\rho_{\omega}-i\partial\bar{\partial}\phi_{s}|_{g}\leq
Ct^{-1-\delta_{s}}$$
for some constant $C>0$, whereas on the other hand, from Lemma \ref{low-pos-bd-scal} we see that
$$|\rho_{\omega}|_{g}\geq C|\RR_{g}|\geq C|\underbrace{|\RR_{\tilde{g}}|}_{\geq\,Ct^{-1}}-\underbrace{
|\RR_{g}-\RR_{\tilde{g}}|}_{\leq\,Ct^{-2}}|\geq Ct^{-1}$$
for another constant $C>0$ and for $t$ sufficiently large. We must therefore have $c_{s}=s$, and so
\begin{equation}\label{wtf2}
\omega(s)=\omega+si\partial\bar{\partial}f+i\partial\bar{\partial}\phi_{s}
\end{equation}
for some $\phi_{s}\in C^{\infty}_{X,\,\delta_{s}}(M)$ with $\delta_{s}>0$.

Finally, combining \eqref{wtf2} with Claim \ref{hotwire}, we conclude that
$$\nu-\omega(c)=i\partial\bar{\partial}u,$$
where now $u\in C_{X,\,\hat{\delta}}^{\infty}(M)$ for some $\hat{\delta}>0$.
After noting the $JX$-invariance of $u$, $\nu$, and $\omega(c)$, Lemma \ref{cma} followed by an application of Lemma \ref{pluri}(ii) yields the equation
\begin{equation*}
\log\left(\frac{(\omega(c)+i\partial\bar{\partial}u)^{n}}{\omega(c)^{n}}\right)
+\frac{X}{2}\cdot u=C
\end{equation*}
satisfied by $u$ for some constant $C$. Since $u\in C_{X,\,\hat{\delta}}^{\infty}(M)$, $C$ must be equal to zero.
The strong maximum principle of Hopf now implies that $u$ is a constant so that $\nu=\omega(c)$ for some $c\in\mathbb{R}$, as required.

\newpage
\appendix

\section{The model metric $\hat{g}$}\label{appendixa}

Let $(C_{0},\,g_{0})$ be a Calabi-Yau cone of complex dimension $n\geq2$ with link $S$, radial function $r$, complex structure $J_{0}$, transverse metric $g^{T}$,
and set $\eta=d^{c}\log(r)$ and  $r^{2}=e^{t}$. We define a K\"ahler form $\hat{\omega}$ on $C_{0}$ by
$$\hat{\omega}:=\frac{i}{2}\partial\bar{\partial}\left(\frac{nt^{2}}{2}\right).$$
One can check that the corresponding K\"ahler metric $\hat{g}$ on $C_{0}$ takes the form
$$\hat{g}:=n\left(\frac{1}{4}dt^{2}+\eta^{2}+tg^{T}\right).$$
We denote the Levi-Civita connection of $\hat{g}$ by $\widehat{\nabla}$. In this appendix, we analyse how the norm and covariant derivatives measured with respect to $\hat{g}$ of various tensors behave as
$t\to+\infty$.

Let $\theta_{1},\ldots,\theta_{2n-2}$ be a local basic orthonormal coframe
of $g^{T}$ on $S$ with $\theta_{i}\circ J_{0}=-\theta_{i+1}$ for $i$ odd and let
$(\omega_{ij})_{1\,\leq\, i,\,j\,\leq\,2n-2}$ denote the matrix of connection one-forms
of $g^{T}$. Then each $\omega_{ij}$ is a basic one-form on $S$ and
$(\omega_{ij})$ solves the Cartan structure equations
\begin{equation*}
\left\{ \begin{array}{ll}
d{\theta}_{i}=\sum_{j\,=\,1}^{2n-2}{\omega}_{ji}\wedge\theta_{j}\\
{\omega}_{ij}+{\omega}_{ji}=0.
\end{array} \right.
\end{equation*}
Next set
$$\hat{\theta}_{i}:=\sqrt{nt}\theta_{i}\qquad\textrm{for $i=1,\ldots,2n-2$},\qquad\hat{\theta}_{2n-1}:=\frac{\sqrt{n}}{2}dt,\qquad\textrm{and}\qquad
\hat{\theta}_{2n}:=\eta\sqrt{n}.$$
The matrix of connection one-forms $(\hat{\omega}_{ij})_{1\,\leq\, i,\,j\,\leq\, 2n}$ of $\hat{g}$ with respect to this coframe is given by
\begin{equation*}
\begin{split}
\hat{\omega}_{ji}&=\left\{ \begin{array}{ll}
\omega_{ji}+\delta_{j,\,i+1}\frac{1}{t\sqrt{n}}\hat{\theta}_{2n},\qquad\textrm{$1\leq i\leq 2n-2$ odd,\qquad$1\leq j\leq 2n-2$},\\
\omega_{ji}-\delta_{j,\,i-1}\frac{1}{t\sqrt{n}}\hat{\theta}_{2n},\qquad\textrm{$1\leq i\leq 2n-2$ even,\qquad$1\leq j\leq 2n-2$},\\
\end{array} \right.\\
\hat{\omega}_{2n-1,\,i}&=-\frac{1}{t\sqrt{n}}\hat{\theta}_{i},\qquad 1\leq i\leq 2n-2,\\
\hat{\omega}_{2n-1,\,2n}&=0,\\
\hat{\omega}_{2n,\,i}&=\left\{ \begin{array}{ll}
\frac{1}{t\sqrt{n}}\hat{\theta}_{i+1},\qquad\textrm{$1\leq i\leq 2n-2$ odd},\\
-\frac{1}{t\sqrt{n}}\hat{\theta}_{i-1},\qquad\textrm{$1\leq i\leq 2n-2$ even}.\\
\end{array} \right.\\
\end{split}
\end{equation*}
With $i=1,\ldots,2n-2$, we have the expressions
\begin{equation}\label{cov}
\begin{split}
\widehat{\nabla}\hat{\theta}_{i}&=\sum_{k\,=\,1}^{2n}\hat{\omega}_{ki}\otimes\hat{\theta}_{k}
=\sum_{k\,=\,1}^{2n-2}\left(\omega_{ki}\pm\delta_{k,\,i+1}\frac{1}{t\sqrt{n}}\hat{\theta}_{2n}\right)\otimes\hat{\theta}_{k}
-\frac{1}{t\sqrt{n}}\hat{\theta}_{i}\otimes\hat{\theta}_{2n-1}\pm\frac{1}{t\sqrt{n}}\hat{\theta}_{i\pm1}\otimes\hat{\theta}_{2n},\\
\widehat{\nabla}\hat{\theta}_{2n-1}&=\sum_{k\,=\,1}^{2n}\hat{\omega}_{k,\,2n-1}\otimes\hat{\theta}_{k}
=\sum_{k\,=\,1}^{2n-2}\frac{1}{t\sqrt{n}}\hat{\theta}_{k}\otimes\hat{\theta}_{k},\\
\widehat{\nabla}\hat{\theta}_{2n}&=\sum_{k\,=\,1}^{2n}\hat{\omega}_{k,\,2n}\otimes\hat{\theta}_{k}=
\sum_{k\,=\,1}^{2n-2}\pm\frac{1}{t\sqrt{n}}\hat{\theta}_{k\pm1}\otimes\hat{\theta}_{k},\\
\end{split}
\end{equation}
where throughout it is understood that $i\pm1,\,k\pm1\,\in\,\{1,\ldots,2n-2\}$.

Let $\zeta$ be any basic one-form on $S$. Then we may write
$$\zeta=\sum_{k\,=\,1}^{2n-2}f_{k}\theta_{k}$$ for some basic functions
$f_{k}$ on $S$ as $\{\theta_{i}\}_{i\,=\,1}^{2n-2}$ has been chosen to be a basic orthonormal coframe of $g^{T}$. It is then easy to see that
$$\zeta=\sum_{k\,=\,1}^{2n-2}\frac{f_{k}}{\sqrt{nt}}\hat{\theta}_{k}$$
so that $$|\zeta|_{\hat{g}}=O\left(t^{-\frac{1}{2}}\right).$$
This in turn implies that
\begin{equation}\label{tea}
|\widehat{\nabla}\hat{\theta}_{i}|_{\hat{g}}=O\left(t^{-\frac{1}{2}}\right)\qquad\textrm{for $1\leq i\leq 2n-2$}.
\end{equation}
Next, from \eqref{cov} it is clear that
\begin{equation}\label{verybasic}
\begin{split}
\widehat{\nabla}\zeta&=\sum_{k\,=\,1}^{2n-2}\left(\frac{1}{{\sqrt{nt}}}df_{k}\otimes\hat{\theta}_{k}-
\frac{f_{k}}{2\sqrt{n}t^{\frac{3}{2}}}dt\otimes\hat{\theta}_{k}+\frac{f_{k}}{\sqrt{nt}}\widehat{\nabla}\hat{\theta}_{k}\right)\\
&=\frac{1}{{\sqrt{nt}}}\sum_{k\,=\,1}^{2n-2}\left(df_{k}\otimes\hat{\theta}_{k}-
\frac{f_{k}}{\sqrt{n}t}\hat{\theta}_{2n-1}\otimes\hat{\theta}_{k}+f_{k}\widehat{\nabla}\hat{\theta}_{k}\right).
\end{split}
\end{equation}
Since the exterior derivative of a basic function is basic, we know that
$$|df_{j}|_{\hat{g}}=O\left(t^{-\frac{1}{2}}\right).$$
Consequently, in light of \eqref{tea}, we must have that
$$|\widehat{\nabla}\zeta|_{\hat{g}}=O\left(t^{-1}\right).$$
By inspection it is also clear that
\begin{equation*}
\begin{split}
|\widehat{\nabla}\hat{\theta}_{2n-1}|_{\hat{g}}&=O\left(t^{-1}\right),\\
|\widehat{\nabla}\hat{\theta}_{2n}|_{\hat{g}}&=O\left(t^{-1}\right),\\
|\widehat{\nabla}^{k}(t^{-1})|_{\hat{g}}&=O\left(t^{-1-k}\right)\qquad\textrm{for $k=0,\,1,\,2.$}\\
\end{split}
\end{equation*}
Collecting all of these estimates together, in summary we have that
\begin{equation}\label{start}
\begin{split}
|\widehat{\nabla}^{k}(t^{-1})|_{\hat{g}}&=O\left(t^{-1-k}\right)\qquad\textrm{for $k=0,\,1,\,2,$}\\
|\widehat{\nabla}^{k}\zeta|_{\hat{g}}&=O\left(t^{-\frac{1}{2}-\frac{k}{2}}\right)\qquad\textrm{for $k=0,\,1,$}\\
|\widehat{\nabla}\hat{\theta}_{i}|_{\hat{g}}&=O\left(t^{-\frac{1}{2}}\right)\qquad\textrm{for $1\leq i\leq 2n-2$},\\
|\widehat{\nabla}\hat{\theta}_{2n-1}|_{\hat{g}}&=O\left(t^{-1}\right),\\
|\widehat{\nabla}\hat{\theta}_{2n}|_{\hat{g}}&=O\left(t^{-1}\right).
\end{split}
\end{equation}
We now derive the following estimates.
\begin{prop}\label{basis}
In the above situation, the following holds true:
\begin{equation*}
\begin{split}
|\widehat{\nabla}^{1+k}(t^{-1})|_{\hat{g}}&=O\left(t^{-3-\frac{(k-1)}{2}}\right)\qquad\textrm{for all $k\geq1$},\\
|\widehat{\nabla}^{k}\zeta|_{\hat{g}}&=O\left(t^{-\frac{1}{2}-\frac{k}{2}}\right)\qquad\textrm{for any basic one-form $\zeta$ on $S$},\\
|\widehat{\nabla}^{k}\hat{\theta}_{i}|_{\hat{g}}&=O\left(t^{-\frac{k}{2}}\right)\qquad\textrm{for $1\leq i\leq 2n-2$ and for all $k\geq1$},\\
|\widehat{\nabla}^{k}\hat{\theta}_{2n-1}|_{\hat{g}}&=O\left(t^{-1-\frac{(k-1)}{2}}\right)\qquad\textrm{for all $k\geq1$},\\
|\widehat{\nabla}^{k}\hat{\theta}_{2n}|_{\hat{g}}&=O\left(t^{-1-\frac{(k-1)}{2}}\right)\qquad\textrm{for all $k\geq1$}.
\end{split}
\end{equation*}
\end{prop}

\begin{proof}
We will prove the proposition by induction, where our induction hypothesis $P(m)$ is that
for every $1\leq k\leq m$,
\begin{equation*}
\begin{split}
|\widehat{\nabla}^{1+k}(t^{-1})|_{\hat{g}}&=O\left(t^{-3-\frac{(k-1)}{2}}\right),\\
|\widehat{\nabla}^{k}\zeta|_{\hat{g}}&=O\left(t^{-\frac{1}{2}-\frac{k}{2}}\right)\qquad\textrm{for any basic one-form $\zeta$ on $S$},\\
|\widehat{\nabla}^{k}\hat{\theta}_{i}|_{\hat{g}}&=O\left(t^{-\frac{k}{2}}\right),\\
|\widehat{\nabla}^{k}\hat{\theta}_{2n-1}|_{\hat{g}}&=O\left(t^{-1-\frac{(k-1)}{2}}\right),\\
|\widehat{\nabla}^{k}\hat{\theta}_{2n}|_{\hat{g}}&=O\left(t^{-1-\frac{(k-1)}{2}}\right).\\
\end{split}
\end{equation*}
$P(1)$ is true by virtue of \eqref{start}. So assume that $P(m)$ holds true for some $m\geq1$.
Then beginning with $\hat{\theta}_{2n-1}$, we have that
\begin{equation}\label{coffee}
\begin{split}
|\widehat{\nabla}^{m+1}\hat{\theta}_{2n-1}|_{\hat{g}}&\leq C\sum_{\substack{i+j+k\,=\,m \\ 1\,\leq\, l\, \leq\, 2n-2}}|\widehat{\nabla}^{i}(t^{-1})|_{\hat{g}}
\underbrace{|\widehat{\nabla}^{j}\hat{\theta}_{l}|_{\hat{g}}|\widehat{\nabla}^{k}\hat{\theta}_{l}|_{\hat{g}}}_{=\,O\left(t^{-\frac{(j+k)}{2}}\right)}\\
&\leq C\sum_{i\,=\,0}^{m}t^{-\frac{(m-i)}{2}}|\widehat{\nabla}^{i}(t^{-1})|_{\hat{g}}\\
&\leq Ct^{-\frac{m}{2}}\Biggl(t^{-1}+t^{\frac{1}{2}}t^{-2}+
\sum_{i\,=\,2}^{m}t^{\frac{i}{2}}\underbrace{|\widehat{\nabla}^{i}(t^{-1})|_{\hat{g}}}_{=\,O\left(t^{-3-\frac{(i-2)}{2}}\right)}\Biggl)\\
&\leq Ct^{-\frac{m}{2}}\left(t^{-1}+t^{-\frac{3}{2}}+t^{-2}\right)\\
&\leq Ct^{-1-\frac{m}{2}}.
\end{split}
\end{equation}
In a similar fashion, one can show that $|\widehat{\nabla}^{m+1}\hat{\theta}_{2n}|_{\hat{g}}\leq Ct^{-1-\frac{m}{2}}$.

Next, we consider the covariant derivatives of $t^{-1}$. We have just shown that\linebreak $|\widehat{\nabla}^{m+1}\hat{\theta}_{2n-1}|_{\hat{g}}
\leq Ct^{-1-\frac{m}{2}}$. Using this together with $P(m)$, we compute that
\begin{equation*}
\begin{split}
|\widehat{\nabla}^{m+2}(t^{-1})|_{\hat{g}}&\leq C
|\widehat{\nabla}^{m+1}((t^{-1})^{2}\hat{\theta}_{2n-1})|_{\hat{g}}\\
&\leq C\sum_{i+j+k\,=\,m+1}|\widehat{\nabla}^{i}(t^{-1})|_{\hat{g}}
|\widehat{\nabla}^{j}(t^{-1})|_{\hat{g}}|\widehat{\nabla}^{k}\hat{\theta}_{2n-1}|_{\hat{g}}\\
&\leq C\Biggl(\sum_{i+j\,=\,m+1}|\widehat{\nabla}^{i}(t^{-1})|_{\hat{g}}
|\widehat{\nabla}^{j}(t^{-1})|_{\hat{g}}|\hat{\theta}_{2n-1}|_{\hat{g}}
+\sum_{\substack{i+j+k\,=\,m+1 \\ k\,\geq\,1}}|\widehat{\nabla}^{i}(t^{-1})|_{\hat{g}}
|\widehat{\nabla}^{j}(t^{-1})|_{\hat{g}}\underbrace{|\widehat{\nabla}^{k}\hat{\theta}_{2n-1}|_{\hat{g}}}_{=\,O\left(t^{-1-\frac{(k-1)}{2}}\right)}\Biggl)\\
&\leq C\Biggl(\sum_{i+j\,=\,m+1}|\widehat{\nabla}^{i}(t^{-1})|_{\hat{g}}
|\widehat{\nabla}^{j}(t^{-1})|_{\hat{g}}
+\sum_{\substack{i+j+k\,=\,m+1 \\ k\,\geq\,1}}t^{-1-\frac{(k-1)}{2}}
|\widehat{\nabla}^{i}(t^{-1})|_{\hat{g}}|\widehat{\nabla}^{j}(t^{-1})|_{\hat{g}}\Biggl)\\
&\leq C\Biggl(\sum_{i+j\,=\,m+1}|\widehat{\nabla}^{i}(t^{-1})|_{\hat{g}}
|\widehat{\nabla}^{j}(t^{-1})|_{\hat{g}}
+t^{-1-\frac{m}{2}}\sum_{0\,\leq\,i+j\,\leq\,m}t^{\frac{i+j}{2}}
|\widehat{\nabla}^{i}(t^{-1})|_{\hat{g}}|\widehat{\nabla}^{j}(t^{-1})|_{\hat{g}}\Biggl)\\
&\leq C\Biggl(\sum_{\substack{i+j\,=\,m+1 \\ i\,\leq\,j}}|\widehat{\nabla}^{i}(t^{-1})|_{\hat{g}}
|\widehat{\nabla}^{j}(t^{-1})|_{\hat{g}}
+t^{-1-\frac{m}{2}}\sum_{\substack{0\,\leq\,i+j\,\leq\,m \\ i\,\leq\,j}}t^{\frac{i+j}{2}}
|\widehat{\nabla}^{i}(t^{-1})|_{\hat{g}}|\widehat{\nabla}^{j}(t^{-1})|_{\hat{g}}\Biggl)\\
&\leq  C\Biggl(t^{-1}|\widehat{\nabla}^{m+1}(t^{-1})|_{\hat{g}}
+|\widehat{\nabla}(t^{-1})|_{\hat{g}}|\widehat{\nabla}^{m}(t^{-1})|_{\hat{g}}
+\sum_{\substack{i+j\,=\,m+1 \\ 2\,\leq\,i\,\leq j}}\underbrace{|\widehat{\nabla}^{i}(t^{-1})|_{\hat{g}}
|\widehat{\nabla}^{j}(t^{-1})|_{\hat{g}}}_{=\,O\left(t^{-4-\frac{(m+1)}{2}}\right)}\\
&\qquad+t^{-1-\frac{m}{2}}\Biggl(t^{-2}+t^{\frac{1}{2}}|\widehat{\nabla}(t^{-1})|_{\hat{g}}t^{-1}
+t|\widehat{\nabla}^{2}(t^{-1})|_{\hat{g}}t^{-1}
+t|\widehat{\nabla}(t^{-1})|_{\hat{g}}^{2}\\
&\qquad+\sum_{\substack{3\,\leq\,i+j\,\leq\,m \\ i\,\leq\,j}}t^{\frac{i+j}{2}}
|\widehat{\nabla}^{i}(t^{-1})|_{\hat{g}}|\widehat{\nabla}^{j}(t^{-1})|_{\hat{g}}\Biggl)\Biggl)\\
&\leq C\Biggl(t^{-\frac{7}{2}-\frac{m}{2}}+\underbrace{t^{-2}|\widehat{\nabla}^{m}(t^{-1})|_{\hat{g}}}_{=\,O\left(t^{-4-\frac{m}{2}}\right)}
+t^{-\frac{9}{2}-\frac{m}{2}}+t^{-1-\frac{m}{2}}\Biggl(t^{-2}+t^{-\frac{5}{2}}+t^{-3}\\
&\qquad+\sum_{\substack{3\,\leq\,i+j\,\leq\,m \\ i\,\leq\,j}}t^{\frac{i+j}{2}}
|\widehat{\nabla}^{i}(t^{-1})|_{\hat{g}}|\widehat{\nabla}^{j}(t^{-1})|_{\hat{g}}\Biggl)\Biggl)\\
&\leq C\Biggl(t^{-3-\frac{m}{2}}+t^{-1-\frac{m}{2}}\Biggl(\sum_{\substack{3\,\leq\,i+j\,\leq\,m \\ i\,\leq\,j}}t^{\frac{i+j}{2}}
|\widehat{\nabla}^{i}(t^{-1})|_{\hat{g}}|\widehat{\nabla}^{j}(t^{-1})|_{\hat{g}}\Biggl)\Biggl)\\
&\leq C\Biggl(t^{-3-\frac{m}{2}}+t^{-1-\frac{m}{2}}\Biggl(\underbrace{\sum_{\substack{3\,\leq\,j\,\leq\,m \\ 0\,\leq\,j}}t^{\frac{j}{2}}
t^{-1}|\widehat{\nabla}^{j}(t^{-1})|_{\hat{g}}}_{=\,O\left(t^{-3}\right)}+\underbrace{\sum_{\substack{3\,\leq\,1+j\,\leq\,m \\ 1\,\leq\,j}}t^{\frac{1+j}{2}}
|\widehat{\nabla}(t^{-1})|_{\hat{g}}|\widehat{\nabla}^{j}(t^{-1})|_{\hat{g}}}_{=\,O\left(t^{-\frac{7}{2}}\right)}\\
&\qquad+\underbrace{\sum_{\substack{3\,\leq\,i+j\,\leq\,m \\ 2\,\leq\,i\,\leq\,j}}t^{\frac{i+j}{2}}
|\widehat{\nabla}^{i}(t^{-1})|_{\hat{g}}|\widehat{\nabla}^{j}(t^{-1})|_{\hat{g}}}_{=\,O\left(t^{-4}\right)}\Biggl)\Biggl)\\
&\leq Ct^{-3-\frac{m}{2}}.
\end{split}
\end{equation*}

As for $\hat{\theta}_{i}$ with $1\leq i\leq 2n-2$, we find that
\begin{equation*}
\begin{split}
|\widehat{\nabla}^{m+1}\hat{\theta}_{i}|_{\hat{g}}&\leq
C\Biggl(\sum_{\substack{p+q\,=\,m \\ 1\,\leq\, k\, \leq\, 2n-2}}
\underbrace{|\widehat{\nabla}^{p}\omega_{ik}|_{\hat{g}}|\widehat{\nabla}^{q}\hat{\theta}_{k}|_{\hat{g}}}_{=\,O\left(t^{-\frac{1}{2}-\frac{m}{2}}\right)}
+\sum_{\substack{p+q+r\,=\,m \\ 1\,\leq\, k\, \leq\, 2n-2}}|\widehat{\nabla}^{p}(t^{-1})|_{\hat{g}}
|\widehat{\nabla}^{q}\hat{\theta}_{2n}|_{\hat{g}}|\widehat{\nabla}^{r}\hat{\theta}_{k}|_{\hat{g}}\\
&\qquad+\sum_{\substack{p+q+r\,=\,m}}
|\widehat{\nabla}^{p}(t^{-1})|_{\hat{g}}
|\widehat{\nabla}^{q}\hat{\theta}_{2n-1}|_{\hat{g}}|\widehat{\nabla}^{r}\hat{\theta}_{i}|_{\hat{g}}
+\sum_{\substack{p+q+r\,=\,m}}|\widehat{\nabla}^{p}(t^{-1})|_{\hat{g}}
|\widehat{\nabla}^{q}\hat{\theta}_{2n}|_{\hat{g}}|\widehat{\nabla}^{r}\hat{\theta}_{i\pm1}|_{\hat{g}}\Biggl)\\
&\leq C\Biggl(t^{-\frac{1}{2}-\frac{m}{2}}
+\sum_{l\,=\,2n-1}^{2n}\Biggl(\sum_{p+q+r\,=\,m}
|\widehat{\nabla}^{p}(t^{-1})|_{\hat{g}}
t^{-\frac{r}{2}}|\widehat{\nabla}^{q}\hat{\theta}_{l}|_{\hat{g}}\Biggl)\Biggl)\\
&\leq C\Biggl(t^{-\frac{1}{2}-\frac{m}{2}}
+\sum_{p+r\,=\,m}
|\widehat{\nabla}^{p}(t^{-1})|_{\hat{g}}
t^{-\frac{r}{2}}+\sum_{l\,=\,2n-1}^{2n}\Biggl(\sum_{\substack{p+q+r\,=\,m \\ q\,\geq\,1}}|\widehat{\nabla}^{p}(t^{-1})|_{\hat{g}}
t^{-\frac{r}{2}}|\widehat{\nabla}^{q}\hat{\theta}_{l}|_{\hat{g}}
\Biggl)\Biggl)\\
&\leq C\Biggl(t^{-\frac{1}{2}-\frac{m}{2}}
+\sum_{p\,=\,0}^{m}
|\widehat{\nabla}^{p}(t^{-1})|_{\hat{g}}
t^{-\frac{(m-p)}{2}}+\sum_{\substack{p+q+r\,=\,m \\ q\,\geq\,1}}|\widehat{\nabla}^{p}(t^{-1})|_{\hat{g}}
t^{-\frac{r}{2}}t^{-1-\frac{(q-1)}{2}}\Biggl)\\
&\leq C\Biggl(t^{-\frac{1}{2}-\frac{m}{2}}
+\sum_{p\,=\,0}^{m}
|\widehat{\nabla}^{p}(t^{-1})|_{\hat{g}}
t^{-\frac{(m-p)}{2}}+t^{-\frac{1}{2}}\sum_{\substack{p+q+r\,=\,m \\ q\,\geq\,1}}t^{-\frac{(q+r)}{2}}|\widehat{\nabla}^{p}(t^{-1})|_{\hat{g}}\Biggl)\\
&\leq C\Biggl(t^{-\frac{1}{2}-\frac{m}{2}}
+\sum_{p\,=\,0}^{m}
|\widehat{\nabla}^{p}(t^{-1})|_{\hat{g}}
t^{-\frac{(m-p)}{2}}+t^{-\frac{1}{2}}\sum_{k\,=\,1}^{m}\sum_{\substack{p\,=\,m-k}}
t^{-\frac{k}{2}}|\widehat{\nabla}^{p}(t^{-1})|_{\hat{g}}\Biggl)\\
&\leq C\Biggl(t^{-\frac{1}{2}-\frac{m}{2}}
+\sum_{p\,=\,0}^{m}
|\widehat{\nabla}^{p}(t^{-1})|_{\hat{g}}
t^{-\frac{(m-p)}{2}}+t^{-\frac{1}{2}}\sum_{k\,=\,1}^{m}t^{-\frac{k}{2}}|\widehat{\nabla}^{m-k}(t^{-1})|_{\hat{g}}\Biggl)\\
&\leq C\Biggl(t^{-\frac{1}{2}-\frac{m}{2}}
+\sum_{p\,=\,0}^{m}
|\widehat{\nabla}^{p}(t^{-1})|_{\hat{g}}
t^{-\frac{(m-p)}{2}}+t^{-\frac{1}{2}}\sum_{p\,=\,0}^{m-1}t^{-\frac{(m-p)}{2}}|\widehat{\nabla}^{p}(t^{-1})|_{\hat{g}}\Biggl)\\
&\leq C\Biggl(t^{-\frac{1}{2}-\frac{m}{2}}
+\underbrace{\sum_{p\,=\,0}^{m}
|\widehat{\nabla}^{p}(t^{-1})|_{\hat{g}}
t^{-\frac{(m-p)}{2}}}_{=\,O\left(t^{-1-\frac{m}{2}}\right)\quad\textrm{as in \eqref{coffee}}}\Biggl)\\
&\leq Ct^{-\frac{1}{2}-\frac{m}{2}}=Ct^{-\frac{(m+1)}{2}}.
\end{split}
\end{equation*}

Finally, we consider a basic one-form $\zeta=\sum_{k\,=\,1}^{2n-2}f_{k}\theta_{k}$ on $S$,
where $\{f_{k}\}$ are basic functions. First note that by inspection it is clear that
$$|\widehat{\nabla}^{k}(t^{-\frac{1}{2}})|_{\hat{g}}=O\left(t^{-\frac{1}{2}-k}\right)\qquad\textrm{for $0\leq k\leq 2$},$$
and since
\begin{equation*}
\begin{split}
||2t^{-\frac{1}{2}}\widehat{\nabla}^{k}(t^{-\frac{1}{2}})|_{\hat{g}}-|\widehat{\nabla}^{k}(t^{-1})|_{\hat{g}}|
&=||2t^{-\frac{1}{2}}\widehat{\nabla}^{k}(t^{-\frac{1}{2}})|_{\hat{g}}-|\widehat{\nabla}^{k}(t^{-\frac{1}{2}}t^{-\frac{1}{2}})|_{\hat{g}}|\\
&\leq |2t^{-\frac{1}{2}}\widehat{\nabla}^{k}(t^{-\frac{1}{2}})-\widehat{\nabla}^{k}(t^{-\frac{1}{2}}t^{-\frac{1}{2}})|_{\hat{g}}\\
&\leq \sum_{\substack{p+q\,=\,k \\ p,\,q\,\leq\,k-1}}
|\widehat{\nabla}^{p}(t^{-\frac{1}{2}})|_{\hat{g}}|\widehat{\nabla}^{q}(t^{-\frac{1}{2}})|_{\hat{g}}\\
\end{split}
\end{equation*}
so that
\begin{equation*}
\begin{split}
t^{\frac{1}{2}}\Biggl(-|\widehat{\nabla}^{k}(t^{-1})|_{\hat{g}}+ \sum_{\substack{p+q\,=\,k \\ p,\,q\,\leq\,k-1}}
|\widehat{\nabla}^{p}(t^{-\frac{1}{2}})|_{\hat{g}}|\widehat{\nabla}^{q}(t^{-\frac{1}{2}})|_{\hat{g}}\Biggl)&\leq
2 |\widehat{\nabla}^{k}(t^{-\frac{1}{2}})|_{\hat{g}} \\
& \leq t^{\frac{1}{2}}\Biggl(|\widehat{\nabla}^{k}(t^{-1})|_{\hat{g}}+ \sum_{\substack{p+q\,=\,k \\ p,\,q\,\leq\,k-1}}
|\widehat{\nabla}^{p}(t^{-\frac{1}{2}})|_{\hat{g}}|\widehat{\nabla}^{q}(t^{-\frac{1}{2}})|_{\hat{g}}\Biggl),
\end{split}
\end{equation*}
an induction argument assuming $P(m)$ shows that
$$|\widehat{\nabla}^{1+k}(t^{-\frac{1}{2}})|_{\hat{g}}=O\left(t^{-\frac{5}{2}-\frac{(k-1)}{2}}\right)\qquad\textrm{for all $1\leq k\leq m$}.$$
Using this together with \eqref{verybasic} and the fact that each $df_{k}$ is basic and that $|\widehat{\nabla}^{m+1}\hat{\theta}_{i}|_{\hat{g}}\leq Ct^{-\frac{(m+1)}{2}}$
for $1\leq i\leq 2n-2$, we derive the following estimate:
\begin{equation*}
\begin{split}
|\widehat{\nabla}^{m+1}\zeta|_{\hat{g}}&
\leq C\Biggl(\sum_{\substack{p+q+r\,=\,m \\ 1\,\leq\,k\,\leq \,2n-2}}|\widehat{\nabla}^{p}(t^{-\frac{1}{2}})|_{\hat{g}}
\underbrace{|\widehat{\nabla}^{q}df_{k}|_{\hat{g}}|\widehat{\nabla}^{r}\hat{\theta}_{k}|_{\hat{g}}}_{=\,O\left(t^{-\frac{1}{2}-\frac{(q+r)}{2}}\right)}\\
&\qquad+\sum_{\substack{p+q+r+s+l\,=\,m \\ 1\,\leq\,k\,\leq\, 2n-2}}|\widehat{\nabla}^{p}(t^{-\frac{1}{2}})|_{\hat{g}}
\underbrace{|\widehat{\nabla}^{q}f_{k}|_{\hat{g}}|\widehat{\nabla}^{r}\hat{\theta}_{k}|_{\hat{g}}}_{=\,O\left(t^{-\frac{(q+r)}{2}}\right)}
|\widehat{\nabla}^{s}(t^{-1})|_{\hat{g}}|\widehat{\nabla}^{l}\hat{\theta}_{2n-1}|_{\hat{g}}\\
&\qquad+\sum_{\substack{p+q+r\,=\,m \\ 1\,\leq\,k\,\leq\, 2n-2}}|\widehat{\nabla}^{p}(t^{-\frac{1}{2}})|_{\hat{g}}
\underbrace{|\widehat{\nabla}^{q}f_{k}|_{\hat{g}}|\widehat{\nabla}^{r+1}\hat{\theta}_{k}|_{\hat{g}}}_{=\,O\left(t^{-\frac{(q+r+1)}{2}}\right)}\Biggl)\\
&\leq C\Biggl(\sum_{p+q+r+s+l\,=\,m}t^{-\frac{(q+r)}{2}}|\widehat{\nabla}^{p}(t^{-\frac{1}{2}})|_{\hat{g}}
|\widehat{\nabla}^{s}(t^{-1})|_{\hat{g}}|\widehat{\nabla}^{l}\hat{\theta}_{2n-1}|_{\hat{g}}\\
&\qquad+t^{-\frac{1}{2}}t^{-\frac{(m+1)}{2}}+|\widehat{\nabla}(t^{-\frac{1}{2}})|_{\hat{g}}t^{-\frac{m}{2}}
+\sum_{\substack{p+q+r\,=\,m \\ p\,\geq\,2}}
t^{-\frac{(q+r+1)}{2}}\underbrace{|\widehat{\nabla}^{p}(t^{-\frac{1}{2}})|_{\hat{g}}}_{=\,O\left(t^{-\frac{5}{2}-\frac{(p-2)}{2}}\right)}\Biggl)\\
&\leq C\Biggl(\sum_{p+q+r+s+l\,=\,m}t^{-\frac{(q+r)}{2}}|\widehat{\nabla}^{p}(t^{-\frac{1}{2}})|_{\hat{g}}
|\widehat{\nabla}^{s}(t^{-1})|_{\hat{g}}|\widehat{\nabla}^{l}\hat{\theta}_{2n-1}|_{\hat{g}}+
t^{-1-\frac{m}{2}}+t^{-\frac{3}{2}-\frac{m}{2}}+t^{-2-\frac{m}{2}}\Biggl)\\
&\leq C\Biggl(t^{-1-\frac{m}{2}}+\sum_{p+q+r+s+l\,=\,m}t^{-\frac{(q+r)}{2}}|\widehat{\nabla}^{p}(t^{-\frac{1}{2}})|_{\hat{g}}
|\widehat{\nabla}^{s}(t^{-1})|_{\hat{g}}|\widehat{\nabla}^{l}\hat{\theta}_{2n-1}|_{\hat{g}}\Biggl).\\
\end{split}
\end{equation*}
Now, the sum in this last expression can be estimated as follows:
\begin{equation*}
\begin{split}
&\sum_{p+q+r+s+l\,=\,m}t^{-\frac{(q+r)}{2}}|\widehat{\nabla}^{p}(t^{-\frac{1}{2}})|_{\hat{g}}
|\widehat{\nabla}^{s}(t^{-1})|_{\hat{g}}|\widehat{\nabla}^{l}\hat{\theta}_{2n-1}|_{\hat{g}}\\
&\leq C\Biggl(\sum_{p+q+r+s\,=\,m}t^{-\frac{(q+r)}{2}}|\widehat{\nabla}^{p}(t^{-\frac{1}{2}})|_{\hat{g}}
|\widehat{\nabla}^{s}(t^{-1})|_{\hat{g}}
+\sum_{\substack{p+q+r+s+l\,=\,m \\ l\,\geq\,1}}t^{-\frac{(q+r)}{2}}|\widehat{\nabla}^{p}(t^{-\frac{1}{2}})|_{\hat{g}}
|\widehat{\nabla}^{s}(t^{-1})|_{\hat{g}}\underbrace{|\widehat{\nabla}^{l}\hat{\theta}_{2n-1}|_{\hat{g}}}_{=\,O\left(t^{-1-\frac{(l-1)}{2}}\right)}\Biggl)\\
&\leq C\Biggl(\sum_{p+q+r+s\,=\,m}t^{-\frac{(q+r)}{2}}|\widehat{\nabla}^{p}(t^{-\frac{1}{2}})|_{\hat{g}}
|\widehat{\nabla}^{s}(t^{-1})|_{\hat{g}}
+t^{-\frac{1}{2}}\sum_{\substack{p+q+r+s+l\,=\,m \\ l\,\geq\,1}}t^{-\frac{(l+q+r)}{2}}|\widehat{\nabla}^{p}(t^{-\frac{1}{2}})|_{\hat{g}}
|\widehat{\nabla}^{s}(t^{-1})|_{\hat{g}}\Biggl)\\
&\leq C\Biggl(\sum_{k\,=\,0}^{m}\sum_{p+s\,=\,m-k}t^{-\frac{k}{2}}|\widehat{\nabla}^{p}(t^{-\frac{1}{2}})|_{\hat{g}}
|\widehat{\nabla}^{s}(t^{-1})|_{\hat{g}}
+t^{-\frac{1}{2}}\sum_{k\,=\,1}^{m}\sum_{\substack{p+s\,=\,m-k}}t^{-\frac{k}{2}}|\widehat{\nabla}^{p}(t^{-\frac{1}{2}})|_{\hat{g}}
|\widehat{\nabla}^{s}(t^{-1})|_{\hat{g}}\Biggl)\\
&\leq C\Biggl(\sum_{k\,=\,0}^{m}\sum_{p+s\,=\,m-k}t^{-\frac{k}{2}}|\widehat{\nabla}^{p}(t^{-\frac{1}{2}})|_{\hat{g}}
|\widehat{\nabla}^{s}(t^{-1})|_{\hat{g}}\Biggl)\\
&\leq C\Biggl(\sum_{r\,=\,0}^{m}\sum_{p+s\,=\,r}t^{-\frac{(m-r)}{2}}|\widehat{\nabla}^{p}(t^{-\frac{1}{2}})|_{\hat{g}}
|\widehat{\nabla}^{s}(t^{-1})|_{\hat{g}}\Biggl)\\
&\leq Ct^{-\frac{m}{2}}\Biggl(t^{-\frac{1}{2}}t^{-1}+t^{\frac{1}{2}}t^{-\frac{5}{2}}+t^{-\frac{7}{2}}t+
\sum_{r\,=\,3}^{m}\sum_{p+s\,=\,r}t^{\frac{r}{2}}|\widehat{\nabla}^{p}(t^{-\frac{1}{2}})|_{\hat{g}}
|\widehat{\nabla}^{s}(t^{-1})|_{\hat{g}}\Biggl)\\
&\leq Ct^{-\frac{m}{2}}\Biggl(t^{-\frac{3}{2}}
+\sum_{r\,=\,3}^{m}\sum_{s\,=\,r}\underbrace{t^{\frac{r}{2}}t^{-\frac{1}{2}}
|\widehat{\nabla}^{s}(t^{-1})|_{\hat{g}}}_{=\,O\left(t^{-\frac{5}{2}+\frac{(r-s)}{2}}\right)}
+\sum_{r\,=\,3}^{m}\sum_{1+s\,=\,r}\underbrace{t^{\frac{r}{2}}|\widehat{\nabla}(t^{-\frac{1}{2}})|_{\hat{g}}
|\widehat{\nabla}^{s}(t^{-1})|_{\hat{g}}}_{_{=\,O\left(t^{-\frac{7}{2}+\frac{(r-s)}{2}}\right)}}\\
&\qquad+\sum_{r\,=\,3}^{m}\sum_{\substack{p+s\,=\,r \\ p\,\geq\,2}}t^{\frac{r}{2}}|\widehat{\nabla}^{p}(t^{-\frac{1}{2}})|_{\hat{g}}
|\widehat{\nabla}^{s}(t^{-1})|_{\hat{g}}\Biggl)\\
&\leq Ct^{-\frac{m}{2}}\Biggl(t^{-\frac{3}{2}}+t^{-\frac{5}{2}}+t^{-3}+
\sum_{r\,=\,3}^{m}\sum_{\substack{p\,=\,r \\ p\,\geq\,2}}\underbrace{t^{\frac{r}{2}}|\widehat{\nabla}^{p}(t^{-\frac{1}{2}})|_{\hat{g}}
t^{-1}}_{=\,O\left(t^{-\frac{5}{2}+\frac{(r-p)}{2}}\right)}+\sum_{r\,=\,3}^{m}\sum_{\substack{p+1\,=\,r \\ p\,\geq\,2}}\underbrace{t^{\frac{r}{2}}|\widehat{\nabla}^{p}(t^{-\frac{1}{2}})|_{\hat{g}}
|\widehat{\nabla}(t^{-1})|_{\hat{g}}}_{=\,O\left(t^{-\frac{7}{2}+\frac{(r-p)}{2}}\right)}\\
&\qquad+\sum_{r\,=\,3}^{m}\sum_{\substack{p+s\,=\,r \\ p,\,s\,\geq\,2}}\underbrace{t^{\frac{r}{2}}|\widehat{\nabla}^{p}(t^{-\frac{1}{2}})|_{\hat{g}}
|\widehat{\nabla}^{s}(t^{-1})|_{\hat{g}}}_{=\,O\left(t^{-\frac{7}{2}+\frac{(r-p-s)}{2}}\right)}\Biggl)\\
&\leq Ct^{-\frac{3}{2}-\frac{m}{2}}.
\end{split}
\end{equation*}
Hence, in light of the above, we arrive at the fact that
$$|\widehat{\nabla}^{m+1}\zeta|_{\hat{g}}\leq Ct^{-1-\frac{m}{2}}=Ct^{-\frac{1}{2}-\frac{(m+1)}{2}},$$
as required. This completes the induction step.
\end{proof}

\newpage
\bibliographystyle{amsalpha}

\bibliography{ref2}

\end{document}